%% file: main.tex
\documentclass{memo-l}

\usepackage{amsmath}
\usepackage{amsfonts}
\usepackage{amssymb}
\usepackage{amscd}
\usepackage{amsthm}
\usepackage{framed}
\usepackage{graphicx}
\usepackage{latexsym}
\usepackage[numbers]{natbib}  
\usepackage{multirow}
\usepackage{tikz}
\usetikzlibrary{positioning}
\usetikzlibrary{arrows}

\usepackage[inline]{enumitem}   
\makeatletter
\newcommand{\inlineitem}[1][]{%
\ifnum\enit@type=\tw@
    {\descriptionlabel{#1}}
  \hspace{\labelsep}%
\else
  \ifnum\enit@type=\z@
       \refstepcounter{\@listctr}\fi
    \quad\@itemlabel\hspace{\labelsep}%
\fi}
\makeatother

\newcommand{\CC}{\mathbb C}
\newcommand{\HH}{\mathbb H}

\newcommand{\NN}{\mathbb N}
\newcommand{\cD}{\mathcal D}
\newcommand{\cA}{\mathcal A}

\newcommand{\PP}{\mathbb P}
\newcommand{\QQ}{\mathbb Q}
\newcommand{\RR}{\mathbb R}
\newcommand{\ZZ}{\mathbb Z}

\newcommand{\SL}{\mathop{\mathrm {SL}}\nolimits}
\newcommand{\SO}{\mathop{\mathrm {SO}}\nolimits}
\newcommand{\Sp}{\mathop{\mathrm {Sp}}\nolimits}
\newcommand{\Orth}{\mathop{\null\mathrm {O}}\nolimits}

\newcommand{\im}{\mathop{\mathrm {Im}}\nolimits}
\newcommand{\rank}{\mathop{\mathrm {rk}}\nolimits}
\newcommand{\latt}[1]{{\langle{#1}\rangle}}
\newcommand{\ord}{\mathop{\mathrm {ord}}\nolimits}
\newcommand{\II}{\mathop{\mathrm {II}}\nolimits}
\def\Grit{\mathbf{G}}
\def\Borch{\mathbf{B}}

\def\dim{\operatorname{dim}}

\def\det{\operatorname{det}}
\def\w{\operatorname{w}}
\newcommand{\bQ}{\mathbf{Q}_\mathfrak{g}}
\newcommand{\bP}{\mathbf{P}_\mathfrak{g}}

\newenvironment{psmallmatrix}
  {\left(\begin{smallmatrix}}
{\end{smallmatrix}\right)}

\usepackage{hyperref}
\hypersetup{colorlinks,linkcolor={red},citecolor={blue},urlcolor={blue}}

\newtheorem{theorem}{Theorem}[chapter]
\newtheorem{lemma}[theorem]{Lemma}
\newtheorem{proposition}[theorem]{Proposition}
\newtheorem{corollary}[theorem]{Corollary}
\newtheorem{conjecture}[theorem]{Conjecture}

\theoremstyle{definition}
\newtheorem{definition}[theorem]{Definition}

\newtheorem{Notation}[theorem]{Notation}
\newtheorem{question}[theorem]{Question}
\newtheorem{Argument}[theorem]{Argument}
\newtheorem{remark}[theorem]{Remark}

\numberwithin{section}{chapter}
\numberwithin{equation}{chapter}
\numberwithin{table}{chapter}

\begin{document}

\frontmatter

\title{Hyperbolization of Affine Lie Algebras}


\author{Kaiwen Sun}
\address{School of Mathematical Sciences, University of Science and Technology of China, Hefei 230026, Anhui, China}
\email{kwsun@ustc.edu.cn}

\author{Haowu Wang}
\address{School of Mathematics and Statistics, Wuhan University, Wuhan 430072, Hubei, China}
\email{haowu.wangmath@whu.edu.cn}

\author{Brandon Williams}

\address{Institute for Mathematics, Heidelberg University, 69120 Heidelberg, Germany}

\email{bwilliams@mathi.uni-heidelberg.de}

\date{December 27, 2023}



\dedicatory{Dedicated to Valery Gritsenko on the occasion of his 70th birthday}

\begin{abstract}
In 1983, Feingold and Frenkel discovered a relation between Siegel modular forms of genus two and a rank-three hyperbolic Kac--Moody algebra extending the affine Lie algebra of type $A_1$. It inspires a problem to explore more general relations between affine Lie algebras, hyperbolic Kac--Moody algebras and modular forms. In this paper, we give an automorphic answer to this problem. We classify hyperbolic Borcherds--Kac--Moody superalgebras whose super-denominators define reflective automorphic products of singular weight on lattices of type $2U\oplus L$. As a consequence, we prove that there are exactly $81$ affine Lie algebras $\hat{\mathfrak{g}}$ which have extensions to hyperbolic BKM superalgebras for which the leading Fourier--Jacobi coefficients of super-denominators coincide with the denominators of $\hat{\mathfrak{g}}$. We find that $69$ of them appear in Schellekens' list of semi-simple $V_1$ structures of holomorphic CFT of central charge $24$, while $8$ of them correspond to the $\mathcal{N}=1$ structures of holomorphic SCFT of central charge $12$ composed of $24$ chiral fermions. The last $4$ cases are related to exceptional modular invariants from nontrivial automorphisms of fusion algebras. This clarifies the relationship of affine Lie algebras, vertex algebras and hyperbolic BKM superalgebras at the level of modular forms. 
\end{abstract}

\maketitle

\tableofcontents

\mainmatter
\include{chap1}

\include{chap2}

\include{chap3}

\include{chap4}

\include{chap5}

\include{chap6}

\include{chap7}

\include{chap8}

\include{chap9}

\include{chap10}

\include{chap11}

\include{chap12}

\include{chap13}

\include{chap14}

\include{tables}


\backmatter
\bibliographystyle{amsplain}
\bibliofont
\bibliography{refs}

\end{document}

%% file: chap1.tex
\chapter{Introduction}

Affine Lie algebras are the simplest class of infinite dimensional Kac--Moody Lie algebras, and they have numerous connections with other areas of mathematics and theoretical physics. The next simplest class of Kac--Moody algebras after the affine Lie algebras are the hyperbolic Lie algebras. In 1983, as an extension of the affine Lie algebra of type $A_1$, Feingold and Frenkel \cite{FF83} investigated the hyperbolic Kac--Moody algebra with Cartan matrix 
$$
\left( \begin{array}{ccc}
2 & -2 & 0 \\
-2 & 2 & -1 \\
0 & -1 & 2
\end{array} \right),
$$
and found that its characters are related to Siegel modular forms of genus 2 and even weight \cite{Igu62}. This motivates us to consider the modularity of the denominator function as Siegel modular forms, and suggests exploring more general relations between affine Lie algebras, hyperbolic Kac--Moody algebras and modular forms of several variables. 

In 1988, Borcherds \cite{Bor88} introduced generalized Kac--Moody algebras, now often called Borcherds--Kac--Moody or simply BKM algebras. These infinite-dimensional Lie algebras are also defined in terms of Chevalley--Serre generators and relations that are encoded in a generalized Cartan matrix, and they differ from Kac--Moody algebras mainly by allowing the diagonal entries of the Cartan matrix to be non-positive. In other words, simple roots are allowed to be imaginary, whereas simple roots in a Kac--Moody algebra are always real.
The best known example of a BKM algebra is the monster Lie algebra. In 1992, Borcherds \cite{Bor92} constructed this algebra as the BRST cohomology related to the monster vertex algebra \cite{Bor86, FLM88} by means of the no-ghost theorem from string theory. By considering the action of the monster group on the denominator identity of the monster Lie algebra, Borcherds proved the celebrated monstrous moonshine conjecture. Furthermore, he observed that the denominator functions of some BKM algebras are modular forms on orthogonal groups of signature $(l,2)$. In 1995 and 1998 Borcherds \cite{Bor95, Bor98} developed the theory of singular theta lift to construct modular forms for arithmetic subgroups of $\Orth(l,2)$ that have infinite product expansions. These are called automorphic products, or Borcherds products, and they are natural candidates for the denominator functions of BKM algebras. Similarly to affine Lie algebras, BKM algebras and automorphic products also have many applications in physics. For example, Harvey and Moore \cite{HV96, HV98} proposed that BKM algebras should play as the underlying organizing structure of BPS states in string compactifications; in particular, the denominators of BKM algebras might be generating functions for BPS states.

In 1996, Gritsenko and Nikulin \cite{GN96a} constructed an automorphic correction of the rank-three hyperbolic Lie algebra considered previously by Feingold and Frenkel. More precisely, they extended this hyperbolic Lie algebra to a (hyperbolic) BKM algebra by adding infinitely many imaginary simple roots, so that the denominator of the resulting BKM algebra is exactly the Igusa cusp form of weight $35$ on $\Sp_2(\ZZ)$ \cite[Theorem 3]{Igu64}. Later, they constructed automorphic corrections for other hyperbolic Lie algebras in a series of papers \cite{GN96b, GN98a, GN98, GN02, GN18}. A common feature of these corrections is that the resulting BKM algbera only has finitely many real simple roots, a Weyl chamber of finite volume and a Weyl vector of negative norm.  Moreover, its denominator usually defines a cuspidal automorphic product on $\Orth(l,2)$. These corrections extend affine Lie algebras in a nice way, but they have several features that are not preferred from our perspective: 
\begin{itemize}
\item[(1)] for the automorphic correction $\mathfrak{G}$ of an affine Lie algebra $\hat{\mathfrak{g}}$, the multiplicity of an imaginary root of $\hat{\mathfrak{g}}$ in $\mathfrak{G}$ is strictly greater than its multiplicity in $\hat{\mathfrak{g}}$;
\item[(2)] the set of imaginary simple roots is very complicated, although the set of real simple roots is easy to describe;
\item[(3)] it is not clear what the vertex algebra does, nor how to construct $ \mathfrak{G}$ naturally beyond simply listing generators and relations.
\end{itemize}

In this paper, we extend affine Lie algebras to (hyperbolic) BKM algebras in a different way. Certain affine Kac--Moody algebras $\hat{\mathfrak{g}}$ will be extended to BKM algebras $\mathcal{G}_\mathfrak{g}$ that have infinitely many real simple roots and that satisfy:
\begin{itemize}
\item[(a)]  for any root $\alpha$ of $\hat{\mathfrak{g}}$, the root multiplicities of $\alpha$ in $\hat{\mathfrak{g}}$ and $\mathcal{G}_\mathfrak{g}$ are the same, which ensures that $\mathcal{G}_\mathfrak{g}$ can be viewed as a graded module over $\hat{\mathfrak{g}}$ in some sense;
\item[(b)] the imaginary simple roots of $\mathcal{G}_\mathfrak{g}$ are negative integral multiples of the Weyl vector; 
\item[(c)] the Lie algebras $\hat{\mathfrak{g}}$ and $\mathcal{G}_\mathfrak{g}$ are closely related to some exceptional vertex algebras, and in many cases $\mathcal{G}_\mathfrak{g}$ have natural constructions with known symmetry groups inherited from  vertex algebras, like the monster Lie algebra. 
\end{itemize}
Our main results are about the classification and construction of such extensions, which are connected with various types of modular forms. We will show that there are exactly $81$ affine Lie algebras $\hat{\mathfrak{g}}$ that extend to hyperbolic BKM algebras in such a nice way. These extensions are related to three special types of vertex algebra, and they are called hyperbolizations of affine Lie algebras.  In the remainder of the introduction, we will explain the setup and state the main theorems.

\section{Denominators of affine Lie algebras and Jacobi forms}
Let $\mathfrak{g}$ be a finite-dimensional simple Lie algebra of rank $r$ and let $\Delta^+_\mathfrak{g}$ be a set of positive roots. The product side of the denominator identity of the associated affine Lie algebra $\hat{\mathfrak{g}}$ is the holomorphic function
$$
\vartheta_{\mathfrak{g}}(\tau,\mathfrak{z})=\eta(\tau)^r \prod_{\alpha\in \Delta_\mathfrak{g}^+}\frac{\vartheta(\tau, \latt{\alpha,\mathfrak{z}})}{\eta(\tau)},
$$
where $\eta$ and $\vartheta$ are the Dedekind eta function and the odd Jacobi theta function, respectively:
\begin{align*}
\eta(\tau)&=q^{\frac{1}{24}}\prod_{n=1}^\infty(1-q^n), \quad \tau\in \HH, \; q=e^{2\pi i\tau},\\
\vartheta(\tau,z)&=-q^{\frac{1}{8}}\zeta^{-\frac{1}{2}}\prod_{n=1}^\infty(1-q^{n-1}\zeta)(1-q^n\zeta^{-1})(1-q^n), \quad z\in \CC, \; \zeta= e^{2\pi iz}.
\end{align*}
The function $\vartheta_{\mathfrak{g}}$ is an example of lattice-index Jacobi forms (see, e.g., \cite{CG13, GSZ19}). Such Jacobi forms are generalizations of classical Jacobi forms introduced by Eichler and Zagier \cite{EZ85}.  Let $L$ be an even positive definite lattice. A Jacobi form of integral weight $k$ and index $L$ is a holomorphic function $\varphi : \HH \times (L \otimes \CC) \rightarrow \CC$ that is modular under $\SL_2(\ZZ)$ and doubly quasi-periodic, namely
\begin{align*}
\varphi \left( \frac{a\tau +b}{c\tau + d},\frac{\mathfrak{z}}{c\tau + d} 
\right)& = (c\tau + d)^k 
\exp{\left(i \pi t \frac{c(\mathfrak{z},\mathfrak{z})}{c 
\tau + d}\right)} \varphi ( \tau, \mathfrak{z} ), \quad A=\begin{psmallmatrix}
a & b \\
c & d
\end{psmallmatrix} \in \SL_2(\ZZ),\\
\varphi (\tau, \mathfrak{z}+ x \tau + y)&= 
\exp{\big(-i \pi t \big( (x,x)\tau +2(x,\mathfrak{z})\big)\big)} 
\varphi (\tau, \mathfrak{z} ), \quad x,y\in L,
\end{align*}
and whose Fourier expansion satisfies a boundary condition. The function $\vartheta_{\mathfrak{g}}$ is a Jacobi form of weight $r/2$ and index $P_\mathfrak{g}^\vee(h_\mathfrak{g}^\vee)$ for some character, where $P_\mathfrak{g}^\vee$ is the dual of the root lattice and $h_\mathfrak{g}^\vee$ is the dual Coxeter number. Note that Jacobi forms defined by an expression similar to $\vartheta_{\mathfrak{g}}$ are called theta blocks following Gritsenko--Skoruppa--Zagier \cite{GSZ19}.

\section{Automorphic products of singular weight}
A modular form of integral weight $k$ and trivial character for an arithmetic subgroup $\Gamma<\Orth(l,2)$ is a holomorphic function on the associated type IV symmetric domain which satisfies 
\begin{align*}
F(t\mathcal{Z})&=t^{-k} F(\mathcal{Z}), \quad t\in \CC^\times,\\
F(g\mathcal{Z})&=F(\mathcal{Z}),   \qquad g \in \Gamma. 
\end{align*}
Let $M$ be an even lattice of signature $(l,2)$.  The input into Borcherds' theta lift is a vector-valued modular form of weight $1-l/2$ with integral Fourier expansion for the Weil representation of $\SL_2(\ZZ)$ attached to the discriminant form $M'/M$, and the output is a meromorphic modular form for a certain subgroup of $\Orth(M)$ which has an infinite product expansion at any $0$-dimensional cusp and whose divisors are linear combinations of hyperplanes. 

Since the denominators of affine Lie algebras satisfy modularity, it is natural to focus on hyperbolic BKM algebras whose denominators are modular. Let $\mathcal{G}$ be a BKM algebra whose denominator function coincides with the Fourier expansion of an $\Orth(l,2)$-modular form $F$ at some $0$-dimensional cusp.  Since $F$ has an infinite product expansion, by Bruinier's converse theorem \cite{Bru02, Bru14} one expects that it can be constructed by the Borcherds lift. In this case, the roots and their multiplicities are encoded in the Fourier expansion of the input. When $F$ has singular weight, that is, weight $l/2-1$, the Fourier expansion is supported only on isotropic vectors, which often forces the imaginary simple roots of $\mathcal{G}$ to be negative integral multiples of the Weyl vector. Moreover, it is conjectured in this particular case that $\mathcal{G}$ can be constructed as the BRST cohomology related to some vertex algebra, similarly to the monster Lie algebra. This suggests focusing on BKM algebras whose denominators are automorphic products of singular weight. The zeros of $F$ that contain the cusp are actually hyperplanes orthogonal to real roots of $\mathcal{G}$, hence $F$ is anti-invariant under the reflections through these hyperplanes. It is natural to expect that $F$ is anti-invariant under all reflections associated with zeros of $F$. This has been proven by the last two named authors \cite{WW23}. It follows that $F$ is a reflective modular form. 

A non-constant modular form on $\Gamma < \Orth(M)$ is called reflective if it vanishes only on mirrors of reflections fixing the lattice $M$. Reflective modular forms were introduced in 1998 by Borcherds \cite{Bor95, Bor98} and Gritsenko--Nikulin \cite{GN98}, and their classification has been an active project for the past thirty years (see \cite{GN02, Bar03, Sch06, DHS15, Sch17, Ma17, Ma18, Dit19, OS19, Wan21, Wan21b, Wan22, DS22, Wan23a, Wan23b}), because they have nice applications to hyperbolic reflection groups \cite{Bor98, Bor00, GN98a, GHS07, GH14}, algebraic geometry \cite{Bor96, BKP98, GHS07, Ma18, Gri18} and free algebras of modular forms \cite{Wan21a, WW20c} in addition to infinite-dimensional Lie algebras. 

\section{Main results} 
BKM algebras whose denominator functions are reflective automorphic products of singular weight are exceptional. The main examples are the fake monster algebra \cite{Bor90} and their twists by the Conway group $\mathrm{Co}_0$ \cite{Bor92, Sch04, WW22}. There are conjecturally only finitely many such algebras and constructing and classifying them is an open problem.  Many partial results have been proved towards such a classification \cite{Bor90, Bor92, Nie02, Bar03, Sch00, Sch04, Sch06, Sch17, Dit19, WW22, DS22}. In this paper we contribute some new results in this direction. 

Let $U$ be an even unimodular lattice of signature $(1,1)$ and let $L$ be an even positive-definite lattice. The input of the Borcherds lift on $2U\oplus L$ can be identified with Jacobi forms of weight $0$ and index $L$. 
We will identify affine Lie algebras, which naturally extends to BKM algebras or superalgebras whose denominators or super-denominators are reflective Borcherds products of singular weight on lattices of type $2U\oplus L$. The setting is inspired by the following result. 

\begin{theorem}\label{MTH1}
Let $F$ be a reflective Borcherds product of singular weight on $2U\oplus L$ whose Jacobi form input has Fourier expansion
$$
\phi(\tau,\mathfrak{z}) = \sum_{n\in \ZZ} \sum_{\ell \in L'} f(n,\ell) q^n \zeta^\ell
$$
satisfying that $f(0,\ell)\geq 0$ for all $\ell\in L'$. 
If $L$ is the Leech lattice, then $F$ is the denominator of the fake monster algebra and $\phi$ is the full character of the Leech lattice vertex operator algebra. Otherwise, the set 
$$
\mathcal{R}=\{ \ell \in L' : \; \ell\neq 0, \; f(0,\ell) \neq 0 \}
$$
determines a finite-dimensional semi-simple Lie algebra 
$$
\mathfrak{g}=\bigoplus_{j=1}^s\mathfrak{g}_{j,k_j}
$$
with the same rank as $L$ such that the identity
\begin{equation}\label{MEQ}
C:= \frac{\dim \mathfrak{g}}{24} - a = \frac{h_j^\vee}{k_j}
\end{equation}
holds for any $1\leq j\leq s$ and such that the leading Fourier--Jacobi coefficient of $F$ at the $1$-dimensional cusp determined by $2U$ coincides with the denominator of the associated affine Lie algebra $\hat{\mathfrak{g}}$. For any $1\leq j\leq s$, $\mathfrak{g}_j$ is a simple ideal of $\mathfrak{g}$, $k_j$ is a positive integer indicating the level of $\mathfrak{g}_j$, and $h_j^\vee$ is the dual Coxeter number of $\mathfrak{g}_j$. The number $a$ equals $f(-1,0)$, which has to be $0$ or $1$. When $a=0$, $k_j>1$ for any $1\leq j\leq s$. The cases $a=0$ and $a=1$ are called symmetric and anti-symmetric, respectively. 
\end{theorem}

If the Fourier expansion of $F$ defines the (super)-denominator of a BKM (super)-algebra $\mathcal{G}$, then the (super)-denominator has the form
$$
e^\rho \prod_{\alpha>0} \big( 1 - e^{-\alpha} \big)^{f(nm,\ell)},
$$
where $\rho$ is the Weyl vector of $F$, and where $(n,\ell,m)$ are coordinates of positive roots $\alpha \in U\oplus L'$ with $n\in\ZZ$, $m\in \NN$, $\ell\in L'$ and $\alpha^2=\ell^2-2nm$. The above $\hat{\mathfrak{g}}$ is embedded into $\mathcal{G}$ as the sum of the root spaces associated with roots of type $\pm (n,\ell, 0)$. In this way, $\mathcal{G}$ can be regarded as a graded module over $\hat{\mathfrak{g}}$ graded by $m\in \NN$. This leads us to define $\mathcal{G}$ as a \textit{hyperbolization} of $\hat{\mathfrak{g}}$ and $F$ as a \textit{hyperbolization} of the denominator of $\hat{\mathfrak{g}}$.  It turns out that there are only $81$ affine Lie algebras with a hyperbolization:

\begin{theorem}\label{MTH2}
There are $81$ possibilities for the semi-simple Lie algebra $\mathfrak{g}$ in Theorem \ref{MTH1} and they fall into three categories:
\begin{enumerate}
\item $69$ make up Schellekens' list of semi-simple $V_1$ structures of holomorphic vertex operator algebras of central charge $24$;
\item $8$ correspond to the $\mathcal{N}=1$ structures of holomorphic vertex operator superalgebras $F_{24}$ of central charge $12$ composed of $24$ fermions;
\item The remaining $4$ cases $A_{1,16}$, $A_{1,8}^2$, $A_{1,4}^4$ and $A_{2,9}$ possess an exceptional modular invariant that comes from a nontrivial automorphism of the fusion algebra. 
\end{enumerate}
\end{theorem}   

Case (1) is anti-symmetric, while Cases (2) and (3) are symmetric. Schellekens' list \cite{Sch93} was established using the solutions of Equation \eqref{MEQ} with $a=1$. Holomorphic vertex operator superalgebras of central charge $12$ were classified by Creutzig, Duncan
and Riedler \cite{CDR18}, and the $\mathcal{N}=1$ structures of $F_{24}$ were determined in \cite{F24}, corresponding to solutions of Equation \eqref{MEQ} with $a=0$ and $C=1$. The exceptional modular invariants mentioned in (3) were discovered around the 1990s by Moore and Seiberg \cite{Moore:1988ss}, Verstegen \cite{Verstegen:1990my} and Gannon \cite{Gannon:1994sp}. The $4$ exotic cases satisfy Equation \eqref{MEQ} with $a=0$ and $C<1$. 

We now present hyperbolizations of these affine Lie algebras. 

\begin{theorem}\label{MTH3}
For any $\mathfrak{g}$ in Theorem \ref{MTH2} there exists a singular-weight reflective Borcherds product $\Psi_\mathfrak{g}$ on some lattice $2U\oplus L_\mathfrak{g}$ whose leading Fourier--Jacobi coefficient equals the denominator of $\hat{\mathfrak{g}}$. Moreover, the Jacobi form input $\phi_\mathfrak{g}$ can be expressed as a $\ZZ$-linear combination of full characters of the affine vertex operator algebra generated by $\hat{\mathfrak{g}}$. 
\end{theorem}

The construction will be briefly summarized here. If $\mathfrak{g}$ is in Schellekens' list, then $L_\mathfrak{g}$ is the orbit lattice in H\"{o}hn's construction \cite{Hoh17} of the holomorphic VOA of central charge $24$ with $V_1=\mathfrak{g}$, and $\phi_\mathfrak{g}$ is the full character of the VOA. If $\mathfrak{g}$ is of symmetric type, $L_\mathfrak{g}$ is the maximal even sublattice of the coweight lattice of $\mathfrak{g}$. If $\mathfrak{g}$ determines an $\mathcal{N}=1$ structure of $F_{24}$, then the Jacobi form input can be expressed in terms of characters of $F_{24}$ as
$$
\phi_{\mathfrak{g}}=(\chi_{\mathrm{NS}} - \chi_{\widetilde{\mathrm{NS}}} - \chi_{\mathrm{R}})/2. 
$$
Finally, we will explain the relation between Jacobi form inputs and exceptional modular invariants for the remaining four $\mathfrak{g}$. The Jacobi form input for $\mathfrak{g}=A_{1,16}$ can be written in terms of affine characters as  
$$
\phi_{A_{1,16}}=\chi^{A_{1,16}}_{2,\frac19}+\chi^{A_{1,16}}_{14,\frac{28}{9}}-\chi^{A_{1,16}}_{8,\frac{10}{9}},
$$
and we find that the difference between the simple current modular invariant and the exceptional modular invariant \cite{Moore:1988ss} is given by $|\phi_{A_{1,16}}|^2$. Similar relations hold for the other three $\mathfrak{g}$. Note that these $L_\mathfrak{g}$ are chosen so that the resulting BKM superalgebra has root lattice $U\oplus L_\mathfrak{g}'$. 

Clearly, the Borcherds products $\Psi_\mathfrak{g}$ in Theorem \ref{MTH3} are closely related to vertex algebras. We therefore expect that the BRST cohomology related to these vertex algebras defines the BKM (super)-algebras with $\Psi_\mathfrak{g}$ as the (super)-denominators. This type of realization has been achieved in \cite{Bor90, HS03, CKS07, HS14, Mol21, DS22, F24} under some technical assumptions for $\mathfrak{g}$ from Schellekens' list and the $\mathcal{N}=1$ structure of $F_{24}$.  However, such a realization is completely open for $\mathfrak{g}$ related to the four exceptional modular invariants.

Affine Lie algebras, vertex algebras and BKM (super)-algebras are therefore closely connected from the point of view of the attached modular forms. The connections are illustrated in Figure \ref{fig:hyperbolization}.

\begin{figure}[h]
\centering
\begin{tikzpicture}[node distance=2.2cm, auto]
\node (A) {Affine Lie algebra};
\node (AA) [right of=A] {$ $};
\node (B) [right of=AA] {Vertex algebra};
\node (AAA) [right of=B] {$ $};
\node (C) [right of=AAA] {BKM superalgebra};
\node (D) [below of=A] {Theta block};
\node (E) [below of=B] {Weight $0$ Jacobi form};
\node (F) [below of=C] {Borcherds product};
\draw[->] (A) to node {} (B);
\draw[->] (B) to node {\small{BRST}} (C);
\draw[->] (D) to node {} (E);
\draw[->] (E) to node {Lift} (F);
\draw[->] (A) to node {\small{Denominator}} (D);
\draw[->] (B) to node {\small{Character}} (E);
\draw[->] (C) to node {\small{Denominator}} (F);
\draw[->, bend left] (F) to node {\small{Leading FJ coefficient}} (D);
\draw[->, bend left] (A) to node {{\color{purple}Hyperbolization}} (C);
\draw[blue, ->] (A) to node {\small{Character}} (E);
\end{tikzpicture}
\caption{Hyperbolization of affine Lie algebras}
\label{fig:hyperbolization}
\end{figure}
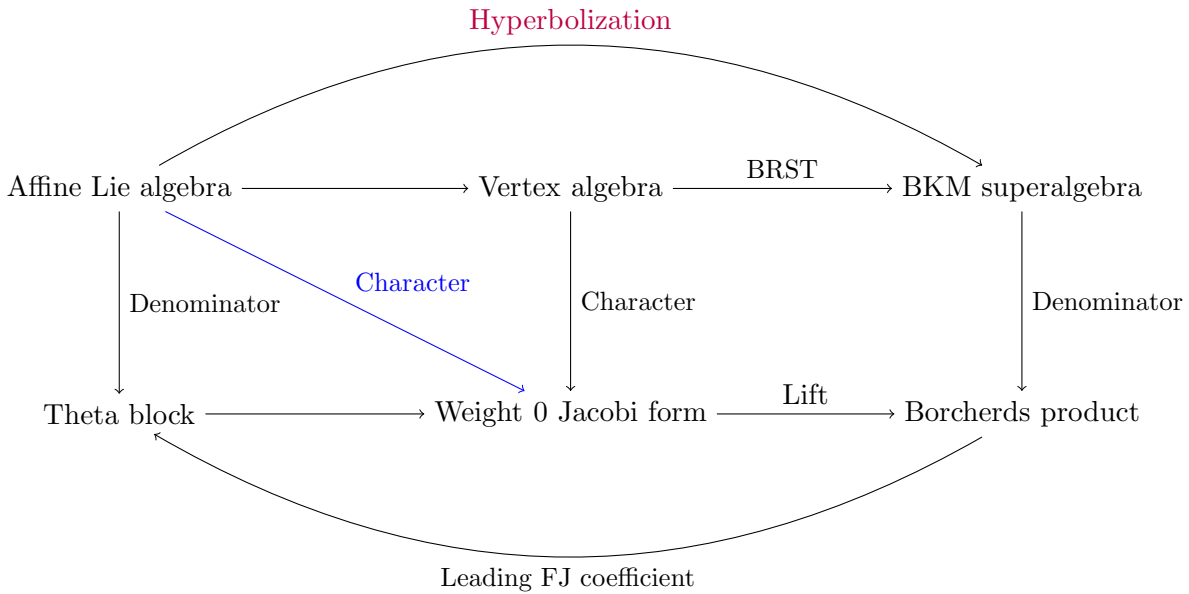

There are some remarks related to the main theorems above. The monster Lie algebra can be regarded as a hyperbolization of the trivial Lie algebra, and the associated vertex algebra is the monster vertex operator algebra. This corresponds to the degenerate case $L=0$ of Theorem \ref{MTH1}. The fake monster Lie algebra is a hyperbolization of the abelian Lie algebra of dimension $24$, and the associated vertex algebra is the Leech lattice vertex operator algebra.  This corresponds to the special case of Theorem \ref{MTH1} where $L$ is equal to the Leech lattice. These two cases are in some sense complementary to Theorem \ref{MTH2}. 
If the $q^0$-term of $\phi$ in Theorem \ref{MTH1} has negative Fourier coefficients, then $\mathcal{R}$ determines a semi-simple Lie superalgebra and the corresponding BKM superalgebra will have odd real roots. We will consider hyperbolizations of affine Lie superalgebras in a separate paper. In addition, it is worthwhile to extend these theorems to reflective Borcherds products of singular weight on lattices of type $U(N)\oplus U\oplus L$. This will cover some interesting BKM superalgebras, including the fake monster Lie superalgebra (see Remark \ref{rem:fE8} and Remark \ref{rem:Conway-SCFT}).

\section{Outline of the proof}
The proof of Theorem \ref{MTH1} relies on some previous results. We know from \cite{WW23} that singular-weight reflective Borcherds products have only simple zeros and are anti-invariant under reflections associated with their zeros. Therefore, the nonzero coefficients $f(0,\ell)$ of the Jacobi form input have to be $1$ if $\ell \neq 0$, and $f(0,0)$ equals the rank of $\mathfrak{g}$ because the Borcherds product has singular weight. The theorem follows by extending an argument used by the second named author \cite{Wan19, Wan21a, Wan23a} to classify reflective modular forms. Theorem \ref{MTH1} further shows that the central charge of the affine vertex operator algebra generated by $\hat{\mathfrak{g}}$ is 
$$
c_\mathfrak{g}= \frac{24(C+a)}{C+1}.
$$
In particular, $c_\mathfrak{g}=24$ if $a=1$, and $c_\mathfrak{g}=12$ if $a=0$ and $C=1$. This motivates the three groupings in Theorem \ref{MTH2}. 

To prove Theorem \ref{MTH2} we first solve equations of type \eqref{MEQ}. 
In the anti-symmetric case, Equation \eqref{MEQ} was first derived by Schellekens \cite{Sch93} in the context of conformal field theories. Schellekens found $221$ solutions to this equation and eliminated $152$ of them to arrive at his list. We also have the same extra solutions to eliminate, but we have to use a completely different approach. In the symmetric case, there are $17$ solutions to Equation \eqref{MEQ} and we have to rule out $5$ of them. In the setting of Theorem \ref{MTH1}, we can prove that $L(C)$ is an integral lattice and that $L$ is bounded by
$$
\bQ < L < \bP,
$$
where $\bQ$ and $\bP$ are the coroot lattice and coweight lattice of $\mathfrak{g}$, respectively. On one hand, for every extraneous $\mathfrak{g}$, we will be able to find an even overlattice $K$ of $L$ for which there is no reflective Borcherds product on $2U\oplus K$ satisfying certain constraints. On the other hand, we prove that if $2U\oplus L$ has a singular-weight reflective Borcherds product then $2U\oplus K$ also has a reflective Borcherds product of the same type. Taken together, this allows us to rule out all $157$ extraneous solutions of Equation \eqref{MEQ}. 

We will now sketch the proof of Theorem \ref{MTH3}, beginning with the anti-symmetric case. Let $V$ be a holomorphic VOA of central charge $24$ with semi-simple $V_1=\mathfrak{g}$. The full character $\chi_V$ of $V$ is known to be a Jacobi form of weight $0$ and lattice index $\bQ$ with non-negative integral Fourier expansion \cite{Zhu96, Miy00, KM12}. This immediately implies that the singular theta lift of $\chi_V$, denoted $\Borch(\chi_V)$, defines a holomorphic Borcherds product of singular weight on $2U\oplus \bQ$. It remains to find an extension $L$ of $\bQ$ for which $\Borch(\chi_V)$ is reflective on $2U\oplus L$. We recognize that $L$ should be the H\"{o}hn' orbit lattice $L_\mathfrak{g}$. H\"{o}hn \cite{Hoh17, Lam19, BLS22} proposed a construction of $V$ as the simple current extension of the tensor product of the lattice VOA $V_{L_\mathfrak{g}}$ and a certain VOA $W_{\mathfrak{g}}$ of central charge $24-\rank(\mathfrak{g})$ with trivial weight-one subspace. This construction corresponds to the theta decomposition of $\chi_V$ as a Jacobi form of index $L_\mathfrak{g}$. It is possible to prove directly that $\Borch(\chi_V)$ is reflective on $2U\oplus L_\mathfrak{g}$ by computing the Fourier expansion of $\chi_V$; however, this is only feasible for certain specific $\mathfrak{g}$. We carry out the calculation for $\mathfrak{g}=B_{12,2}$, $A_{2,2}F_{4,6}$ and $C_{4,10}$, in which case $[L_\mathfrak{g}:\bQ]\leq 2$. 

To complete the proof for the remaining $\mathfrak{g}$, we relate $\Borch(\chi_V)$ to a twisted denominator of the fake monster algebra. By \cite{Hoh17, Lam19}, there exists a conjugacy class $[g]$ of the Conway group $\mathrm{Co}_0$ such that $W_\mathfrak{g}$ is isomorphic to the orbifold  $V_{\Lambda_g}^{\hat{g}}$, where $\Lambda_g$ is the coinvariant sublattice of the Leech lattice $\Lambda$. Moreover, $\Borch(\chi_V)$ for distinct $V$ define the same modular form on type IV symmetric domain if $V_1=\mathfrak{g}$ correspond to the same $[g]$. Combining this fact with our previous calculation of $\chi_V$, we prove that $\Borch(\chi_V)$ is reflective if $V_1$ corresponds to a class $[g]$ whose order is distinct from its level (see the last sentence of Section \ref{subsec:BKM}).  The last two named authors \cite{WW22} proved that the $g$-twisted denominator of the fake monster algebra defines a reflective Borcherds product $\Phi_g$ of singular weight on $U(n_g)\oplus U\oplus \Lambda^g$ if $g$ has the same level and order $n_g$, where $\Lambda^g$ is the sublattice of $\Lambda$ fixed by $g$. We identify $\Borch(\chi_V)$ with the associated $\Phi_g$ and thus prove that $\Borch(\chi_V)$ is reflective if $V_1$ corresponds to some $[g]$ with the equal level and order. As a by-product, we find that the BKM algebra constructed as the BRST cohomology related to $V$ and the $g$-twist of the fake monster algebra are isomorphic if $g$ has equal order and level, but they are not isomorphic if the order and level of $g$ are distinct. 

Then we consider the symmetric case with $C=1$. For any $\mathcal{N}=1$ structure of $F_{24}$, we show that there exists a $\mathrm{Co}_0$-conjugacy class $[g]$ with level equal to its order such that the Borcherds product $\Borch(\phi_\mathfrak{g})$ is the $g$-twisted denominator of the fake monster algebra. This is proven by identifying their Jacobi form inputs and relies on the construction of these $g$-twisted denominators as Gritsenko (additive) lifts due to Dittmann and the second named author \cite{DW21}. We then confirm that $\Borch(\phi_\mathfrak{g})$ is a singular-weight reflective Borcherds product and that the BKM superalgebra constructed in \cite{F24} as the BRST cohomology related to $F_{24}$ is isomorphic to the $g$-twist of the fake monster algebra. 

In the final case, we determine the Jacobi form inputs in terms of affine characters and prove that their singular theta lifts are reflective by direct calculation.

\section{Outline of the paper}

In Chapter \ref{sec:preliminaries}, we quickly introduce Jacobi forms, affine Kac--Moody algebras, reflective modular forms, automorphic products, and Borcherds--Kac--Moody algebras. We also fix some notation that will be used frequently later on.  

In Chapter \ref{sec:hyperbolization} we define the hyperbolization of an affine Lie algebra and explain its motivation. 

In Chapter \ref{sec:root-system}, we present the proof of Theorem \ref{MTH1} and the solutions of Equation \eqref{MEQ}.  

In Chapter \ref{sec:main-results} we state the full version of the main results.

In Chapter \ref{sec:construct-anti}, we first review holomorphic VOA of central charge $24$ and H\"{o}hn's construction, as well as their relations with the twisted denominators of the fake monster algebra. We then construct hyperbolizations of the affine Lie algebras on Schellekens' list. 

In Chapter \ref{sec:construct-symmetric-C=1}, we begin by introducing holomorphic SVOA of central charge $12$ and type $F_{24}$, and then construct hyperbolizations associated with the eight $\mathcal{N}=1$ structures of $F_{24}$.  As an application, we construct many exceptional modular invariants by considering the conformal embedding from the affine VOA generated by the $\mathcal{N}=1$ structure to $F_{24}$ and considering an automorphism of the $D_{12,1}$ fusion algebra.

In Chapter \ref{sec:construct-symmetric-C<1}, we construct hyperbolizations of the remaining affine Lie algebras and explain their connection with some exceptional modular invariants.  

Combining Chapters \ref{sec:construct-anti}-\ref{sec:construct-symmetric-C<1} completes the proof of Theorem \ref{MTH3}.

In Chapter \ref{sec:additive lifts}, we present a uniform construction of the $12$ symmetric singular-weight reflective Borcherds products as Gritsenko (additive) lifts. 

In Chapter \ref{sec:Fourier-expansion}, we compute the Fourier expansions of these reflective Borcherds products of singular weight at the $0$-dimensional cusp determined by one copy of $U$. 

Chapter \ref{sec:exclude-anti-root} is devoted to the proof of the anti-symmetric case of Theorem \ref{MTH2}.

Chapter \ref{sec:exclude-root} contains the proof of the symmetric case of Theorem \ref{MTH2}.

In Chapter \ref{sec:paramodular} we give an application of our main results. We use the pull-back to construct a new infinite series of anti-symmetric Siegel paramodular forms of weight $3$. 

In Chapter \ref{sec:problems}, we raise some related questions and conjectures.

This article ends with several long tables. 

\vspace{5mm}
\noindent
\textbf{Acknowledgements} 
KS thanks Thomas Creutzig, Zhihao Duan, Kimyeong Lee, and Hiraku Nakajima for useful discussions. HW thanks Valery Gritsenko, Xingjun Lin, and Nils Scheithauer for fruitful discussions. KS and HW are grateful for the hospitality of the IBS Center for Geometry and Physics, where part of this work was done. KS would also like to thank KIAS, MPIM Bonn, Uppsala University and Isaac Newton Institute for Mathematical Sciences in Cambridge (during the program \textit{Black holes: bridges between number theory and holographic quantum information}, by EPSRC grant no EP/R014604/1) for support and hospitality where work on this paper was undertaken. 
The authors sincerely thank the referee for careful reading and valuable comments.


%% file: chap2.tex
\chapter{Preliminaries}\label{sec:preliminaries}
In this chapter, we will first review Jacobi forms and denominator identities of affine Lie algebras. We then define reflective modular forms on orthogonal groups $\Orth(n,2)$ and review Borcherds' theory of automorphic products, as well as the denominator identities of Borcherds--Kac--Moody algebras. This background is necessary to state and prove the main theorems. 

\section{Jacobi forms of lattice index}
Let $\ZZ$ and $\NN$ denote the sets of integers and non-negative integers, respectively. Let $L$ be an even integral positive-definite lattice with bilinear form $(-,-)$ and dual lattice 
$$
L'=\{ v\in L\otimes \QQ : \; (x,y)\in \ZZ, \; \text{for all $y\in L$}  \}.
$$
For any nonzero integer $a$, the lattice with abelian group $L$ and bilinear form $a(-,-)$ is denoted by $L(a)$. Let $\HH$ be the complex upper half plane. The Dedekind eta function 
$$
\eta(\tau)=q^{1/24}\prod_{j=1}^\infty(1-q^j), \quad q=e^{2\pi i\tau}, \; \tau\in \HH
$$
is a modular form of weight $\frac{1}{2}$ on $\SL_2(\ZZ)$ with a multiplier system of order $24$. We denote this multiplier system by $\upsilon_{\eta}$. 
\begin{definition}
Let $D$ be a positive integer, $k\in\frac{1}{2}\ZZ$ and $t\in \frac{1}{2}\NN$. A holomorphic function $\varphi : \HH \times (L \otimes \CC) \rightarrow \CC$ is 
called a {\it weakly holomorphic} Jacobi form of weight $k$, index $t$ and character (or multiplier system) $\upsilon_\eta^D$ for $L$
if it satisfies 
\begin{align*}
\varphi \left( \frac{a\tau +b}{c\tau + d},\frac{\mathfrak{z}}{c\tau + d} 
\right)& = \upsilon_\eta^D(A) (c\tau + d)^k 
\exp{\left(i \pi t \frac{c(\mathfrak{z},\mathfrak{z})}{c 
\tau + d}\right)} \varphi ( \tau, \mathfrak{z} ),\\
\varphi (\tau, \mathfrak{z}+ x \tau + y)&= 
(-1)^{t\big((x,x)+(y,y)\big)}\exp{\big(-i \pi t \big( (x,x)\tau +2(x,\mathfrak{z})\big)\big)} 
\varphi (\tau, \mathfrak{z} ),
\end{align*}
for $A=\begin{psmallmatrix}
a & b \\
c & d
\end{psmallmatrix} \in \SL_2(\ZZ)$ and $x,y\in L$, 
and if its Fourier expansion  takes the form
\begin{equation*}
\varphi ( \tau, \mathfrak{z} )=\sum_{ \substack{-\infty \ll n\in  \ZZ +\frac{D}{24} \\ \ell \in \frac{1}{2}L' } }f(n,\ell)q^n\zeta^\ell, \quad \zeta^\ell=e^{2\pi i (\ell,
\mathfrak{z})},
\end{equation*}
where $n\gg -\infty$ means that $n$ is bounded from below.
If $f(n,\ell) = 0$ whenever $n < 0$,
then $\varphi$ is called a {\it weak} Jacobi form.  If $ f(n,\ell) = 0 $ whenever
$ 2nt - (\ell,\ell) < 0 $,
then $\varphi$ is called a {\it holomorphic}
Jacobi form. 
\end{definition}

We denote the vector spaces of weakly holomorphic, weak and holomorphic Jacobi forms of weight $k$, index $t$ and character $\upsilon_\eta^D$ for $L$ by 
$$
J^{!}_{k,L,t}(\upsilon_\eta^D) \supset J^{\w}_{k,L,t}(\upsilon_\eta^D) \supset J_{k,L,t}(\upsilon_\eta^D),
$$
respectively. We simply write $J^{!}_{k,L,t}\supset J^{\w}_{k,L,t} \supset  J_{k,L,t}$ if the character is trivial. 
The spaces $J_{k, t}$ of classical Jacobi forms introduced by Eichler--Zagier \cite{EZ85} 
are simply $J_{k,A_1, t}$ for the lattice $A_1=\ZZ$ with $(x,x)=2x^2$. If $\varphi \in J^{!}_{k,L,t}(\upsilon_\eta^D)$ is nonzero for some $t>0$ then $L(t)$ is necessarily an integral lattice \cite{CG13}. Jacobi forms of index $t$ for $L$ are also called Jacobi forms of (lattice) index $L(t)$ and we sometimes write 
$$
J^!_{k,L(t)}=J^!_{k,L,t}, \quad J^{\w}_{k,L(t)}=J^{\w}_{k,L,t}, \quad J_{k,L(t)}=J_{k,L,t}. 
$$

We will use the following Hecke operators to construct Jacobi forms later. 

\begin{proposition}[Proposition 3.1 in \cite{CG13}]\label{prop:Hecke}
Let $\varphi \in J_{k,L,t}^!(\upsilon_{\eta}^D)$. Assume that $k\in \ZZ$ and $D$ is an even divisor of $24$. If $Q=24/D$ is odd, we further assume that $t\in \ZZ$. Then for any positive integer $m$ coprime to $Q$ we have 
$$
\varphi \lvert_{k,t}T_{-}^{(Q)}(m)(\tau,\mathfrak{z})=m^{-1}\sum_{\substack{ad=m,\,a>0\\ 0\leq b <d}}a^k \upsilon_{\eta}^D(\sigma_a)\varphi\left(\frac{a\tau+bQ}{d},a\mathfrak{z}\right)\in J_{k,L,mt}^!(\upsilon_{\eta}^{D x}),
$$
where $x,y\in \ZZ$ such that $mx+Qy=1$, 
and
$$\sigma_a= \left(\begin{array}{cc}
dx+Qdxy & -Qy \\ 
Qy & a
\end{array}  \right) \in \SL_2(\ZZ).$$ 
Moreover, the Fourier coefficients $f_m(-,-)$ of $\varphi \lvert_{k,t}T_{-}^{(Q)}(m)$ are linear combinations of the Fourier coefficients $f(-,-)$ of $\varphi$, i.e. 
$$
f_m(n,\ell)=\sum_{a \mid (n,\ell,m)}a^{k-1} \upsilon_{\eta}^D(\sigma_a)f\left( \frac{nm}{a^2},\frac{\ell}{a}\right),
$$
where $a\! \mid\! (n,\ell,m)$ means that $a\! \mid\! nQ$, $a^{-1}\ell\in \frac{1}{2}L'$ and $a\! \mid\! m$.
\end{proposition}

\section{Affine Kac--Moody algebras and theta blocks}\label{subsec:theta-blocks}
We review untwisted affine Kac--Moody algebras following \cite{Kac90} and identify their denominators as important examples of Jacobi forms. 

Let $\mathfrak{g}$ be a simple Lie algebra of rank $r$ with Cartan subalgebra $\mathfrak{h}$ and root system $\Delta_\mathfrak{g}$. We fix a set of simple roots $\{ \alpha_1, ..., \alpha_r \} \subset \mathfrak{h}^*$ and denote the set of positive roots by $\Delta_\mathfrak{g}^+$ and the highest root by $\theta$. Note that $\dim \mathfrak{g} = r + |\Delta_\mathfrak{g}|$. The invariant symmetric bilinear form $\latt{-,-}$ on $\mathfrak{h}^*$ is normalized such that long roots have (square) norm two; in particular, $\latt{\theta, \theta}=2$. We identify $\mathfrak{h}$ with $\mathfrak{h}^*$ and define the coroot of a root $\alpha$ as 
$$
\alpha^\vee=2\alpha/\latt{\alpha, \alpha}.
$$
The fundamental weights $w_i\in \mathfrak{h}^*$ are defined by 
$$
\latt{w_i, \alpha_i^\vee} = \delta_{i,j}, \quad 1\leq i,j \leq r,
$$
where $\delta_{i,j}=1$ if $i=j$ and $0$ otherwise. 
The Weyl vector $\rho_\mathfrak{g}$ is defined as
$$
\rho_\mathfrak{g}=\frac{1}{2}\sum_{\alpha\in \Delta_\mathfrak{g}^+} \alpha = \sum_{j=1}^r w_j.
$$
Let $Q_\mathfrak{g}$ be the rational lattice generated by the roots of $\mathfrak{g}$ and let $Q_\mathfrak{g}^\vee$ be the even integral lattice generated by the coroots, or equivalently by the long roots of $\mathfrak{g}$. The weight lattice $P_\mathfrak{g}$, generated by the fundamental weights, is the dual of the coroot lattice:
$$
P_\mathfrak{g}=(Q_\mathfrak{g}^\vee)'=\{ x\in Q_\mathfrak{g}\otimes\QQ : \latt{x, \alpha^\vee}\in \ZZ, \; \alpha\in \Delta_\mathfrak{g}  \}\supset Q_\mathfrak{g}.
$$
The coweight lattice $P_\mathfrak{g}^\vee$ is the dual of the root lattice:
$$
P_\mathfrak{g}^\vee=Q_\mathfrak{g}'=\{ x\in Q_\mathfrak{g}\otimes\QQ : \latt{x, \alpha}\in \ZZ, \; \alpha\in \Delta_\mathfrak{g}  \}\supset Q_\mathfrak{g}^\vee.
$$
The reflection associated with a root $\alpha$ is defined as
$$
\sigma_{\alpha}(x)= x - \latt{x,\alpha^\vee}\alpha, \quad x\in Q_\mathfrak{g}\otimes \RR.
$$
The reflections associated with simple roots generate the Weyl group $W_\mathfrak{g}$. 
The coroot of $\theta$ can be written as an $\NN$-linear combination of the coroots of simple roots,
$$
\theta = \theta^\vee = \sum_{j=1}^r a_j^\vee \alpha_j^\vee.
$$
The integers $a_j^\vee$ are called comarks. The number 
$$
h_\mathfrak{g}^\vee=1+\sum_{j=1}^r a_j^\vee = \frac{1}{r}\sum_{\alpha\in \Delta_\mathfrak{g}^+} \latt{\alpha, \alpha}
$$
is the dual Coxeter number and it satisfies the identity
$$
\sum_{\alpha\in \Delta_\mathfrak{g}^+} \latt{\alpha, \mathfrak{z}}^2 = h_\mathfrak{g}^\vee\latt{\mathfrak{z},\mathfrak{z}}, \quad \mathfrak{z}\in Q_\mathfrak{g}\otimes\CC.
$$
The above identity implies that the rescaled lattice $P_\mathfrak{g}^\vee(h_\mathfrak{g}^\vee)$ is integral. 

The classification of irreducible root systems into types $A_n$ for $n\geq 1$, $B_n$ for $n\geq 2$, $C_n$ for $n\geq 3$, $D_n$ for $n\geq 4$, and the exceptional systems $E_6$, $E_7$, $E_8$, $F_4$ and $G_2$ is well-known. We use the same symbol to stand for the corresponding root lattice with its normalized bilinear form.  Some useful data is summarized in Table \ref{tab:data} for convenience; for the coordinates of the fundamental weights and the values of the comarks, see \cite{Bou82}.

\begin{table}[ht]
\caption{Data related to the irreducible root systems}
\label{tab:data}
\renewcommand\arraystretch{1.5}
\[
\begin{array}{|c|c|c|c|c|c|c|c|c|c|}
\hline 
\Delta_\mathfrak{g} & A_n & B_n & C_n & D_n & E_6 & E_7 & E_8 & G_2 & F_4  \\ 
\hline 
|\Delta_\mathfrak{g}| & n(n+1) & 2n^2 & 2n^2 & 2n(n-1) & 72 & 126 & 240 & 12 & 48 \\ 
\hline
h^\vee_{\mathfrak{g}} & n+1 & 2n-1 & n+1 & 2(n-1)& 12 & 18 & 30 & 4 & 9 \\
\hline
Q^\vee_\mathfrak{g} & A_n & D_n & nA_1 & D_n & E_6 & E_7 & E_8 & A_2 & D_4 \\
\hline
P^\vee_\mathfrak{g} & A_n' & \ZZ^n & D_n'(2) & D_n' & E_6' & E_7' & E_8 & A_2 & D_4 \\
\hline
\end{array} 
\]
\end{table}

The untwisted affine Kac--Moody algebra $\hat{\mathfrak{g}}$ is an extension of $\mathfrak{g}$ defined by
$$
\hat{\mathfrak{g}} = \CC[t,t^{-1}]\otimes \mathfrak{g} \oplus \CC K \oplus \CC d,
$$
where $K$ is a central element and $d$ is a derivative. The algebra $\hat{\mathfrak{g}}$ is an infinite-dimensional Lie algebra with affine Cartan subalgebra $\hat{\mathfrak{h}}=\mathfrak{h}\oplus \CC K \oplus \CC d$. We write $\hat{\mathfrak{h}}^*=\mathfrak{h}^*\oplus \CC \hat{w}_0 \oplus \CC \delta$ with
\begin{align*}
&\hat{w}_0(\mathfrak{h}\oplus \CC d)=0, \qquad\hat{w}_0(K)=1,\\
&\delta(\mathfrak{h}\oplus \CC K)=0, \qquad\ \delta(d)=1.
\end{align*}
We embed $\mathfrak{h}^*$ into $\hat{\mathfrak{h}}^*$ and define 
$$
\alpha_0=\delta-\theta \quad \text{and} \quad \alpha_0^\vee=K-\theta^\vee.
$$
Then $\{\alpha_0, \alpha_1,\cdots, \alpha_r\}$ is a set of simple roots of $\hat{\mathfrak{g}}$ and $\{\alpha_0^\vee, \alpha_1^\vee,\cdots, \alpha_r^\vee\}$ is the corresponding set of coroots. Setting  
$$
\hat{w}_i = w_i + a_i^\vee \hat{w}_0, \quad 1\leq i\leq r,
$$
$\{ \hat{w}_0, \hat{w}_1,\cdots ,\hat{w}_r \}$ is a set of fundamental weights of $\hat{\mathfrak{g}}$. The Weyl vector of $\hat{\mathfrak{g}}$ is defined by 
$$
\hat{\rho}_\mathfrak{g}= \rho_\mathfrak{g} + h_\mathfrak{g}^\vee \hat{w}_0 - (\dim \mathfrak{g} / 24)\delta,
$$
such that the norm of $\hat{\rho}_\mathfrak{g}$ is zero; indeed, by the ``strange formula'' of Freudenthal--de Vries,
$$
h_\mathfrak{g}^\vee \cdot \dim \mathfrak{g} =12 \latt{\rho_\mathfrak{g},\rho_\mathfrak{g}}.
$$
The integrable highest weight representations of $\hat{\mathfrak{g}}$ are indexed by dominant integral weights, which are elements of
$$
\hat{P}_+ = \sum_{j=0}^r \NN \hat{w}_j + \CC \delta. 
$$
The level of $\lambda = \sum_{j=0}^r x_j \hat{w}_j + c\delta \in \hat{P}_+$ is the integer
$$
\lambda(K) = \lambda_0 + \sum_{j=1}^r x_j a_j^\vee.
$$
The character of the irreducible highest weight module labeled by a weight $\lambda \in \hat{P}_+$ of level $k$ is given by the \textit{Weyl--Kac character formula},
\begin{equation}
\chi_\lambda^{\hat{\mathfrak{g}}} = \frac{\sum_{\sigma\in W_{\hat{\mathfrak{g}}}} (-1)^{\ell(\sigma)} e^{\sigma(\lambda+\rho_{\hat{\mathfrak{g}}})} }{e^{\rho_{\hat{\mathfrak{g}}}} \prod_{\alpha \in \Delta_{\hat{\mathfrak{g}}}^+}(1-e^{-\alpha})^{\mathrm{mult}(\alpha)}},  
\end{equation}
where $W_{\hat{\mathfrak{g}}}$ is the Weyl group of $\hat{\mathfrak{g}}$, a semi-direct product of $W_\mathfrak{g}$ by a certain group of translations, $\ell(\sigma)$ is the length of $\sigma$, $\Delta_{\hat{\mathfrak{g}}}^+$ is the set of positive roots of $\hat{\mathfrak{g}}$, and $\mathrm{mult}(\alpha)$ is the multiplicity of $\alpha$, i.e. the dimension of the root space $\hat{\mathfrak{g}}_\alpha$. When $\lambda=0$, we have $\chi_\lambda^{\hat{\mathfrak{g}}}=1$, which gives the \textit{Macdonald--Weyl denominator identity}:
\begin{equation}
e^{\rho_{\hat{\mathfrak{g}}}} \prod_{\alpha \in \Delta_{\hat{\mathfrak{g}}}^+}(1-e^{-\alpha})^{\mathrm{mult}(\alpha)} = \sum_{\sigma\in W_{\hat{\mathfrak{g}}}} (-1)^{\ell(\sigma)}e^{\sigma(\rho_{\hat{\mathfrak{g}}})}.    
\end{equation}
The algebra $\hat{\mathfrak{g}}$ has real roots (with norm $>0$) and imaginary roots (with norm $= 0$). The set of positive real roots is
$$
\Delta_{\hat{\mathfrak{g}}}^{+,\mathrm{re}} = \{ \alpha + n\delta : 0 < n \in \ZZ, \alpha \in \Delta_\mathfrak{g} \} \cup \Delta_\mathfrak{g}^+. 
$$
The set of positive imaginary roots is 
$$
\Delta_{\hat{\mathfrak{g}}}^{+,\mathrm{im}} = \{ n\delta : 0< n \in \ZZ \}.
$$
Every real root has multiplicity one, but every imaginary root has multiplicity $r$. 

Let $(\tau, \mathfrak{z})\in \HH \times \mathfrak{h}$. By interpreting the formal variable $e^{\alpha}$ as $e^{2\pi i \alpha(\mathfrak{z}-\tau d)}$, the character $\chi_\lambda^{\hat{\mathfrak{g}}}(\tau, \mathfrak{z})$ defines a holomorphic function on $\HH\times \CC^r$. The product side of the denominator identity can be written as a theta block (introduced by Gritsenko--Skoruppa--Zagier \cite{GSZ19})
\begin{equation}
\vartheta_{\mathfrak{g}}(\tau,\mathfrak{z})=\eta(\tau)^r \prod_{\alpha\in \Delta_\mathfrak{g}^+}\frac{\vartheta(\tau, \latt{\alpha,\mathfrak{z}})}{\eta(\tau)},
\end{equation}
where the odd Jacobi theta function
$$
\vartheta(\tau,z)=\sum_{n\in\ZZ}\left(\frac{-4}{n}\right)q^{\frac{n^2}{8}}\zeta^{\frac{n}{2}} =-q^{\frac{1}{8}}\zeta^{-\frac{1}{2}}\prod_{n=1}^\infty(1-q^{n-1}\zeta)(1-q^n\zeta^{-1})(1-q^n) 
$$
is the denominator of $\chi_\lambda^{\hat{\mathfrak{g}}}$ for the $A_1$ Lie algebra $\mathfrak{g}$. Here, $z\in \CC$ and $\zeta= e^{2\pi iz}$ as before. 
Note that $\vartheta_\mathfrak{g}$ is a holomorphic Jacobi form of (singular) weight ${r}/{2}$ and lattice index $P_\mathfrak{g}^\vee(h_\mathfrak{g}^\vee)$ with multiplier system $\upsilon_\eta^{\dim \mathfrak{g}}$.  By \cite{Kac90}, all characters $\chi_\lambda^{\hat{\mathfrak{g}}}$ of fixed level $k$ form a vector-valued weakly holomorphic Jacobi form of weight $0$ and lattice index $Q_\mathfrak{g}^\vee(k)$ which is invariant under the action of the Weyl group $W_\mathfrak{g}$ on  $\mathfrak{z}$. We refer to \cite[Theorem 13.8]{Kac90} for the precise transformation laws with respect to the generators of $\SL_2(\ZZ)$. 

Let $k$ be a positive integer. The irreducible $\hat{\mathfrak{g}}$-module $L_{\hat{\mathfrak{g}}}(k,0)$ associated with $k\hat{w}_0$ has a canonical structure as a simple rational vertex operator algebra. This is called the affine VOA generated by $\hat{\mathfrak{g}}$ at level $k$. Such VOAs are physically realized as the well-known Wess--Zumino--Witten (WZW) models, which are nonlinear sigma models describing mapping fields from Riemann surfaces to Lie group manifolds. The central charge of $L_{\hat{\mathfrak{g}}}(k,0)$ is 
\begin{equation}\label{cc}
    c_\mathfrak{g}=\frac{k \dim \mathfrak{g}}{k + h_\mathfrak{g}^\vee}.
\end{equation}
The irreducible modules of $L_{\hat{\mathfrak{g}}}(k,0)$ are the irreducible $\hat{\mathfrak{g}}$-modules $L_{\hat{\mathfrak{g}}}(k,\lambda)$ associated with level $k$ dominant integral weights $k\hat{w}_0+\lambda \in \hat{P}_+$ for any $\lambda\in \sum_{j=1}^r \NN w_j$ that satisfies
$$
\latt{\lambda, \theta^\vee}=\sum_{j=1}^r a_j^\vee \lambda_j\leq k, \quad \text{where} \quad \lambda=\sum_{j=1}^r \lambda_j w_j.
$$
The conformal weight of $L_{\hat{\mathfrak{g}}}(k,\lambda)$ is 
\begin{equation}\label{conformalw}
    h_\lambda=\frac{\latt{\lambda, \lambda + 2\rho_\mathfrak{g}}}{2(k+h^\vee_\mathfrak{g})}.
\end{equation}
The full character of $L_{\hat{\mathfrak{g}}}(k,\lambda)$ is actually $\chi_{k\hat{w}_0+\lambda}^{\hat{\mathfrak{g}}}(\tau,\mathfrak{z})$ as defined above. When the root system of $\mathfrak{g}$ is $R_r$ as in Table \ref{tab:data} and $\lambda=\sum_{j=1}^r \lambda_j w_j$, we often write $\chi_{k\hat{w}_0+\lambda}^{\hat{\mathfrak{g}}}$ as
\begin{equation}\label{eq:symbol-char}
\chi^{R_{r,k}}_{\lambda_1\cdots\lambda_r,h_\lambda}. 
\end{equation}

\begin{Notation}\label{notation}
Let $\mathfrak{g}=\oplus_{j=1}^s \mathfrak{g}_j$ be a semi-simple Lie algebra. We will often have to associate with each simple ideal $\mathfrak{g}_j$ a positive integer $k_j$, called the level. In this case, we indicate the levels by writing $\mathfrak{g}=\oplus_{j=1}^s \mathfrak{g}_{j,k_j}$. For simplicity we write
$$
h_j^\vee = h_{\mathfrak{g_j}}^\vee, \quad Q_j = Q_{\mathfrak{g}_j}, \quad Q_j^\vee = Q_{\mathfrak{g}_j}^\vee, \quad P_j = P_{\mathfrak{g}_j}, \quad P_j^\vee = P_{\mathfrak{g}_j}^\vee. 
$$
We further fix two lattices
$$
\mathbf{Q}_\mathfrak{g} = \bigoplus_{j=1}^s Q_j^\vee(k_j) \quad \text{and} \quad \mathbf{P}_\mathfrak{g} = \bigoplus_{j=1}^s P_j^\vee(k_j).
$$
Let $V_\mathfrak{g}$ denote the affine vertex operator algebra 
$$
\bigotimes_{j=1}^s L_{\hat{\mathfrak{g}}_j}(k_j,0).
$$
Clearly, the denominator of $\hat{\mathfrak{g}}$ is given by the theta block
\begin{equation}\label{eq:theta-blcok}
\vartheta_\mathfrak{g}:=\vartheta_{\mathfrak{g}_1}\otimes \vartheta_{\mathfrak{g}_2}\otimes \cdots \otimes \vartheta_{\mathfrak{g}_s}.   
\end{equation}
\end{Notation}

\begin{remark} Some remarks on the denominator identities are in order:
\begin{enumerate}
\item Gritsenko, Skoruppa and Zagier \cite{GSZ19} found a direct proof of the denominator identity based on their theory of theta blocks.
\item  The function $\vartheta_\mathfrak{g}(\tau,\mathfrak{z})/\eta(\tau)^{\dim \mathfrak{g}}$ is equal to the modular Jacobian of the generators of the ring of weak Jacobi forms for $Q_\mathfrak{g}^\vee$ invariant under $W_\mathfrak{g}$ (see \cite{Wan21}).
\item If a theta block defines a holomorphic Jacobi form of singular weight, then it has to be the denominator of an affine Lie algebra (\cite{Wan23-star}). Therefore, there is a one-to-one correspondence between affine Lie algebras and singular-weight theta blocks.  
\end{enumerate}
\end{remark}

\section{Reflective modular forms on orthogonal groups}
Let $M$ be an even integral lattice of signature $(l,2)$ with $l\geq 3$. We choose one of the two connected components of 
$$
\{\mathcal{Z} \in  M\otimes \CC:  (\mathcal{Z}, \mathcal{Z})=0, (\mathcal{Z},\bar{\mathcal{Z}}) < 0\}
$$
and label it $\cA(M)$. The symmetric domain of type IV attached to $M$ is
$$
\cD(M):=\cA(M) / \mathbb{C}^{\times}=\{[\mathcal{Z}] \in  \PP(M\otimes \CC):  \mathcal{Z} \in \cA(M) \}.
$$
Let $\Orth^+(M)$ denote the subgroup of $\Orth(M\otimes\RR)$ that preserves $M$ and $\cA(M)$.  Let $\Gamma$ be a finite-index subgroup of $\Orth^+(M)$. The most important example of $\Gamma$ will be the \textit{discriminant kernel}
$$
\widetilde{\Orth}^+(M)=\{ g \in \Orth^+(M):\; g(v) - v \in M, \; \text{for all $v\in M'$} \},
$$
where $M'$ is the dual lattice of $M$. 

\begin{definition}
Let $k\in \ZZ$ and $\chi: \Gamma\to \CC^\times$ be a character. A holomorphic function $F: \cA(M)\to \CC$ is called a modular form of weight $k$ and character $\chi$ on $\Gamma$ if it satisfies
\begin{align*}
F(t\mathcal{Z})&=t^{-k}F(\mathcal{Z}), \quad\  \text{for all $t \in \CC^\times$},\\
F(g\mathcal{Z})&=\chi(g)F(\mathcal{Z}), \quad \text{ for all $g\in \Gamma$}.
\end{align*}
\end{definition}
Modular forms can be represented by their Fourier expansions. Let $c$ be a primitive isotropic vector of $M$ and choose $c'\in M'$ satisfying $(c,c')=1$. Then $M_{c,c'}=M\cap c^\perp \cap (c')^\perp$ is an even lattice of signature $(l-1,1)$. Around the cusp $c$, one can identify the symmetric domain $\cD(M)$ with a tube domain $\HH_{c,c'}$, a connected component of 
$$
\{ Z = X + iY :\; X, Y \in M_{c,c'}\otimes\RR, \; (Y,Y)<0  \}.
$$
This induces an action of $\Orth^+(M)$ on $\HH_{c,c'}$ and an automorphy factor, which allows us to define modular forms of half-integral weight. A modular form of trivial character on $\widetilde{\SO}^+(M)$ can be expanded on $\HH_{c,c'}$ as
$$
F(Z)=\sum_{\lambda \in M_{c,c'}', \; (\lambda, \lambda)\leq 0} c(\lambda) e^{2\pi i(\lambda, Z)}.
$$
Modular forms $F$ on general $\Gamma$ have similar expansions with $M_{c,c'}$ replaced by a finite-index sublattice. If $F$ is nonzero, then either $k=0$, in which case $F$ is constant, or $k\geq l/2-1$. The smallest possible positive weight $l/2-1$ is called the \textit{singular weight}.  When $F$ has singular weight, its Fourier coefficients $c(\lambda)$ are zero whenever $(\lambda,\lambda)\neq 0$.

Let $\lambda \in M\otimes\QQ$ be a vector of positive norm. The \textit{rational quadratic divisor} associated with $\lambda$ is 
$$
\lambda^\perp=\{ [\mathcal{Z}] \in \cD(M) : (\mathcal{Z}, \lambda)=0 \}.
$$
We define the associated reflection $\sigma_\lambda \in \Orth^+(M\otimes\QQ)$ as
$$
\sigma_\lambda(x)=x-\frac{2(x,\lambda)}{(\lambda,\lambda)}\lambda, \quad x\in M\otimes\QQ.
$$
The divisor $\lambda^\perp$ is called \textit{reflective} if $\sigma_\lambda$ fixes $M$, i.e. $\sigma_\lambda \in \Orth^+(M)$. If $\lambda$ is a primitive vector in $M'$, then $\lambda^\perp$ is reflective if and only if there exists a positive integer $d$ such that $(\lambda, \lambda)=2/d$ and $d\lambda \in M$. More precisely, the order of $\lambda$ in $M'/M$ is either $d$, or $d/2$ in which case $d/2$ is necessarily even. A non-constant (holomorphic) modular form on $\Gamma$ is called \textit{reflective} if its zero divisor is a linear combination of reflective rational quadratic divisors. 

\section{Borcherds products}\label{subsec:Borcherds}
Let $M$ be an even lattice of signature $(b^+,b^-)$ with discriminant form $D_M:=(M'/M, \mathfrak{q})$, where $\mathfrak{q}(x)=(x,x)/2$ is the induced quadratic form. Let $\mathrm{Mp}_2(\mathbb{Z})$ be the metaplectic group, which consists of pairs $A = (A, \phi_A)$, where $A = \begin{psmallmatrix} a & b \\ c & d \end{psmallmatrix} \in \mathrm{SL}_2(\mathbb{Z})$ and $\phi_A$ is a holomorphic square root of $\tau \mapsto c \tau + d$ on $\mathbb{H}$, with the standard generators 
$$
T = \big(\begin{psmallmatrix} 1 & 1 \\ 0 & 1 \end{psmallmatrix}, 1\big) \quad  \text{and} \quad S = \big(\begin{psmallmatrix} 0 & -1 \\ 1 & 0 \end{psmallmatrix}, \sqrt{\tau}\big).
$$
The \emph{Weil representation} $\rho_M$ is the unitary representation of $\mathrm{Mp}_2(\mathbb{Z})$ on the group ring $\mathbb{C}[D_M] = \mathrm{span}(e_x: \, x \in D_M)$ defined by 
$$
\rho_M(T) e_x = \mathbf{e}(-\mathfrak{q}(x)) e_x \quad \text{and} \quad \rho_M(S) e_x = \frac{\mathbf{e}( \mathrm{sign}(M) / 8)}{\sqrt{|D_M|}} \sum_{y \in D_M} \mathbf{e}(( x,y)) e_y,
$$
where $\mathbf{e}(t)=e^{2\pi it}$ for $t\in \CC$, and $\mathrm{sign}(M)=b^+ - b^- \mod 8$. The dual representation of $\rho_M$ is the complex conjugate of $\rho_M$; moreover, $\bar{\rho}_M=\rho_{M(-1)}$. 

Let $k \in \frac{1}{2}\mathbb{Z}$. A holomorphic function $f : \mathbb{H} \rightarrow \mathbb{C}[D_M]$ is called a \emph{weakly holomorphic modular form} of weight $k$ if $f$ satisfies
$$
\phi_A(\tau)^{-2k} f(A \cdot \tau) = \rho_M(A) f(\tau), \quad \text{for all $A \in \mathrm{Mp}_2(\mathbb{Z})$,}
$$ 
and if $f$ is represented by a Fourier series of the form 
\begin{equation}
f(\tau) = \sum_{x \in D_M} \sum_{\substack{n \in \mathbb{Z} - \mathfrak{q}(x)\\ n \gg -\infty}} c_x(n) q^n e_x.    
\end{equation}
The finite sum $c_x(n)q^n e_x$ with $n<0$ is called the \textit{principal part} of $f$. 
If $f$ is holomorphic at infinity, i.e. its principal part is zero, then it is a \textit{holomorphic modular form}. Note that $k+\mathrm{sign}(M)/2 \in \ZZ$ if non-zero $f$ exist, and that if $\mathrm{sign}(M)$ is even then $\rho_M$ factors through a representation of $\SL_2(\ZZ)$. We denote the spaces of weakly holomorphic and holomorphic modular forms of weight $k$ for $\rho_M$ respectively by
$$
M_k^!(\rho_M) \supset M_k(\rho_M).
$$ 

There is a natural homomorphism $\Orth(M) \to \Orth(D_M)$ with kernel $\widetilde{\Orth}(M)$. We define an action of $\Orth(D_M)$ on modular forms $f = \sum c_x(n) q^n e_x$ by
$$
\sigma(f)=\sum c_x(n)q^n e_{\sigma(x)}, \quad \sigma \in \Orth(D_M).
$$

We now assume that $M$ has signature $(l,2)$ with $l\geq 3$.
Let $f$ be a weakly holomorphic modular form of weight $1-l/2$ for $\rho_M$ with integral principal part. The Borcherds singular theta lift \cite{Bor95, Bor98} produces a meromorphic modular form $\Borch(f)$ of weight $c_0(0)/2$ and some character (or multiplier system) on 
$$
\Orth^+(M,f)=\{ \sigma \in \Orth^+(M): \sigma(f)=f\} \supset \widetilde{\Orth}^+(M)
$$
which has an infinite product expansion at any $0$-dimensional cusp involving the Fourier coefficients of $f$ (see below). Moreover, the divisor of $\Borch(f)$ is a linear combination of rational quadratic divisors $\lambda^\perp$, each with multiplicity 
$$
\sum_{d=1}^\infty c_{d\lambda}(-d^2\lambda^2/2),
$$
where $\lambda \in M'$ are primitive vectors of positive norm.

There is an identification between modular forms for the Weil representation and Jacobi forms. Let $L$ be an even positive-definite lattice of rank $\rank(L)$. Then one has the isomorphism
\begin{equation}
\begin{split}
M_{k-\frac{1}{2}\rank(L)}^!(\rho_L) &\stackrel{\sim}{\longrightarrow} J_{k,L}^! \\
f(\tau)=\sum_{\gamma\in L'/L} f_\gamma(\tau)e_\gamma &\longmapsto \sum_{\gamma\in L'/L} f_\gamma(\tau) \Theta_{L,\gamma}(\tau,\mathfrak{z}),
\end{split}
\end{equation}
where $\Theta_{L,\gamma}$ is the Jacobi theta function of $L+\gamma$, defined as
$$
\Theta_{L,\gamma}(\tau,\mathfrak{z}) = \sum_{\ell \in L+\gamma} e^{\pi i(\ell,\ell)\tau + 2\pi i(\ell, \mathfrak{z})}, \quad (\tau,\mathfrak{z})\in\HH\times (L\otimes\CC). 
$$
This map can be extended to modular forms with characters and it induces an isomorphism between holomorphic modular forms. It follows that a holomorphic Jacobi form of weight $\frac{1}{2}\rank(L)$ and index $L$ is a $\CC$-linear combination of $\Theta_{L,\gamma}$. We call $\frac{1}{2}\rank(L)$ the \textit{singular weight} for Jacobi forms. 

It will be convenient to represent Borcherds products in terms of Jacobi forms. Let $\phi \in J_{0,L,1}^!$ be a weakly holomorphic Jacobi form of trivial character with Fourier expansion
$$
\phi(\tau,\mathfrak{z})=\sum_{n\in \ZZ}\sum_{\ell\in L'}f(n,\ell)q^n\zeta^\ell.
$$
The terms $f(n,\ell)q^n\zeta^\ell$ with $2n-(\ell,\ell)<0$ are called \textit{singular Fourier coefficients} and they correspond to the principal part of the preimage of $\phi$ under the above isomorphism. 

\begin{remark}\label{rem:singular}
The Fourier expansion of $\phi$ satisfies $f(n,\ell)=f(n,-\ell)$ and $f(n_1,\ell_1)=f(n_2,\ell_2)$ if $2n_1-(\ell_1,\ell_1)=2n_2-(\ell_2,\ell_2)$ and if $\ell_1-\ell_2\in L$. Therefore, any singular Fourier coefficient of $\phi$ already appears as a summand $f(n,\ell)q^n\zeta^\ell$ where $n\leq \hat{\delta}_L$, $\ell \in L'$ and $2n<(\ell,\ell)$, where $\hat{\delta}_L$ is the largest integer less than $\delta_L/2$ and
\begin{equation}
\delta_L := \max\big\{ \min\{(y,y): y\in L + x \} : x \in L' \big\}.
\end{equation}
\end{remark}

Let $U$ be a hyperbolic plane, i.e. an even unimodular lattice of signature $(1,1)$. We write 
$$
U=\ZZ e + \ZZ f, \quad \text{where} \quad e^2=f^2=0 \quad \text{and} \quad (e,f)=-1.
$$
Let $U_1=\ZZ e_1 +\ZZ f_1$ be a second hyperbolic plane and define $M=U_1\oplus U\oplus L$. We fix coordinates on the tube domain about $U_1$ by writing
$$
\HH(L)=\{ Z= \tau f + \mathfrak{z} +\omega e : \tau, \omega \in \HH,\; \mathfrak{z}\in L\otimes\CC, \; 2\im(\tau)\im(\omega) -(\mathfrak{z},\mathfrak{z})>0  \},
$$
such that 
$$
-(\alpha, Z)=n\tau - (\ell, \mathfrak{z}) + m\omega  \quad \text{for} \quad \alpha=ne+\ell+mf\in U\oplus L'.
$$
These coordinates are chosen such that the infinite expansions below match the denominators of Borcherds--Kac--Moody algebras. 

\begin{Notation}\label{notation2}
We write $v=x_1e_1+xe+\ell+yf+y_1f_1\in U_1\oplus U\oplus L'$ in the coordinate $(x_1,x,\ell,y,y_1)$ and $\alpha=xe+\ell+yf\in U\oplus L'$ in the coordinate $(x,\ell,y)$. In particular, $v^2=(\ell,\ell)-2(xy+x_1y_1)$. 
\end{Notation}

\begin{theorem}[{\cite[Theorem 4.2]{Gri18}}]\label{th:product}
Assume the above $\phi$ has integral singular Fourier coefficients. Then the Borcherds theta lift of $\phi$ is a meromorphic modular form of weight $f(0,0)/2$ on $\widetilde{\Orth}^+(M)$ and it can be expanded on an open subset of the tube domain $\HH(L)$ as
$$
\Borch(\phi)(Z)=q^A \zeta^{\vec{B}} \xi^C\prod_{\substack{n,m\in\ZZ,\; \ell
\in L'\\ (n,\ell,m)>0}}\Big(1-q^n \zeta^{-\ell} \xi^m\Big)^{f(nm,\ell)}, 
$$ 
where $q=\exp(2\pi i \tau)$, $\zeta^\ell=\exp(2\pi i (\ell, \mathfrak{z}))$, $\xi=\exp(2\pi i \omega)$, where the positivity condition $(n,\ell,m)>0$ means that either $m>0$, or $m=0$ 
and $n>0$, or $m=n=0$ and $\ell>0$, and where the Weyl vector $\rho= (-A,\vec{B}, -C)$ of $\Borch(\phi)$ is defined by
$$
A=\frac{1}{24}\sum_{\ell\in L'}f(0,\ell),\quad
\vec{B}=\frac{1}{2}\sum_{\ell>0} f(0,\ell)\ell, \quad C=\frac{1}{2\rank(L)}\sum_{\ell\in L'}f(0,\ell)(\ell,\ell).
$$
The Fourier--Jacobi expansion of $\Borch(\phi)$ on $\HH(L)$ is
\begin{equation}
\Borch(\phi)(Z)=
\biggl(\Theta_{f(0,\ast)}
(\tau,\mathfrak{z})\cdot \xi^C \biggr)
\exp \left(-\sum_{m=1}^\infty \Big(\phi |_0 T_{-}^{(1)}(m)\Big) (\tau, \mathfrak{z}) \cdot \xi^m \right),
\end{equation}
where the leading Fourier--Jacobi coefficient is given by a generalized theta block
\begin{equation*}
\Theta_{f(0,\ast)}(\tau,\mathfrak{z})
=\eta(\tau)^{f(0,0)}\prod_{\ell >0}
\biggl(\frac{\vartheta(\tau,(\ell,\mathfrak{z}))}{\eta(\tau)} 
\biggr)^{f(0,\ell)}.
\end{equation*}
\end{theorem} 

We also need the following additive lifts with characters to construct orthogonal modular forms.

\begin{theorem}[Theorem 3.2 in \cite{CG13}]\label{th:additive}
Let $D$ be an even divisor of $24$, $k\in\NN$, $t\in \frac{1}{2}\NN$. If $Q=24/D$ is odd, we further assume that $t\in\NN$. Let $\varphi \in J_{k,L,t}(\upsilon_{\eta}^D)$.  Let $G_k(\tau)$ be the normalized Eisenstein series of weight $k$ on $\SL_2(\ZZ)$ whose Fourier coefficient at $q$ is 1. Then the function 
$$ 
\Grit(\varphi)(Z)=f(0,0)G_k(\tau)+\sum_{0<m\in 1+Q\ZZ}\Big(\varphi \lvert_{k,t}T_{-}^{(Q)}(m)\Big)(\tau,\mathfrak{z})\cdot \xi^{m/Q}
$$
is a holomorphic modular form of weight $k$ and a certain character on the group $\widetilde{\Orth}^+(U_1\oplus U\oplus L(Qt))$. The form $\Grit(\varphi)$ is always invariant under the action of the involution $(\omega,\mathfrak{z},\tau)\mapsto (\tau,\mathfrak{z},\omega)$. 
\end{theorem}

For a given lattice $M$, we will often need to determine all holomorphic, reflective Borcherds products and in some cases prove that no such products exist. To do this, we apply Borcherds's obstruction criterion \cite{Bor99} which states that a formal sum $$\sum_{n < 0} \sum_{x \in D_M} c_x(n) q^n e_x$$ occurs as the principal part of a weakly-holomorphic modular form of weight $\kappa$ for $\rho_M$ if and only if the identity $$\sum_{n < 0} \sum_{x \in D_M} c_x(n) a_x(-n) = 0$$ holds for every cusp form $\sum_{n > 0} \sum_{x \in D_M} a_x(n) q^n e_x$  of weight $2-\kappa$ for the dual representation $\rho_{M(-1)}$.
We construct a basis of cusp forms for $\rho_{M(-1)}$ following \cite{Wil2018} and realize the problem of computing holomorphic reflective products as the problem of enumerating integral lattice points in a polyhedral cone defined by finitely many inequalities (non-negative order along reflective divisors) and finitely many linear equations. 
The solution can be conveniently expressed by the notion of a Hilbert basis, i.e. a minimal system of reflective products $F_1,...,F_r$ such that every holomorphic reflective product can be written in the form $$F = F_1^{n_1} \cdot ... \cdot F_r^{n_r}$$ with non-negative integers $n_1,...,n_r$. Given the polyhedral cone, we used the software Normaliz \cite{Normaliz} to find a Hilbert basis.

\begin{remark}
This computation will mainly be applied in Chapter \ref{sec:exclude-anti-root} to prove that certain lattices $L$ are ``forbidden components" in the sense that they cannot occur as direct summands in a lattice that admits a holomorphic, reflective Borcherds product (of any weight). For this, it turns out to be sufficient to prove that $L$ itself does not admit holomorphic reflective products, i.e. that the above polyhedral cone contains no integral points. This is an easier problem than computing a Hilbert basis that can also be solved using Normaliz \cite{Normaliz}.
\end{remark}

\section{Borcherds--Kac--Moody superalgebras}\label{subsec:BKM}
Borcherds--Kac--Moody superalgebras are infinite dimensional Lie superalgebras introduced by Borcherds \cite{Bor88} in 1988 which generalize affine Kac--Moody algebras. They can be defined in terms of Chevalley--Serre generators and relations which are encoded in a generalized Cartan matrix. The restrictions on the generalized Cartan matrix are weaker; in particular, simple roots are allowed to have non-positive norm (i.e. they can be imaginary roots). BKM superalgebras also have a character formula for highest-weight modules and a denominator identity. We will review the denominator identity following \cite{Ray06} because it contains valuable information about the roots, the root multiplicities and the Weyl group. 

Let $\mathcal{G}=\mathcal{G}_0\oplus \mathcal{G}_1$ be a BKM superalgebra with even and odd components $\mathcal{G}_0$ and $\mathcal{G}_1$, respectively. If $\mathcal{G}_1$ is trivial, $\mathcal{G}$ is a BKM algebra.  Let $\mathcal{H}$ be the generalized Cartan subalgebra of $\mathcal{G}$. Let $\Delta^+$ be the set of positive roots of $\mathcal{G}$. We have the root space decomposition
$$
\mathcal{G}=\bigoplus_{\alpha \in \Delta^+} \mathcal{G}_\alpha \oplus \mathcal{H} \oplus \bigoplus_{\alpha \in \Delta^+} \mathcal{G}_{-\alpha}.
$$
For $\alpha \in \Delta^+$, we define
$$
\mathrm{mult}_0(\alpha) = \dim (\mathcal{G}_\alpha\cap \mathcal{G}_0),  \quad \mathrm{mult}_1(\alpha) = \dim (\mathcal{G}_\alpha\cap \mathcal{G}_1), \quad
\mathrm{mult}(\alpha) = \dim (\mathcal{G}_\alpha),
$$
and define the super-multiplicity $s$-$\mathrm{mult}(\alpha)$ as 
$$
\mathrm{mult}_0(\alpha) - \mathrm{mult}_1(\alpha). 
$$

The sets of positive even roots and odd roots are respectively
$$
\Delta^+_0=\{ \alpha\in \Delta^+: \mathrm{mult}_0(\alpha) >0 \} \quad \text{and} \quad \Delta^+_1=\{ \alpha\in \Delta^+: \mathrm{mult}_1(\alpha) >0 \}.
$$
Let $I$ be a countable set indexing the simple roots $\alpha_i$ and $S$ be the subset of $I$ indexing odd roots. By expanding a root as $\alpha=\sum_{i\in I} k_i \alpha_i$, we define the height and even height of $\alpha$ respectively as
$$
\mathrm{ht}(\alpha) = \sum_{i\in I} k_i \quad \text{and} \quad \mathrm{ht}_0(\alpha) = \sum_{i\in I \backslash S } k_i.
$$
We further define two formal sums
$$
T=e^\rho \sum_{\mu}(-1)^{\mathrm{ht}(\mu)} e^{-\mu}, \quad T'=e^\rho \sum_{\mu}(-1)^{\mathrm{ht}_0(\mu)} e^{-\mu},
$$
where $\rho$ is the Weyl vector of $\mathcal{G}$ and the sums are taken over all sums $\mu$ of distinct pairwise orthogonal imaginary simple roots. Let $W$ be the Weyl group generated by reflections associated with real simple roots. In this paper, we assume that $\mathcal{G}$ has no odd real roots. Under this assumption, one has the denominator identity 
\begin{equation}
    e^\rho \cdot \frac{\prod_{\alpha \in \Delta_0^+}(1-e^{-\alpha})^{\mathrm{mult}_0(\alpha)}}{\prod_{\alpha \in \Delta_1^+}(1+e^{-\alpha})^{\mathrm{mult}_1(\alpha)}} = \sum_{\sigma\in W} \det(\sigma) \sigma(T)
\end{equation}
and the super-denominator identity 
\begin{equation}
    e^\rho \cdot \frac{\prod_{\alpha \in \Delta_0^+}(1-e^{-\alpha})^{\mathrm{mult}_0(\alpha)}}{\prod_{\alpha \in \Delta_1^+}(1-e^{-\alpha})^{\mathrm{mult}_1(\alpha)}} = \sum_{\sigma\in W} \det(\sigma) \sigma(T'). 
\end{equation}
An affine Kac--Moody  algebra $\mathcal{G}$ is a BKM algebra with no imaginary simple roots. Therefore, $T=e^{\rho}$, and we recover the denominator identity for affine Kac--Moody algebras. 

Following Borcherds, we study BKM superalgebras that have a hyperbolic root lattice (i.e. its signature is of type $(l,1)$) and whose denominator or super-denominator is the Fourier expansion of a holomorphic modular form $F$ on $\Orth(l+1,2)$ at some $0$-dimensional cusp. The product side of the denominator identity suggests that $F$ can usually be constructed as a Borcherds product \cite{Bru02, Bru14}, in which case the roots and their multiplicities are encoded in a weakly holomorphic modular form for the Weil representation of $\SL_2(\ZZ)$. The real roots of $\mathcal{G}$ are orthogonal to hyperplanes through the cusp. By definition of Weyl group, these hyperplanes are reflective as zeros of $F$. Thus $F$ is locally reflective at the cusp in some sense. In general, $F$ also vanishes on hyperplanes not through the cusp. If such hyperplanes are also mirrors of reflections in the modular group of $F$, then $F$ would be a reflective modular form, i.e. $F$ is globally reflective. This holds in many cases (see \cite{WW23} for the singular-weight case). 

This is particularly interesting when $F$ is of singular weight. In this case, imaginary simple roots are easy to describe, and in many cases they are just the multiples of the Weyl vector. Moreover, one can write the sum side of the denominator identity explicitly as in \cite{Bor98, Sch04, Sch06}, and it is expected that such a BKM superalgebra has a natural construction as the BRST cohomology of a suitable vertex algebra. In this paper, we will classify this type of BKM superalgebra under some mild conditions and work out the corresponding vertex algebras. Our classification is closely related to the fake monster algebra and its twists, so we will review that theory as well. 

In 1990 Borcherds \cite{Bor90} constructed the fake monster algebra $\mathbb{G}$ as the BRST cohomology related to the Leech lattice VOA. The root lattice of the BKM algebra $\mathbb{G}$ is $U\oplus \Lambda$, where $\Lambda$ is the Leech lattice. The Weyl vector is $\rho=(-1,0,0)$ (see Notation \ref{notation2}) and the real roots are $\alpha\in U\oplus \Lambda$ with $\alpha^2=2$. The real simple roots are characterized by the equality $(\rho,\alpha)=1$. The imaginary simple roots are $n\rho$ for negative integers $n$, each with multiplicity $24$. The denominator of $\mathbb{G}$ defines a reflective Borcherds product $\Phi_{12}$ of singular weight on $2U\oplus \Lambda$ whose input is the full character of the Leech lattice VOA (see \cite{Bor95}). Let $\mathrm{Co}_0=\Orth(\Lambda)$ be the Conway group.  Borcherds \cite{Bor92} proved that the denominator identity of $\mathbb{G}$ is a cohomological identity and that $\mathrm{Co}_0$ acts on $\mathbb{G}$ naturally. In this way, he obtained a twisted denominator identity for each conjugacy class $[g]$ of $\mathrm{Co}_0$ and proved that it is the (untwisted) super-denominator identity of a BKM superalgebra, denoted $\mathbb{G}_g$. The root lattice of $\mathbb{G}_g$ is $L_g=U\oplus \Lambda^g$, where 
$$
\Lambda^g=\{ v \in \Lambda: g(v)=v \}
$$
is the fixed-point sublattice. The real simple roots are the roots $\alpha$ of $L_g$ satisfying 
$$
(\rho, \alpha) = \alpha^2/2
$$
and the imaginary simple roots are $m\rho$ with super-multiplicity $\sum_{k|(m,n_g)}b_k$ for all negative integers $m$, where $n_g$ is the order of $g$ and where $b_k$ describe the cycle shape $\prod_{b_k\neq 0}k^{b_k}$ of $g$. The last two named authors \cite{WW22} proved that the twisted denominator function associated with $g$ defines a holomorphic Borcherds product (denoted $\Phi_g$) of singular weight on $U(N_g)\oplus U\oplus \Lambda^g$, where $N_g$ is the level of $g$, and that $\Phi_g$ is reflective on $U(N_g)\oplus U\oplus \Lambda^g$ if $N_g=n_g$, confirming a conjecture of Borcherds \cite[\S 15, Example 3]{Bor95}. Note that the level $N_g$ is the smallest multiple of $n_g$ for which $N_g\sum_{k|n_g}\frac{b_k}{k}$ is divisible by $24$, and the associated eta quotient $\eta_g(\tau)=\prod_{k|n_g}\eta(k\tau)^{b_k}$ is a weakly holomorphic modular form of weight $\frac{1}{2}\sum_{k|n_g}b_k$ on $\Gamma_0(N_g)$. 

%% file: chap3.tex
\chapter{The hyperbolization of affine Kac--Moody algebras}\label{sec:hyperbolization}

It is well-known how to realize affine Kac--Moody algebras as extensions of semi-simple Lie algebras. In this chapter, we introduce extensions of 
affine Kac--Moody algebras to BKM superalgebras that we call their hyperbolizations.

Let $L$ be an even positive-definite lattice of rank $\rank(L)$, let $U$ and $U_1$ be two hyperbolic planes, and define $M=U_1\oplus U\oplus L$. Let $\mathcal{G}$ be a BKM superalgebra with root lattice $U\oplus L'$ and generalized Cartan subalgebra $\mathcal{H}$ with no odd real roots. We assume that the super-denominator of $\mathcal{G}$ is given by the Fourier expansion of a Borcherds product $F$ for $\widetilde{\Orth}^+(M)$ at the $0$-dimensional cusp determined by $U_1$. 

Every divisor $\lambda^\perp$ of $F$ lies in the $\widetilde{\Orth}^+(M)$-orbit of the rational quadratic divisor associated with some primitive vector $\alpha\in U\oplus L'$. If $\alpha^\perp$ appears in the divisor of $F$, then $\alpha$ is a real root. Since we have assumed that $\mathcal{G}$ has no odd real roots, the divisor $\alpha^\perp$ is always a zero and therefore $F$ is holomorphic. Moreover, the associated reflection $\sigma_\alpha$ lies in the Weyl group of $\mathcal{G}$. Therefore, $\sigma_\alpha$ fixes the root lattice $U\oplus L'$, so it also fixes $M'=U_1\oplus U\oplus L'$, which implies that $\sigma_\alpha \in \Orth^+(M')=\Orth^+(M)$ and finally that $F$ is a reflective modular form. 

We denote by $\phi$ the Jacobi form input of $F$ and write its expansion at $\infty$ as
$$
\phi(\tau,\mathfrak{z}) = \sum_{n\in\ZZ,\; \ell\in L'} f(n,\ell)q^n \zeta^\ell \in J_{0,L}^!.
$$
By Theorem \ref{th:product}, the product side of the denominator identity is (see Notation \ref{notation2})
$$
e^\rho \prod_{\alpha>0}\left(1-e^{-\alpha}\right)^{f(nm,\ell)}, \quad \alpha=(n,\ell,m)\in U\oplus L',\quad \rho=(-A,\vec{B}, -C).
$$
In particular, if $f(nm,\ell)\neq 0$ then $\alpha$ is a root of $\mathcal{G}$ and 
$$
\mathrm{mult}_0(\alpha)-\mathrm{mult}_1(\alpha)=f(nm,\ell).
$$
The weight of $F$ is given by $f(0,0)/2$. The root $\alpha$ is positive if and only if either $m>0$, or $m=0$ and $n>0$, or $m=n=0$ and $\ell>0$. 

Let $\mathcal{R}_+$ be the set of positive roots of $\mathcal{G}$ of type $(n,\ell,0)$. For any positive integer $n$, $(n,0,0)$ is an imaginary root in $\mathcal{R}_+$. All other roots in $\mathcal{R}_+$ are (even) real and have multiplicity $f(0,\ell)$. Since the multiplicity of a real root is always $1$ (see e.g. \cite[Corollary 2.3.42]{Ray06}), we have $f(0,\ell)=1$. 

Assume further that none of $(n,0,0)$ is odd. Then the even roots $(n,0,0)$ always have multiplicity $f(0,0)$. Now consider the sum of root spaces
$$
\mathcal{G}^0:=\bigoplus_{\alpha\in \mathcal{R}_+} \mathcal{G}_{\alpha} \oplus \mathcal{H} \oplus \bigoplus_{\alpha\in \mathcal{R}_+} \mathcal{G}_{-\alpha}.
$$
If $F$ is of singular weight, i.e. $f(0,0)=\rank(L)$, then $\mathcal{G}^0$ defines an affine Kac--Moody algebra $\hat{\mathfrak{g}}$ for which $\mathfrak{g}$ is a semi-simple Lie algebra whose positive roots are the roots of type $(0,\ell,0)$ in $\mathcal{R}^+$ (here, we can view $K=\delta=(1,0,0)$ and $d=\hat{w}_0=(0,0,-1)$). The leading Fourier--Jacobi coefficient of $F$ at the $1$-dimensional cusp determined by $U_1\oplus U$ is a holomorphic Jacobi form of singular weight given by the theta block
$$
\eta(\tau)^{\rank(L)} \prod_{(0,\ell,0)\in \mathcal{R}^+} \frac{\vartheta(\tau, (\ell, \mathfrak{z}))}{\eta(\tau)}, \quad \mathfrak{z}\in L\otimes\CC,
$$
which equals the denominator $\vartheta_\mathfrak{g}$ of $\hat{\mathfrak{g}}$. In particular, $\mathcal{G}$ is an extension of $\hat{\mathfrak{g}}$. For $m>0$ we define $\mathcal{G}^m$ to be the sum of root spaces associated with roots of type $\pm (*,*,m)$. Then   
$$
\mathcal{G}=\bigoplus_{m=0}^\infty \mathcal{G}^m
$$
is a graded module over $\hat{\mathfrak{g}}=\mathcal{G}^0$. From this point of view, it is natural to regard $\mathcal{G}$ as a \textit{hyperbolization} of $\hat{\mathfrak{g}}$. 

Motivated by the discussion above, we introduce the following definition. 

\begin{definition}\label{def:hyperbolization}
Let $\hat{\mathfrak{g}}$ be an affine Kac--Moody algebra. A BKM superalgebra $\mathcal{G}$ without odd real roots is called a \textit{hyperbolization} of $\hat{\mathfrak{g}}$ if there exists an even positive-definite lattice $L$ such that the root lattice of $\mathcal{G}$ is $U\oplus L'$ and the super-denominator of $\mathcal{G}$ defines a holomorphic Borcherds product $F$ of (singular) weight $\rank(L)/2$ on $2U\oplus L$ whose leading Fourier--Jacobi coefficient is the denominator $\vartheta_\mathfrak{g}$ of $\hat{\mathfrak{g}}$. We also call $F$ a \textit{hyperbolization} of $\vartheta_\mathfrak{g}$.  
\end{definition}

The notion of hyperbolization was first introduced by Gritsenko and the second named author in \cite[Theorem 1.2]{Gri12}, \cite[Section 6.3]{GW19} and \cite[Remark 3.11]{GW20}.

In this paper, we will classify affine Kac--Moody algebras that admit hyperbolizations and construct a hyperbolization for every such affine Kac--Moody algebra. Moreover, we express every Jacobi form input $\phi$ as a $\ZZ$-linear combination of full characters of an associated vertex operator (super)algebra. 

%% file: chap4.tex
\chapter[Root systems associated with singular reflective modular forms]{Root systems associated with reflective Borcherds products of singular weight}\label{sec:root-system}
In this chapter, we classify reflective automorphic products of singular weight on lattices of type $2U\oplus L$ whose Jacobi form inputs have non-negative $q^0$-terms. Although Jacobi form inputs that have negative coefficients in their $q^0$-terms can lift to holomorphic Borcherds products, the corresponding BKM superalgebras in these cases will have odd real roots.

By an argument similar to \cite[Theorem 6.2]{Wan19} and \cite[Lemma 4.5]{Wan21a}, we will show that any such Borcherds product belongs to a root system satisfying certain strong restrictions. These restrictions yield a finite list of root systems and a finite list of candidates for affine Kac--Moody algebras that admit hyperbolizations.

\begin{theorem}\label{th:root-system}
Let $L$ be an even positive-definite lattice of rank $\rank(L)$ and let $U$ be a hyperbolic plane. Suppose there is a reflective Borcherds product $F$ of singular weight on $\widetilde{\Orth}^+(2U\oplus L)$ whose Jacobi form input has non-negative $q^0$-term. Then either $(a)$ or $(b)$ below holds:
\begin{itemize}
    \item[(a)] The lattice $L$ is the Leech lattice, $F$ is the denominator $\Phi_{12}$ of the fake monster algebra, and the Jacobi form input is the full character of the Leech lattice VOA. 
    \item[(b)] There exists a semi-simple Lie algebra $\mathfrak{g}=\oplus_{j=1}^s\mathfrak{g}_{j,k_j}$ of rank $\rank(L)$ satisfying restrictions $(1)$ or $(1')$ and $(2)$ and $(3)$, where $\mathfrak{g}_{j}$ are simple Lie algebras and $k_j$ are positive integers for $1\leq j\leq s$. We refer to Notation \ref{notation} for the symbols below. Let $\lambda\in U$ with $\lambda^2=2$.
    \begin{itemize}
        \item[(1)] If $F$ vanishes on $\lambda^\perp$, then 
        \begin{equation}\label{antiC}
        C:=\frac{\dim \mathfrak{g}}{24} -1 = \frac{h_j^\vee}{k_j}, \quad \text{for $1\leq j\leq s$}. 
        \end{equation}
        
        \item[($1'$)] If $F$ does not vanish on $\lambda^\perp$, then 
        \begin{equation}\label{symC}
        C:=\frac{\dim \mathfrak{g}}{24} = \frac{h_j^\vee}{k_j} \quad \text{and} \quad k_j>1, \quad \text{for $1\leq j\leq s$}. 
        \end{equation}

        \item[(2)] The rescaled lattice $L(C)$ is integral and $L$ is bounded by
        \begin{equation}
        \mathbf{Q}_\mathfrak{g}<L< \mathbf{P}_\mathfrak{g}.
    \end{equation} 

        \item[(3)] The leading Fourier--Jacobi coefficient of $F$ at the $1$-dimensional cusp determined by $2U$ coincides with the denominator of the affine Kac--Moody algebra $\hat{\mathfrak{g}}$.
    \end{itemize}
\end{itemize}
\smallskip
    
    Let $c_\mathfrak{g}$ denote the central charge of the affine VOA generated by $\hat{\mathfrak{g}}=\oplus_{j=1}^s\hat{\mathfrak{g}}_{j,k_j}$. Then $c_\mathfrak{g}$ is always $24$ in case $(1)$ and in case $(1')$ we have 
    \begin{equation}
        c_\mathfrak{g} = \frac{24C}{C+1}. 
    \end{equation}
\end{theorem}
We call case $(1)$ (resp. $(1')$) the \textit{anti-symmetric} (resp. \textit{symmetric}) case, since $F$ is anti-invariant (resp. invariant) under the involution $(\omega,\mathfrak{z},\tau)\mapsto (\tau,\mathfrak{z},\omega)$. 

\begin{proof}
This is a more precise version of \cite[Theorem 2.2]{Wan23a} in the specific case of singular weight. We provide a detailed proof here as it combines the proof of \cite[Theorem 2.2]{Wan23a} with some recent results proved in \cite{WW23}, and because it explains the construction of the Lie algebra $\mathfrak{g}$.

Let $\phi\in J_{0,L}^!$ be the Jacobi form input of $F$, and write its Fourier expansion as
$$
\phi(\tau,\mathfrak{z}) = \sum_{n\in \ZZ,\; \ell\in L'} f(n,\ell) q^n \zeta^\ell. 
$$
By assumption, $f(0,\ell)$ are non-negative integers. 
Each vector $v=(n,\ell,1) \in U\oplus L'$ is primitive in $M:=2U\oplus L$ and satisfies $v^2=(\ell,\ell)-2n$. The multiplicity of $v^\perp$ in the divisor of $F$ is
$$
\sum_{d=1}^\infty f(d^2n, d\ell).
$$ 
By \cite[Lemma 2.1 and Equation (2.1)]{Wan23a}, the Fourier expansion of $F$ begins with
$$
\phi(\tau,\mathfrak{z}) = f(-1,0)q^{-1} + \sum_{\ell \in L'}f(0,\ell)\zeta^\ell + O(q). 
$$
Since $F$ is of singular weight, $f(0,0)=\rank(L)$. We know from \cite[Theorem 1.2]{WW23} that a holomorphic Borcherds product of singular weight has only simple zeros. Therefore, $f(-1,0)$ and $f(0,\ell)$ can only be $0$ or $1$. By definition, $f(-1,0)=1$ in the anti-symmetric case and $f(-1,0)=0$ in the symmetric case. We define
$$
\mathcal{R}=\{ \ell \in L'\backslash \{0\} : f(0,\ell)=1 \}.
$$
By \cite[Proposition 2.6]{Gri18}, we have
\begin{align}\label{eq:A}
\frac{|\mathcal{R}|+\rank(L)}{24} -f(-1,0) &= \frac{1}{2\rank(L)}\sum_{\ell \in \mathcal{R}} (\ell, \ell) =:C, \\
\label{eq:B}
\sum_{\ell \in \mathcal{R}} (\ell, \mathfrak{z})^2 &= 2C (\mathfrak{z}, \mathfrak{z}).
\end{align}

If $\mathcal{R}$ is empty, then $C=0$, $f(-1,0)=1$ and $\rank(L)=24$. In this case, \cite[Proposition 3.2]{Ma17} shows that $M$ is isomorphic to the even unimodular lattice of signature $(26,2)$, which then forces $L$ to be the Leech lattice $\Lambda$ and $F$ to be the Borcherds form $\Phi_{12}$ (see \cite[Theorem 3.5]{WW22}). The Jacobi form input of $\Phi_{12}$ is given by 
$$
\Theta_{\Lambda,0}(\tau,\mathfrak{z})/\eta^{24}(\tau)=q^{-1}+24+O(q),
$$
which equals the full character of the Leech lattice VOA. 

Now assume that $\mathcal{R}$ is non-empty. Equation \eqref{eq:B} implies that $C> 0$ and that $L\otimes\QQ$ is spanned by $\mathcal{R}$ over $\QQ$. We will now show that $\mathcal{R}$ is a rescaled root system. Let $\ell_1, \ell_2 \in \mathcal{R}$. Then $-\ell_1 \in \mathcal{R}$ because $f(0,\ell)=f(0,-\ell)$. For any integer $a>1$, $a\ell_1 \not\in \mathcal{R}$, as otherwise $(0,\ell_1,1)^\perp$ would have multiplicity at least $2$ in the divisor of $F$, contradicting the fact that all zeros of $F$ are simple. Since $F$ is reflective, the reflection $\sigma_{(0,0,\ell_1,1,0)}$ fixes $M$ (see Notation \ref{notation2}). From
$$
\sigma_{(0,0,\ell_1,1,0)}\big((0,0,\ell_2,1,0)\big)= \big(0,0,\sigma_{\ell_1}(\ell_2),1-2(\ell_1,\ell_2)/(\ell_1,\ell_1),0\big)=:\lambda \in M
$$
we obtain
$$
\sigma_{\ell_1}(\ell_2)\in L' \quad \text{and} \quad 2(\ell_1,\ell_2)/(\ell_1,\ell_1) \in \ZZ.
$$
We know from \cite[Theorem 1.2]{WW23} that $F$ is anti-invariant under $\sigma_{(0,0,\ell_1,1,0)}$ and therefore vanishes on $\lambda^\perp$. Since $(0,0,\sigma_{\ell_1}(\ell_2),1,0)$ is primitive in $M'$, by the Eichler criterion (see e.g. \cite[Proposition 3.3]{GHS09}) there exists $g\in \widetilde{\Orth}^+(M)$ such that $g(\lambda)=(0,0,\sigma_{\ell_1}(\ell_2),1,0)$. Therefore, $F$ also vanishes on $(0,0,\sigma_{\ell_1}(\ell_2),1,0)^\perp$, from which it follows that $f(0,\sigma_{\ell_1}(\ell_2))=1$ and that $\sigma_{\ell_1}(\ell_2)\in \mathcal{R}$. This shows that $\mathcal{R}$ is a (rescaled) root system.

Now we can decompose $\mathcal{R}$ into rescaled irreducible root systems
$$
\mathcal{R}=\bigoplus_{j=1}^s \mathcal{R}_j(a_j),
$$
where $\mathcal{R}_j$ are standard irreducible root systems and $a_j$ are rational numbers by which the (squared) norms of the roots in $\mathcal{R}_j$ are rescaled. Since $(0,0,\ell,1,0)^\perp$ is reflective for any $\ell \in \mathcal{R}$, there exists a positive integer $d$ such that $\ell^2=2/d$ and $d \ell \in L$. It follows that $a_j=1/k_j$ for some positive integers $k_j$. In the symmetric case, there are no $2$-roots in $\mathcal{R}$, so every $k_j$ is larger than $1$. Let $\mathfrak{g}_{j,k_j}$ be the simple Lie algebra of type $\mathcal{R}_j$ with level $k_j$. We see from \eqref{eq:A} and \eqref{eq:B} that the semi-simple Lie algebra $\mathfrak{g}=\oplus_{j=1}^s \mathfrak{g}_{j,k_j}$ satisfies condition $(1)$ or $(1')$ of the theorem. 

By \cite[Corollary 4.5]{Wan19}, $L(C)$ is an integral lattice. The set $\mathcal{R} = \oplus_{j=1}^s \mathcal{R}_j(1/k_j)$ generates a sublattice of $L'$ over $\ZZ$, from which it follows that
$$
\bigoplus_{j=1}^s Q_j(1/k_j) < L' \quad \text{and therefore}  \quad L<\bigoplus_{j=1}^s P_j^\vee(k_j).
$$
The vectors $\ord(\ell) \ell$ for $\ell \in \mathcal{R}$ generate a sublattice of $L$ over $\ZZ$, where $\ord(\ell)$ is the order of $\ell$ in $L'/L$. It is easy to check that this lattice contains the lattice generated by long roots of $\oplus_{j=1}^s \mathcal{R}_j(k_j)$, and therefore $\oplus_{j=1}^s Q_j^\vee(k_j) < L$. This proves claim (2).

Finally, let $\Delta_j^+$ be the set of positive roots of $\mathcal{R}_j$ and let $\latt{-,-}$ be the normalized bilinear form on $\mathcal{R}_j$ that was fixed in Section \ref{subsec:theta-blocks}. For $\alpha \in \Delta_j^+$, we have $\alpha / k_j \in \mathcal{R} \subsetneq L'$ and $$(\alpha/k_j, \mathfrak{z})=k_j\latt{\alpha/k_j, \mathfrak{z}}=\latt{\alpha, \mathfrak{z}}.$$
Therefore, the leading Fourier--Jacobi coefficient of $F$ is a Jacobi form of lattice index $L(C)$ given by
$$
\eta(\tau)^{\rank(L)} \prod_{0<\ell \in \mathcal{R}} \frac{\vartheta(\tau, (\ell, \mathfrak{z}))}{\eta(\tau)} = \prod_{j=1}^s \Big( \eta(\tau)^{\rank(\mathfrak{g}_{j})} \prod_{\alpha \in \Delta_j^+}\frac{\vartheta(\tau, \latt{\alpha, \mathfrak{z}})}{\eta(\tau)} \Big),
$$
which is actually the denominator of $\hat{\mathfrak{g}}$. Note that 
$$
\bigoplus_{j=1}^s Q^\vee_j(h_j^\vee)<L(C)< \bigoplus_{j=1}^s P^\vee_j(h_j^\vee),
$$
which matches the index $L(C)$ of the above leading Fourier--Jacobi coefficient of $F$. 

The last assertion follows from the central charge formula
\begin{align*}
    c_\mathfrak{g} = \sum_{j=1}^s \frac{k_j \dim \mathfrak{g}_{j}}{ k_j+h_j^\vee } = \sum_{j=1}^s \frac{\dim \mathfrak{g}_{j}}{1 + C} = \frac{\dim \mathfrak{g}}{1+C}
\end{align*}
and because $\dim \mathfrak{g} = 24(C+1)$ in the anti-symmetric case and $\dim \mathfrak{g} = 24C$ in the symmetric case. 
\end{proof}

\begin{remark}
In the above theorem, when $\mathcal{R}_j$ is of type $E_8$, $F_4$ or $G_2$, we have $Q_j^\vee = P_j^\vee$. Let $J$ denote the subset of such $j$. Then we have (see \cite[Theorem 2.2 (4)]{Wan23a})
$$
L=K\oplus \bigoplus_{j\in J}P_j^\vee(k_j), \quad \text{where} \quad \bigoplus_{j\neq J}Q_j^\vee(k_j) < K < \bigoplus_{j\neq J}P_j^\vee(k_j). 
$$  
\end{remark}

\begin{proposition}\label{cor:solutions}
Equation \eqref{antiC} has $221$ solutions, and they are listed in Table \ref{table:anti1}. Equation \eqref{symC} has $17$ solutions, and they are listed in Table \ref{table:sym}.     
\end{proposition}
\begin{proof}
Equation \eqref{antiC} was first derived by Schellekens \cite{Sch93} in 1993 in the context of holomorphic vertex operator algebras of central charge $24$, and he determined the solutions that are listed in Table \ref{table:anti1}. The solutions of Equation \eqref{symC} can be found in a similar way.
\end{proof}

%% file: chap5.tex
\chapter[The classification of hyperbolizations]{The classification of affine Lie algebras with hyperbolizations}\label{sec:main-results}

In this chapter, we state the main theorems and explain the ideas behind the proofs. 

Our first main result is the classification of root systems associated with reflective Borcherds products of singular weight on lattices of type $2U\oplus L$, which was introduced in Theorem \ref{th:root-system}. 

\begin{theorem}\label{th:non-existence}
Suppose $2U\oplus L$ has a reflective Borcherds product of singular weight whose Jacobi form input has non-negative $q^0$-term. Then the associated semi-simple Lie algebra $\mathfrak{g}$ defined in Theorem \ref{th:root-system} is one of the $81$ Lie algebras colored blue in Tables \ref{table:anti1} and \ref{table:sym}.   
\end{theorem}

To prove this, we have to rule out most of the solutions of Equation \eqref{antiC} and Equation \eqref{symC}. There are $238$ root systems in Tables \ref{table:anti1} and \ref{table:sym} altogether. A root system will be called \textit{extraneous} if it is not one of the $81$ root systems described in Theorem \ref{th:non-existence}; there are $152$ extraneous root systems of anti-symmetric type and $5$ extraneous root systems of symmetric type.

Let $\mathfrak{g}$ be a semi-simple Lie algebra associated with a singular-weight reflective Borcherds product on $2U\oplus L$ that is extraneous. By Theorem \ref{th:root-system} (2), the underlying lattice $L$ satisfies the bounds
$$
\bQ < L < \bP. 
$$
We will determine an explicit even overlattice $K$ of $L$ with the property that any singular-weight reflective Borcherds product on $2U\oplus L$ lifts to a reflective Borcherds product satisfying certain conditions on $2U\oplus K$, but for which no such reflective products on $2U\oplus K$ exist. This leads to a contradiction. The complete proof appears in Chapter \ref{sec:exclude-anti-root} for anti-symmetric root systems and in Chapter \ref{sec:exclude-root} for symmetric root systems.

The second main result is the construction of hyperbolizations for the $81$ semi-simple Lie algebras in the theorem above. 

\begin{theorem}\label{th:construction}
For each semi-simple Lie algebra $\mathfrak{g}$ in Theorem \ref{th:root-system}, there exists an even positive definite lattice $L_{\mathfrak{g}}$ with the same rank as $\mathfrak{g}$ that satisfies the following conditions:
\begin{enumerate}
\item There is a singular-weight reflective Borcherds product $\Psi_{\mathfrak{g}}$ on $2U\oplus L_{\mathfrak{g}}$;
\item The leading Fourier--Jacobi coefficient of $\Psi_{\mathfrak{g}}$ at the $1$-dimensional cusp determined by $2U$ is the denominator of the affine Lie algbera $\hat{\mathfrak{g}}$;
\item The Jacobi form input $\phi_\mathfrak{g}$ of $\Psi_{\mathfrak{g}}$ is a $\ZZ$-linear combination of full characters of the affine vertex operator algebra generated by $\hat{\mathfrak{g}}$;
\item The lattice generated by $d\rho$ and by those $\lambda=(n,\ell,m)\in U\oplus L_{\mathfrak{g}}'$ for which $\Psi_{\mathfrak{g}}$ vanishes on $\lambda^\perp$ is $U\oplus L_{\mathfrak{g}}'$; that is, $U\oplus L_{\mathfrak{g}}'$ is the root lattice of the BKM (super)algebra with (super)-denominator $\Psi_{\mathfrak{g}}$. Here, $\rho$ is the Weyl vector of $\Psi_{\mathfrak{g}}$ and $d$ is the denominator of $C$ as defined in Theorem \ref{th:root-system}. 
\item When $\mathfrak{g}$ is of symmetric type, the Borcherds product $\Psi_{\mathfrak{g}}$ coincides with the Gritsenko (additive) lift of the denominator of $\hat{\mathfrak{g}}$. 
\end{enumerate}
The above lattices $L_{\mathfrak{g}}$ and Jacobi forms $\phi_\mathfrak{g}$ are constructed as follows.
\begin{itemize}
\item[(a)] Let $\mathfrak{g}$ be one of the $69$ semi-simple Lie algebras of anti-symmetric type. In this case, $\mathfrak{g}$ appears as the semi-simple $V_1$ structure of a holomorphic vertex operator algebra $V$ of central charge $24$ in Schellekens’s list. Then $L_\mathfrak{g}$ is the orbit lattice in H\"{o}hn's construction of $V$ and $\phi_\mathfrak{g}$ is the unique full character of $V$.
\item[(b)] Let $\mathfrak{g}$ be one of the $8$ semi-simple Lie algebras of symmetric type with $C=1$. In this case, $\mathfrak{g}$ appears as the $\mathcal{N}=1$ structure of a holomorphic vertex operator superalgebra of central charge $12$ composed of $24$ chiral fermions. Then $L_\mathfrak{g}$ is the maximal even sublattice of $\bP$, and $\phi_\mathfrak{g}$ can be expressed as a $\ZZ$-linear combination of the full $\mathrm{NS}$-, $\widetilde{\mathrm{NS}}$- and $\mathrm{R}$-characters of the associated SVOA:
$$
\phi_{\mathfrak{g}}=(\chi_{\mathrm{NS}} - \chi_{\widetilde{\mathrm{NS}}} - \chi_{\mathrm{R}})/2.  
$$
\item[(c)] Let $\mathfrak{g}=A_{1,16}$, $A_{1,8}^2$, $A_{1,4}^4$ or $A_{2,9}$. These semi-simple Lie algebras are of symmetric type with $C<1$. In these cases, the affine Lie algebras $\hat{\mathfrak{g}}$ admit exceptional modular invariants coming from a nontrivial automorphism of the fusion algebra. Then $L_\mathfrak{g}=\bP$, the expressions of $\phi_\mathfrak{g}$ in terms of affine characters are given in Theorem \ref{th:input-character}, and the relationship between $\phi_\mathfrak{g}$ and the exceptional modular invariants is explained in Remark \ref{rem:exceptional-modular-invariants}. 
\end{itemize}
\end{theorem}

The proof of Theorem \ref{th:construction} is divided into four chapters. In Chapters \ref{sec:construct-anti}-\ref{sec:construct-symmetric-C<1} we prove parts (a), (b) and (c) respectively. In Chapter \ref{sec:additive lifts}, we prove property (5). The connections between our construction and the twists of the fake monster algebra are also explained in Chapters \ref{sec:construct-anti} and \ref{sec:construct-symmetric-C=1}. 

Our last main result is the complete classification of affine Lie algebras with hyperbolizations. This is a direct consequence of Theorem \ref{th:non-existence} and Theorem \ref{th:construction}. 

\begin{theorem}\label{th:classification}
There are exactly $81$ affine Kac--Moody algebras which have a hyperbolization in the sense of Definition \ref{def:hyperbolization}. The $81$ algebras are colored blue in Tables \ref{table:anti1} and \ref{table:sym}. 
\end{theorem}

Combining Theorem \ref{th:non-existence} and Theorem \ref{th:construction}, we obtain the following constraint on the existence of singular automorphic products:

\begin{corollary}
There are reflective Borcherds products of singular weight on lattices of type $2U\oplus L$ whose input has non-negative principal part if and only if 
$$
\rank(L) \in \{1, 2, 4, 6, 8, 10, 12, 16, 24\}. 
$$
\end{corollary}

%% file: chap6.tex
\chapter[Antisymmetric case: holomorphic CFTs of central charge 24]{The antisymmetric case: holomorphic CFTs of central charge 24}\label{sec:construct-anti}

The $69$ semi-simple Lie algebras $\mathfrak{g}$ in the anti-symmetric case of Theorem \ref{th:non-existence} coincide with the semi-simple $V_1$ structures of holomorphic vertex operator algebras (VOA) of central charge $24$ in Schellekens' list \cite{Sch93}. In 2017, H\"{o}hn \cite{Hoh17} found a uniform construction of the holomorphic VOAs. For each $\mathfrak{g}$, we will take H\"{o}hn's orbit lattice $L_\mathfrak{g}$ as the underlying lattice $L$ and construct the reflective Borcherds product of singular weight as the Borcherds lift of the full character of the holomorphic VOA. We also explain the connection between these hyperbolizations and the twisted denominators of the fake monster algebra.

\section{Holomorphic CFTs of central charge 24 and Schellekens' list}
Let $V$ be a holomorphic vertex operator algebra of central charge $24$. The weight-one subspace $V_1$ has a natural Lie algebra structure and by \cite{Sch93, DM04}, $V_1$ is either trivial, abelian of dimension $24$, or semi-simple. In the first case, it was conjectured in \cite{FLM88} that $V$ is isomorphic to the monster VOA, which was proved under some conditions in \cite{dong2007uniqueness}. In the second case, $V$ is isomorphic to the Leech lattice VOA. 

We focus on the third case where $V_1$ is a semi-simple Lie algebra $\mathfrak{g}$. Let $(-,-)$ be the unique symmetric, non-degenerate, invariant bilinear form on $V$ normalized such that $(\mathbf{1},\mathbf{1})=-1$, where $\mathbf{1}$ is the vacuum vector. The restriction of $(-,-)$ to a simple ideal $\mathfrak{g}_j$ of $\mathfrak{g}$ satisfies $(-,-)=k_j\latt{-,-}$ for some positive integer $k_j$, where $\latt{-,-}$ is the normalized bilinear form of $\mathfrak{g}_j$ (see \cite{Sch93, DM06}).  We indicate
these integers by writing
$$
\mathfrak{g}=\mathfrak{g}_{1,k_1}\oplus \cdots \oplus \mathfrak{g}_{s,k_s}.
$$
Then the affine vertex operator algebra 
$$
V_\mathfrak{g}\cong L_{\hat{\mathfrak{g}}_1}(k_1,0)\otimes \cdots \otimes L_{\hat{\mathfrak{g}}_s}(k_s,0)
$$
generated by $V_1$ is a full vertex operator subalgebra of $V$. As a $V_\mathfrak{g}$-module, one can decompose $V$ into finitely many irreducible $V_\mathfrak{g}$-modules
\begin{equation}\label{eq:V-decomposition}
   V=\bigoplus_{\lambda_1,...,\lambda_s} m_{\lambda_1,...,\lambda_s} L_{\hat{\mathfrak{g}}_1}(k_1,\lambda_1)\otimes \cdots \otimes L_{\hat{\mathfrak{g}}_s}(k_s,\lambda_s),
\end{equation}
where the sum runs over the dominant integral weights $\lambda_j$ of $\mathfrak{g}_j$ that satisfy the inequality $\latt{\lambda_j, \theta_j^\vee}\leq k_j$, where $\theta_j$ is the highest root of $\mathfrak{g}_j$.

In 1993 Schellekens \cite{Sch93} established the equation (see Notation \ref{notation})
$$
\frac{\dim \mathfrak{g}}{24} -1 = \frac{h_j^\vee}{k_j}, \quad \text{for $1\leq j\leq s$}
$$
and showed that it has exactly $221$ solutions. By excluding $152$ of them, Schellekens proved that $V_1$ is isomorphic to one of the $69$ semi-simple Lie algebras in \cite[Table 1]{Sch93} (named \textit{Schellekens's list}) and further determined the induced decompositions \eqref{eq:V-decomposition}. This result was reproved in \cite{EMS20} using similar arguments and in \cite{ELM21} by means of the ``very strange formula". 

Schellekens also conjectured that there exists a unique holomorphic VOA of central charge $24$ with $V_1=\mathfrak{g}$ for every $\mathfrak{g}$ on his list. By the work of many authors over the past three decades, this celebrated conjecture was finally proved (see \cite{Lam11, LS12, LS16a, LS16b, LL20, SS16, EMS20} for the existence and \cite{DM04b, EMS20b, LL20, KLL18, LS15, LS19, LS20, LS22} for the uniqueness). There are at least three uniform proofs of the conjecture: the Leech lattice method of H\"{o}hn (\cite{Hoh17, Lam19, BLS22}), the generalized deep hole method of M\"oller--Scheithauer \cite{SM19, SM21}, and the Niemeier lattice method of H\"{o}hn--M\"oller \cite{HM20}.

\section{H\"{o}hn's construction of holomorphic CFT of central charge 24} 
Our construction of hyperbolizations relies heavily on H\"{o}hn's argument, which is reviewed below.

Let $\mathbf{Q}_\mathfrak{g}$ be the rescaled coroot lattice (see Notation \ref{notation}). Then $V$ contains the lattice VOA associated with $\mathbf{Q}_\mathfrak{g}$ denoted $V_{\mathbf{Q}_\mathfrak{g}}$. 
If we let $W_\mathfrak{g}=\mathrm{Com}_V(V_{\mathbf{Q}_\mathfrak{g}})$ be the commutant (or centralizer),  then the double commutant 
$$
\mathrm{Com}_V(W_\mathfrak{g})=\mathrm{Com}_V(\mathrm{Com}_V(V_{\mathbf{Q}_\mathfrak{g}}))
$$
is a lattice VOA extending $V_{\mathbf{Q}_\mathfrak{g}}$. Therefore, there exists an even positive definite lattice $L_\mathfrak{g}\supset \mathbf{Q}_\mathfrak{g}$ such that $\mathrm{Com}_V(W_\mathfrak{g})$ is isomorphic to the lattice VOA $V_{L_\mathfrak{g}}$. It is well known that $V_{L_\mathfrak{g}}$ has group-like fusion: all irreducible $V_{L_\mathfrak{g}}$-modules are simple current modules. In this case, the set of all irreducible modules $R(V_{L_\mathfrak{g}})$ forms an abelian group with respect to the fusion product, and it also carries the quadratic form $\mathfrak{q}: R(V_{L_\mathfrak{g}})\cong L_\mathfrak{g}'/L_\mathfrak{g} \to \QQ/\ZZ$ defined via
$$
\mathfrak{q}(V_{\alpha + L_\mathfrak{g}}) = \mathrm{wt}(V_{\alpha + L_\mathfrak{g}}) = (\alpha,\alpha)/2 \mod \ZZ,
$$
where $\mathrm{wt}(-)$ denotes the conformal weight of the module. The full character of the irreducible module labelled by $\alpha+L_\mathfrak{g} \in L_\mathfrak{g}'/L_\mathfrak{g}$ can be expressed in terms of the Jacobi theta function and the $\eta$-function:
\begin{equation}
\chi_{V_{\alpha+L_\mathfrak{g}}}(\tau, \mathfrak{z}) = \frac{\Theta_{L_\mathfrak{g},\alpha}(\tau,\mathfrak{z})}{\eta(\tau)^{\rank(L_\mathfrak{g})}}, \quad (\tau,\mathfrak{z})\in \HH \times (L_\mathfrak{g}\otimes\CC).    
\end{equation}

It is also known that $W_\mathfrak{g}$ is strongly rational and also has group-like fusion (see e.g. \cite{CKM22, Lin17}). The set $R(W_\mathfrak{g})$ of irreducible modules forms a quadratic space isomorphic to $(R(V_{L_\mathfrak{g}}), -\mathfrak{q})$, where the quadratic form is defined by reducing the conformal weight modulo $\ZZ$. Therefore, $V$ is a simple current extension of $W_\mathfrak{g}\otimes V_{L_\mathfrak{g}}$, i.e. there is an isometry $\iota: L_\mathfrak{g}'/L_\mathfrak{g} \to R(W_\mathfrak{g})$ such that
$$
V \cong \bigoplus_{\alpha + L_\mathfrak{g} \in L_\mathfrak{g}'/L_\mathfrak{g}} W_{\iota(\alpha + L_\mathfrak{g})}\otimes V_{\alpha + L_\mathfrak{g}}.
$$
Note that $V_{L_\mathfrak{g}}$ has central charge $\rank(L_\mathfrak{g})=\rank(V_1)=\rank(\mathfrak{g})$, that $W_\mathfrak{g}$ has central charge $24-\rank(\mathfrak{g})$, and that the weight-one subspace of $W_\mathfrak{g}$ is zero. H\"{o}hn computed the lattice $L_\mathfrak{g}$ for each of the $69$ semi-simple $\mathfrak{g}$ and called it the \textit{orbit lattice} of $\mathfrak{g}$. 

Let $\mathfrak{h}$ be the Cartan subalgebra of $\mathfrak{g}$. Recall that the full character of $V$,  defined by
\begin{equation}
\chi_{V}(\tau, \mathfrak{z}) = \mathrm{Tr}_V (e^{2\pi i\mathfrak{z}} q^{L_0-1}), \quad q=e^{2\pi i\tau}, \quad (\tau, \mathfrak{z})\in\HH\times \mathfrak{h},     
\end{equation}
is a weakly holomorphic Jacobi form of weight $0$ and lattice index $\mathbf{Q}_\mathfrak{g}$ (with trivial character on $\SL_2(\ZZ)$) (see e.g. \cite{Zhu96, Miy00, KM12}). The above construction implies that $\chi_V$ also defines a weakly holomorphic Jacobi form of weight $0$ and lattice index $L_\mathfrak{g}$. Moreover, the character decomposition 
\begin{equation}\label{eq:full-theta}
\chi_V(\tau, \mathfrak{z}) = \sum_{\alpha+L_\mathfrak{g} \in L_\mathfrak{g}'/L_\mathfrak{g}} \chi_{V_{\alpha+L_\mathfrak{g}}}(\tau,\mathfrak{z}) \cdot \chi_{W_{\iota(\alpha+L_\mathfrak{g})}}(\tau)    
\end{equation}
determines the theta decomposition of $\chi_V$ as a Jacobi form of lattice index $L_\mathfrak{g}$. In particular, the vector-valued function 
\begin{equation}\label{eq:character-orbifold}
f_\mathfrak{g}(\tau):= \sum_{\alpha \in L_\mathfrak{g}'/L_\mathfrak{g}} \eta(\tau)^{-\rank(\mathfrak{g})}\cdot\chi_{W_{\iota(\alpha+L_\mathfrak{g})}}(\tau) e_\alpha     
\end{equation}
is a weakly holomorphic modular form of weight $-\rank(\mathfrak{g})/2$ for the Weil representation $\rho_{L_\mathfrak{g}}$.

H\"{o}hn described $W_\mathfrak{g}$ in an elegant way. Let $\Lambda$ continue to be the Leech lattice and let
$$
\Lambda_g:=(\Lambda^g)^\perp=\{ x\in \Lambda: (x,y)=0 \quad \text{for all $y\in \Lambda$ satisfying $g(y)=y$} \}
$$
be the coinvariant sublattice for $g\in \mathrm{Co}_0=\Orth(\Lambda)$. 
Let $\hat{g}\in \mathrm{Aut}(V_\Lambda)$ be a standard lift of $g$; that is, $\hat{g}$ is a lift of $g$ to the Leech lattice VOA that acts trivially on $\Lambda^g$. H\"{o}hn \cite{Hoh17} conjectured that $W_\mathfrak{g}$ is isomorphic to an orbifold VOA $V_{\Lambda_g}^{\hat{g}}$ where $[g]$ is one of $11$ particular conjugacy classes $[g]$ of $\mathrm{Co}_0$ (see Tables \ref{table:8cycleshapes} and \ref{table:3cycleshapes}). This would imply that $V$ is isomorphic to a simple-current extension of $V_{\Lambda_g}^{\hat{g}}\otimes V_{L_\mathfrak{g}}$. Lam \cite{Lam19} proved that the orbifold $V_{\Lambda_g}^{\hat{g}}$ with $\Lambda^g\neq \{0\}$ has group-like fusion and described the corresponding quadratic form explicitly, and in this way was able to confirm H\"{o}hn's conjecture. The simple-current extension of $V_{\Lambda_g}^{\hat{g}}\otimes V_{L_\mathfrak{g}}$ depends on the double cosets
\begin{equation}\label{eq:double-coset}
\Orth(L_\mathfrak{g}) \backslash \Orth(R(W_\mathfrak{g}),-\mathfrak{q}) / \mathrm{Aut}(W_\mathfrak{g}).    
\end{equation}
By counting the numbers of the above double cosets denoted $n(L_\mathfrak{g})$, Betsumiya, Lam and Shimakura \cite{BLS22} supplemented H\"{o}hn's proof of the uniqueness of holomorphic VOAs of central charge $24$ with semi-simple $V_1$. 

We will now describe the $69$ simple current extensions explicitly. The construction involves 11 distinct $\mathrm{Co}_0$-conjugacy classes $[g]$, which are characterized by the property that the discriminant form $(R(V_{\Lambda_g}^{\hat{g}}), -\mathfrak{q})$ of signature $\rank(\Lambda^g) \mod 8$ can be realized by an even positive definite lattice of rank $\rank(\Lambda^g)$. 
The orbit lattices $L_\mathfrak{g}$ lie in the $11$ associated genera.
\begin{enumerate}
    \item Let $[g]$ be one of the $8$ conjugacy classes in Table \ref{table:8cycleshapes}. For each orbit lattice of genus $[g]$, the number of double cosets \eqref{eq:double-coset} is always $1$, so the number of inequivalent simple-current extensions equals to the number of classes in the genus. There are in total $58$ classes (including the Leech lattice) in the $8$ genera, so the $8$ conjugacy classes $[g]$ induce $57$ holomorphic VOAs of central charge $24$ with semi-simple $V_1=\mathfrak{g}$. 
    \item Let $[g]$ be the conjugacy class of cycle shape $2^{12}$.  The associated (lattice) genus has $2$ classes $D_{12}(2)$ and $E_8(2)\oplus D_4(2)$ and the corresponding numbers of double cosets are respectively $6$ and $3$. Therefore, this $[g]$ induces $9$ simple-current extensions.
    \item Let $[g]$ be the conjugacy class of cycle shape $2^3 6^3$.  The associated (lattice) genus has a unique class and the corresponding number of double cosets is $2$. Therefore, this $[g]$ induces $2$ simple current extensions.
    \item Let $[g]$ be the conjugacy class of cycle shape $2^2 10^2$. The associated (lattice) genus has a unique class and the corresponding number of double cosets is $1$. Therefore, this $[g]$ induces only $1$ simple current extension.
\end{enumerate}
The relevant data appear in Tables \ref{table:8cycleshapes} and \ref{table:3cycleshapes}.

\section{Constructing hyperbolizations} 

\begin{theorem}\label{th:VOA}
Let $V$ be a holomorphic VOA of central charge $24$ with semi-simple $V_1=\mathfrak{g}$. Let $L_\mathfrak{g}$ be H\"{o}hn's orbit lattice of $\mathfrak{g}$ and let $\chi_V$ be the full character of $V$. Then the theta lift $\Borch(\chi_V)$ is a reflective Borcherds product of singular weight on $2U\oplus L_\mathfrak{g}$. Moreover, the leading Fourier--Jacobi coefficient of $\Borch(\chi_V)$ coincides with the denominator of the affine Kac--Moody algebra $\hat{\mathfrak{g}}$. 
\end{theorem}
\begin{proof}
Let $\Delta_\mathfrak{g}$ be the set of roots of $\mathfrak{g}$. 
By definition, the Fourier expansion of $\chi_V$ begins 
$$
\chi_V(\tau,\mathfrak{z}) = q^{-1} + \sum_{\alpha\in \Delta_\mathfrak{g}} e^{2\pi i\latt{\alpha, \mathfrak{z}}} + \rank(\mathfrak{g}) + O(q),
$$
and all Fourier coefficients of $\chi_V$ are non-negative integers. It follows from H\"{o}hn's construction that $\chi_V$ is a weakly holomorphic Jacobi form of weight $0$ and lattice index $L_\mathfrak{g}$. We conclude immediately that $\Borch(\chi_V)$ is a holomorphic Borcherds product of singular weight on $2U\oplus L_\mathfrak{g}$. It remains to show that $\Borch(\chi_V)$ is reflective. 

Suppose first that $\mathfrak{g}$ comes from one of the conjugacy classes $[g]$ in Table \ref{table:8cycleshapes}. Then the level and order of $g$ are the same (denoted $n_g$ as in Section \ref{subsec:BKM}). The orbit lattice $L_\mathfrak{g}$ is determined by 
$$
U(n_g) \oplus \Lambda^g \cong U\oplus L_\mathfrak{g}. 
$$
As mentioned at the end of Section \ref{subsec:BKM}, it was proved in \cite[Theorem 6.5, Remark 6.14]{WW22} that the twisted denominator of the fake monster algebra associated with $g$ defines a reflective Borcherds product $\Phi_g$ of singular weight on $U(n_g)\oplus U\oplus \Lambda^g$. M\"{o}ller \cite{Mol21} calculated the characters of the orbifold VOA $V_{\Lambda_g}^{\hat{g}}$, and by comparing M\"{o}ller's result and \cite[Theorem 6.5]{WW22}, it was also proved \cite[Remark 6.13]{WW22} that the input of $\Phi_g$ equals the full character of $V_{\Lambda_g}^{\hat{g}}$ (up to a factor of $\eta^{-\rank(\mathfrak{g})}$) as a vector-valued modular form. By H\"{o}hn's construction, the input of $\Phi_g$ is actually $f_\mathfrak{g}$ as defined in Equation \eqref{eq:character-orbifold}. We can then use Equation \eqref{eq:full-theta} to see that $\Borch(\chi_V)=\Phi_g$; in particular, $\Borch(\chi_V)$ is reflective.

Otherwise, $\mathfrak{g}$ comes from one of the conjugacy classes $[g]$ in Table \ref{table:3cycleshapes}, the level of $g$ is twice its order, and the isomorphism $U(n_g) \oplus \Lambda^g \cong U\oplus L_\mathfrak{g}$ does not hold. In fact, $\Borch(\chi_V)$ is not obviously related to $\Phi_g$ in these cases. We will prove that $\Borch(\chi_V)$ is reflective by directly calculating the input forms $\chi_V$. By H\"{o}hn's construction, it is sufficient to do this for any semi-simple $\mathfrak{g}$ from a fixed $[g]$, because the input for any other semi-simple Lie algebra from $[g]$ can be expressed as $\sigma(f_\mathfrak{g})$ for some $\sigma \in \Orth(L_\mathfrak{g}'/L_\mathfrak{g})$. The proof by cases is given in Lemmas \ref{lem:B12}, \ref{lem:F4A2} and \ref{lem:C4} below. 
\end{proof}

\begin{remark}
All products $\Borch(\chi_V)$ for $V_1=\mathfrak{g}$ within a fixed conjugacy class $[g]$ define the same modular form on the type IV symmetric domain of dimension $2+\rank(\mathfrak{g})$ up to automorphism. We denote this modular form by $\Psi_g$. The different products $\Borch(\chi_V)$ can be viewed as the distinct Fourier--Jacobi expansions of $\Psi_g$ at different $1$-dimensional cusps represented by a splitting of two hyperbolic planes. Letting $D_g$ denote the discriminant group of $L_\mathfrak{g}$, the automorphism group $\mathrm{Aut}(V_{\Lambda_g}^{\hat{g}})$ equals the subgroup of $\Orth(D_g)$ that fixes $f_\mathfrak{g}$. If $g$ is in Table \ref{table:8cycleshapes}, then 
$$
\mathrm{Aut}(V_{\Lambda_g}^{\hat{g}})=\Orth(D_g),
$$
and therefore $\Psi_g$ is modular for the full orthogonal group. This property does not hold if $g$ is in Table \ref{table:3cycleshapes}. 
\end{remark}

\begin{remark}
By the proof of Theorem \ref{th:VOA}, $\Psi_g=\Phi_g$ for $[g]$ in Table \ref{table:8cycleshapes}. The eight $\Phi_g$ were first constructed by Borcherds \cite{Bor95, Bor98} and Scheithauer in \cite{Sch04, Sch09}, but they only considered the Fourier expansions at $0$-dimensional cusps. Gritsenko \cite{Gri12} first calculated the $24$ distinct Fourier--Jacobi expansions of $\Phi_g$ for $g$ of cycle shape $1^{24}$ at the distinct $1$-dimensional cusps, which correspond to the $24$ Niemeier lattices. The Fourier--Jacobi expansions of $\Phi_g$ for $g$ of cycle shapes $1^3 7^3$ and $1^4 5^4$ were first determined in \cite{GW19}. We also remark that Scheithauer \cite{Sch15} calculated $\Psi_g$ at $0$-dimensional cusps for $g$ of cycle shape $2^{12}$.  
\end{remark}

\begin{remark}
Borcherds \cite{Bor90, Bor92} proved that the BRST cohomology related to the Leech lattice VOA naturally defines the fake monster algebra. It is natural to expect that the BRST cohomology related to any holomorphic VOA of central charge $24$ defines a BKM algebra whose denominator function is $\Borch(\chi_V)$. This type of construction has been realized for holomorphic vertex operator algebras with $V_1$ structures $A_{1,2}^{16}$, $A_{2,3}^8$, $A_{4,5}^2$, $A_{6,7}$ and $B_{12,2}$ by 
Creutzig, H\"{o}hn, Klauer and Scheithauer \cite{HS03, CKS07, HS14}. A uniform construction for all vertex operator algebras was given by Driscoll-Spittler in his thesis \cite{DS22} under several technical assumptions. This (conditional) natural construction also implies that $\Borch(\chi_V)$ is reflective on $2U\oplus L_\mathfrak{g}$. 
\end{remark}

\begin{remark}\label{rem:isomorphic}
Let $\mathcal{G}_\mathfrak{g}$ be the BKM algebra that arises as the BRST cohomology of a holomorphic VOA with $V_1=\mathfrak{g}$. Let $[g]$ be the $\mathrm{Co}_0$-conjugacy class corresponding to $\mathfrak{g}$. Clearly $\mathcal{G}_\mathfrak{g}$ depends only on $[g]$. 
\begin{enumerate}
\item When $[g]$ is in Table \ref{table:8cycleshapes}, we have the two isomorphisms  
$$
U(n_g) \oplus \Lambda^g \cong U\oplus L_\mathfrak{g} \quad \text{and} \quad U\oplus \Lambda^g \cong (U \oplus L_\mathfrak{g}')(n_g),
$$
since $(\Lambda^g)'(n_g)\cong \Lambda^g$. 
This induces the identifications
\begin{equation}\label{eq:groups}
\begin{split}
&\Orth\!\big(U_1(n_g)\oplus U\oplus \Lambda^g\big) \cong  \Orth\!\big(U_1(n_g)\oplus (U\oplus L_\mathfrak{g}')(n_g)\big) \\
=& \Orth\!\big(U_1\oplus U\oplus L_\mathfrak{g}'\big) = \Orth\!\big(U_1\oplus U\oplus L_\mathfrak{g}\big).
\end{split}
\end{equation}
It follows that the Fourier expansion of $\Borch(\chi_V)$ at $U_1$ equals the Fourier expansion of $\Phi_g$ at $U_1(n_g)$. Recall that the latter is also identical to the denominator of the twist $\mathbb{G}_g$ of the fake monster algebra by $g$ (see Section \ref{subsec:BKM}). Therefore, the denominator of $\mathcal{G}_\mathfrak{g}$ is the same as the denominator of $\mathbb{G}_g$. The root lattice of 
$\mathbb{G}_g$ is always $U\oplus \Lambda^g$. From \eqref{eq:groups} 
we conclude that the root lattice of $\mathcal{G}_\mathfrak{g}$ is $U\oplus L_\mathfrak{g}'$, and moreover that the two BKM algebras $\mathcal{G}_\mathfrak{g}$ and $\mathbb{G}_g$ are isomorphic.  
\item When $[g]$ is in Table \ref{table:3cycleshapes}, $\mathcal{G}_\mathfrak{g}$ and $\mathbb{G}_g$ are not isomorphic because their root lattices are not even isomorphic up to scaling. In addition, the Weyl vector of $\mathbb{G}_g$ lies in the root lattice, but the Weyl vector of $\mathcal{G}_\mathfrak{g}$ does not lie in the root lattice (recall that $C$ is half-integral and the Weyl vector is of type $(-C-1,*,-C)$ in this case). For instance, when $g$ has cycle shape $2^{12}$, the root lattices of $\mathcal{G}_\mathfrak{g}$ and $\mathbb{G}_g$ are respectively
$$
U\oplus D_{12}'(1/2) \quad \text{and} \quad  U\oplus D_{12}^+(2).
$$
In particular, we disagree with \cite[Remark 3.11]{HS14}: the denominator there is not a twisted denominator function of the fake monster algebra. We conjecture that the denominator $\Phi_g$ of $\mathbb{G}_g$ is given by the Fourier expansion of $\Psi_\mathfrak{g}=\Borch(\chi_V)$ at some other $0$-dimensional cusp.
\end{enumerate}
\end{remark}

\begin{remark}
The $\ZZ$-lattice generated by $d\rho$ and the $\lambda=(n,\ell,m)\in U\oplus L_\mathfrak{g}'$ for which $\Borch(\chi_V)$ vanishes on $\lambda^\perp$ is exactly the dual lattice $U\oplus L_\mathfrak{g}'$. This follows from Remark \ref{rem:isomorphic} and a direct calculation of zero divisors (see Lemmas \ref{lem:B12}, \ref{lem:F4A2} and \ref{lem:C4} below). Here, $\rho$ is the Weyl vector of $\Borch(\chi_V)$ and $d=1$ or $2$ if the corresponding $[g]$ lies in Table \ref{table:8cycleshapes} or Table \ref{table:3cycleshapes}, respectively.    
\end{remark}

\begin{remark}\label{rem:Lam-paper}
Let $[g]$ be a $\mathrm{Co}_0$-conjugacy class of level $N_g$ and order $n_g$. Let $M_g$ be an even lattice of signature $(\rank(\Lambda^g)+2,2)$ whose discriminant form is isomorphic to $(R(V_{\Lambda_g}^{\hat{g}}),-\mathfrak{q})$, as determined by Lam \cite{Lam19}. Let $\Phi_g$ be the $g$-twisted denominator of the fake monster algebra as before, and recall that $\Phi_g$ is a holomorphic Borcherds product of singular weight on $U(N_g)\oplus U\oplus \Lambda^g$ (see \cite[Theorem 6.5]{WW22}). We have the following:
\begin{enumerate}
\item If $N_g=n_g$, then $M_g\cong U(N_g)\oplus U\oplus \Lambda^g$ (see \cite[Theorem 5.3]{Lam19}) and $\Phi_g$ is reflective on $M_g$ (see \cite[Remark 6.14]{WW22}).
\item If $N_g\neq n_g$, then $M_g$ is not isomorphic to $U(N_g)\oplus U\oplus \Lambda^g$ and $\Phi_g$ is not reflective on the lattice $U(N_g)\oplus U\oplus \Lambda^g$. Motivated by H\"{o}hn's construction and the discussions above,  we conjecture that the vector-valued characters of $V_{\Lambda_g}^{\hat{g}}$ (divided by $\eta^{\rank(\Lambda^g)}$) can be lifted to a reflective Borcherds product of singular weight on $M_g$. Obviously, the theta lift defines a holomorphic Borcherds product of singular weight on $M_g$. However, we have to compute the conformal weights to confirm that it is reflective.  We further conjecture that the Fourier expansion of this product at a certain $0$-dimensional cusp recovers the $g$-twisted denominator of the fake monster algebra. In other words, $\Phi_g$ defines a reflective Borcherds product on $M_g$, i.e. $M_g$ plays the role of the lattice $\mathbb{M}$ in \cite[Theorem 1.4]{WW23}.
\end{enumerate}
\end{remark}

We can now complete the proof of Theorem \ref{th:VOA}. By \eqref{eq:V-decomposition}, we can express $\chi_V$ as an $\NN$-linear combination of characters of the affine VOA generated by $V_1=\mathfrak{g}$ which have integral conformal weight. Expressions of this type have been determined by Schellekens \cite{Sch93}. We will compute the characters $\chi_V$ for $V_1=B_{12,2}$, $F_{4,6}A_{2,2}$ and $C_{4,10}$. In these cases, the index $[L_\mathfrak{g}:\mathbf{Q}_\mathfrak{g}]$ is $1$ or $2$.  To determine the divisor of $\Borch(\chi_V)$, it suffices to compute the singular Fourier coefficients of $\chi_V$, and by Remark \ref{rem:singular} it is enough to consider Fourier coefficients of the form
\begin{equation}\label{eq:check-list}
f(n,\ell)q^n\zeta^\ell, \quad n\leq \hat{\delta}_{L_\mathfrak{g}}, \; \ell \in \mathbf{Q}_\mathfrak{g}', \; 2n<(\ell,\ell).    
\end{equation}
Note that the above $f(n,\ell)$ are $0$ or $1$ because $f(n,\ell)\geq 0$ and $\Borch(\chi_V)$ has only simple zeros, hence $\Borch(\chi_V)$ is reflective on $2U\oplus L_\mathfrak{g}$ as soon as we can show that for every nonzero Fourier coefficient of form \eqref{eq:check-list} there exists a positive integer $t$ such that $(\ell,\ell)-2n=2/t$ and $t\ell \in L_\mathfrak{g}$. 

In the three lemmas below, we will express the Fourier expansion of $\chi_V$ in terms of Weyl orbits. Suppose $\mathfrak{g}=\mathfrak{g}_{1,k_1}\oplus \mathfrak{g}_{2,k_2}$, where $\mathfrak{g}_{2,k_2}$ may be zero, and let $W_{j}$ be the Weyl group of $\mathfrak{g}_j$. For a dominant integral weight $\lambda_j = \sum_{i} x_i w_i$ of $\mathfrak{g}_j$, we define 
$$
O_{\lambda_j, n_j} = \sum_{\ell \in W_{j}\cdot\lambda_j} e^{2\pi i\latt{\ell, \mathfrak{z}}}, \quad \mathfrak{z} \in L_\mathfrak{g}\otimes\CC,
$$
where $n_j=\latt{\lambda_j,\lambda_j}/k_j$. Recall (from Notation \ref{notation}) that
$$
\mathbf{Q}_\mathfrak{g} = Q_1^\vee(k_1)\oplus Q_2^\vee(k_2), \quad  \mathbf{Q}_\mathfrak{g}' = P_1(1/k_1)\oplus P_2(1/k_2)
$$ 
and 
$$
\mathbf{Q}_\mathfrak{g}< L_\mathfrak{g} < L_\mathfrak{g}' < \mathbf{Q}_\mathfrak{g}'.
$$
When we view $O_{\lambda_j, n_j}$ as part of the Fourier series of $\chi_V$, the vector $\ell \in W_j\cdot \lambda_j \subsetneq \bQ'$ is identified with $\ell / k_j$ and its norm is $(\ell,\ell)=k_j\latt{\ell/k_j,\ell/k_j}=n_j$. Therefore, $q^n \cdot O_{\lambda_1, n_1}\otimes O_{\lambda_2, n_2}$ induces reflective zeros if and only if there exists a positive integer $t$ such that 
$$
n_1+n_2-2n=2/t \quad \text{and} \quad t(\lambda_1/k_1 + \lambda_2/k_2)\in L_\mathfrak{g}.
$$
For simplicity, we set 
$$
O_{\lambda_1,\lambda_2, n_1+n_2}=O_{\lambda_1,n_1}\otimes O_{\lambda_2,n_2}.
$$
 
\begin{lemma}\label{lem:B12}
When $V_1=\mathfrak{g}$ is of type $B_{12,2}$, the product $\Borch(\chi_V)$ is reflective on $2U\oplus D_{12}(2)$. 
\end{lemma}
\begin{proof}
$\chi_V$ can be expressed in terms of characters of the affine VOA generated by $\mathfrak{g}$ (using the notation of \eqref{eq:symbol-char}) as
$$
\chi_V=\chi^{B_{12,2}}_{0,0}+\chi^{B_{12,2}}_{w_1+w_{12},2}+\chi^{B_{12,2}}_{w_{10},3}+\chi^{B_{12,2}}_{w_5,2}.
$$
In this case, $\mathbf{Q}_\mathfrak{g}=L_\mathfrak{g}=D_{12}(2)$. Note that $\delta_{L}=6$ (see \cite[Lemma 2.3]{HS14}), so we only need to calculate  $\chi_V$ up to its $q^2$-term.  We first use SageMath to compute the Fourier coefficients of $\chi^{B_{12,2}}$ as representations of the simple Lie algebra of type $B_{12}$, and then we decompose those representations into Weyl orbits as defined above. We find 
$$
\chi_V=q^{-1}+(O_{w_2,1}+O_{w_1,\frac12}+12)+\sum_{i=1}^\infty c_iq^i, 
$$ 
where
\begin{align*}
c_1=&\,O_{2w_2,4}+O_{w_1+w_3,3}+O_{w_5,\frac52}+O_{w_1+w_{12},\frac52}+O_{w_1+w_2,\frac52}+4O_{w_4,2}\\
&+12O_{2w_1,2}+12O_{w_3,\frac32}+12O_{w_{12},\frac32}+44O_{w_2,1}+90O_{w_1,\frac12}+300,\\
c_2=&\,O_{w_{10},5}+O_{2w_1+w_2,5}+O_{w_2+w_4,5}+O_{w_9,\frac92}+O_{3w_1,\frac92}+O_{w_2+w_3,\frac92}+O_{2w_1+w_{12},\frac92}\\
&+O_{w_3+w_{12},\frac92}+O_{w_1+w_{6},\frac92}+4 O_{w_8,4}+12 O_{2w_2 ,4}+4 O_{w_1 +w_5,4}+12 O_{w_7,\frac{7}{2}}\\
&+12 O_{w_2+w_{12},\frac{7}{2}}+12 O_{w_1+w_4,\frac{7}{2}}+32 O_{w_6,3}+44 O_{w_1+w_3,3}+90 O_{w_5,\frac{5}{2}}\\
&+90 O_{w_1 +w_{12},\frac{5}{2}}+90 O_{w_1 +w_{2},\frac{5}{2}}+224 O_{w_4,2}+288 O_{2w_1,2}\\
&+520 O_{w_{12},\frac{3}{2}}+520 O_{w_3,\frac{3}{2}}
+1242 O_{w_2,1}
+2535 O_{w_1,\frac{1}{2}}
+5792.
\end{align*}
The proof follows by verifying that every singular Weyl orbit is reflective. As an example, consider the orbit $q^2\cdot O_{w_9,9/2}$. We write $w_9=(1,1,1,1,1,1,1,1,1,0,0,0)$ as coordinates in the simple roots as in \cite{Bou82}. Then $9/2 - 2\times 2=1/2=2/4$ and $4\cdot(w_9/2)=2w_9 \in D_{12}(2)$, so $q^2\cdot O_{w_9,9/2}$ determines reflective divisors. 
\end{proof}

\begin{lemma}\label{lem:F4A2}
When $V_1=\mathfrak{g}$ is of type $A_{2,2}F_{4,6}$, the product $\Borch(\chi_V)$ is reflective on $2U\oplus D_4(6)\oplus A_2(2)$. 
\end{lemma}
\begin{proof}
We express $\chi_V$ as the linear combination
\begin{align*}
&\Big(\chi^{F_{4,10}}_{0000,0} + \chi^{F_{4,10}}_{0004,2} + \chi^{F_{4,10}}_{0030,3} + \chi^{F_{4,10}}_{1100,2}\Big)\otimes\chi^{A_{2,2}}_{00,0}\\
+&\Big(\chi^{F_{4,10}}_{0101,\frac{26}{15}} + \chi^{F_{4,10}}_{1012,\frac{41}{15}}\Big)\otimes\Big( \chi^{A_{2,2}}_{10,\frac{4}{15} }+ \chi^{A_{2,2}}_{01,\frac{4}{15}}\Big) \\
+&\Big(\chi^{F_{4,10}}_{0003,\frac75} + \chi^{F_{4,10}}_{0006,\frac{17}{5}} + \chi^{F_{4,10}}_{0021,\frac{12}{5}} + \chi^{F_{4,10}}_{2010,\frac{12}{5}}\Big)\otimes\chi^{A_{2,2}}_{11,\frac35}\\
+& \Big(\chi^{F_{4,10}}_{0102,\frac73} + \chi^{F_{4,10}}_{2000,\frac43}\Big)\otimes\Big(\chi^{A_{2,2}}_{02,\frac23} + \chi^{A_{2,2}}_{20,\frac23}\Big).
\end{align*}
Note that
$$
\mathbf{Q}_\mathfrak{g}=L_\mathfrak{g}=A_2(2)\oplus D_4(6).
$$
Since $\delta_L = 22/3$, we only need to calculate  $\chi_V$ up to its $q^3$-term. We write 
$$
\chi_V=q^{-1}+ \big(O_{11,0000,1}+O_{00,1000,1}+O_{10,0000,\frac{1}{3}}+O_{01,0000,\frac{1}{3}}+O_{00,0001,\frac{1}{2}}+6\big)+ \sum_{i=1}^\infty c_iq^i
$$
and list the singular Weyl orbits in $c_i$ for $i\leq 3$. There are $17$ orbits with $\mathrm{norm}>2$ in $c_1$:
\begin{align*}
& O_{22,0000,4},O_{03,0000,3},O_{30,0000,3},O_{00,0004,\frac{8}{3}},O_{02,2000,\frac{8}{3}},O_{20,2000,\frac{8}{3}}, \\
& O_{11,0003,\frac{5}{2}},O_{00,1100,\frac{7}{3}},O_{02,0100,\frac{7}{3}}, O_{20,0100,\frac{7}{3}},O_{00,0012,\frac{13}{6}},O_{00,2001,\frac{13}{6}},\\
& O_{01,0101,\frac{13}{6}},O_{02,1001,\frac{13}{6}},O_{10,0101,\frac{13}{6}},O_{11,0011,\frac{13}{6}},O_{20,1001,\frac{13}{6}} .
\end{align*}
There are $26$ orbits with $\mathrm{norm}>4$ in $c_2$:
\begin{align*}
&O_{00,0030,\frac{9}{2}},O_{03,0003,\frac{9}{2}},O_{30,0003,\frac{9}{2}},O_{00,1004,\frac{13}{3}},O_{00,2100,\frac{13}{3}},O_{02,0102,\frac{13}{3}},O_{02,3000,\frac{13}{3}},\\
&O_{20,0102,\frac{13}{3}},O_{20,3000,\frac{13}{3}},O_{22,1000,\frac{13}{3}},O_{00,0005,\frac{25}{6}},O_{00,0111,\frac{25}{6}},O_{00,3001,\frac{25}{6}},O_{01,1012,\frac{25}{6}},\\
&O_{02,0110,\frac{25}{6}},O_{02,1003,\frac{25}{6}},O_{03,0011,\frac{25}{6}},O_{10,1012,\frac{25}{6}},O_{11,0021,\frac{25}{6}},O_{11,2010,\frac{25}{6}},O_{12,0101,\frac{25}{6}},\\
&O_{20,0110,\frac{25}{6}},O_{20,1003,\frac{25}{6}},O_{21,0101,\frac{25}{6}},O_{22,0001,\frac{25}{6}},O_{30,0011,\frac{25}{6}} .
\end{align*}
There are $41$ orbits with $\textrm{norm}>6$ in $c_3$:
\begin{align*}
&O_{11,0006,7},O_{14,0000,7},O_{41,0000,7},O_{00,2004,\frac{20}{3}},O_{02,4000,\frac{20}{3}},O_{04,2000,\frac{20}{3}},O_{20,4000,\frac{20}{3}},\\
&O_{22,0004,\frac{20}{3}},O_{40,2000,\frac{20}{3}},O_{00,0104,\frac{19}{3}},O_{00,1200,\frac{19}{3}},O_{02,1102,\frac{19}{3}},O_{04,0100,\frac{19}{3}},
O_{20,1102,\frac{19}{3}},\\
&O_{22,1100,\frac{19}{3}},O_{40,0100,\frac{19}{3}},O_{00,0031,\frac{37}{6}},O_{00,1005,\frac{37}{6}},O_{00,2012,\frac{37}{6}},O_{01,0112,\frac{37}{6}},O_{01,2101,\frac{37}{6}},\\
&O_{02,1110,\frac{37}{6}},O_{02,2003,\frac{37}{6}},O_{03,0021,\frac{37}{6}},O_{03,2010,\frac{37}{6}},O_{04,1001,\frac{37}{6}},O_{10,0112,\frac{37}{6}},O_{10,2101,\frac{37}{6}},\\
&O_{11,0014,\frac{37}{6}},O_{11,1021,\frac{37}{6}},O_{12,1012,\frac{37}{6}},O_{13,0101,\frac{37}{6}},O_{20,1110,\frac{37}{6}},O_{20,2003,\frac{37}{6}},O_{21,1012,\frac{37}{6}},\\
&O_{22,0012,\frac{37}{6}},O_{22,2001,\frac{37}{6}},O_{30,0021,\frac{37}{6}},O_{30,2010,\frac{37}{6}},O_{31,0101,\frac{37}{6}},O_{40,1001,\frac{37}{6}}.
\end{align*}
We verify that each of these orbits is reflective by direct computation. 
\end{proof}

\begin{lemma}\label{lem:C4}
When $V_1=\mathfrak{g}$ is of type $C_{4,10}$, the product $\Borch(\chi_V)$ is reflective on $2U\oplus D_{4}(10)$. 
\end{lemma}
\begin{proof}
We have the expression
\begin{align*}
\chi_V
=\sum_{1}&\Big(\chi^{C_{4,10}}_{0000,0} + \chi^{C_{4,10}}_{0024,4} + \chi^{C_{4,10}}_{0040,2} + \chi^{C_{4,10}}_{0044,6} + \chi^{C_{4,10}}_{00,10,0,8} + \chi^{C_{4,10}}_{0260,5} \\[-2mm]
 &+ \chi^{C_{4,10}}_{0321,3}
+\chi^{C_{4,10}}_{0323,5} + \chi^{C_{4,10}}_{0500,2} + \chi^{C_{4,10}}_{0800,4}+ \chi^{C_{4,10}}_{1051,4} + \chi^{C_{4,10}}_{1430,4} \\
 &+ \chi^{C_{4,10}}_{1431 ,5}
+\chi^{C_{4,10}}_{2242,6} + \chi^{C_{4,10}}_{3031,3} + \chi^{C_{4,10}}_{4140,4}\Big)+2 \chi^{C_{4,10}}_{ 2222,4},
\end{align*}
where the glue vector $1$ exchanges the affine Dynkin labels $\hat{w}_0,\hat{w}_1,\hat{w}_2,\hat{w}_3,\hat{w}_4$ and $\hat{w}_4,\hat{w}_3,\hat{w}_2,\hat{w}_1,\hat{w}_0$ (such that e.g. $\chi^{C_{4,10}}_{0000,0}$ becomes $\chi^{C_{4,10}}_{000,10,10}$ and $\chi^{C_{4,10}}_{0024,4}$ becomes $\chi^{C_{4,10}}_{2004,3}$; see \eqref{eq:symbol-char} and the explanations there). In this case, we have 
$$
\mathbf{Q}_\mathfrak{g}=\ZZ^4(20) \quad  \text{and}  \quad L_\mathfrak{g}=D_4'(20)\cong D_4(10).
$$
Since $\delta_L = 10$, we only need to calculate  $\chi_V$ up to its $q^4$-term. We write 
$$
\chi_V=q^{-1}+\big(O_{2 0 0 0,\frac{1}{5}}+ O_{0 1 0 0,\frac{1}{10}}+4\big)+\sum_{i=1}^\infty c_i q^i 
$$
and find that $c_1$ equals 
\begin{align*}
&O_{0 5 0 0,\frac{5}{2}}+O_{0 0 4 0,\frac{12}{5}}+O_{4 0 0 2,\frac{12}{5}}+O_{4 0 2 0,\frac{11}{5}}+O_{0 1 2 1,\frac{21}{10}}+O_{1 3 1 0,\frac{21}{10}}+O_{2 1 0 2,\frac{21}{10}}\\
+&O_{4 1 0 1,\frac{21}{10}}+2 O_{0 2 0 2,2}+2 O_{4 2 0 0,2}+2 O_{1 0 1 2,\frac{19}{10}}+2 O_{2 1 2 0,\frac{19}{10}}+2 O_{5 0 1 0,\frac{19}{10}}+2 O_{0 0 0 3,\frac{9}{5}}\\
+&4 O_{0 2 2 0,\frac{9}{5}}+2 O_{2 2 0 1,\frac{9}{5}}+4 O_{6 0 0 0,\frac{9}{5}}+5 O_{0 3 0 1,\frac{17}{10}}+5 O_{1 0 3 0,\frac{17}{10}}+5 O_{2 3 0 0,\frac{17}{10}}+5 O_{3 0 1 1,\frac{17}{10}}\\
+&10 O_{0 4 0 0,\frac{8}{5}}+10 O_{1 1 1 1,\frac{3}{2}}+10 O_{3 1 1 0,\frac{3}{2}}+12 O_{0 0 2 1,\frac{7}{5}}+17 O_{2 0 0 2,\frac{7}{5}}+12 O_{4 0 0 1,\frac{7}{5}}\\
+&20 O_{0 1 0 2,\frac{13}{10}}+20 O_{1 2 1 0,\frac{13}{10}}+20 O_{4 1 0 0,\frac{13}{10}}+32 O_{2 0 2 0,\frac{6}{5}}+38 O_{0 1 2 0,\frac{11}{10}}+38 O_{2 1 0 1,\frac{11}{10}}\\
+&46 O_{0 2 0 11}+56 O_{2 2 0 01}+69 O_{0 3 0 0,\frac{9}{10}}+69 O_{1 0 1 1,\frac{9}{10}}+69 O_{3 0 1 0,\frac{9}{10}}+101 O_{0 0 0 2,\frac{4}{5}}\\
+&101 O_{4 0 0 0,\frac{4}{5}}+120 O_{1 1 1 0,\frac{7}{10}}+168 O_{0 0 2 0,\frac{3}{5}}+148 O_{2 0 0 1,\frac{3}{5}}+205 O_{0 1 0 1,\frac{1}{2}}+205 O_{2 1 0 0,\frac{1}{2}}\\
+&280 O_{0 2 0 0,\frac{2}{5}}+340 O_{1 0 1 0,\frac{3}{10}}+418 O_{0 0 0 1,\frac{1}{5}}+456 O_{2 0 0 0,\frac{1}{5}}+558 O_{0 1 0 0,\frac{1}{10}}+748 .
\end{align*}
There are $10$ orbits with $\textrm{norm}>4$ in $c_2$:
\begin{align*}
&O_{10, 0 0 0,5},O_{2 0 0 4,\frac{21}{5}},O_{4 2 2 0,\frac{21}{5}},O_{0 1 0 4,\frac{41}{10}},O_{0 3 2 1,\frac{41}{10}},\\
&O_{2 3 0 2,\frac{41}{10}},O_{3 0 3 1,\frac{41}{10}},O_{4 3 0 1,\frac{41}{10}},O_{5 0 3 0,\frac{41}{10}},O_{8 1 0 0,\frac{41}{10}}.
\end{align*}
There are $11$ orbits with $\textrm{norm}>6$ in $c_3$:
\begin{align*}
&O_{0 8 0 0,\frac{32}{5}},O_{0 0 2 4,\frac{31}{5}},O_{6 2 0 2,\frac{31}{5}},O_{0 1 0 5,\frac{61}{10}},O_{10, 1 0 0,\frac{61}{10}},O_{1 0 5 1,\frac{61}{10}},\\
&O_{1 4 3 0,\frac{61}{10}},O_{3 4 1 1,\frac{61}{10}},O_{4 1 4 0,\frac{61}{10}},O_{5 0 1 3,\frac{61}{10}},O_{7 0 1 2,\frac{61}{10}}.
\end{align*}
There are $16$ orbits with $\textrm{norm}>8$ in $c_4$:
\begin{align*}
&O_{0 2 6 0,\frac{41}{5}},O_{2 8 0 0,\frac{41}{5}},O_{8 2 2 0,\frac{41}{5}},O_{0 3 2 3,\frac{81}{10}},O_{0 9 0 0,\frac{81}{10}},O_{1 0 7 0,\frac{81}{10}},O_{1 4 3 1,\frac{81}{10}},O_{2 3 0 4,\frac{81}{10}},\\
&O_{3 0 3 3,\frac{81}{10}},O_{3 4 1 2,\frac{81}{10}},O_{4 1 4 1,\frac{81}{10}},O_{5 4 1 1,\frac{81}{10}},O_{6 1 4 0,\frac{81}{10}},O_{7 0 1 3,\frac{81}{10}},O_{8 3 0 1,\frac{81}{10}},O_{9 0 3 0,\frac{81}{10}}.
\end{align*} 
To finish the proof, we verify that the required condition on the orders of vectors in the singular Weyl orbits is satisfied. Here we work out $q^1 \cdot O_{4002,12/5}$ as an example. The associated dominant weight is 
$$
\lambda=4w_1+2w_4=(6,2,2,2)
$$
in coordinates as in \cite{Bou82}. The corresponding vector in $L_\mathfrak{g}'=D_4(1/20)$ is $\lambda / 20$, which has norm $20\times (12/100)=12/5$, and we have $12/5 - 2\times 1 =2/5$. The order of $\lambda/20$ modulo $\mathbf{Q}_\mathfrak{g}=\ZZ^4(20)$ is $10$, but its order modulo $L_\mathfrak{g}=D_4'(20)$ is $5$. Hence the induced zero divisor is reflective on $2U\oplus L_\mathfrak{g}$ although it is not reflective on $2U\oplus \bQ$.
\end{proof}

%% file: chap7.tex
\chapter[Symmetric case: holomorphic SCFTs of central charge 12]{The symmetric case: holomorphic SCFTs of central charge 12}\label{sec:construct-symmetric-C=1}

In this chapter, we construct hyperbolizations of affine Kac--Moody algebras $\hat{\mathfrak{g}}$ for the eight $\mathfrak{g}$ of symmetric type with $C=1$ in Theorem \ref{th:non-existence}. These semi-simple $\mathfrak{g}$ first appeared in \cite{DW21}, where Dittmann and the second named author proved that the additive lifts of the denominators of these $\hat{\mathfrak{g}}$ are reflective Borcherds products of singular weight on the maximal even sublattices of the (integral) lattices $\mathbf{P}_\mathfrak{g}$. Shortly after \cite{DW21} appeared on arXiv, Harrison, Paquette, Persson and Volpato \cite{F24} described these $\mathfrak{g}$ in terms of the $\mathcal{N}=1$ structures of holomorphic SCFT $F_{24}$ of central charge $12$, and gave a natural construction of a certain BKM superalgebra (denoted $\mathcal{G}_\mathfrak{g}$) as the BRST cohomology  related to $F_{24}$ with a fixed $\mathcal{N}=1$ structure under some technical assumptions. They also determined the super-denominators of these BKM superalgebras. 

We will prove that these super-denominators are actually the singular-weight reflective Borcherds products constructed in \cite{DW21}. We will also describe the connection between these BKM superalgebras and the twists of the fake monster algebra, which resolves some open questions proposed in \cite{F24}. Furthermore, we express the inputs of these Borcherds products as $\ZZ$-linear combinations of the full characters of the affine VOAs generated by $\mathfrak{g}$, and construct some exceptional modular invariants of $\hat{\mathfrak{g}}$ as an application.

\section{Holomorphic SCFTs of central charge 12 composed of 24 chiral fermions} 
Holomorphic vertex operator superalgebras (SVOA) of central charge $12$ were classified by Creutzig, Duncan and Riedler \cite{CDR18}, and they fall into three types: the Conway SCFT $V^{f\natural}$ with trivial weight-$1/2$ subspace, the theory $V^{fE_8}$ of $8$ chiral bosons and $8$ chiral fermions, and the theory $F_{24}$ of $24$ chiral fermions. The hyperbolizations of affine Kac--Moody algebras are related to $F_{24}$. We review the theory of $F_{24}$ following \cite{F24}. 

The vertex algebra $F_{24}$ is constructed from the lattice VOA associated with the $D_{12}$ lattice
$$
D_{12}=\left\{ (x_1,x_2,...,x_{12})\in\ZZ^{12} : \quad x_1+x_2+...+x_{12} \in 2\ZZ \right\}.
$$
The lattice VOA $V_{D_{12}}$ has $4$ irreducible modules labeled by the $4$ cosets of $D_{12}'/D_{12}$, 
$$
0,\quad  v=(1,0,..,0),\quad  s=\Big(\frac{1}{2},...,\frac{1}{2},\frac{1}{2}\Big) \quad \text{and} \quad c=\Big(\frac{1}{2},...,\frac{1}{2},-\frac{1}{2}\Big).
$$
The four modules $V_0$, $V_v$, $V_s$ and $V_c$ have conformal weights $0$, $1/2$, $3/2$ and $3/2$, respectively. Their full characters can be expressed as the quotients of Jacobi theta functions by $\eta^{12}$:
$$
\chi_*(\tau,z) = \frac{\Theta_{D_{12},*}(\tau, z)}{\eta(\tau)^{12}}, \quad *= 0, v, s, c, \quad z\in D_{12}\otimes\CC.
$$
The sum $V_0\oplus V_v$ is the $\ZZ_2$-graded superspace of $F_{24}$ (i.e. the $\rm NS$ sector in 2d SCFTs). Let $(-1)^F$ be the canonical involution, such that 
$$
(-1)^F|_{V_0}=I \quad  \text{and} \quad (-1)^F|_{V_v}=-I.
$$
The SVOA $F_{24}$ has a unique irreducible $(-1)^F$-stable canonically twisted module which is given by $V_s\oplus V_c$ (i.e. the $\rm R$ sector in 2d SCFTs). Let $F_{24}=\bigoplus_{n=0}^\infty W_{n/2}$. The weight-$1$ space $W_1$ has a natural Lie algebraic structure $\widehat{so}(24)_1$ of dimension $276$.  An $\mathcal{N}=1$ superconformal structure of $F_{24}$ is determined by a suitable choice of the weight-$3/2$ vectors, which is constructed from a linear combination of cubic terms of fermions whose coefficients are the structure constants of a certain semi-simple Lie algebra $\mathfrak{g}$ (see \cite[Section 2.1]{F24}). This $\mathfrak{g}$ is a $24$-dimensional subalgebra of $\widehat{so}(24)_1$. In a sense, $\mathfrak{g}$ can be regarded as a (virtual) Lie algebraic structure of the weight-$1/2$ subspace $W_{1/2}$. Since $\mathcal{N}=1$ superconformal structures are one-to-one corresponding to the Lie algebras $\mathfrak{g}$, we do not distinguish between them in this paper. For $\mathfrak{z}\in \mathfrak{h}$, the full characters of $F_{24}$ with an $\mathcal{N}=1$ structure are defined by
\begin{align*}
   \chi_{\mathrm{NS}}\big(\tau, \mathfrak{z}\big) &= \mathrm{Tr}_{\mathrm{NS}}\big(q^{L_0-\frac{1}{2}}e^{2\pi i\mathfrak{z}}\big),& \chi_{\widetilde{\mathrm{NS}}}\big(\tau, \mathfrak{z}\big) &= \mathrm{Tr}_{\mathrm{NS}}\big(q^{L_0-\frac{1}{2}}e^{2\pi i\mathfrak{z}}(-1)^F\big),& \\
   \chi_{\mathrm{R}}(\tau, \mathfrak{z}) &= \mathrm{Tr}_{\mathrm{R}}\big(q^{L_0-\frac{1}{2}}e^{2\pi i\mathfrak{z}}\big),& \chi_{\widetilde{\mathrm{R}}}(\tau, \mathfrak{z}) &= \mathrm{Tr}_{\mathrm{R}}\big(q^{L_0-\frac{1}{2}}e^{2\pi i\mathfrak{z}}(-1)^F\big).&
\end{align*}

The restriction of the unique normalized invariant bilinear form of $F_{24}$ to a simple ideal $\mathfrak{g}_j$ of $\mathfrak{g}$ satisfies $(-,-)=k_j\latt{-,-}$, where $\latt{-,-}$ is the standard bilinear form of $\mathfrak{g}_j$ as before (see \cite{F24}). We indicate these integers by writing 
$$
\mathfrak{g}=\mathfrak{g}_{1,k_1}\oplus \cdots \oplus \mathfrak{g}_{s,k_s}. 
$$
The semi-simple Lie algebra $\mathfrak{g}$ is of dimension $24$. There are exactly eight distinct possibilities for $\mathfrak{g}$ and they are characterized by the identity (see Notation \ref{notation})
$$
1=\frac{\dim \mathfrak{g}}{24} = \frac{h_j^\vee}{k_j}, \quad 1\leq j\leq s.
$$
The affine VOA $V_\mathfrak{g}$ generated by $\mathfrak{g}$ is a full sub-VOA of $F_{24}$.  Let $\Delta_\mathfrak{g}^+$ denote the set of positive roots of $\mathfrak{g}$. We have the following embedding of integral lattices
$$
\iota_\mathfrak{g}:\quad  \mathbf{P}_\mathfrak{g}:=\bigoplus_{j=1}^s P_{j}^\vee(h_j^\vee) \hookrightarrow \ZZ^{12}, \quad
\mathfrak{z} \mapsto \iota_\mathfrak{g}(\mathfrak{z})=\{\latt{\mathfrak{z},\alpha} \}_{\alpha \in \Delta_\mathfrak{g}^+},
$$
which induces an embedding of the maximal even sublattice $L_\mathfrak{g}$ of $\mathbf{P}_\mathfrak{g}$ into $D_{12}$. From the conformal embedding $V_\mathfrak{g} \hookrightarrow F_{24}$ we deduce that the full characters of $F_{24}$ with the $\mathcal{N}=1$ structure $\mathfrak{g}$ can be expressed as
\begin{align*}
\chi_{\mathrm{NS}}(\tau, \mathfrak{z}) &= \chi_0 (\tau, \iota_\mathfrak{g}(\mathfrak{z})) + \chi_v (\tau, \iota_\mathfrak{g}(\mathfrak{z})),\\
\chi_{\widetilde{\mathrm{NS}}}(\tau, \mathfrak{z}) &= \chi_0 (\tau, \iota_\mathfrak{g}(\mathfrak{z})) - \chi_v (\tau, \iota_\mathfrak{g}(\mathfrak{z})),\\
\chi_{\mathrm{R}}(\tau, \mathfrak{z}) &= \chi_s (\tau, \iota_\mathfrak{g}(\mathfrak{z})) + \chi_c (\tau, \iota_\mathfrak{g}(\mathfrak{z})),\\
\chi_{\widetilde{\mathrm{R}}}(\tau, \mathfrak{z}) &=\chi_s (\tau, \iota_\mathfrak{g}(\mathfrak{z})) - \chi_c (\tau, \iota_\mathfrak{g}(\mathfrak{z})) \equiv 0,
\end{align*}
where the last equality follows from the fact that $\iota(L_\mathfrak{g}) \hookrightarrow D_{10}$. 
It is easy to check that, up to scaling,
\begin{equation}\label{eq:phig}
\begin{split}
\phi_\mathfrak{g}(\tau, \mathfrak{z}) &= \frac{1}{2}\Big( \chi_{\mathrm{NS}}(\tau, \mathfrak{z}) - \chi_{\widetilde{\mathrm{NS}}}(\tau, \mathfrak{z}) - \chi_{\mathrm{R}}(\tau, \mathfrak{z}) \Big) \\
&= \chi_v (\tau, \iota_\mathfrak{g}(\mathfrak{z})) - \chi_c (\tau, \iota_\mathfrak{g}(\mathfrak{z}))   \\
&= \rank(\mathfrak{g}) + \sum_{\alpha \in \Delta_\mathfrak{g}^+}\Big( e^{2\pi i\latt{\alpha, \mathfrak{z}}} + e^{-2\pi i\latt{\alpha, \mathfrak{z}}} \Big) + O(q)
\end{split}    
\end{equation}
is the unique non-trivial $\CC$-linear combination of the full characters of $F_{24}$ that is modular under $\SL_2(\ZZ)$. In particular, $\phi_\mathfrak{g}(\tau,\mathfrak{z})$ is a weak Jacobi form of weight $0$ and lattice index $L_\mathfrak{g}$ (with trivial character on $\SL_2(\ZZ)$). It was proved in \cite{F24} under some assumptions that Fourier coefficients of $\phi_{\mathfrak{g}}$ determine the roots and root multiplicities of the BKM superalgebra $\mathcal{G}_\mathfrak{g}$, realized naturally as the BRST cohomology related to $F_{24}$ with the $\mathcal{N}=1$ structure $\mathfrak{g}$. More precisely, the root $(n,\ell,m)\in U\oplus L_\mathfrak{g}'$ has multiplicity $a_\mathfrak{g}(nm,\ell)$ as an even root and has multiplicity $b_\mathfrak{g}(nm,\ell)$ as an odd root, where 
\begin{align*}
\frac{1}{2}\big( \chi_{\mathrm{NS}}(\tau, \mathfrak{z}) - \chi_{\widetilde{\mathrm{NS}}}(\tau, \mathfrak{z}) \big) &= \chi_v (\tau, \iota_\mathfrak{g}(\mathfrak{z})) = \sum_{n,\ell} a_\mathfrak{g}(n,\ell)q^n\zeta^\ell,    \\
\frac{1}{2} \chi_{\mathrm{R}}(\tau, \mathfrak{z}) &=\chi_c (\tau, \iota_\mathfrak{g}(\mathfrak{z})) = \sum_{n,\ell} b_\mathfrak{g}(n,\ell)q^n\zeta^\ell.
\end{align*}

\section{Construction of hyperbolizations}
In this section, we prove the following theorem. 

\begin{theorem}\label{th:SVOA}
Let $\mathfrak{g}$ correspond to one of the $\mathcal{N}=1$ structures of $F_{24}$. Then the theta lift $\Borch(\phi_{\mathfrak{g}})$ is a reflective Borcherds product of singular weight on $2U\oplus L_\mathfrak{g}$. Moreover, the leading Fourier--Jacobi coefficient of $\Borch(\phi_{\mathfrak{g}})$ coincides with the denominator of the affine Kac--Moody algebra $\hat{\mathfrak{g}}$. 
\end{theorem}
\begin{proof}
We will use \eqref{eq:phig} and Lemma \ref{lem:chi-c} below to conclude that $\phi_{\mathfrak{g}}$ has integral Fourier coefficients and non-negative singular Fourier coefficients. Therefore, the theta lift of $\phi_{\mathfrak{g}}$ is a holomorphic Borcherds product of singular weight on $2U\oplus L_\mathfrak{g}$. We see from the $q^0$-term of $\phi_{\mathfrak{g}}$ that the leading Fourier--Jacobi coefficient is given by the denominator $\vartheta_\mathfrak{g}$ of $\hat{\mathfrak{g}}$. It remains to prove that $\Borch(\phi_{\mathfrak{g}})$ is reflective. This can be verified by calculating the Fourier expansions of $\phi_\mathfrak{g}$ up to the $q^{\hat{\delta}_L}$-term for $L=L_\mathfrak{g}$, where both $\delta_L$ and $\hat{\delta}_L$ are defined in Remark \ref{rem:singular}. We refer to Table \ref{tab:sym-cycle} for values of these $\delta_L$. As we mentioned at the beginning of this chapter, this also follows from Lemma \ref{lem:equality} and \cite[Theorems 5.1 and 6.1]{DW21}. 
\end{proof}

\begin{lemma}\label{lem:D12}
Let $z=(z_1,z_2,...,z_{12})\in \CC^{12}$. We set 
$$
\phi_{*}(\tau,z)=\eta(\tau)^{-12}\prod_{j=1}^{12} \vartheta_{*}(\tau,z_j), \quad *=00, 01, 10, 11,
$$
where $\vartheta_{11}(\tau,u):=\vartheta(\tau,u)$ and
\begin{align*}
\vartheta_{00}(\tau,u) &= \prod_{n=1}^\infty (1-q^n)(1+q^{n-1/2}\zeta)(1+q^{n-1/2}\zeta^{-1}),\\
\vartheta_{01}(\tau,u) &= \prod_{n=1}^\infty (1-q^n)(1-q^{n-1/2}\zeta)(1-q^{n-1/2}\zeta^{-1}),\\
\vartheta_{10}(\tau,u) &= q^{1/8}(\zeta^{1/2}+\zeta^{-1/2})\prod_{n=1}^\infty (1-q^n)(1+q^{n}\zeta)(1+q^{n}\zeta^{-1}),
\end{align*}
here $u\in \CC$ and $\zeta=e^{2\pi iu}$. 
Then the following identities hold:
$$
\chi_0=(\phi_{00}+\phi_{01})/2, \; \chi_v=(\phi_{00}-\phi_{01})/2, \; \chi_s=(\phi_{10}+\phi_{11})/2, \; \chi_c=(\phi_{10}-\phi_{11})/2,
$$
$$
\frac{\phi_{11}(2\tau,2z)}{\phi_{11}(\tau,z)} = \phi_{10}(\tau,z), \; \frac{\phi_{11}(\tau/2,z)}{\phi_{11}(\tau,z)} = \phi_{01}(\tau,z), \; 
\frac{\phi_{11}(\tau/2+1/2,z)}{\phi_{11}(\tau,z)} = -\phi_{00}(\tau,z).
$$
In particular, we have (see Proposition \ref{prop:Hecke})
$$
\phi_{D_{12}}(\tau,z):=\frac{1}{2}\Big( \phi_{00}(\tau,z) - \phi_{01}(\tau,z) - \phi_{10}(\tau,z) \Big) = -\frac{\phi_{11}|_{0}T_{-}^{(1)}(2)(\tau,z)}{\phi_{11}(\tau,z)}.
$$
\end{lemma}
\begin{proof}
The formulas for $\chi_*$ can be established by considering $(\chi_*)$ as a vector-valued Jacobi form for the dual Weil representation $\overline{\rho}_{D_{12}}$. The other formulas follow from the definition. 
\end{proof}

By restricting the above identities along the embedding $\iota_\mathfrak{g}: L_\mathfrak{g} \hookrightarrow D_{12}$, we obtain the following lemma:
\begin{lemma}\label{lem:equality} 
The image $\iota_\mathfrak{g}(L_\mathfrak{g})$ is a sublattice of $D_{12-\rank(\mathfrak{g})/2} \subsetneq D_{12}$. By considering the quasi-pullback we can express the denominator of $\hat{\mathfrak{g}}$ as
$$
\vartheta_\mathfrak{g}(\tau,\mathfrak{z}) = \frac{1}{(2\pi i)^{\rank(\mathfrak{g})/2}}\Bigg(\Big(\prod_{j=1}^{\rank(\mathfrak{g})/2} \partial_{z_{12-\rank(\mathfrak{g})/2+j}}\Big) \phi_{11}(\tau,z)\Bigg)(\tau,\iota_\mathfrak{g}(\mathfrak{z})).
$$
The full characters of $F_{24}$ can be expressed as
\begin{align*}
\chi_{\mathrm{NS}}(\tau, \mathfrak{z})&=\phi_{00}(\tau,\iota_\mathfrak{g}(\mathfrak{z})) = -\frac{\vartheta_\mathfrak{g}(\tau/2+1/2,\mathfrak{z})}{\vartheta_\mathfrak{g}(\tau,\mathfrak{z})},\\  
\chi_{\widetilde{\mathrm{NS}}}(\tau, \mathfrak{z})&=\phi_{01}(\tau,\iota_\mathfrak{g}(\mathfrak{z}))=\frac{\vartheta_\mathfrak{g}(\tau/2,\mathfrak{z})}{\vartheta_\mathfrak{g}(\tau,\mathfrak{z})},\\
\chi_{\mathrm{R}}(\tau, \mathfrak{z})&=\phi_{10}(\tau,\iota_\mathfrak{g}(\mathfrak{z})) = 2^{\rank(\mathfrak{g})/2} \frac{\vartheta_\mathfrak{g}(2\tau,2\mathfrak{z})}{\vartheta_\mathfrak{g}(\tau,\mathfrak{z})}.
\end{align*}
The input form has the representation
$$
\phi_\mathfrak{g}(\tau,\mathfrak{z}) = \phi_{D_{12}}(\tau,\iota_\mathfrak{g}(\mathfrak{z}))  = -\frac{\vartheta_\mathfrak{g}|_{\rank(\mathfrak{g})/2}T_{-}^{(1)}(2)(\tau,\mathfrak{z})}{\vartheta_\mathfrak{g}(\tau,\mathfrak{z})}.
$$
\end{lemma}

We can now prove directly that the BKM superalgebra $\mathcal{G}_\mathfrak{g}$ has no odd real roots.
\begin{lemma}\label{lem:chi-c}
The function $\chi_c (\tau, \iota_\mathfrak{g}(\mathfrak{z}))=O(q)$ is holomorphic at infinity, that is, its nonzero Fourier coefficients $f(n,\ell)q^n\zeta^\ell$ satisfy that $2n\geq (\ell,\ell)$. 
\end{lemma}
\begin{proof}
This can be read off of the Fourier expansion of $\chi_c (\tau, \iota_\mathfrak{g}(\mathfrak{z}))$ up to their $q^{\hat{\delta}_L}$-terms, since the analogue of Remark \ref{rem:singular} holds. The numbers $\delta_L$ for $L=L_\mathfrak{g}$ are listed in Table \ref{tab:sym-cycle}. This claim can also be proved by a generalization of \cite[Lemma 6.3]{GPY15}: if $\phi(\tau,\mathfrak{z})$ is a holomorphic Jacobi form of lattice index $L$ which can be constructed as a theta block then $\phi(2\tau,2\mathfrak{z})/\phi(\tau,\mathfrak{z})$ is holomorphic at infinity. The claim then follows from Lemma \ref{lem:equality}. 
\end{proof}

Let us now explain the connection between $\Borch(\phi_{\mathfrak{g}})$ and the twisted denominators of the fake monster algebra. This is similar to H\"{o}hn's observation in the anti-symmetric case. 

\begin{theorem}
Let $\mathfrak{g}$ correspond to one of the $\mathcal{N}=1$ superconformal structures of $F_{24}$. Then there exists a unique $\mathrm{Co}_0$-conjugacy class $[g]$ of the same order and level $n_g$ (see Table \ref{tab:sym-cycle}) satisfying the following properties:
\begin{enumerate}
    \item The lattice $L_\mathfrak{g}$ is characterized by the isomorphism
    $$
    U\oplus L_\mathfrak{g} \cong U(n_g)\oplus (\Lambda^g)'(n_g). 
    $$
    \item Recall from Section \ref{subsec:BKM} that the $g$-twisted denominator of the fake monster algebra is the Fourier expansion of a reflective Borcherds product $\Phi_g$ of singular weight on $U_1(n_g)\oplus U\oplus \Lambda^g$ at the $0$-dimensional cusp related to $U_1$, and that the input form that lifts to $\Phi_g$ is given by the vector-valued characters of the orbifold VOA $V_{\Lambda_g}^{\hat{g}}$. The Fourier expansion of $\Borch(\phi_\mathfrak{g})$ at $U_1$ equals the Fourier expansion of $\Phi_g$ at $U_1(n_g)$ after making the identifications
    \begin{align*}
     \Orth\!\big(U_1(n_g)\oplus U\oplus \Lambda^g\big)&=\Orth\!\big(U_1(1/n_g)\oplus U\oplus (\Lambda^g)'\big)\\
     &=\Orth\!\big(U_1\oplus U(n_g)\oplus (\Lambda^g)'(n_g)\big)
     \cong \Orth\!\big(U_1\oplus U\oplus L_\mathfrak{g}\big).   
    \end{align*}
    \item The BKM superalgebra $\mathcal{G}_\mathfrak{g}$, i.e. the BRST cohomology related to $F_{24}$ with $\mathcal{N}=1$ structure $\mathfrak{g}$, is isomorphic to the BKM superalgebra $\mathbb{G}_g$ obtained as the twist of the fake monster algebra by $g$. 
    \item The $\ZZ$-lattice generated by the Weyl vector $\rho$ and the vectors of type $\lambda=(n,\ell,m)\in U\oplus L_\mathfrak{g}'$ for which $\Borch(\phi_\mathfrak{g})$ vanishes on $\lambda^\perp$, i.e. the root lattice of $\mathcal{G}_\mathfrak{g}$, is exactly $U\oplus L_\mathfrak{g}'$. Recall that the root lattice of $\mathbb{G}_g$ is $U\oplus \Lambda^g$. The two lattices are related by
    $$
    (U\oplus L_\mathfrak{g}')(n_g) \cong U\oplus \Lambda^g. 
    $$
\end{enumerate}
\end{theorem}
\begin{proof}
This follows from Theorem \ref{th:SVOA} and its proof and from \cite[Theorem 6.1, Proposition 6.3]{DW21}.      
\end{proof}

\begin{remark}
For each $\mathfrak{g}$, the lattice $L_\mathfrak{g}$ is the unique class in the corresponding genus.  The Borcherds product $\Borch(\phi_\mathfrak{g})=\Phi_g$ is modular under the full orthogonal group $\Orth^+(2U\oplus L_\mathfrak{g})$.   
\end{remark}

\begin{remark}\label{rem:fE8}
In 2000 Scheithauer \cite{Sch00} found a natural realization of the fake monster superalgebra (introduced by Borcherds \cite{Bor92}) as the BRST cohomology related to the holomorphic SCFT $V^{fE_8}$. This BKM superalgbera has no real roots and its denominator (not its super-denominator!) is given by the Fourier expansion of $\Borch(\phi_\mathfrak{g})=\Phi_g$ at $U_1(2)$ for $\mathfrak{g}=A_{1,2}^8$ and $g$ with cycle shape $1^{-8}2^{16}$, where we view $\Borch(\phi_\mathfrak{g})$ as a modular form on $U_1(2)\oplus U\oplus E_8 \cong U_1\oplus U\oplus D_8$. By \cite[Theorem 5.9]{WW22} we can express the input as a pair of Jacobi forms $(\phi_1,\phi_2)$, where $\phi_j$ is a weakly holomorphic Jacobi form of weight $0$ and lattice index $E_8$ with trivial character on $\Gamma_0(2/j)$ for $j=1,2$. Let $\mathfrak{z}\in E_8\otimes\CC$. Note that $\chi_{\widetilde{\mathrm{R}}}=0$. The other full characters of $V^{fE_8}$ have the expressions 
\begin{align*}
\chi_{\mathrm{NS}}(\tau,\mathfrak{z})&=\frac{\Theta_{E_8,0}(\tau,\mathfrak{z})}{\eta(\tau)^8}\prod_{j=1}^4\frac{\vartheta_{00}(\tau,0)}{\eta(\tau)},\\
\chi_{\widetilde{\mathrm{NS}}}(\tau,\mathfrak{z})&=\frac{\Theta_{E_8,0}(\tau,\mathfrak{z})}{\eta(\tau)^8}\prod_{j=1}^4\frac{\vartheta_{01}(\tau,0)}{\eta(\tau)},\\
\chi_{\mathrm{R}}(\tau,\mathfrak{z})&=\frac{\Theta_{E_8,0}(\tau,\mathfrak{z})}{\eta(\tau)^8}\prod_{j=1}^4\frac{\vartheta_{10}(\tau,0)}{\eta(\tau)}.
\end{align*}
The input forms can be expressed as linear combinations of the characters:
\begin{align*}
\phi_1(\tau,\mathfrak{z})&=\big( \chi_{\mathrm{NS}} - \chi_{\widetilde{\mathrm{NS}}} \big)(\tau,\mathfrak{z})=\sum_{n\in \NN,\ell\in E_8}c_1(n,\ell)q^n\zeta^\ell,\\
\phi_2(\tau,\mathfrak{z})&=\big( \chi_{\mathrm{NS}} - \chi_{\widetilde{\mathrm{NS}}} - \chi_{\mathrm{R}} \big)(\tau,\mathfrak{z}) = 0,
\end{align*}
The denominator of the fake monster superalgebra can be written in the form
$$
\prod_{0<\alpha=(n,\ell,m)\in U\oplus E_8}\left(\frac{1-e^{-\alpha}}{1+e^{-\alpha}}\right)^{c_1(nm,\ell)/2}.
$$
\end{remark}

\begin{remark}\label{rem:Conway-SCFT}
The Conway SCFT $V^{f\natural}$ was constructed by Duncan \cite{Dun07} in 2007. Harrison, Paquette and Volpato \cite{HPV19} proved in 2019 that the BRST cohomology related to $V^{f\natural}$ defines a BKM superalgebra with no real roots.   
Note that the denominator of this superalgebra is the Fourier expansion of a Borcherds product of weight $0$ on $U_1(2)\oplus U$ at the $0$-dimensional cusp determined by $U_1(2)$, whose input can be expressed as a pair $(\phi_1,\phi_2)$ with 
\begin{align*}
\phi_1(\tau)&=(\chi_{\mathrm{NS}}-\chi_{\widetilde{\mathrm{NS}}})(\tau),\\
\phi_2(\tau)&=(\chi_{\mathrm{NS}}-\chi_{\widetilde{\mathrm{NS}}}-\chi_{\mathrm{R}}+\chi_{\widetilde{\mathrm{R}}})(\tau) = -48,
\end{align*}
and where the characters of $V^{f\natural}$ are given by
\begin{align*}
\chi_{\mathrm{NS}}(\tau)&=\frac{\eta(\tau)^{48}}{\eta(\tau/2)^{24}\eta(2\tau)^{24}}-24,& \chi_{\widetilde{\mathrm{NS}}}(\tau)&=\frac{\eta(\tau/2)^{24}}{\eta(\tau)^{24}}+24, \\
\chi_{\mathrm{R}}(\tau)&=2^{12}\frac{\eta(2\tau)^{24}}{\eta(\tau)^{24}}+24,& \chi_{\widetilde{\mathrm{R}}}(\tau)&=-24.
\end{align*}
Let $[g]$ be the $\mathrm{Co}_0$-conjugacy class of cycle shape $1^{-24}2^{24}$. The $g$-twisted denominator of the fake monster algebra is identical to the Fourier expansion of the above Borcherds product at the other $1$-dimensional cusp determined by $U$ through the identification 
$$
\Orth(U\oplus U_1(2)) = \Orth(U \oplus U_1(1/2)) =\Orth(U(2)\oplus U_1). 
$$
\end{remark}

Inspired by the two remarks above, we hope to systematically study the connection between vertex algebras and reflective Borcherds products of singular weight on lattices of type $U(N)\oplus U\oplus L$ in the near future. 

\section{Exceptional modular invariants}
The 2D rational CFT consists of both holomorphic and anti-holomorphic parts, which together can yield $\mathrm{SL}_2(\mathbb{Z})$ modular invariants. The classification of modular invariants of affine Kac--Moody algebras was an interesting and vast topic particularly in the 1980s and 90s. We will reveal some new connections to the theory of Borcherds products here and in the next chapter. 

Let us first recall some basics about modular invariants. Clearly, the CFT torus partition function $\sum|\chi_i|^2$, i.e. the summation of the norm-squares of the characters of all primaries, is a diagonal modular invariant. However, there can exist other non-diagonal modular invariants. In general, a $\mathrm{SL}_2(\mathbb{Z})$ modular invariant of a 2D rational CFT can be written as a sesquilinear combination of characters
$$
Z=\sum M_{ab} \chi_a\overline{\chi}_b,\qquad M_{ab} \in\mathbb{N}.
$$
Two major approaches to constructing modular invariants are the \emph{conformal embedding} and \emph{simple currents extension}, see e.g. the textbook \cite[Chapter 17]{CFT}. Modular invariants that cannot be obtained from simple currents are often called \emph{exceptional}. For example, $A_{1,k}$ is known to have three exceptional modular invariants at levels $k=10,16,28$, respectively called $E_6,E_7,E_8$ modular invariants. Exceptional modular invariants are rare and can be constructed by either special conformal embeddings or via nontrivial automorphisms of fusion algebras. The full classification of modular invariants for simple affine Lie algebras of all levels has been achieved for $A_1$ \cite{Cappelli,Kato} and $A_2$ \cite{Gannon}, see some recent progress in \cite{gannon2023exotic}.

We can construct many exceptional modular invariants using the conformal embedding $$
V_\mathfrak{g} \hookrightarrow F_{24}
$$
and the automorphism $\chi_v \leftrightarrow \chi_c$ of the $D_{12,1}$ fusion algebra. First we have the diagonal modular invariant of $D_{12}$, i.e.
$$
Z_{D_{12}} = |\chi_0|^2 + |\chi_v|^2 + |\chi_s|^2 + |\chi_c|^2.
$$
The diagonal modular invariant of $F_{24}$ with the $\mathcal{N}=1$ structure $\mathfrak{g}$ is given by
\begin{equation}\label{eq:modular-F24}
\begin{split}
Z_{F_{24}} &= \frac{1}{2}\Big( |\chi_{\mathrm{NS}}|^2 + |\chi_{\widetilde{\mathrm{NS}}}|^2 + |\chi_{\mathrm{R}}|^2 + |\chi_{\widetilde{\mathrm{R}}}|^2 \Big) \\
&= |\chi_0(\tau, \iota_\mathfrak{g}(\mathfrak{z}))|^2 + |\chi_v(\tau, \iota_\mathfrak{g}(\mathfrak{z}))|^2 + |\chi_s(\tau, \iota_\mathfrak{g}(\mathfrak{z}))|^2 + |\chi_c(\tau, \iota_\mathfrak{g}(\mathfrak{z}))|^2.
\end{split}
\end{equation}
By letting the automorphism $\chi_v \leftrightarrow \chi_c$ act on the holomorphic part, we obtain an exceptional modular invariant of $V_\mathfrak{g}$,
\begin{equation}\label{eq:exc}
\begin{split}
Z^{\mathrm{exc}}_{F_{24}} =& |\chi_0(\tau, \iota_\mathfrak{g}(\mathfrak{z}))|^2 + |\chi_s(\tau, \iota_\mathfrak{g}(\mathfrak{z}))|^2 + \chi_v(\tau, \iota_\mathfrak{g}(\mathfrak{z})) \overline{\chi_c(\tau, \iota_\mathfrak{g}(\mathfrak{z}))} \\
& \quad + \chi_c(\tau, \iota_\mathfrak{g}(\mathfrak{z})) \overline{\chi_v(\tau, \iota_\mathfrak{g}(\mathfrak{z}))}.
\end{split}
\end{equation}
The following relation is immediate:
\begin{equation}\label{eq:relation-Z}
Z_{F_{24}} - Z^{\mathrm{exc}}_{F_{24}} = |\phi_\mathfrak{g}|^2.     
\end{equation}

In order to write out the exceptional modular invariants \eqref{eq:exc} explicitly, it is enough to express the full characters of $F_{24}$ as $\ZZ$-linear combinations of full characters of the affine VOA $V_\mathfrak{g}$. We have calculated these expressions by cases and have found that the full character $\chi_{\mathrm{R}}$ is always given as follows (see \eqref{eq:symbol-char} for the notation):
\begin{equation}
   \chi_{\mathrm{R}} = 2^{\rank(\mathfrak{g})/2}\bigotimes_{j=1}^s \chi^{\mathfrak{g}_{j,k_j}}_{\rho_j,3/2},
\end{equation}
where $\rho_j$ is the Weyl vector of $\mathfrak{g}_j$. 
Note that if we write $\phi_\mathfrak{g}$ as a $\ZZ$-linear combination of characters of  $V_\mathfrak{g}$ then there is a unique negative term, and it is given by
$$
\phi_{\mathfrak{g},-} = s \bigotimes_{j=1}^s \chi^{\mathfrak{g}_{j,k_j}}_{\rho_j,3/2}.  
$$
We will see that 
$$
\chi_{\rm R} \neq 2\phi_{\mathfrak{g},-}
$$
if and only if $\mathfrak{g}$ is of type $A_{4,5}$ or $A_{2,3}^3$, in which cases there are certain imaginary roots that are simultaneously even and odd. Note that the $q$-order of $\chi_\mathrm{R}$ is always $1$. 

The expressions of $\chi_{\mathrm{NS}}$ and $\chi_{\widetilde{\mathrm{NS}}}$ for the eight $\mathcal{N}=1$ structures are given below.

\vspace{3mm}

\textbf{Case $A_{1,2}^8$}: The affine VOA of type $A_{1,2}$ has nonzero conformal weights $\frac{3}{16}$ and $\frac12$ associated with the weights $w_1$ and $2w_1$, respectively. The (formal) fermionization of $A_{1,2}$ is well-known to be
$$
\chi_{\rm NS}^{A_{1,2}}=\chi_{0,0}^{A_{1,2}}+\chi_{2,\frac12}^{A_{1,2}},\quad \chi_{\rm \widetilde{NS}}^{A_{1,2}}=\chi_{0,0}^{A_{1,2}}-\chi_{2,\frac12}^{A_{1,2}},\quad \chi_{\rm R}^{A_{1,2}}=\sqrt{2}\chi_{1,\frac{3}{16}}^{A_{1,2}},
$$
This theory describes three 2D chiral fermions.
For the holomorphic SVOA of type $A_{1,2}^8 $, the fermionic characters are naturally given by
$$
\chi_{\rm NS}=\bigotimes\chi_{\rm NS}^{A_{1,2}},\quad \chi_{\rm \widetilde{NS}}=\bigotimes\chi_{\rm \widetilde{NS}}^{A_{1,2}},\quad \chi_{\rm R}=\bigotimes\chi_{\rm R}^{A_{1,2}}.
$$
where the tensor products take over all eight copies of $A_{1,2}$. We remark that \eqref{eq:exc} for $A_{1,2}^8$ gives exactly the exceptional modular invariant found in \cite[Equation (5.12b)]{Gannon:1994sp}.

\vspace{3mm}

\textbf{Case $A_{2,3}^3$ }: We use the well-known conformal embedding $A_{2,3}\subset D_{4,1}$. Denote
$$
\phi_0=\chi_{00,0}^{A_{2,3}}+\chi_{03,1}^{A_{2,3}}+\chi_{30,1}^{A_{2,3}},\qquad 
\phi_1=\chi_{11,\frac12}^{A_{2,3}}.
$$
As an SVOA, the fermionization of $A_{2,3}$ has the fermionic characters
$$
\chi_{\rm NS}^{A_{2,3}}=\phi_0+\phi_1 ,\quad \chi_{\rm \widetilde{NS}}^{A_{2,3}}=\phi_0-\phi_1,\quad \chi_{\rm R}^{A_{2,3}}=2\phi_1 .
$$
Therefore, for the holomorphic SVOA $A_{2,3}^3$, we have the fermionic characters
$$
\chi_{\rm NS}=\bigotimes\chi_{\rm NS}^{A_{2,3}},\quad \chi_{\rm \widetilde{NS}}=\bigotimes\chi_{\rm \widetilde{NS}}^{A_{2,3}},\quad \chi_{\rm R}=\bigotimes\chi_{\rm R}^{A_{2,3}}.
$$
where the tensor products run over the three copies of $A_{2,3}$.

\vspace{3mm}

\textbf{Case  $A_{4,5}$}: Several modular invariants of $A_{4,5}$ were found in  \cite[Appendix B]{Schellekens:1989uf}. The one related to the conformal embedding $A_{4,5}\subset D_{12,1}$ is actually the fermionic modular invariant. Let us denote
\begin{align*}
&\phi_0=\sum_{1}\chi^{A_{4,5}}_{0000,0},\quad \phi_1=\sum_{1}\chi^{A_{4,5}}_{0102,1},\quad 
\phi_2=\sum_{1}\chi^{A_{4,5}}_{1001,\frac12},\\
& \phi_3=\sum_{1}\chi^{A_{4,5}}_{0021,\frac65},\quad 
\phi_4=\sum_{1}\chi^{A_{4,5}}_{0110,\frac45},\quad \phi_5=\chi^{A_{4,5}}_{1111,\frac32}.
\end{align*}
Here the glue vector 1 describes $\mathbb{Z}_5$ permutation on the affine Dynkin labels, i.e., $\hat{w}_i\to \hat{w}_{i-1}$ for $1\leq i\leq 4$ and $\hat{w}_0\to \hat{w}_4$. For the holomorphic SVOA, we find that the fermionic characters can be expressed as 
$$
\chi_{\rm NS}=\phi_0+\phi_1+\phi_2+\phi_5 ,\quad \chi_{\rm \widetilde{NS}}=\phi_0+\phi_1-\phi_2-\phi_5 ,\quad \chi_{\rm R}=4\phi_5.
$$
The modular invariant \eqref{eq:modular-F24} is exactly the (B.6) modular invariant in \cite{Schellekens:1989uf}. The exceptional modular invariant \eqref{eq:exc} coincides with the (B.5) modular invariant in \cite{Schellekens:1989uf}. 
Moreover, the summation 
$$
Z_{F_{24}}+|\phi_{\mathfrak{g}}|^2 =|\phi_0+\phi_1|^2+2|\phi_2|^2+10|\phi_5|^2
$$
gives the (B.3) modular invariant in \cite{Schellekens:1989uf} (compare \eqref{eq:relation-Z}). Let us comment on the extra modular invariants. The simple current extended modular invariant was given  in \cite[Equation (B.2)]{Schellekens:1989uf} as
$$
Z_{\rm (B.2)}=\sum_{i=0}^4|\phi_i|^2+5|\phi_5|^2.
$$
In addition, a different exceptional modular invariant was given in  \cite[Equation (B.4)]{Schellekens:1989uf}  as
$$
Z_{\rm (B.4)}=|\phi_0|^2+|\phi_1|^2+|\phi_3|^2+|\phi_4|^2+4|\phi_5|^2+(\phi_2\bar\phi_5+\bar{\phi}_2\phi_5).
$$
These two modular invariants are related by 
$$
Z_{\rm (B.2)}-Z_{\rm(B.4)}=|\phi_{\mathfrak{g}}|^2.
$$
(Compare \eqref{eq:relation-Z}.)

\vspace{3mm}

\textbf{Case $A_{3,4}A_{1,2}^3$}: In this case we use the conformal embedding $A_{3,4}\subset B_{7,1}$. 
We find that the fermionic characters can be expressed as 
\begin{align*}
\chi_{\rm NS}=&\,\Big(\sum_{1}\big(\chi_{000,*}^{A_{3,4}}+\chi_{101,*}^{A_{3,4}}\big)\Big)\otimes\bigotimes\big(\chi^{A_{1,2}}_{0,0}+\chi^{A_{1,2}}_{2,\frac12}\big),\\
\chi_{\rm \widetilde{NS}}= &\, \Big(\sum_{1,even}\big(\chi_{000,*}^{A_{3,4}}+\chi_{101,*}^{A_{3,4}}\big)-\sum_{1,odd}\big(\chi_{000,*}^{A_{3,4}}+\chi_{101,*}^{A_{3,4}}\big)\Big)\otimes\bigotimes\big(\chi^{A_{1,2}}_{0,0}-\chi^{A_{1,2}}_{2,\frac12}\big).
\end{align*}
The tensor products run over the $3$ copies of $A_{1,2}$.
The glue vector 1 describes the $\mathbb{Z}_4$ permutation on the affine Dynkin labels, i.e., $\hat{w}_i\to \hat{w}_{i-1}$ and $\hat{w}_0\to \hat{w}_3$. The subscript \textit{even} means the summation is taken over all primaries generated by the simple current $1$ with integral conformal weights, while \textit{odd} means all those generated by the simple current $1$ with half-integral conformal weights.

\vspace{3mm}

\textbf{Case $B_{2,3}G_{2,4}$}: Consider the decomposition $D_{5,1}\times D_{7,1}\subset D_{12,1}$ and the conformal embeddings $B_{2,3}\subset D_{5,1}$ and $G_{2,4}\subset D_{7,1}$.  
We find the following formulas for the fermionic characters:
\begin{align*}
\chi_{\rm NS}&=\Big(\sum_{even}\chi^{B_{2,3}}+\sum_{odd}\chi^{B_{2,3}}\Big)\otimes \Big(\sum_{even}\chi^{G_{2,4}}+\sum_{odd}\chi^{G_{2,4}}\Big),\\
\chi_{\rm \widetilde{NS}}&=\Big(\sum_{even}\chi^{B_{2,3}}
-\sum_{odd}\chi^{B_{2,3}}\Big)\otimes  \Big(\sum_{even}\chi^{G_{2,4}}-\sum_{odd}\chi^{G_{2,4}}\Big).
\end{align*}
Here \textit{even} indicates the summation is taken over all primaries with integral conformal weights, while \textit{odd} means all those with half-integral conformal weights. For both ${B_{2,3}}$ and ${G_{2,4}}$, the even part contains two primaries with conformal weights $0$ and $1$, while the odd part contains two primaries with conformal weights $1/2$ and $3/2$.

\vspace{3mm}

\textbf{Case $B_{2,3}A_{2,3}A_{1,2}^2$}: We use the conformal embeddings from the previous cases and find
\begin{small}
\begin{align*}
\chi_{\rm NS}&=\Big(\sum_{even}\chi^{B_{2,3}}+\sum_{odd}\chi^{B_{2,3}}\Big)\otimes \big(\chi_{00,0}^{A_{2,3}}+\chi_{03,1}^{A_{2,3}}+\chi_{30,1}^{A_{2,3}}+\chi_{11,\frac12}^{A_{2,3}}\big)\otimes\bigotimes\big(\chi_{0,0}^{A_{1,2}}+\chi_{2,\frac12}^{A_{1,2}}\big),\\
\chi_{\rm \widetilde{NS}}&=\Big(\sum_{even}\chi^{B_{2,3}}
-\sum_{odd}\chi^{B_{2,3}}\Big)\otimes \big(\chi_{00,0}^{A_{2,3}}+\chi_{03,1}^{A_{2,3}}+\chi_{30,1}^{A_{2,3}}-\chi_{11,\frac12}^{A_{2,3}}\big)\otimes\bigotimes\big(\chi_{0,0}^{A_{1,2}}-\chi_{2,\frac12}^{A_{1,2}}\big).
\end{align*}
\end{small}
Here \textit{even} and \textit{odd} are defined as in the case of $B_{2,3}G_{2,4}$.

\vspace{3mm}

\textbf{Case $B_{3,5}A_{1,2}$}: We use the conformal embedding $B_{3,5}\subset B_{10,1}$ and find
\begin{align*}
\chi_{\rm NS}&=\Big(\sum_{even}\chi^{B_{3,5}}+\sum_{odd}\chi^{B_{3,5}}\Big)\otimes\big(\chi_{0,0}^{A_{1,2}}+\chi_{2,\frac12}^{A_{1,2}}\big) ,\\
\chi_{\rm \widetilde{NS}}&= \Big(\sum_{even}\chi^{B_{3,5}}-\sum_{odd}\chi^{B_{3,5}}\Big)\otimes\big(\chi_{0,0}^{A_{1,2}}-\chi_{2,\frac12}^{A_{1,2}}\big).
\end{align*}
Here the even part of $B_{3,5}$ contains four primaries with conformal weights $0,1,2,2$, while the odd part contains four primaries with conformal weights $1/2,3/2,3/2,5/2$.

\vspace{3mm}

\textbf{Case $C_{3,4}A_{1,2}$}: This case is very similar to $B_{3,5}A_{1,2}$. We find
\begin{align*}
\chi_{\rm NS}&=\Big(\sum_{even}\chi^{C_{3,4}}+\sum_{odd}\chi^{C_{3,4}}\Big)\otimes\big(\chi_{0,0}^{A_{1,2}}+\chi_{2,\frac12}^{A_{1,2}}\big) ,\\
\chi_{\rm \widetilde{NS}}&= \Big(\sum_{even}\chi^{C_{3,4}}-\sum_{odd}\chi^{C_{3,4}}\Big)\otimes\big(\chi_{0,0}^{A_{1,2}}-\chi_{2,\frac12}^{A_{1,2}}\big).
\end{align*}
Here the even part of $C_{3,4}$ contains four primaries with conformal weights $0,1,2,3$, while the odd part contains four primaries with conformal weights $1/2,3/2,3/2,5/2$.

%% file: chap8.tex
\chapter[Symmetric case: four exotic CFTs]{The symmetric case: four exotic CFTs}\label{sec:construct-symmetric-C<1}

In this chapter, we construct hyperbolizations of affine Kac--Moody algebras $\hat{\mathfrak{g}}$ for the remaining $4$ semi-simple Lie algebras $\mathfrak{g}$ of symmetric type with $C<1$ in Theorem \ref{th:non-existence}.  The data for these $4$ cases are given in Table \ref{tab:extra}. Note that in these cases the lattices $\mathbf{P}_\mathfrak{g}$ (see Notation \ref{notation}) are even.

\begin{table}[h]
\caption{The $4$ exotic cases related to exceptional modular invariants.}
\label{tab:extra}
\renewcommand\arraystretch{1.5}
\[
\begin{array}{|c|c|c|c|c|}
\hline 
\mathfrak{g} & A_{1,16} & A_{1,8}^2 & A_{1,4}^4 & A_{2,9} \\ \hline
\mathbf{P}_\mathfrak{g} & A_1(4) & 2A_1(2) & 4A_1 & A_2(3) \\
\hline
C & \frac{1}{8} & \frac{1}{4} & \frac{1}{2} & \frac{1}{3} \\
\hline
c_\mathfrak{g} & \frac{8}{3} & \frac{24}{5} & 8 & 6 \\
\hline
\end{array} 
\]
\end{table}

The $4$ exceptional $\mathfrak{g}$ correspond to $4$ exotic 2D CFTs that already made appearances in the physics literature around three decades ago. The common feature of these $4$ exotic 2D CFTs is that they possess certain exceptional modular invariants that come from the nontrivial automorphisms of the fusion algebra of the simple current modular invariants. 
This accidental phenomenon was first found by Moore and Seiberg for $A_{1,16}$ and $A_{2,9}$ \cite{Moore:1988ss}, later by Verstegen for $A_{1,8}^2$ \cite{Verstegen:1990my} and finally by Gannon for $A_{1,4}^4$ \cite{Gannon:1994sp}. Our new observation here is that the two stories are actually deeply connected; that is, both the existence of reflective Borcherds products of singular weight and the existence of exceptional modular invariants come from the same identities among the characters of the affine Kac--Moody algebras. 

Let $\Delta_\mathfrak{g}^+$ be the set of positive roots of $\mathfrak{g}$. Note that $\dim\mathfrak{g}/C=24$.  
Similarly to the previous chapter, we work with the lattice embedding
$$
\iota_\mathfrak{g}:\quad  \mathbf{P}_\mathfrak{g} \hookrightarrow D_{12}, \quad
\mathfrak{z} \mapsto \iota_\mathfrak{g}(\mathfrak{z})=\{ \text{$1/C$ copies of $\latt{\mathfrak{z},\alpha}$} \}_{\alpha \in \Delta_\mathfrak{g}^+},
$$
such that the following identity holds:
$$
(\mathfrak{z},\mathfrak{z}) = (\iota_\mathfrak{g}(\mathfrak{z}), \iota_\mathfrak{g}(\mathfrak{z})) = \frac{1}{C} \sum_{\alpha \in \Delta_\mathfrak{g}^+}\latt{\alpha, \mathfrak{z}}^2. 
$$
\begin{theorem}\label{th:extra}
For each $\mathfrak{g}$ in Table \ref{tab:extra}, we define the pullback 
$$
\phi_\mathfrak{g}(\tau,\mathfrak{z}):=C\cdot\phi_{D_{12}}(\tau, \iota_\mathfrak{g}(\mathfrak{z})),
$$
where $\phi_{D_{12}}$ is defined in Lemma \ref{lem:D12}. Then the theta lift $\Borch(\phi_\mathfrak{g})$ is a reflective Borcherds product of singular weight on $2U\oplus \mathbf{P}_\mathfrak{g}$. Morover, the leading Fourier--Jacobi coefficient of $\Borch(\phi_\mathfrak{g})$ coincides with the denominator of $\hat{\mathfrak{g}}$.
\end{theorem}
\begin{proof}
By definition, we find that
$$
\phi_\mathfrak{g}(\tau,\mathfrak{z}) = \rank(\mathfrak{g}) + \sum_{\alpha\in\Delta_\mathfrak{g}^+} \Big( e^{2\pi i\latt{\alpha,\mathfrak{z}}} + e^{-2\pi i\latt{\alpha,\mathfrak{z}}} \Big) +O(q)
$$
and this is a weak Jacobi form of weight $0$ and lattice index $\mathbf{P}_\mathfrak{g}$. It is easy to check that $\delta_L=2$ for these four cases. From Remark \ref{rem:singular}, it follows that all singular Fourier coefficients of  $\phi_\mathfrak{g}(\tau,\mathfrak{z})$ appear in the $q^0$-term given above. Therefore, the theta lift is reflective and of singular weight, and its leading Fourier--Jacobi coefficient has the desired form. 
\end{proof}

The next theorem expresses the inputs of the Borcherds products as $\ZZ$-linear combinations of the full characters of the affine VOA generated by $\hat{\mathfrak{g}}$.
\begin{theorem}\label{th:input-character}
The following identities hold:
\begin{align*}
\phi_{A_{1,16}}&=\chi^{A_{1,16}}_{2,\frac19}+\chi^{A_{1,16}}_{14,\frac{28}{9}}-\chi^{A_{1,16}}_{8,\frac{10}{9}}, \\  
\phi_{A_{1,8}^2}&=\big(\chi^{A_{1,8}}_{0,0}+\chi^{A_{1,8}}_{8,2}\big)\otimes\big(\chi^{A_{1,8}}_{2,\frac15}+\chi^{A_{1,8}}_{6,\frac65}\big)+\big(\chi^{A_{1,8}}_{2,\frac15}+\chi^{A_{1,8}}_{6,\frac65}\big)\otimes \big(\chi^{A_{1,8}}_{0,0}+\chi^{A_{1,8}}_{8,2}\big)\\
& \quad \quad  -2\chi^{A_{1,8}}_{4,\frac35}\otimes\chi^{A_{1,8}}_{4,\frac35},\\
\phi_{A_{2,9}}&=\chi^{A_{2,9}}_{11,\frac14}+\chi^{A_{2,9}}_{17,\frac{9}{4}}+\chi^{A_{2,9}}_{71,\frac{9}{4}}-\chi^{A_{2,9}}_{33,\frac{5}{4}},\\
\phi_{A_{1,4}^4}&= \sum_{cyclic}\big(\chi^{A_{1,4}}_{0,0}+\chi^{A_{1,4}}_{4,1}\big)\otimes \big(\chi^{A_{1,4}}_{0,0}+\chi^{A_{1,4}}_{4,1}\big)\otimes \big(\chi^{A_{1,4}}_{0,0}+\chi^{A_{1,4}}_{4,1}\big)\otimes\chi^{A_{1,4}}_{2,\frac13}\\
&\quad \quad -4\bigotimes\chi^{A_{1,4}}_{2,\frac13},
\end{align*}
where the sum contains $4$ terms by permutation. By taking $\mathfrak{z}=0$ we obtain the identities among unflavored characters given by
$$
\phi_\mathfrak{g}(\tau,0) = \dim \mathfrak{g}. 
$$
\end{theorem}
\begin{proof}
First one proves that the $\ZZ$-linear combinations of affine characters above define Jacobi forms by checking their transformations under the action of the generators of $\SL_2(\ZZ)$.
To prove the identities, it is enough to check that their $q^0$-terms match, since a weak Jacobi form of weight $-12$ and lattice index $\mathbf{P}_\mathfrak{g}$ has to be zero (see e.g. \cite{Wan21}). 
\end{proof}

The embedding $\iota_\mathfrak{g}$ for $\mathfrak{g}=A_{2,9}$ induces an embedding $A_2(3) \hookrightarrow 3A_2$. By considering the corresponding pullback we obtain an identity between affine characters of type $A_2$ at levels $3$ and $9$:
$$
\big(\chi_{00,0}^{A_{2,3}}+\chi_{03,1}^{A_{2,3}}+\chi_{30,1}^{A_{2,3}}\big)^2\chi_{11,\frac12}^{A_{2,3}}-\big(\chi_{11,\frac12}^{A_{2,3}}\big)^3=\chi^{A_{2,9}}_{11,\frac14}+\chi^{A_{2,9}}_{17,\frac94}+\chi^{A_{2,9}}_{71,\frac94}-\chi^{A_{2,9}}_{33,\frac54}.
$$
Both sides of this identity reduce to the constant $8$ if we set $\mathfrak{z}\in A_2\otimes\CC$ equal to $0$. 
Similarly, by considering the embedding 
\begin{align*}
&A_{1,16} \to A_{1,8}^2 \to A_{1,4}^4 \to A_{1,2}^8, \\  
&A_1(4) \hookrightarrow 2A_1(2) \hookrightarrow 4A_1 \hookrightarrow D_8,
\end{align*}
we obtain identities between affine characters of type $A_1$ at levels $2$, $4$, $8$ and $16$:
\begin{align*}
&\chi^{A_{1,16}}_{2,\frac19}+\chi^{A_{1,16}}_{14,\frac{28}{9}}-\chi^{A_{1,16}}_{8,\frac{10}{9}}\\
=\,&\big(\chi^{A_{1,8}}_{0,0}+\chi^{A_{1,8}}_{8,2}\big)\big(\chi^{A_{1,8}}_{2,\frac15}+\chi^{A_{1,8}}_{6,\frac65}\big)-\big(\chi^{A_{1,8}}_{4,\frac35}\big)^2\\
=\,&\big(\chi^{A_{1,4}}_{0,0}+\chi^{A_{1,4}}_{4,1}\big)^3\chi^{A_{1,4}}_{2,\frac13}-\big(\chi^{A_{1,4}}_{2,\frac13}\big)^4  \\
=\,&\big(\chi_{0,0}^{A_{1,2}}\big)^7\chi_{2,\frac12}^{A_{1,2}}+7\big(\chi_{0,0}^{A_{1,2}}\big)^5\big(\chi_{2,\frac12}^{A_{1,2}}\big)^3\!+7\big(\chi_{0,0}^{A_{1,2}}\big)^3\big(\chi_{2,\frac12}^{A_{1,2}}\big)^5\\
& \quad +\chi_{0,0}^{A_{1,2}}\big(\chi_{2,\frac12}^{A_{1,2}}\big)^7\!-\big(\chi_{1,\frac{3}{16}}^{A_{1,2}}\big)^8. 
\end{align*}
All four expressions reduce to the constant $3$ if we set $\mathfrak{z}\in A_1\otimes\CC$ equal to $0$. 

\begin{remark}
The $\ZZ$-lattice generated by those $\lambda=(n,\ell,m)\in U\oplus \mathbf{P}_\mathfrak{g}'$ for which $\Borch(\phi_\mathfrak{g})$ vanishes on $\lambda^\perp$ is exactly $U\oplus \mathbf{P}_\mathfrak{g}'$.
\end{remark}

\begin{remark}
The singular Borcherds product $\Borch(\phi_\mathfrak{g})$ for $\mathfrak{g}=A_{1,16}$ was first constructed by Gritsenko and Nikulin \cite[\S 1.4]{GN98} in 1998. They showed that this product equals a certain even theta constant when viewed as a Siegel paramodular form of genus $2$ and level $4$. They also studied the corresponding BKM superalgebra \cite[Section 5.1]{GN98}.      
\end{remark}

\begin{remark}\label{rem:exceptional-modular-invariants}
As we mentioned at the beginning of this chapter, the inputs $\phi_\mathfrak{g}$ are related to exceptional modular invariants. 
\begin{enumerate}
\item For $A_{1,16}$, there exist the $D_8$ modular invariant \cite[Equation (5.8)]{Moore:1988ss} by simple current extension and the $E_7$ exceptional modular invariant \cite[Equation (5.11)]{Moore:1988ss}. They satisfy the identity
$$
Z_{D_{10}}-Z_{E_7}=|\phi_{A_{1,16}}|^2
$$
\item For $A_{2,9}$, there exist the $\mathcal{D}_9$ modular invariant \cite[Equation (5.12)]{Moore:1988ss} by simple current extension and the $\mathcal{E}'_9$ exceptional modular invariant \cite[Equation (5.14)]{Moore:1988ss}. They are related by the identity 
$$
Z_{\mathcal{D}_9}-Z_{\mathcal{E}'_9}=|\phi_{A_{2,9}}|^2. 
$$
\item For $A_{1,8}^2$, there exist the $2D_6$ modular invariant  by simple current extension for each $A_{1,8}$ \cite[Equation (5.8)]{Moore:1988ss} and the exceptional modular invariant \cite[Equation (4.1a)]{Verstegen:1990my}. 
\item For $A_{1,4}^4$, there exist the $4D_4$ modular invariant  by simple current extension for each $A_{1,4}$ \cite[Equation (5.8)]{Moore:1988ss} and the exceptional modular invariant \cite[Equation (5.12a)]{Gannon:1994sp}. 
\item In cases (3)-(4), the difference of the simple current modular invariant and the exceptional modular invariant is given by $\sum_{j}|\phi_{\mathfrak{g},j}|^2$ for a suitable finite decomposition $\phi_\mathfrak{g}=\sum_j \phi_{\mathfrak{g},j}$. 
\end{enumerate}
\end{remark}

\begin{remark}
Theorem \ref{th:input-character} relates the Borcherds products of Theorem \ref{th:extra} to vertex algebras. However, we do not know whether they are super-denominators of BKM superalgebras that can be constructed as BRST cohomology related to vertex algebras. 
\end{remark}

%% file: chap9.tex
\chapter[Symmetric cases: construction as additive lifts]{The 12 symmetric cases: the construction as additive lifts}\label{sec:additive lifts}
Some Borcherds products are known to have constructions as additive lifts. Since additive lifts are always invariant under the involution $(\omega,\mathfrak{z},\tau)\to (\tau,\mathfrak{z},\omega)$, we cannot hope to express anti-symmetric Borcherds products as additive lifts. In this chapter, we will give a uniform construction of the 12 symmetric Borcherds products as additive lifts. The inputs into the additive lift turn out to be the denominators of the associated affine Kac--Moody algebras $\hat{\mathfrak{g}}$. 

Recall that for each $\mathfrak{g}$ in Table \ref{tab:sym-cycle} or Table \ref{tab:extra} we have constructed a reflective Borcherds product $\Borch(\phi_\mathfrak{g})$ of singular weight on $2U\oplus L_\mathfrak{g}$ in the previous two chapters, where $L_\mathfrak{g}$ is the maximal even sublattice of $\mathbf{P}_\mathfrak{g}$. The product $\Borch(\phi_\mathfrak{g})$ has a Fourier--Jacobi expansion of the form
$$
\Borch(\phi_\mathfrak{g})(\omega,\mathfrak{z},\tau)=\vartheta_{\mathfrak{g}}(\tau,\mathfrak{z})\cdot e^{2\pi iC\omega} \cdot \exp\left( -\sum_{m=1}^\infty \Big(\phi_\mathfrak{g}|_{0} T_{-}^{(1)}(m)\Big)(\tau,\mathfrak{z})\cdot e^{2\pi im\omega} \right).
$$
When $\mathfrak{g}$ is not of type $A_{1,16}$, the weight of $\vartheta_\mathfrak{g}$ is integral, so we can define the additive lift of $\vartheta_\mathfrak{g}$ following Theorem \ref{th:additive}:
$$
\Grit(\vartheta_\mathfrak{g})(\omega,\mathfrak{z},\tau)=\sum_{0<m\in 1+\frac{1}{C}\ZZ} \Big( \vartheta_\mathfrak{g}\lvert_{\frac{1}{2}\rank(\mathfrak{g})}T_{-}^{(\frac{1}{C})}(m)\Big)(\tau,\mathfrak{z})\cdot e^{2\pi imC\omega}
$$
When $\mathfrak{g}=A_{1,16}$, we have $\vartheta_\mathfrak{g}(\tau,\mathfrak{z})=\vartheta(\tau,z)$. In this case, we use Gritsenko's ``trivial lifting'' of $\vartheta$, defined in \cite[Theorem 1.11]{GN98} to be
$$
\Grit(\vartheta)(\omega,z,\tau) = \sum_{m=1}^\infty  \left( \frac{-4}{m} \right) \vartheta(\tau,mz)\cdot e^{2\pi i m^2\omega/8}.
$$

\begin{theorem}\label{th:G=B}
Let $\mathfrak{g}$ be one of the semi-simple Lie algebras in Table \ref{tab:sym-cycle} or Table \ref{tab:extra}. Then we have
$$
\Borch(\phi_\mathfrak{g}) = \Grit(\vartheta_\mathfrak{g}).
$$
\end{theorem}
\begin{proof}
The identity follows from Koecher's principle if we can show that the divisor of the additive lift contains the divisor of the Borcherds product. The $8$ cases with $C=1$ were proved in \cite[Theorem 5.1]{DW21}, and a much simpler proof can be found in \cite[Corollary 4.5]{WW21}. In particular, the identities for $\mathfrak{g}=A_{1,2}^8$ and $A_{2,3}^3$ were first established by Gritsenko in \cite[Theorems 3.2, 4.2]{Gri10}. 

The case $\mathfrak{g}=A_{1,16}$ was proved in \cite[Example 2.3]{GN98}, the case $\mathfrak{g}=A_{1,4}^4$ was proved in \cite[Theorem 5.1]{Gri10}, and the case $\mathfrak{g}=A_{2,9}$ was proved in \cite[Theorem 13.5 (2)]{GSZ19}.  The remaining case $\mathfrak{g}=A_{1,8}^2$ can be proved in a similar way; we will give a sketch. Here we have $L_\mathfrak{g}=2A_1(2)$ and the number $\delta_L$ defined in Remark \ref{rem:singular} is $2$. Therefore, all of the singular Fourier coefficients of $\phi_\mathfrak{g}$ already appear in its $q^0$-term. By an argument analogous to \cite[Remark 3.4, Lemma 3.5]{GW20}, $\Grit(\vartheta_{\mathfrak{g}})/\Borch(\phi_\mathfrak{g})$ is holomorphic, and therefore constant. 
\end{proof}

Compare the first two Fourier--Jacobi coefficients of the Borcherds products and additive lifts,
\begin{align*}
\Borch(\phi_\mathfrak{g}) &= \vartheta_\mathfrak{g}(\tau,\mathfrak{z})\cdot \zeta^C - \vartheta_\mathfrak{g}(\tau,\mathfrak{z})\phi_\mathfrak{g}(\tau,\mathfrak{z})\cdot\zeta^{C+1} + O(\zeta^{C+2}),\\
\Grit(\vartheta_\mathfrak{g}) &= \vartheta_\mathfrak{g}(\tau,\mathfrak{z})\cdot \zeta^C + \Big( \vartheta_\mathfrak{g}\lvert_{\frac{1}{2}\rank(\mathfrak{g})}T_{-}^{(\frac{1}{C})}(1/C+1)\Big)(\tau,\mathfrak{z})\cdot \zeta^{C+1} +  O(\zeta^{C+2}), 
\end{align*}
where for $\mathfrak{g}=A_{1,16}$ we define
$$
\Big(\vartheta|_{\frac12} T_{-}^{(8)}(9)\Big)(\tau,z) = -\vartheta(\tau,3z). 
$$
This implies the following uniform expression.
\begin{corollary}
Let $\mathfrak{g}$ be one of the semi-simple Lie algebras in Table \ref{tab:sym-cycle} or Table \ref{tab:extra}. Then the input can be expressed as
$$
\phi_\mathfrak{g}(\tau,\mathfrak{z}) = - \frac{\Big(\vartheta_\mathfrak{g}\lvert_{\frac{1}{2}\rank(\mathfrak{g})}T_{-}^{(\frac{1}{C})}(1/C+1)\Big)(\tau,\mathfrak{z})}{\vartheta_\mathfrak{g}(\tau,\mathfrak{z})}.
$$
\end{corollary}

\begin{remark}
This corollary shows that in the symmetric cases the input, which was originally constructed as a $\ZZ$-linear combination of characters of some vertex algebra, can be read off the denominator of $\hat{\mathfrak{g}}$. There is no expression of this kind in the anti-symmetric cases. 
\end{remark}

\begin{remark}
Recall from Corollary \ref{cor:solutions} that Equation \eqref{symC} has $17$ solutions. The $12$ solutions of Equation \eqref{symC} that correspond to symmetric cases with hyperbolizations are characterized by $1/C$ being integral. This is related to the expression of the input $\phi_\mathfrak{g}$ as a $\ZZ$-linear combination of characters of the affine VOA generated by $\hat{\mathfrak{g}}$. This expression contains a unique term with negative coefficients, namely
$$
s\cdot \bigoplus_{j=1}^s \chi_{k_j\rho_j/h_j^\vee}^{\mathfrak{g}_{j,k_j}}.
$$
The condition that $1/C$ is integral guarantees that the dominant integral weight $k_j\rho_j/h_j^\vee$ is well defined. The equality $\dim \mathfrak{g} / 24 = C = h_j^\vee/k_j$ implies that the first nonzero Fourier coefficient in the above character is $q^1$.  We do not have a conceptual explanation for this. A similar phenomenon for exceptional modular invariants has been observed by Schellekens and Yankielowicz \cite[p.100]{Schellekens:1989uf}. 
\end{remark}

%% file: chap10.tex
\chapter[Fourier expansions of singular reflective Borcherds products]{Fourier expansions of singular-weight reflective Borcherds products}\label{sec:Fourier-expansion}

In this chapter, we will work out the Fourier expansions of singular-weight reflective Borcherds products at the standard $0$-dimensional cusp as in \cite[Example 13.7]{Bor98}, \cite[Section 9]{Sch04} and \cite[Section 5]{DHS15}. 

Let $F$ be a reflective Borcherds product of singular weight on $U_1\oplus U\oplus L$ and set $K=U\oplus L$. Assume that the input as a Jacobi form,
$$
\phi(\tau,\mathfrak{z}) = \sum_{n\in\ZZ} \sum_{\ell \in L'} f(n,\ell) q^n \zeta^\ell,
$$
has only non-negative singular Fourier coefficients. We will consider the Fourier expansion of $F$ at the $0$-dimensional cusp related to $U_1$. Let $W$ be the Weyl group, i.e. the subgroup of $\Orth(U\oplus L)$ generated by reflections associated with vectors $(n,\ell,m)\in U\oplus L'$ for which $2n<\ell^2$ and $f(nm,\ell)>0$. We fix a Weyl chamber $\mathcal{C}$ of $F$ and denote its closure by $\overline{\mathcal{C}}$. Let $\rho$ be the corresponding Weyl vector. 

We will view $F$ as a function on the associated tube domain. This is anti-invariant under the Weyl group $W$, i.e.
$$
F(\sigma(Z)) = \det(\sigma) F(Z), \quad \text{for all $\sigma\in W$}. 
$$
It is known that $W$ acts simply transitively on the Weyl chambers of the negative cone of $K\otimes\RR$. Therefore, the Fourier expansion of $F$ has the form
$$
F(Z)=\sum_{\sigma\in W} \det(\sigma) \sum_{\substack{\lambda \in K', \; \lambda + \rho \in \overline{\mathcal{C}} \\ (\lambda, \mathcal{C})<0 }} c(\lambda) \mathbf{e}\Big(\sigma(\lambda+\rho), Z\Big),
$$
where $\mathbf{e}(t)=e^{2\pi it}$ as before. Since $F$ has singular weight, $c(\lambda)=0$ whenever $(\lambda+\rho, \lambda+\rho)\neq 0$.
Note that $(\lambda, \mathcal{C})<0$ and $\rho \in \overline{\mathcal{C}}$ together imply that $(\lambda, \rho) \leq 0$. But on the other hand, if $(\lambda+\rho, \lambda+\rho)= 0$, then we have
$$
2(\lambda, \rho) = (\lambda+\rho, \lambda+\rho) - (\lambda, \lambda) - (\rho,\rho) = -\lambda^2
$$
and 
$$
(\lambda, \rho) = (\lambda, \lambda+\rho) - (\lambda, \lambda) \leq -\lambda^2,
$$
since $(\lambda, \lambda+\rho)\leq 0$. Since $-\lambda^2/2\leq -\lambda^2$, $\lambda^2\leq 0$ and therefore $(\lambda, \rho)\geq 0$, so in this case $(\lambda,\rho)=0$.

Write $\lambda$ and $\rho$ in coordinates as $(x_1,x_2,...,x_l,x_0) \in \RR^{l,1}$ and $(y_1,y_2, ...,y_r,y_0)\in \RR^{l,1}$ respectively, such that 
\begin{align*}
x_1^2+x_2^2+...+x_r^2-x_0^2&=(\lambda, \lambda)\leq 0,\\
y_1^2+y_2^2+...+y_r^2-y_0^2&=(\rho, \rho)=0,\\
x_1y_1+x_2y_2+...+x_ry_r-x_0y_0&=(\lambda,\rho)=0.
\end{align*}
By the Cauchy--Schwarz inequality, we have 
$$
x_0^2y_0^2=\Big(\sum_{j=1}^r x_jy_j\Big)^2 \leq \sum_{j=1}^r x_j^2 \sum_{j=1}^r y_j^2 \leq x_0^2y_0^2,
$$
from which it follows that $\lambda=c\rho \in K'$ for some $c\in \QQ$. Note that $(v, \mathcal{C})<0$ for $v\in \mathcal{C}$.   Since $(\lambda, \mathcal{C})<0$ and $\rho\in \overline{\mathcal{C}}$, we find $c\geq 0$.

Write $\lambda=c\rho=\sum \lambda_j$ with $\lambda_j\in K'$ and $(\lambda_j, \mathcal{C})<0$, and suppose that these $\lambda_j$ contribute to $c(\lambda)$ in the product expansion of $F$ at $U_1$. From $(\lambda_j, \rho)\leq 0$ and $\sum (\lambda_j, \rho)=(c\rho, \rho)=0$ we obtain $(\lambda_j, \rho)=0$. 

If $(\lambda_j,\lambda_j)\leq 0$, then a similar argument as before shows that $\lambda_j$ is a positive rational multiple of $\rho$. Otherwise, suppose $(\lambda_j, \lambda_j)>0$ for some $j$. Then $F$ vanishes on $\lambda_j^\perp$ and the reflection $\sigma_{\lambda_j}$ lies in the Weyl group $W$.  Since $\sigma_{\lambda_j}$ maps $\mathcal{C}$ to a new Weyl chamber, it maps $\rho$ to the corresponding new Weyl vector, i.e. $\sigma_{\lambda_j}(\rho)\neq \rho$. It follows that $(\lambda_j, \rho)\neq 0$, leading to a contradiction. Therefore, every $\lambda_j$ is a positive rational multiple of $\rho$ and $\lambda_j\in K'$. If we choose a positive $c\in \QQ$ such that $c\rho$ is a primitive vector in $K'$, then every $\lambda_j$ is a positive integral multiple of $c\rho$. 

This gives us the following expression:
\begin{equation}\label{eq:Singular-Fourier}
\begin{split}
F(Z)&=\sum_{\sigma\in W} \det(\sigma) \; \mathbf{e}\big((\sigma(\rho), Z)\big)\prod_{n=1}^\infty \Big(1-\mathbf{e}\big((\sigma(nc\rho), Z)\big)\Big)^{f(n^2c^2\rho^2/2,\, nc\rho)}\\
&=\sum_{\sigma\in W} \det(\sigma)\, \sigma\left( \mathbf{e}\big((\rho, Z)\big)\prod_{n=1}^\infty \Big(1-\mathbf{e}\big(nc(\rho, Z)\big)\Big)^{f(n^2c^2\rho^2/2,\, nc\rho)} \right)\\
&=: \sum_{\sigma\in W} \det(\sigma)\, \sigma\Big( \Delta_F\big( (\rho, Z) \big) \Big).
\end{split}   
\end{equation}
The value of $f(n^2c^2\rho^2/2,\, nc\rho)$ depends only on the coset of $nc\rho$ in $K'/K$, and 
$$
f(*, nc\rho)=f(*, -nc\rho)=f(*, mc\rho)
$$
if $(n+m)c\rho \in K$. Note that the function $\Delta_F$ above is often an eta quotient. 

We will now work out the expressions of type \eqref{eq:Singular-Fourier} for the singular-weight reflective Borcherds products constructed in the previous chapters. Let $\mathfrak{g}=\oplus_j \mathfrak{g}_{j,k_j}$ be one of the $81$ semi-simple Lie algebras with a hyperbolization. Let $\Psi_\mathfrak{g}$ be the corresponding singular-weight reflective Borcherds product on $2U\oplus L_\mathfrak{g}$. 
Using the notation of Theorem \ref{th:product} and Theorem \ref{th:root-system}, the Weyl vector of $\Psi_\mathfrak{g}$ is a norm zero vector given by
$$
(-C-1,\; \rho_\mathfrak{g},\; -C)
$$
in the anti-symmetric cases, and by
$$
(-C,\; \rho_\mathfrak{g},\; -C)
$$
in the symmetric cases, where $\rho_\mathfrak{g}$ is the normalized Weyl vector of $\mathfrak{g}$ defined by
$$
\rho_\mathfrak{g}=\sum_j \rho_j / k_j,
$$
where $\rho_j$ is the Weyl vector of the simple Lie algebra $\mathfrak{g}_j$. The function $\Delta_F$ that appears in \eqref{eq:Singular-Fourier} for $\Psi_\mathfrak{g}$ can be described as follows. 
\begin{enumerate}
\item  Let $\mathfrak{g}$ be of anti-symmetric type. Then $\mathfrak{g}$ corresponds to a $\mathrm{Co}_0$-conjugacy class denoted $[g]$. Let $\prod_{b_k\neq 0}k^{b_k}$ be the cycle shape of $g$ and define the associated eta quotient by
$$
\eta_g(\tau) := \prod_{k\geq 1} \eta(k\tau)^{b_k}, \quad \eta(\tau)=q^{\frac{1}{24}}\prod_{n=1}^\infty(1-q^n). 
$$
Then
$$
\Delta_{\Psi_\mathfrak{g}}\big( (\rho, Z) \big) = \eta_g\big( (\rho, Z) \big). 
$$
Here the number $c$ in \eqref{eq:Singular-Fourier} is $1$ if $g$ has order equal to its level, and it is $2$ otherwise. 

\item Let $\mathfrak{g}$ be of symmetric type with $C=1$. Then $\mathfrak{g}$ also corresponds to a $\mathrm{Co}_0$-conjugacy class denoted $[g]$. The number $c$ in \eqref{eq:Singular-Fourier} is always $1$, and
$$
\Delta_{\Psi_\mathfrak{g}}\big( (\rho, Z) \big) = \eta_g\big( (\rho, Z) \big). 
$$

\item Let $\mathfrak{g}$ be of symmetric type with $C<1$. In this case, $\mathfrak{g}$ does not correspond to any $\mathrm{Co}_0$-conjugacy class. However, we obtain a similar expression for $\Delta_F$ as an eta quotient associated to a (fake) cycle shape. More precisely,
$$
\Delta_{\Psi_\mathfrak{g}}\big( (\rho, Z) \big) = \eta_g\big( (\rho, Z) \big),
$$
where the (fake) cycle shapes of $g$ are given in Table \ref{tab:fake-cycle-shape}. The number $c$ in \eqref{eq:Singular-Fourier} equals $1/C$, which also appears in Table \ref{tab:fake-cycle-shape}. 

\begin{table}[ht]
\caption{The (fake) cycle shapes of the 4 exceptional semi-simple Lie algebras}
\label{tab:fake-cycle-shape}
\renewcommand\arraystretch{1.5}
\[
\begin{array}{|c|c|c|c|c|}
\hline 
\mathfrak{g} & A_{1,16} & A_{1,8}^2 & A_{1,4}^4 & A_{2,9}  \\ 
\hline 
c & 8 & 4 & 2 & 3 \\
\hline 
g & 8^{-1}16^2 & 4^{-2}8^4 & 2^{-4}4^8 & 3^{-1} 9^3 \\ 
\hline
\end{array} 
\]
\end{table}
\end{enumerate}

%% file: chap11.tex
\chapter{The classification of anti-symmetric root systems}\label{sec:exclude-anti-root}
In this chapter, we prove Theorem \ref{th:non-existence} in the anti-symmetric case.

\begin{theorem}\label{th:non-existence-anti-symmetric}
If $2U\oplus L$ has an anti-symmetric reflective Borcherds product of singular weight whose Jacobi form input has non-negative $q^0$-term, then the associated semi-simple Lie algebra $\mathfrak{g}$ defined in Theorem \ref{th:root-system} lies in Schellekens’ list of $69$ semi-simple $V_1$ structures of holomorphic vertex operator algebras of central charge $24$. 
\end{theorem}

\section{The main argument}
To explain the proof of Theorem \ref{th:non-existence-anti-symmetric}, we begin by introducing a useful concept.

\begin{definition}
An even positive definite lattice $K$ is called a \emph{forbidden component} if there is no reflective Borcherds product of any weight on $2U\oplus K$.  By considering its symmetrization, we find that $K$ is forbidden if and only if there is no reflective Borcherds product of any weight that is modular for the full orthogonal group $\Orth^+(2U\oplus K)$. 
\end{definition}

Let $F=\Borch(f)$ be the potential Borcherds product in Theorem \ref{th:non-existence-anti-symmetric}. By \cite[Lemma 3.3]{Wan23a}, for any even overlattice $L_1$ of $L$, there exists an anti-symmetric reflective Borcherds product of some weight on $2U\oplus L_1$. Moreover, if there is a decomposition
$$
2U\oplus L_1 \cong 2U\oplus L_2\oplus L_3,
$$
then the quasi-pullback will yield an anti-symmetric reflective Borcherds product on $2U\oplus L_2$. 

\vspace{3mm}

To rule out the existence of a singular, reflective, anti-symmetric Borcherds product on $2U \oplus L$, it is therefore sufficient to find an overlattice $L_1$ of $L$ that admits a direct sum decomposition $L_1 \cong K \oplus L_0$ with a forbidden component $K$. 

\begin{Argument}\label{Argument}
To exclude the $152$ extraneous root systems, we use the bounds 
$$
\bQ < L < \bP
$$
and the property that $L(C)$ is integral to determine an even overlattice $K=K_1\oplus K_2$ of $L$ that contains a forbidden component $K_1$.
\begin{itemize}
\item[(a)] In many cases, the upper bound $\bP$ is integral, and we can take $K$ to be the unique maximal even sublattice $\bP^{\mathrm{ev}}$ of $\bP$, or take $K$ to be a certain even overlattice of $\bP^{\mathrm{ev}}$.
\item[(b)] When $\bP$ is not integral, $L$ is contained in a certain maximal even overlattice $Q$ of the lower bound $\bQ$, and we usually find $K$ as a suitable sublattice of $Q$. 
\end{itemize}
\end{Argument}

We will use the following forbidden components:

\vspace{3mm}

\textbf{Rank 2:}
\begin{enumerate}
\item $2A_1(16)$;
\inlineitem $2A_1(18)$;
\inlineitem $2A_1(20)$;
\inlineitem $A_2(16)$;
\inlineitem $A_2(24)$;
\end{enumerate}

\textbf{Rank 3:}
\begin{enumerate}[resume]
\item $A_1 \oplus 2A_1(8)$;
\inlineitem $2A_1(2) \oplus A_1(5)$;
\inlineitem $A_1(3) \oplus 2A_1(9)$;
\item $A_1(3) \oplus A_2(5)$;
\inlineitem $A_1 \oplus A_2(6)$;
\inlineitem $A_1 \oplus A_2(8)$;
\item $A_1(2) \oplus A_2(8)$;
\inlineitem $A_1 \oplus A_2(12)$;
\item $A_1(2) \oplus A_2(12)$;
\inlineitem $A_3'(24)$;
\end{enumerate}

\textbf{Rank 4:}
\begin{enumerate}[resume]
\item $2A_1(6) \oplus A_2$;
\inlineitem $A_1(2) \oplus A_3(3)$;
\inlineitem $A_1 \oplus A_3'(16)$;
\item $A_1(7) \oplus L_1$, where $L_1$ is the rank three lattice with Gram matrix $\begin{psmallmatrix} 4 & 4 & 6 \\ 4 & 14 & 0 \\ 6 & 0 & 14 \end{psmallmatrix}$. The lattice $L_1$ is a maximal even overlattice of $3A_1(7)$ with discriminant $56$.
\end{enumerate}

\textbf{Rank 5:}
\begin{enumerate}[resume]
\item $3A_1 \oplus 2A_1(5)$;
\inlineitem $A_1 \oplus 2A_1(2) \oplus A_2(4)$;
\inlineitem $A_1(3) \oplus A_2 \oplus A_2(4)$;
\item $A_1(3) \oplus A_4$;
\inlineitem $A_1(4) \oplus A_4(2)$;
\end{enumerate}

\textbf{Rank 6:}
\begin{enumerate}[resume]
\item $2A_1(3) \oplus A_2 \oplus A_2(3)$;
\inlineitem $A_2 \oplus A_2(2) \oplus A_2(4)$;
\inlineitem $A_1 \oplus A_2 \oplus A_3'(8)$;
\item $A_3 \oplus A_3'(16)$;
\inlineitem $2A_1(2) \oplus A_4$;
\inlineitem $2A_1 \oplus A_4(2)$;
\item $A_2(6) \oplus D_4$;
\item $A_2 \oplus L_2$, where $L_2$ is a maximal even overlattice of $A_4(3)$ (such that $\mathrm{discr}(L_2)=45$);
\item $2L_3 \oplus L_3(2)$, where $L_3$ is the rank two lattice with Gram matrix $\begin{psmallmatrix} 2 & 1 \\ 1 & 4 \end{psmallmatrix}$;
\end{enumerate}

\textbf{Rank 7:}
\begin{enumerate}[resume]
\item $2A_2 \oplus A_3'(8)$;
\end{enumerate}

\textbf{Rank 8:}
\begin{enumerate}[resume]
\item $4A_2(3)$;
\inlineitem $2A_2 \oplus A_4$;
\inlineitem $A_1 \oplus A_3'(8) \oplus D_4$;
\item $3A_1 \oplus A_1(4) \oplus D_4$;
\inlineitem $2A_2(2) \oplus D_4$;
\inlineitem $A_4 \oplus D_4$;
\item $2D_4(3)$;
\inlineitem $A_1 \oplus A_2(2) \oplus D_5$;
\inlineitem $A_2(4) \oplus D_6$;
\item $A_4'(5) \oplus L_4$, where the lattice $L_4$ is the maximal even sublattice of $\mathbb{Z} \oplus \mathbb{Z}^3(5)$ of genus $2_{\mathrm{II}}^2 5^3$;
\item $D_4 \oplus L_5$, where $L_5$ has rank four and Gram matrix $\begin{psmallmatrix} 2 & 0 & 1 & -1 \\ 0 & 2 & -1 & -1 \\ 1 & -1 & 6 & 0 \\ -1 & -1 & 0 & 6 \end{psmallmatrix}$, and is of genus $2_{\mathrm{II}}^{-2}5^2$;
\end{enumerate}

\textbf{Rank 9:}
\begin{enumerate}[resume]
\item $3A_1 \oplus 3A_2$;
\inlineitem $2A_1 \oplus A_1(3) \oplus A_2 \oplus D_4$;
\inlineitem $A_1 \oplus 3A_1(2) \oplus D_5$;
\end{enumerate}

\textbf{Rank 10:}
\begin{enumerate}[resume]
\item $2A_1 \oplus 2A_2 \oplus D_4$;
\inlineitem $2A_2 \oplus A_2(2) \oplus D_4$;
\inlineitem $2A_2(2) \oplus D_6$;
\item $A_4 \oplus D_6$;
\inlineitem $3A_1(2) \oplus D_7$;
\inlineitem $A_3 \oplus D_7$;
\inlineitem $A_2 \oplus D_8(3)$;
\item $3A_2 \oplus L_7$, where $L_7$ is the maximal even sublattice of $\mathbb{Z}^2 \oplus \mathbb{Z}^2(3)$ in the genus $2_{\mathrm{II}}^2 3^2$;
\end{enumerate}

\textbf{Rank 11:}
\begin{enumerate}[resume]
\item $3A_2 \oplus D_5$;
\inlineitem $A_1 \oplus A_2(2) \oplus E_8$;
\item $D_4 \oplus L_6$, where $L_6$ has rank seven and Gram matrix $\begin{psmallmatrix} 2 & 1 & 1 & 1 & 1 & 1 & 1 \\ 1 & 2 & 1 & 1 & 0 & 0 & 0 \\ 1 & 1 & 2 & 1 & 0 & 1 & 1 \\ 1 & 1 & 1 & 2 & 0 & 0 & 0 \\ 1 & 0 & 0 & 0 & 2 & 1 & 1 \\ 1 & 0 & 1 & 0 & 1 & 4 & 0 \\ 1 & 0 & 1 & 0 & 1 & 0 & 4\end{psmallmatrix}$, and is of genus $2_7^1 4_{\mathrm{II}}^2$;
\end{enumerate}

\textbf{Rank 12:}
\begin{enumerate}[resume]
\item $D_6 \oplus E_6$;
\inlineitem $2A_2 \oplus D_8$;
\inlineitem $A_4 \oplus D_8$;
\item $A_1 \oplus A_3 \oplus E_8$;
\inlineitem $3A_1 \oplus D_9$;
\item The maximum even sublattice of $\mathbb{Z}^6 \oplus 2A_3'(4)$ in the genus $2_{\mathrm{II}}^{-2} 4_{\mathrm{II}}^4$;
\end{enumerate}

\textbf{Rank 13:}
\begin{enumerate}[resume]
\item $3A_1 \oplus A_2 \oplus E_8$;
\end{enumerate}

\textbf{Rank 14:}
\begin{enumerate}[resume]
\item $2A_2 \oplus A_2(2) \oplus E_8$;
\inlineitem $A_2(2) \oplus D_4 \oplus E_8$;
\item $A_1 \oplus D_5 \oplus E_8$;
\inlineitem $E_6 \oplus E_8$;
\item $E_8 \oplus L$, where $L$ has rank six and Gram matrix $\begin{psmallmatrix} 2 & -1 & 0 & 0 & 0 & 0 \\ -1 & 2 & -1 & -1 & -1 & -1 \\ 0 & -1 & 2 & 0 & 0 & 0 \\ 0 & -1 & 0 & 4 & 2 & 2 \\ 0 & -1 & 0 & 2 & 4 & 0 \\ 0 & -1 & 0 & 2 & 0 & 4 \end{psmallmatrix}$, and is of genus $2_6^2 
4_{\mathrm{II}}^2$;
\item The maximal even sublattice of $\mathbb{Z}^8 \oplus 2A_3'(4)$;
\end{enumerate}

\textbf{Rank 16:}
\begin{enumerate}[resume]
\item $A_2(2) \oplus D_6 \oplus E_8$;
\inlineitem $A_1 \oplus D_7 \oplus E_8$;
\end{enumerate}

\textbf{Rank 18:}
\begin{enumerate}[resume]
\item $A_2 \oplus 2D_4 \oplus E_8$;
\inlineitem $2A_1 \oplus D_8 \oplus E_8$;
\item $A_1 \oplus A_1(2) \oplus 2E_8$;
\inlineitem $2A_1(2) \oplus 2E_8$.
\end{enumerate}

\vspace{3mm}

The forbidden components above can be verified using the algorithm described at the end of Section \ref{subsec:Borcherds} except for the following three cases, which would have been prohibitively time-consuming to check directly due to the large lattice discriminant.

\begin{lemma}
The lattice $4A_2(3)$ is a forbidden component.
\end{lemma}
\begin{proof}
We use the overlattice
$$
2U\oplus 4A_2(3) \cong 2U(9)\oplus 4A_2 < U\oplus U(9)\oplus E_8
$$
and can verify quickly that there are no reflective Borcherds products on $U\oplus U(9)\oplus E_8$. By \cite[Lemma 3.3]{Wan23a}, $2U\oplus 4A_2(3)$ has no anti-symmetric reflective Borcherds product of any weight. Suppose that $2U\oplus 4A_2(3)$ has a symmetric reflective Borcherds product, and denote its input as a vector-valued modular form by $f$. Since the components of its divisor must be of the form $v^\perp$ for primitive vectors $v\in 2U\oplus 4A_2'(1/3)$ with $v^2=2/3$ or $2/9$, the function $\eta^8 f$ defines a holomorphic vector-valued modular form of weight $0$, so it is constant, i.e. an invariant of the Weil representation $\rho_{2U\oplus 4A_2(3)}$. This is impossible because the constant coefficient $c(0, 0)$ of $f$ and therefore the coefficient of $q^{1/3} e_0$ of $\eta^8 f$ is nonzero, which proves the lemma.
\end{proof}

\begin{lemma}
The lattice $L_4\oplus A_4'(5)$ is a forbidden component.  
\end{lemma}
\begin{proof}
There is a positive definite even lattice $K$ of rank $4$ such that
$$
2U\oplus L_4 \oplus A_4'(5) \cong U\oplus U(10)\oplus K \oplus A_4'(5). 
$$
It follows from the bound of \cite[Lemma 4.5 and Table 2]{Dit19} with $N=10$ that there is no reflective Borcherds product on  $2U\oplus L_4\oplus A_4'(5)$. 
\end{proof} 

\begin{lemma}
The lattice $A_2\oplus D_8(3)$ is a forbidden component.  
\end{lemma}
\begin{proof}
Use the isometry
$$2U\oplus A_2\oplus D_8(3)\cong U\oplus U(6)\oplus 2A_2\oplus E_6'(3).$$
By \cite[Lemma 4.5 and Table 2]{Dit19} with $N=6$, there is no reflective Borcherds product of any weight on  $2U\oplus A_2\oplus D_8(3)$. 
\end{proof}

\section{Excluding the extraneous root systems}
In this section we rule out the 152 extraneous anti-symmetric root systems using Argument \ref{Argument} and the $78$ forbidden components that were listed above.  We refer to \cite{CS99} for lattice genera. 
\subsection{Rank 4} There are \textbf{3} extraneous root systems of rank $4$:
\begin{enumerate}
    \item $D_{4,36}$: $C=1/6$, $\bP=D_4(18)$. Here, 
    $$
    L < D_4(18) < A_2(6) \oplus A_2(12) < A_1(3) \oplus A_1(9) \oplus A_2(6) < A_1 \oplus A_1(3) \oplus A_2(6)
    $$ 
    using the overlattices $D_4(3) < A_2 \oplus A_2(2)$, $A_2(4) < A_1 \oplus A_1(3)$ and $A_1(9) < A_1$. This overlattice of $L$ contains the forbidden component $A_1 \oplus A_2(6)$ labelled (10).
    \item $G_{2,24}^2$:  $C=1/6$, $L=2A_2(24)$. This lattice contains the forbidden component $A_2(24)$ labelled (5).
    \item $B_{4,14}$: $C=1/2$, $\bP=4A_1(7)$. Then 
    $$
    L < 4A_1(7) < A_1(7) \oplus L_1,
    $$
    where $L_1$ is a certain maximal even overlattice of $3A_1(7)$.
    This is the forbidden component of rank 4 labelled (19).
\end{enumerate}

\subsection{Rank 5} There are \textbf{6} extraneous root systems of rank $5$.
\begin{enumerate}[resume]
\item $A_{1,48}A_{2,72}G_{2,96}$: $C=1/24$, $\bP=A_1(12)\oplus A_2(24)\oplus A_2(96)$. This lattice contains the forbidden component $A_2(24)$ labelled (5). 
\item $A_{3,96}B_{2,72}$: $C=1/24$, $\bP=A_3'(96)\oplus 2A_1(36)$. Then 
$$
L < A_3'(96) \oplus 2A_1(36) < A_1(12) \oplus 2A_1(24) \oplus 2A_1(36) < A_1(3) \oplus 2A_1(9) \oplus 2A_1(24)
$$ 
using the rules $A_3'(8) < A_1 \oplus 2A_1(2)$ and $A_1(4) < A_1$. This overlattice contains the forbidden component $A_1(3) \oplus 2A_1(9)$ that we labelled (8). 
\item $A_{1,16}B_{2,24}G_{2,32}$: $C=1/8$, $\bP=A_1(4)\oplus 2A_1(12)\oplus A_2(32)$. Then 
$$
L < A_1(4) \oplus 2A_1(12) \oplus A_2(32) < A_1 \oplus 2 A_1(12) \oplus A_2(8)
$$ 
using the rules $A_1(4) < A_1$ and $A_2(4) < A_2$. This overlattice contains forbidden component (11), i.e. $A_1 \oplus A_2(8)$.
\item $A_{1,16}^2 C_{3,32}$: $C=1/8$, $\bP=2A_1(4)\oplus A_3'(64)$. Then $$
L < 2A_1(4) \oplus A_3'(64) < 2A_1(4) \oplus A_1(8) \oplus 2A_1(16)
$$
using the overlattice $A_3'(8) < A_1 \oplus 2A_1(2)$. This contains forbidden component (1), i.e. $2A_1(16)$.
\item $A_{1,16}^2 B_{3,40}$: $C=1/8$, $\bP=2A_1(4)\oplus 3A_1(20)$. Then 
$$
L < 2A_1(4) \oplus 3A_1(20).
$$ 
This overlattice contains forbidden component (3), i.e. $2A_1(20)$.
\item $A_{1,16} A_{4,40}$: $C=1/8$, $\bP=A_1(4)\oplus A_4'(40)$. Then 
$$
L < A_1(4) \oplus A_4'(40) < A_1(4) \oplus A_4(8) < A_1(4) \oplus A_4(2)
$$ 
using the overlattice $A_4'(5) < A_4$ and $A_4(4) < A_4$. This case is ruled out because the overlattice contains forbidden component (24), that is, $A_1(4) \oplus A_4(2)$. 
\end{enumerate}

\subsection{Rank 6} There are \textbf{22} extraneous root systems of rank $6$.
\begin{enumerate}[resume]
\item $A_{1,24}^4 G_{2,48}$: $C=1/12$, $\bP=4A_1(6)\oplus A_2(48)$. Then 
$$
L < 4 A_1(6) \oplus A_2(48) < 4 A_1(6) \oplus A_2
$$ 
using the overlattices $A_2(4) < A_2$ and $A_2(3) < A_2$. This case can be excluded because the overlattice contains forbidden component (16), i.e. $2A_1(6) \oplus A_2$.
\item $A_{1,24}^2 B_{2,36}^2$: $C=1/12$, $\bP=2A_1(6)\oplus 4A_1(18)$. Then $$
L<2A_1(6)\oplus 4A_1(18).
$$
This overlattice contains forbidden component (2), i.e. $2A_1(18)$. 
\item $A_{2,36}^2 B_{2,36}$: $C=1/12$, $\bP=2A_2(12)\oplus 2A_1(18)$. Then $$
L<2A_2(12)\oplus 2A_1(18).
$$
This overlattice contains forbidden component (2), i.e. $2A_1(18)$. 
\item $A_{1,24} A_{2,36} A_{3,48}$: $C=1/12$, $\bP=A_1(6)\oplus A_2(12)\oplus A_3'(48)$. Then 
$$
L < A_1(6) \oplus A_2(12) \oplus A_3'(48) < 4A_1(6) \oplus A_2(12) < 4A_1(6) \oplus A_2,
$$ 
using the overlattices $A_3'(8) < 3A_1$, $A_2(4) < A_2$ and $A_2(3) < A_2$. This overlattice contains forbidden component (16), i.e. $2A_1(6) \oplus A_2$.
\item $A_{1,12}^2 A_{2,18} G_{2,24}$: $C=1/6$, $\bP=2A_1(3)\oplus A_2(6)\oplus A_2(24)$. Then 
$$
L < 2A_1(3) \oplus A_2(6) \oplus A_2(24).
$$ 
This overlattice contains forbidden component (5), i.e. $A_2(24)$.
\item $A_{1,12} A_{3,24} B_{2,18}$: $C=1/6$, $\bP=A_1(3)\oplus A_3'(24)\oplus 2A_1(9)$. Then 
$$
L < A_1(3) \oplus A_3'(24) \oplus 2A_1(9).
$$ 
This overlattice contains the forbidden component (8), i.e. $A_1(3) \oplus 2A_1(9)$. 
\item $A_{2,18} B_{2,18}^2$: $C=1/6$, $\bP=A_2(6)\oplus 4A_1(9)$. Then 
$$
L < A_2(6) \oplus 4A_1(9) < A_1 \oplus 3A_1(9) \oplus A_2(6).
$$ 
This overlattice contains forbidden component (10), i.e. $A_1 \oplus A_2(6)$.
\item $A_{1,8}^3 C_{3,16}$: $C=1/4$, $\bP=3A_1(2)\oplus A_3'(32)$. Then 
$$
L < 3A_1(2) \oplus A_3'(32) < 3A_1(2) \oplus A_1(4) \oplus 2A_1(8) < A_1 \oplus 3A_1(2) \oplus 2A_1(8)
$$ 
using $A_3'(8) < A_1 \oplus 2A_1(2)$. This overlattice contains forbidden component (6), i.e. $A_1 \oplus 2A_1(8)$.
\item $A_{1,8}^2 B_{2,12} G_{2,16}$: $C=1/4$, $\bP=2A_1(2)\oplus 2A_1(6)\oplus A_2(16)$. Then 
$$
L < 2A_1(2) \oplus 2A_1(6) \oplus A_2(16).
$$ 
This overlattice contains forbidden component (4), i.e. $A_2(16)$.
\item $A_{1,8}^3 B_{3,20}$: $C=1/4$, $\bP=3A_1(2)\oplus 3A_1(10)$. Then 
$$
L < 3A_1(2) \oplus 3A_1(10) < 3A_1(2) \oplus 2A_1(5) \oplus A_1(10)
$$ 
using $2A_1(2) < 2A_1$. This overlattice contains forbidden component (7), i.e. $2A_1(2) \oplus A_1(5)$.
\item $A_{3,16}^2$: $C=1/4$, $\bP=2A_3'(16)$. Then 
$$
L < 2A_3'(16) < A_3 \oplus A_3'(16)
$$ 
using the rule $A_3'(16) < A_3$. The overlattice $A_3 \oplus A_3'(16)$ is forbidden component (28). 
\item $A_{1,8}^2 A_{4,20}$: $C=1/4$, $\bP=2A_1(2)\oplus A_4'(20)$. Then 
$$
L < 2A_1(2) \oplus A_4'(20) < 2A_1(2) \oplus A_4(4) < 2A_1(2)\oplus A_4
$$
using the rule $A_4'(5) < A_4$. The overlattice $2A_1(2)\oplus A_4$ is forbidden component (29).
\item $A_{2,12}^2 G_{2,16}$: $C=1/4$, $\bP=2A_2(4)\oplus A_2(16)$. Then 
$$
L < 2A_2(4) \oplus A_2(16).
$$ 
This overlattice contains forbidden component (4), i.e. $A_2(16)$.
\item $B_{2,12}^3$: $C=1/4$, $\bP=6A_1(6)$. Then 
$$
L<6A_1(6)<3A_1(6)\oplus A_1(2)\oplus A_2
$$
using the rule $3A_1(6)<A_1(2)\oplus A_2$. This overlattice contains forbidden component (16), i.e. $2A_1(6)\oplus A_2$. 
\item $A_{2,9} B_{2,9} G_{2,12}$: $C=1/3$, $\bP=A_2(3)\oplus \ZZ^2(9)\oplus A_2(12)$. Since $L$ is an even sublattice of $\bP$, 
$$
L < A_2(3) \oplus 2 A_1(9) \oplus A_2(12) < A_1(3) \oplus 3A_1(9) \oplus A_2(3)
$$ 
using $A_2(4) < A_1 \oplus A_1(3)$. This overlattice contains  forbidden component (8), i.e. $A_1(3) \oplus 2A_1(9)$.
\item $A_{1,6} A_{2,9} C_{3,12}$: $C=1/3$, $\bP=\ZZ(3)\oplus A_2(3)\oplus A_3'(24)$. Since $L$ is an even sublattice of $\bP$, 
$$
L < A_1(6) \oplus A_2(3) \oplus A_3'(24).
$$ 
This overlattice contains forbidden component (15), i.e. $A_3'(24)$.
\item $A_{2,9} A_{4,15}$: $C=1/3$, $\bP=A_2(3)\oplus A_4'(15)$. Then 
$$
L < A_2(3) \oplus A_4'(15) < A_2\oplus A_4(3) < A_2 \oplus L_2,
$$ 
where $L_2$ is a maximal overlattice of $A_4(3)$, using $A_2(3) < A_2$ and $A_4'(5) < A_4$. This overlattice contains forbidden component (32), i.e. $A_2 \oplus L_2$.
\item $A_{1,6} A_{2,9} B_{3,15}$: $C=1/3$, $\bP=\ZZ(3)\oplus A_2(3)\oplus \ZZ^3(15)$. The maximal even sublattice of $\ZZ(3)\oplus \ZZ^3(15)$ is an index two sublattice in $A_4'(15)$, and it is contained in $A_4(3)$. Since $L$ is even, we have 
$$
L<A_2(3)\oplus A_4(3) < A_2\oplus L_2
$$
as in the previous case. This overlattice is forbidden component (32). 
\item $A_{1,6} A_{3,12} G_{2,12}$: $C=1/3$, $\bP=\ZZ(3)\oplus A_3'(12)\oplus A_2(12)$. Using the rule $A_3'(4)<\ZZ\oplus 2A_1$, 
$$
\bP < A_2(12)\oplus \ZZ^2(3)\oplus 2A_1(3).
$$
Since $L$ is an even sublattice of $\bP$, we have
$$
L<4A_1(3)\oplus A_2(12) < A_1(3) \oplus A_1\oplus A_2(2) \oplus A_2(12)
$$
using $3A_1(3)<A_1\oplus A_2(2)$. This overlattice contains forbidden component (13), i.e. $A_1\oplus A_2(12)$. 
\item $A_{2,6} G_{2,8}^2$: $C=1/2$, $\bP=A_2(2)\oplus 2A_2(8)$. Then 
$$
L < A_2(2) \oplus 2A_2(8) < A_1(2) \oplus A_1(6) \oplus A_2(2) \oplus A_2(8)
$$ 
using $A_2(4) <A_1 \oplus A_1(3)$. This overlattice contains  forbidden component (12), i.e. $A_1(2) \oplus A_2(8)$.
\item $A_{3,8} B_{3,10}$: $C=1/2$, $\bP=A_3'(8)\oplus 3A_1(5)$. Then 
$$
L < A_3'(8) \oplus 3A_1(5) < A_1 \oplus 2A_1(2) \oplus 3A_1(5)
$$
using $A_3'(8) < A_1 \oplus 2A_1(2)$. This overlattice contains forbidden component (7), i.e. $2A_1(2) \oplus A_1(5)$.
\item $A_{3,8} C_{3,8}$: $C=1/2$, $\bP=A_3'(8)\oplus A_3'(16)$. Then 
$$
L < A_3'(8) \oplus A_3'(16) < A_1 \oplus 2A_1(2) \oplus A_3'(16)
$$ 
using the rule $A_3'(8) < A_1 \oplus 2A_1(2)$. This overlattice contains forbidden component (18), i.e. $A_1 \oplus A_3'(16)$.
\end{enumerate}

\subsection{Rank 7} There are \textbf{5} extraneous root systems of rank $7$.
\begin{enumerate}[resume]
\item $A_{1,48}^3 A_{2,72}^2$: $C=1/24$, $\bP=3A_1(12)\oplus 2A_2(24)$. Then $$
L < 3 A_1(12) \oplus 2A_2(24).
$$ 
This overlattice contains forbidden component (5), i.e. $A_2(24)$.
\item $A_{1,48}^5 B_{2,72}$: $C=1/24$, $\bP=5A_1(12)\oplus 2A_1(36)$. Then $$
L < 5A_1(12) \oplus 2A_1(36) < 5A_1(3) \oplus 2A_1(9),
$$ 
using $A_1(4N) < A_1(N)$. This overlattice contains forbidden component (8), i.e. $A_1(3) \oplus 2A_1(9)$.
\item $A_{1,16} A_{2,24}^3$: $C=1/8$, $\bP=A_1(4)\oplus 3A_2(8)$. Then 
$$
L < A_1(4) \oplus 3A_2(8) < A_1 \oplus 3A_2(8).
$$ 
This overlattice contains forbidden component (11), i.e. $A_1 \oplus A_2(8)$.
\item $A_{1,16}^3 A_{2,24} B_{2,24}$: $C=1/8$, $\bP=3A_1(4)\oplus A_2(8)\oplus 2A_1(12)$. Then 
$$
L < 3A_1(4) \oplus A_2(8) \oplus 2A_1(12) < 3A_1 \oplus 2A_1(12) \oplus A_2(8).
$$ 
This overlattice contains forbidden component (11), i.e. $A_1 \oplus A_2(8)$.
\item $A_{1,16}^4 A_{3,32}$: $C=1/8$, $\bP=4A_1(4)\oplus A_3'(32)$. Then 
$$
L < 4A_1(4) \oplus A_3'(32) < 5A_1(4) \oplus 2A_1(8) < 5A_1 \oplus 2A_1(8)
$$ 
using $A_3'(8) < A_1 \oplus 2A_1(2)$ and $A_1(4) < A_1$. This overlattice contains forbidden component (6), i.e. $A_1 \oplus 2A_1(8)$.
\end{enumerate}

\subsection{Rank 8} There are \textbf{38} extraneous root systems of rank $8$.
\begin{enumerate}[resume]
\item $A_{1,24}^6 A_{2,36}$:    $C=1/12$, $\bP=6A_1(6)\oplus A_2(12)$. Then $$
L < 6A_1(6) \oplus A_2(12) < 6 A_1(6) \oplus A_2
$$ 
using $A_2(4) < A_2$ and $A_2(3) < A_2$. This overlattice contains forbidden component (16), i.e. $2A_1(6) \oplus A_2$.
\item $A_{1,12}^6 B_{2,18}$:     $C=1/6$, $\bP=6A_1(3)\oplus 2A_1(9)$. Then $$
L < 6A_1(3) \oplus 2A_1(9).
$$ 
This overlattice contains forbidden component (8), i.e. $A_1(3) \oplus 2A_1(9)$.
\item $A_{1,12}^4 A_{2,18}^2$:  $C=1/6$, $\bP=4A_1(3)\oplus 2A_2(6)$. By the rule $4A_1(3) < 4A_1$ we obtain
$$
L < 4A_1(3) \oplus 2A_2(6) < 4A_1 \oplus 2A_2(6).
$$ 
This overlattice contains forbidden component (10), i.e. $A_1 \oplus A_2(6)$.
\item $A_{1,8}^2 A_{2,12}^3$:   $C=1/4$, $\bP=2A_1(2)\oplus 3A_2(4)$. Then
$$
L<2A_1(2)\oplus 3A_2(4) < 2A_1(2)\oplus A_2(4)\oplus A_2\oplus A_1\oplus A_1(3)
$$
using $A_2(4)<A_1\oplus A_1(3)<A_2$. This overlattice contains forbidden component (22), i.e. $A_1(3)\oplus A_2\oplus A_2(4)$. 
\item $A_{1,8}^5 A_{3,16}$:     $C=1/4$, $\bP=5A_1(2)\oplus A_3'(16)$. Then 
$$
L<5A_1(2)\oplus A_3'(16) < 3A_1(2)\oplus 2A_1 \oplus A_3'(16)
$$
using $2A_1(2)<2A_1$. This overlattice contains forbidden component (18), i.e. $A_1\oplus A_3'(16)$. 
\item $A_{1,8}^4 A_{2,12} B_{2,12}$:   $C=1/4$, $\bP=4A_1(2)\oplus A_2(4)\oplus 2A_1(6)$. Then 
$$
L < 4A_1(2) \oplus A_2(4) \oplus 2A_1(6) < 2A_1\oplus 2A_1(2)\oplus A_2(4)\oplus 2A_1(6)
$$ 
using $2A_1(2) < 2A_1$. This overlatice contains forbidden component (21), i.e. $A_1\oplus 2A_1(2)\oplus A_2(4)$.
\item $A_{1,6}^6 G_{2,12}$:            $C=1/3$, $\bP=\ZZ^6(3)\oplus A_2(12)$. Since $L$ is an even sublattice of $\bP$,  we have
$$
L < D_6(3) \oplus A_2(12) < 2A_1 \oplus 2A_2 \oplus A_2(12)
$$ 
using $D_6(3) < 2A_1 \oplus 2A_2.$ This overlattice contains forbidden component (13), i.e. $A_1 \oplus  A_2(12)$.
\item $A_{1,6}^2 A_{2,9}^2 B_{2,9}$:    $C=1/3$, $\bP=\ZZ^2(3)\oplus 2A_2(3)\oplus \ZZ^2(9)$. The maximal even sublattice of $\ZZ^2(3)\oplus  \ZZ^2(9)$ is the rescaling by $3$ of a lattice in the genus $2_{\II}^2 3^2$, and it is contained in $2A_2(3)$. Therefore, $L<4A_2(3)$. This case is excluded because it is forbidden component (35).
\item $A_{1,6}^4 B_{2,9}^2$:           $C=1/3$, $\bP=\ZZ^4(3)\oplus \ZZ^4(9)$. The maximal even sublattice of $\bP(1/3)$ is 
$$
2A_2(3)\oplus M_4(3),
$$
where $M_4$ is a maximal even overlattice of $\ZZ^2\oplus\ZZ^2(3)$   (such that $M_4$ has discriminant $36$) with $M_4 < 2A_2$. Therefore we have 
$$
L<2A_2(3)\oplus M_4(3) < 4A_2(3).
$$
This overlattice is forbidden component (35).
\item $A_{1,6}^3 A_{2,9} A_{3,12}$:     $C=1/3$, $\bP=\ZZ^3(3)\oplus A_2(3)\oplus A_3'(12)$. Using the rule $A_3'(4)<\ZZ\oplus 2A_1$, 
$$
\bP < A_2(3)\oplus 2A_1(3)\oplus \ZZ^4(3).
$$
Since $L$ is an even sublattice of $\bP$, we have
$$
L<A_2(3)\oplus 2A_1(3)\oplus D_4(3) < A_2(3)\oplus 2A_1(3) \oplus A_2\oplus A_2(2)
$$
using $D_4(3)<A_2\oplus A_2(2)$. This overlattice contains forbidden component (25), i.e. $2A_1(3) \oplus A_2 \oplus A_2(3)$.
\item $A_{2,9}^4$:         $C=1/3$, $\bP=4A_2(3)$. Then
$$
L<4A_2(3).
$$
This overlattice is the forbidden component (35).
\item $A_{1,4}^2 A_{3,8}^2$:             $C=1/2$, $\bP=2A_1\oplus 2A_3'(8)$. Then 
$$
L<2A_1\oplus 2A_3'(8) < A_3'(8)\oplus 5A_1 < A_3'(8)\oplus D_4\oplus A_1
$$
using $A_3'(8)<3A_1$ and $4A_1<D_4$. This overlattice is the forbidden component (37).
\item $A_{1,4}^4 A_{4,10}$:   $C=1/2$, $\bP=4A_1\oplus A_4'(10)$. Using $A_4'(5) < A_4$ we have
$$
L < 4A_1 \oplus A_4'(10) < 4A_1 \oplus A_4(2)
$$ 
This overlattice contains forbidden component (30), i.e. $2A_1 \oplus A_4(2)$.
\item $A_{1,4} A_{2,6} A_{3,8} B_{2,6}$:   $C=1/2$, $\bP=A_1\oplus A_2(2)\oplus A_3'(8)\oplus 2A_1(3)$. Then 
$$
L<A_1\oplus A_2(2)\oplus A_3'(8)\oplus 2A_1(3) < A_3'(8) \oplus A_2\oplus 3A_1
$$
using $A_1(3)\oplus A_2(2) < 3A_1$ and $A_1\oplus A_1(3)<A_2$. This overlattice  contains forbidden component (27), i.e. $A_1 \oplus A_2 \oplus A_3'(8)$. 
\item $A_{1,4}^5 B_{3,10}$:                $C=1/2$, $\bP=5A_1\oplus 3A_1(5)$. Then 
$$
L < 5A_1 \oplus 3A_1(5).
$$ 
This overlattice contains forbidden component (20), i.e. $3A_1 \oplus 2A_1(5)$.
\item $A_{2,6}^2 B_{2,6}^2$:              $C=1/2$, $\bP=2A_2(2)\oplus 4A_1(3)$. Then 
$$
L < 2A_2(2) \oplus 4A_1(3) < 2A_2(2) \oplus D_4
$$ 
using the rule $4A_1(3) < D_4$. This overlattice is forbidden component (39).
\item $A_{1,4}^2 A_{2,6}^2 G_{2,8}$:        $C=1/2$, $\bP=2A_1\oplus 2A_2(2)\oplus A_2(8)$. Then 
$$
L < 2A_1 \oplus 2A_2(2) \oplus A_2(8).
$$ 
This overlattice contains forbidden component (11), i.e. $A_1 \oplus A_2(8)$.
\item $A_{1,4}^4 B_{2,6} G_{2,8}$:            $C=1/2$, $\bP=4A_1\oplus 2A_1(3)\oplus A_2(8)$. Then 
$$
L < 4A_1 \oplus 2A_1(3) \oplus A_2(8).
$$ 
This overlattice contains forbidden component (11), i.e. $A_1 \oplus A_2(8)$.
\item $A_{1,4}^5 C_{3,8}$:                  $C=1/2$, $\bP=5A_1\oplus A_3'(16)$. Then 
$$
L<5A_1\oplus A_3'(16).
$$
This overlattice contains forbidden component (28), i.e. $A_1\oplus A_3'(16)$. 
\item $A_{1,4}^2 B_{2,6}^3$:                $C=1/2$, $\bP=2A_1\oplus 6A_1(3)$. Then
$$
L<2A_1\oplus 6A_1(3) < 4A_1\oplus 2A_2(2) < D_4\oplus 2A_2(2)
$$
using $3A_1(3)<A_1\oplus A_2(2)$ and $4A_1<D_4$. This overlattice is forbidden component (39).
\item $A_{1,3}^4 G_{2,6}^2$:                $C=2/3$, $L=K\oplus 2A_2(6)$ with $4A_1(3)<K<\ZZ^4(3/2)$. Then 
$$
\ZZ^4(4)<K(2/3)<\ZZ^4.
$$
Since $K(2/3)$ is integral and $K=K(2/3)(3/2)$ is even, $K(2/3)$ is the $2$-scaling of an even lattice. Therefore, $K(2/3)<D_4(2)$. Using the overlattice $D_4(3)<D_4$ we then have
$$
L<2A_2(6)\oplus D_4(3) < 2A_2(6)\oplus D_4
$$
This overlattice contains forbidden component (31), i.e. $A_2(6) \oplus D_4$. 
\item $A_{1,3}^4 D_{4,9}$:                   $C=2/3$, $\bP=\ZZ^4(3/2)\oplus D_4(9/2)$ and $\bQ=4A_1(3)\oplus D_4(9)$. Then $L(2/3)$ is integral and
$$
4A_1(2)\oplus D_4(6) < L(2/3) < \ZZ^4 \oplus D_4(3). 
$$
We see from $L=L(2/3)(3/2)$ that $L(2/3)$ is the $2$-scaling of an even lattice. Therefore, $L(1/3)$ is an even lattice bounded by 
$$
4A_1\oplus D_4(3) < L(1/3) < D_4(1/2) \oplus D_4(3/2) \cong D_4'\oplus D_4'(3). 
$$
It follows that $L(1/3)<D_4\oplus D_4(3)$. Then
$$
L<D_4(3)\oplus D_4(9) < A_2(3)\oplus A_2(6)\oplus D_4,
$$
using $D_4(3)<A_2\oplus A_2(2)<D_4$. This overlattice contains forbidden component (31), i.e. $A_2(6) \oplus D_4$.
\item $A_{1,2} A_{4,5} C_{3,4}$:            $C=1$, $\bP=\ZZ\oplus A_4'(5)\oplus A_3'(8)$. Since $L$ is an even sublattice of $\bP$,  
$$
L < A_1(2) \oplus A_4'(5) \oplus A_3'(8) < A_1 \oplus 3A_1(2) \oplus A_4,
$$
where used $A_4'(5) < A_4$ and $A_3'(8) < A_1 \oplus 2A_1(2)$. This overlattice contains forbidden component (29), i.e. $2A_1(2) \oplus A_4$.
\item $A_{1,2}^4 B_{4,7}$:                 $C=1$, $\bP=\ZZ^4\oplus \ZZ^4(7)$. Since $L$ is even, it is contained in the maximal even sublattice of $\bP$: 
$$
L < 3 L_3 \oplus L_3(2)
$$ 
where $L_3$ is $\mathbb{Z}^2$ with Gram matrix $\begin{psmallmatrix} 2 & 1 \\ 1 & 4 \end{psmallmatrix}$. This overlattice contains forbidden component (33), i.e. $2L_3 \oplus L_3(2)$.
\item $B_{2,3}^2 D_{4,6}$:                   $C=1$, $\bP=\ZZ^4(3)\oplus D_4(3)$. Since $L$ is an even sublattice of $\bP$, we have $L<2D_4(3)$. This overlattice is forbidden component (41).
\item $A_{1,2} B_{2,3} B_{3,5} G_{2,4}$:      $C=1$, $\bP=\ZZ\oplus \ZZ^2(3)\oplus \ZZ^3(5)\oplus A_2(4)$. Since $L$ is even, it is contained in the maximal even sublattice $$Q \oplus A_2(4) \oplus A_2(5)$$ of $\bP$,  where $Q$ has genus $[2^2]_6 3^1 5^{-1}$ and discriminant $60$. Therefore, $$L < Q \oplus A_2(4) \oplus A_2(5) < Q \oplus A_1 \oplus A_1(3) \oplus A_2(5)$$ using $A_2(4) < A_1 \oplus A_1(3)$. This overlattice contains forbidden component (9), i.e. $A_1(3) \oplus A_2(5)$.
\item $A_{1,2} A_{4,5} B_{3,5}$:               $C=1$, $\bP=\ZZ\oplus A_4'(5)\oplus \ZZ^3(5)$. Let $L_4$ denote the maximal even sublattice of $\ZZ\oplus \ZZ^3(5)$. The overlattice $L_4\oplus A_4'(5)$ of $L$ is forbidden component (44).  
\item $A_{1,2}^2 B_{3,5}^2$:                    $C=1$, $\bP=\ZZ^2\oplus \ZZ^6(5)$. The maximal even sublattice of $\bP$ is $L_4\oplus A_4'(5)$, which is forbidden component (44) and is an even overlattice of $L$.
\item $A_{1,2} B_{2,3} C_{3,4} G_{2,4}$:        $C=1$, $\bP=\ZZ\oplus \ZZ^2(3)\oplus A_3'(8)\oplus A_2(4)$. Since $L$ is even, it is contained in the maximal even sublattice $$L < Q \oplus A_2(4) \oplus A_3'(8)$$ of $\bP$, where $Q$ has genus $[4^1]_7 3^2$. Then $$L < A_1 \oplus 2A_1(2) \oplus A_2(4) \oplus Q$$ using the rule $A_3'(8) < A_1 \oplus 2A_1(2).$ This overlattice contains forbidden component (21), i.e. $A_1 \oplus 2A_1(2) \oplus A_2(4)$.
\item $A_{1,2}^2 G_{2,4}^3$:                     $C=1$, $\bP=\ZZ^2\oplus 3A_2(4)$. Since $L$ is an even sublattice of $\bP$, we have
$$
L<2A_1\oplus 3A_2(4) < 3A_1\oplus A_1(3)\oplus A_2\oplus A_2(4)
$$
using $A_2(4)<A_1\oplus A_1(3)<A_2$. This overlattice contains forbidden component (22), i.e. $A_1(3)\oplus A_2\oplus A_2(4)$. 
\item $A_{1,2}^2 C_{3,4}^2$:                    $C=1$, $\bP=\ZZ^2\oplus 2A_3'(8)$. The maximal even sublattice of $\bP$ is $2A_1\oplus 2A_3'(8)$. Then
$$
L<2A_1\oplus 2A_3'(8) < A_3'(8) \oplus 5A_1 < A_3'(8)\oplus D_4\oplus A_1
$$
using $A_3'(8)<3A_1$ and $4A_1<D_4$. The overlattice $A_1 \oplus A_3'(8)\oplus D_4$ is forbidden component (37). 
\item $A_{1,2}^2 B_{3,5} C_{3,4}$:             $C=1$, $\bP=\ZZ^2\oplus \ZZ^3(5)\oplus A_3'(8)$. Since $L$ is even, it is contained in the maximal even sublattice $A_3'(8) \oplus Q$ of $\bP$, where $Q$ has genus $[4^1]_1 5^3$. Then 
$$
L < A_1(2) \oplus A_3'(8) \oplus A_4 < A_1 \oplus 3A_1(2) \oplus A_4
$$ 
using $Q < A_1(2) \oplus A_4$ and $A_3'(8) < A_1 \oplus 2A_1(2)$. This overlattice contains forbidden component (29), i.e. $2A_1(2) \oplus A_4$.
\item $B_{2,3}^2 G_{2,4}^2$:    $C=1$, $\bP=\ZZ^4(3)\oplus 2A_2(4)$. Since $L$ is an even sublattice of $\bP$, we have
$$
L<D_4(3)\oplus 2A_2(4) < D_4\oplus 2A_2(2)
$$
using $2A_2(2)<2A_2$ and $D_4(3)<D_4$. The overlattice $D_4\oplus 2A_2(2)$ is forbidden component (39).
\item $A_{4,5} B_{2,3} G_{2,4}$:      $C=1$, $\bP=A_4'(5)\oplus \ZZ^2(3)\oplus A_2(4)$. Since $L$ is an even sublattice of $\bP$ and $A_4'(5)<A_4$ we have
$$
L<A_4'(5)\oplus 2A_1(3)\oplus A_2(4) < A_4\oplus 2A_1(3)\oplus A_2(4).
$$
This overlattice contains forbidden component (23), i.e. $A_1(3)\oplus A_4$. 
\item $A_{1,2}^4 C_{4,5}$:                  $C=1$, $\bP=\ZZ^4\oplus D_4(5)$. Since $L$ is even,  
$$
L<D_4\oplus D_4(5) < D_4 \oplus L_5,
$$
where $L_5$ is an even overlattice of $D_4(5)$ of genus $2_{\II}^{-2} 5^2$. This overlattice is forbidden component (45).
\item $A_{1,2}^2 D_{4,6} G_{2,4}$:            $C=1$, $\bP=\ZZ^2\oplus D_4(3)\oplus A_2(4)$. Since $L$ is even,
$$
L<2A_1\oplus D_4(3)\oplus A_2(4) < 2A_1\oplus A_2\oplus A_2(2)\oplus A_2(4)
$$
using $D_4(3)<A_2\oplus A_2(2)$. This overlattice contains forbidden component (26), i.e. $A_2\oplus A_2(2)\oplus A_2(4)$. 
\item $C_{3,3}^2 G_{2,3}$:                  $C=4/3$, $L=A_2(3)\oplus K$ with $6A_1(3)<K<2A_3'(6)$. Then 
$$
6A_1 < K(1/3) < 2A_3'(2). 
$$
Since $K$ is even and $K(4/3)$ is integral, $K(1/3)$ is even. It follows that $K(1/3)=6A_1$, and thus
$$
L< A_2(3)\oplus 6A_1(3) < A_2(3)\oplus A_1\oplus A_2(2) \oplus 3A_1(3) < A_2(2)\oplus 2A_1(3)\oplus A_2 \oplus A_2(3)
$$
using $3A_1(3)<A_1\oplus A_2(2)$ and $A_1\oplus A_1(3)<A_2$. This overlattice contains forbidden component (25), i.e. $2A_1(3) \oplus A_2 \oplus A_2(3)$. 
\item $G_{2,3}^4$:   $C=4/3$, $L=4A_2(3)$ is forbidden component (35). 
\end{enumerate}

\subsection{Rank 9} There is only \textbf{1} extraneous root system of rank $9$.
\begin{enumerate}[resume]
\item $A_{1,16}^9$: $C=1/8$, $\bP=9A_1(4)$. Then
$$
L<9A_1(4)<4A_1\oplus A_1(4)\oplus D_4.
$$
This overlattice contains forbidden component (38), i.e. $3A_1 \oplus A_1(4) \oplus D_4.$
\end{enumerate}

\subsection{Rank 10} There are \textbf{24} extraneous root systems of rank $10$.
\begin{enumerate}[resume]
\item $A_{1,8}^{10}$:  $C=1/4$, $\bP=10A_1(2)$. Then 
$$
L < 10A_1(2) < 3A_1(2) \oplus D_7.
$$ 
This overlattice is forbidden component (53).
\item $A_{1,6}^8 A_{2,9}$: $C=1/3$, $\bP=\mathbb{Z}^8(3)\oplus A_2(3)$.  Since $L$ is even, we have
$$
L<D_8(3)\oplus A_2(3) < D_8(3)\oplus A_2.
$$ 
The overlattice $D_8(3)\oplus A_2$ is forbidden component (55).
\item $A_{1,4}^4 A_{2,6}^3$:  $C=1/2$, $\bP=4A_1\oplus 3A_2(2)$. Using $4A_1 < D_4$ we obtain 
$$
L < 4A_1 \oplus 3A_2(2) < 3A_2(2) \oplus D_4
$$ 
This overlattice contains forbidden component (39), i.e. $2A_2(2) \oplus D_4$. 
\item $A_{1,4}^7 A_{3,8}$:  $C=1/2$, $\bP=7A_1\oplus A_3'(8)$. Then
$$
L<7A_1\oplus A_3'(8) < 3A_1\oplus D_4 \oplus A_3'(8).
$$
This overlattice contains forbidden component (37), i.e. $A_1 \oplus A_3'(8)\oplus D_4$. 
\item $A_{1,4}^6 A_{2,6} B_{2,6}$: $C=1/2$, $\bP=6A_1\oplus A_2(2)\oplus 2A_1(3)$. Then
$$
L<6A_1\oplus A_2(2)\oplus 2A_1(3) < D_4\oplus 2A_2\oplus A_2(2)
$$
using the rules $4A_1<D_4$ and $A_1\oplus A_1(3) < A_2$. The overlattice $D_4\oplus 2A_2\oplus A_2(2)$ is forbidden component (50). 
\item $A_{2,4}^4 B_{2,4}$:  $C=3/4$, $\bP=4A_2'(4)\oplus 2A_1(2)$, $\bQ=4A_2(4)\oplus 2A_1(4)$. Then
$$
2A_1\oplus 4A_2 < L(1/4) < \ZZ^2 \oplus 4A_2'.
$$
Since $L(3/4)$ is integral and $L$ is even, $L(1/4)$ is integral. It follows that $L(1/4)<\ZZ^2 \oplus E_8$, and therefore that
$
L<2A_1(2)\oplus E_8(4).
$
Note that
$$
2U\oplus 2A_1(2)\oplus E_8(4) \cong 2U\oplus 10A_1(2)
$$
and
$$
10A_1(2) < 4A_1(2)\oplus 6A_1 <  4A_1(2)\oplus D_4\oplus 2A_1 < 3A_1(2)\oplus D_5\oplus 2A_1 
$$
using $2A_1(2)<2A_1$ and $4A_1<D_4$ and $D_4\oplus A_1(2)<D_5$. This overlattice contains forbidden component (48), i.e. $A_1\oplus 3A_1(2)\oplus D_5$.
\item $A_{1,2} A_{2,3}^3 B_{3,5}$:  $C=1$, $\bP=\ZZ\oplus 3A_2\oplus \ZZ^3(5)$. Since $L$ is even, it is contained in the maximal even sublattice $Q$ of $\bP$, where $Q$ is an index two sublattice of $A_4$ in the genus $2^2_{\II}5^3$. Therefore we have
$$
L<Q\oplus 3A_2 < A_4\oplus 3A_2. 
$$
This overlattice contains forbidden component (36), i.e. $2A_2 \oplus  A_4$.
\item $A_{1,2}^3 A_{2,3} B_{2,3} C_{3,4}$: $C=1$, $\bP=\ZZ^3\oplus A_2\oplus \ZZ^2(3)\oplus A_3'(8)$. The maximal even sublattice of $\ZZ^3\oplus \ZZ^2(3)$ is $A_1(2)\oplus 2A_2$. Since $L$ is an even sublattice of $\bP$, we have
$$
L<A_1(2)\oplus 3 A_2\oplus A_3'(8).
$$
This overlattice contains forbidden component (34), i.e. $2A_2\oplus A_3'(8)$. 
\item $A_{1,2} A_{3,4} B_{2,3}^3$: $C=1$, we have
$$
\bP=\mathbb{Z}\oplus A_3'(4)\oplus \mathbb{Z}^6(3)<\mathbb{Z}^4\oplus 2A_1\oplus 2A_2
$$
using the rules $A_3'(4)<\mathbb{Z}\oplus 2A_1$ and $\mathbb{Z}^3(3)<\mathbb{Z}\oplus A_2$. 
Therefore, $L<D_4\oplus 2A_1\oplus 2A_2$. The overlattice $D_4\oplus 2A_1\oplus 2A_2$ is forbidden component (49). 
\item $A_{1,2}^4 A_{2,3} G_{2,4}^2$: $C=1$, $\bP=\ZZ^4\oplus A_2\oplus 2A_2(4)$. Since $L$ is even, we have
$$
L < D_4 \oplus A_2 \oplus 2A_2(4) < A_2 \oplus 2A_2(2) \oplus D_4
$$ 
using $2A_2(2) < 2A_2$. This overlattice contains forbidden component (39), i.e. $2A_2(2) \oplus D_4$. 
\item $A_{1,2}^2 A_{2,3} B_{2,3}^2 G_{2,4}$:  $C=1$, $\bP=\ZZ^2\oplus A_2\oplus \ZZ^4(3)\oplus A_2(4)$. The maximal even sublattice of $\bP$ is
$
2A_1(3) \oplus 3A_2 \oplus A_2(4),
$
and it is an even overlattice of $L$. 
This overlattice contains forbidden component (22), that is, $A_1(3) \oplus  A_2 \oplus A_2(4)$.
\item $A_{2,3}^3 A_{4,5}$: $C=1$, $\bP=3A_2\oplus A_4'(5)$. Then 
$$
L < 3A_2 \oplus A_4'(5) < 3A_2 \oplus A_4.
$$ 
This overlattice contains forbidden component (36), i.e. $2A_2 \oplus A_4$. 
\item $A_{1,2}^2 A_{2,3} A_{4,5} B_{2,3}$:  $C=1$, $\bP=\ZZ^2\oplus A_2\oplus A_4'(5)\oplus \ZZ^2(3)$. The maximal even sublattice of $\bP$ decomposes as $Q \oplus A_2 \oplus A_4'(5)$, where $Q$ has genus $2_{\II}^2 3^2$. Since $L$ is even, we have
$$
L < Q \oplus A_2 \oplus A_4'(5) < 3A_2 \oplus A_4,
$$ 
using $Q < 2A_2$ and $A_4'(5) < A_4$. 
This overlattice contains forbidden component (36), i.e. $2A_2 \oplus A_4$. 
\item $A_{1,2} A_{2,3}^3 C_{3,4}$: $C=1$, $\bP=\ZZ\oplus 3A_2\oplus A_3'(8)$. Since $L$ is even, 
$$
L < A_1(2) \oplus 3A_2 \oplus A_3'(8).
$$ 
This overlattice contains forbidden component (34), i.e. $A_3'(8)\oplus 2A_2$.
\item $A_{2,3}^3 B_{2,3} G_{2,4}$:  $C=1$, $\bP=3A_2\oplus \ZZ^2(3)\oplus A_2(4)$. Since $L$ is even, we have
$$
L < 2A_1(3) \oplus 3A_2 \oplus A_2(4).
$$ 
This overlattice contains forbidden component (22), that is, $A_1(3)\oplus A_2\oplus A_2(4)$. 
\item $A_{1,2}^4 A_{2,3} D_{4,6}$: $C=1$, $\bP=\ZZ^4\oplus A_2\oplus D_3$. Since $L$ is even and $D_4(3) < A_2 \oplus A_2(2)$ we have
$$
L < A_2 \oplus D_4 \oplus D_4(3) < 2A_2 \oplus A_2(2) \oplus D_4
$$ 
The overlattice  $2A_2 \oplus A_2(2) \oplus D_4$ is forbidden component (50). 
\item $A_{1,2}^4 A_{3,4} C_{3,4}$:  $C=1$, $\bP=\ZZ^4\oplus A_3'(4)\oplus A_3'(8)$. The maximal even sublattice of $\bP$ is $2A_1(2) \oplus D_5 \oplus A_3'(8) $. Since $L$ is even,  
$$
L<2A_1(2) \oplus D_5 \oplus A_3'(8) < A_1 \oplus 4A_1(2) \oplus D_5
$$
using the rule $A_3'(8) < A_1 \oplus 2A_1(2)$. This overlattice contains  forbidden component (48), i.e. $A_1 \oplus 3A_1(2) \oplus D_5$.
\item $A_{1,2}^3 A_{3,4} A_{4,5}$:  $C=1$, we have
$$
\bP=\ZZ^3\oplus A_3'(4)\oplus A_4'(5) < \ZZ^4 \oplus 2A_1\oplus A_4'(5)
$$
using $A_3'(4)<\ZZ\oplus 2A_1$. Since $L$ is even, 
$$
L<D_4\oplus 2A_1\oplus A_4'(5) < D_4\oplus 2A_1\oplus A_4.
$$
This overlattice contains forbidden component (40), i.e. $A_4\oplus D_4$. 
\item $A_{2,3} A_{3,4}^2 B_{2,3}$:  $C=1$, $\bP=A_2\oplus 2A_3'(4)\oplus \ZZ^2(3)$. The maximal even sublattice of $\bP$ decomposes as $A_1(2) \oplus A_3(3) \oplus Q$, where $Q$ has genus $[4^3]_5$.  The overlattice $A_1(2) \oplus A_3(3) \oplus Q$ of $L$ contains forbidden component (17), i.e. $A_1(2) \oplus A_3(3)$.
\item $A_{2,3} B_{2,3}^4$:  $C=1$, $\bP=A_2\oplus \ZZ^8(3)$. Since $L$ is even, we have 
$$
L<A_2\oplus D_8(3).
$$
This overlattice is forbidden component (55). 
\item $A_{1,2}^4 A_{3,4} B_{3,5}$:  $C=1$, $\bP=\ZZ^4\oplus A_3'(4)\oplus \ZZ^3(5)$.  The maximal even sublattice of $\bP$ splits as $A_4 \oplus Q$, where $Q$ is a sublattice of $D_6$ with genus $[2^2]_6 4^2 5^2$ and discriminant $1600$. Then $$
L < A_4 \oplus Q < A_4 \oplus D_6.
$$ 
The overlattice $A_4 \oplus D_6$ is forbidden component (52). 
\item $A_{1,2} A_{2,3}^2 A_{3,4} G_{2,4}$: $C=1$, $\bP=\ZZ\oplus 2A_2\oplus A_3'(4)\oplus A_2(4)$. The maximal even sublattice of $\bP$ is $2A_2 \oplus A_2(4) \oplus D_4(2)$. Therefore,
$$
L<2A_2 \oplus A_2(4) \oplus D_4(2) < 4A_1 \oplus 3A_2,
$$
using $A_2(4) < A_2$ and $D_4(2) < 4A_1$. This overlattice contains forbidden component (46), i.e. $3A_1 \oplus 3A_2$.
\item $A_{1,2}^3 A_{2,3} B_{2,3} B_{3,5}$: $C=1$, $\bP=\ZZ^3\oplus A_2\oplus \ZZ^2(3)\oplus \ZZ^3(5)$. The maximal even sublattice of $\bP$ splits as $2A_2 \oplus Q$ where $Q$ has genus $2_{\II}^2 3^{-1} 5^3$. Therefore, 
$$
L < 2A_2 \oplus Q < 3A_2 \oplus A_4
$$ 
since $Q < A_2 \oplus A_4$. This overlattice contains forbidden component (36), i.e. $2A_2 \oplus A_4$.
\item $A_{1,2}^3 A_{3,4} B_{2,3} G_{2,4}$: $C=1$, $\bP=\ZZ^3\oplus A_3'(4)\oplus \ZZ^2(3)\oplus A_2(4)$. The maximal even sublattice of $\bP$ is $A_2 \oplus 2A_2(4) \oplus D_4$. Therefore,
$$
L<A_2 \oplus 2A_2(4) \oplus D_4 < 2A_1 \oplus 2A_1(3) \oplus A_2 \oplus D_4
$$
using $A_2(4) < A_1 \oplus A_1(3)$. This overlattice contains  forbidden component (47), i.e. $2A_1 \oplus A_1(3) \oplus A_2 \oplus D_4$.
\end{enumerate}

\subsection{Rank 12} There are \textbf{21} extraneous root systems of rank $12$.
\begin{enumerate}[resume]
\item $A_{1,2}^6 A_{2,3}^2 G_{2,4}$:    $C=1$, $\bP=\ZZ^6\oplus 2A_2\oplus A_2(4)$. Since $L$ is even, 
$$
L < 2A_2 \oplus A_2(4) \oplus D_6.
$$ 
This overlattice contains forbidden component (43), i.e. $A_2(4) \oplus D_6$.
\item $A_{1,2}^9 B_{3,5}$:             $C=1$, $\bP=\ZZ^9\oplus \ZZ^3(5)$.  The maximal even sublattice of $\bP$ is $A_4'(5) \oplus D_8$. Therefore,
$$L < A_4'(5) \oplus D_8 < A_4 \oplus D_8.$$
The overlattice $A_4 \oplus D_8$ is forbidden component (62).
\item $A_{1,2}^3 A_{2,3}^3 A_{3,4}$:      $C=1$, $\bP=\ZZ^3\oplus 3A_2\oplus A_3'(4)$. The maximal even sublattice of $\bP$ splits as  $Q \oplus 3A_2$, where $Q$ is an index two sublattice of $2A_1 \oplus D_4$ in the genus $2^2_6 4^2$. Therefore,
$$
L < Q \oplus 3A_2 < 2A_1 \oplus 3A_2 \oplus D_4.
$$ 
This overlattice contains forbidden component (49), i.e. $2A_1 \oplus 2A_2 \oplus D_4$.
\item $A_{1,2}^5 A_{2,3} A_{3,4} B_{2,3}$:     $C=1$, $\bP=\ZZ^5\oplus A_2\oplus A_3'(4)\oplus \ZZ^2(3)$. The maximal even sublattice of $\bP$ is $2A_2\oplus A_2(4)\oplus D_6$. This is an even overlattice of $L$ containing forbidden component (43), i.e. $A_2(4)\oplus D_6$.
\item $A_{1,2}^8 B_{2,3} G_{2,4}$:        $C=1$, $\bP=\ZZ^8\oplus \ZZ^2(3)\oplus A_2(4)$. The maximal even sublattice of $\bP$ is $2A_2\oplus A_2(4)\oplus D_6$. This is an even overlattice of $L$ containing forbidden component (43), i.e. $A_2(4)\oplus D_6$.
\item $A_{1,2}^6 A_{3,4}^2$:             $C=1$, $\bP=\ZZ^6\oplus 2A_3'(4)$. The maximal even sublattice of $\bP$ is a sublattice of $A_1 \oplus D_4 \oplus L_6$, where $L_6$ is a rank seven lattice in the genus $2^1_7 4_{\II}^2$. Therefore,
$$
L< A_1 \oplus D_4 \oplus L_6.
$$
This overlattice contains forbidden component (59), i.e. $L_6 \oplus D_4$.
\item $A_{1,2}^2 A_{2,3}^4 B_{2,3}$:        $C=1$, $\bP=\ZZ^2\oplus 4A_2\oplus \ZZ^2(3)$. The maximal even sublattice of $\bP$ is $4A_2 \oplus L_7$, where $L_7$ has genus $2_{\II}^2 3^2$. The overlattice $4A_2 \oplus L_7$ of $L$ contains forbidden component (56), i.e. $3A_2 \oplus L_7$.
\item $A_{1,2}^8 A_{4,5}$:                   $C=1$, $\bP=\ZZ^8\oplus A_4'(5)$. The maximal even sublattice of $\bP$ is $A_4'(5) \oplus D_8 $. Therefore, 
$$
L<A_4'(5) \oplus D_8 < A_4 \oplus D_8.
$$
The overlattice $A_4 \oplus D_8$ is forbidden component (62). 
\item $A_{1,2}^4 A_{2,3}^2 B_{2,3}^2$:        $C=1$, $\bP=\ZZ^4\oplus 2A_2\oplus \ZZ^4(3)$. The maximal even sublattice of $\bP$ is $4A_2 \oplus L_7$, where $L_7$ has genus $2_{\II}^2 3^2$ as before. The overlattice $4A_2 \oplus L_7$ of $L$ contains forbidden component (56), i.e. $3A_2 \oplus L_7$.
\item $A_{1,2}^9 C_{3,4}$:                  $C=1$, $\bP=\ZZ^9\oplus A_3'(8)$. Since $L$ is even, 
$$
L < A_3'(8) \oplus D_9 < 3A_1 \oplus D_9,
$$ 
using the rule $A_3'(8) < 3A_1$. The overlattice $D_9\oplus 3A_1$ is forbidden component (64).
\item $A_{1,2}^6 B_{2,3}^3$:                $C=1$, $\bP=\ZZ^6\oplus \ZZ^6(3)$. The maximal even sublattice of $\bP$ is $4A_2 \oplus L_7$, where $L_7$ is the maximal even sublattice of $\ZZ^2 \oplus \ZZ^2(3)$ as before. The overlattice $4A_2 \oplus L_7$ of $L$ contains forbidden component (56), i.e. $3A_2 \oplus L_7$.
\item $A_{3,2}^2 G_{2,2}^3$:                $C=2$, $L=3A_2(2)\oplus K$ with $2A_3(2)<K<2A_3'(2)$. Then $K<\mathbb{Z}^2(1/2)\oplus \mathbb{Z}^4$. Thus $K<\mathbb{Z}^6$ and further $K<D_6$. Therefore, $L<3A_2(2)\oplus D_6$. This overlattice contains forbidden component (51), i.e. $2A_2(2)\oplus D_6$. 

\item $A_{1,1}^3 C_{3,2} D_{4,3} G_{2,2}$:       $C=2$, $L=A_2(2)\oplus K$ with
$$
3A_1\oplus 3A_1(2)\oplus D_4(3) < K < \mathbb{Z}^3(1/2)\oplus A_3'(4)\oplus D_4'(3).
$$ 
Then $K$ is contained in a maximal even overlattice of $3A_1\oplus 3A_1(2)\oplus D_4(3)$ in the genus of $A_1\oplus A_1(2)\oplus E_8$. Therefore, $L$ is contained in a lattice in the genus of $A_1\oplus A_1(2)\oplus A_2(2)\oplus E_8$ which contains forbidden component (58), i.e. $A_1\oplus A_2(2)\oplus E_8$. 

\item $A_{1,1}^3 C_{3,2} G_{2,2}^3$:     $C=2$,  $L=3A_2(2)\oplus K$ with 
$$
3A_1\oplus 3A_1(2)<K<\mathbb{Z}^3(1/2)\oplus A_3'(4).
$$
Then $K$ is contained in a maximal even overlattice of $3A_1\oplus 3A_1(2)$ in the genus of $A_1\oplus D_5$. Therefore, $L$ is contained in a lattice in the genus of $A_1\oplus D_5\oplus 3A_2(2)$ which contains forbidden component (42), i.e. $A_1\oplus A_2(2)\oplus D_5$. 
\item $A_{1,1}^3 C_{3,2}^3$:    $C=2$, $\bQ=3A_1\oplus 9A_1(2)$. Then $L$ is contained in a maximal even overlattice of $3A_1\oplus 9A_1(2)$ in the genus of $A_1\oplus A_3\oplus E_8$ which is forbidden component (63). 
\item $A_{3,2}^2 D_{4,3} G_{2,2}$:         $C=2$, $L=A_2(2)\oplus K$ with $2A_3(2)\oplus D_4(3)<K$. Then $K$ is contained in a maximal even overlattice of $2A_3(2)\oplus D_4(3)$ in the genus of $2A_1\oplus E_8$. Therefore, $L$ is contained in a lattice in the genus of $2A_1\oplus A_2(2)\oplus E_8$ which contains forbidden component (58), i.e. $A_1\oplus A_2(2)\oplus E_8$. 
\item $A_{3,2}^2 C_{3,2}^2$:          $C=2$, 
$$
2A_3(2)\oplus 6A_1(2) < L < 2A_3'(2)\oplus 2A_3'(4).
$$
We see from the upper bound that $L<\ZZ^{2}(1/2)\oplus \ZZ^4\oplus 2A_3'(4)$. Therefore, $L<\ZZ^6\oplus 2A_3'(4)$. The maximal even sublattice of $\ZZ^6\oplus 2A_3'(4)$ is forbidden component (65). 
\item $A_{1,1}^3 A_{5,3} G_{2,2}^2$:       $C=2$, $L=2A_2(2)\oplus K$ with
$$
3A_1\oplus A_5(3)<K.
$$
Then $K$ is contained in a maximal even overlattice of $3A_1\oplus A_5(3)$ in the genus of $2A_2\oplus D_4$. Therefore, $L$ is contained in a lattice in the genus of $2A_2\oplus D_4\oplus 2A_2(2)$ which contains forbidden component (39), i.e. $2A_2(2)\oplus D_4$. 
\item $A_{1,1}^4 A_{3,2} D_{5,4}$:         $C=2$, $\bQ=4A_1\oplus A_3(2)\oplus D_5(4)<L$.  Then $L$ is contained in a maximal even overlattice of $4A_1\oplus A_3(2)\oplus D_5(4)$ in the genus of $A_1\oplus A_3\oplus E_8$ which is forbidden component (63). 
\item $A_{2,1}^2 D_{4,2} F_{4,3}$:  $C=3$, $L=D_4(3)\oplus K$ with
$$
2A_2\oplus D_4(2) < K < 2A_2'\oplus D_4.
$$
Thus $K<2A_2\oplus D_4$. Using the rule $D_4(3)<A_2\oplus A_2(2)$, we have
$$
L<2A_2\oplus D_4\oplus D_4(3) < 3A_2\oplus D_4\oplus A_2(2). 
$$
This overlattice contains forbidden component (50), i.e. $D_4\oplus 2A_2\oplus A_2(2)$. 
\item $C_{3,1}^2 E_{6,3}$:     $C=4$, $L$ is bounded by
$$
6A_1\oplus E_6(3) < L<2A_3'(2)\oplus E_6'(3).
$$
Therefore, $L$ is contained in a maximal even overlattice of $6A_1\oplus E_6(3)$ in the genus of $D_6\oplus E_6$, which is forbidden component (60). 
\end{enumerate}

\subsection{Rank 14} There are \textbf{10} extraneous root systems of rank $14$.
\begin{enumerate}[resume]
\item $A_{1,2}^8 A_{2,3}^3$: $C=1$, $\bP=\ZZ^8\oplus 3A_2$. Since $L$ is even, we have
$$
L < 3A_2 \oplus D_8.
$$ 
This overlattice contains forbidden component (61), i.e. $2A_2 \oplus D_8$. 
\item $A_{1,2}^{11} A_{3,4}$: $C=1$, $\bP=\ZZ^{11}\oplus A_3'(4)$.  The maximal even sublattice of $\bP$ is $L_8 \oplus E_8$, where $L_8$ has genus $2^2_6 4_{\II}^2$. The overlattice $L_8 \oplus E_8$ of $L$ is forbidden component (71). 
\item $A_{1,2}^{10} A_{2,3} B_{2,3}$: $C=1$, $\bP=\ZZ^{10}\oplus A_2\oplus \ZZ^2(3)$. The maximal even sublattice of $\bP$ is $A_2 \oplus E_8 \oplus L_7$, where $L_7$ has genus $2_{\II}^2 3^2$ as before. Then 
$$
L < A_2 \oplus E_8 \oplus L_7 < 3A_2 \oplus E_8 < E_6\oplus E_8
$$ 
using $L_7 < 2A_2$ and $3A_2<E_6$. The overlattice $E_6 \oplus E_8$ of $L$ is forbidden component (70). 
\item $A_{2,2}^5 B_{2,2}^2$: $C=3/2$, $L$ is bounded by
$$
5A_2(2)\oplus 4A_1(2) < L < 5A_2'(2) \oplus 4A_1.
$$
The projection $K$ of $L$ to the first component $5A_2'(2)$ is even and integral, because $L$ itself is even and $4A_1$ is even. Therefore, $L<K\oplus 4A_1$, where $K$ is an even sublattice of $5A_2'(2)$ containing $5A_2(2)$. Then $K$ is contained in the maximal even overlattice $A_2\oplus E_8$ of $5A_2(2)$. Therefore,
$$
L<A_2\oplus 4A_1\oplus E_8. 
$$
This overlattice contains forbidden component (66), i.e. $3A_1\oplus A_2\oplus E_8$. 
\item $A_{1,1}^5 A_{3,2} C_{3,2}^2$: $C=2$, $L$ is bounded by
$$
5A_1\oplus A_3(2)\oplus 6A_1(2) < L < \ZZ^5(1/2)\oplus A_3'(2)\oplus 2A_3'(4).
$$
We see from the upper bound that $L<\ZZ^6(1/2)\oplus \ZZ^2\oplus 2A_3'(4)$. Then $L$ is contained in a maximal integral lattice of $\ZZ^6(1/2)\oplus \ZZ^2\oplus 2A_3'(4)$ which is in the genus of $\ZZ^8\oplus 2A_3'(4)$. Since $L$ is even, it is contained in the maximal even sublattice of that, which is forbidden component (72). 
\item $A_{1,1}^9 D_{5,4}$: $C=2$, $\bQ=9A_1\oplus D_5(4)<L$. Then $L$ is contained in a maximal even overlattice of $9A_1\oplus D_5(4)$, in the genus of $A_1\oplus D_5\oplus E_8$. This is forbidden component (69). 
\item $A_{1,1}^5 A_{3,2} G_{2,2}^3$: $C=2$, $L=3A_2(2)\oplus K$ with
$$
5A_1\oplus A_3(2) < K < \mathbb{Z}^5(1/2)\oplus A_3'(2).
$$
The maximal even overlattice of $5A_1\oplus A_3(2)$ is $E_8$. Using $2A_2(2)<2A_2$, we have
$$
L<E_8\oplus 3A_2(2) < E_8\oplus 2A_2\oplus A_2(2)
$$
This overlattice is forbidden component (67). 
\item $A_{1,1}^5 A_{3,2} D_{4,3} G_{2,2}$: $C=2$, $L=A_2(2)\oplus K$ with $5A_1\oplus A_3(2)\oplus D_4(3)<K$. The maximal even overlattice of $5A_1\oplus A_3(2)\oplus D_4(3)$ is in the genus of $E_8\oplus D_4$. Therefore, 
$$
L<A_2(2)\oplus E_8 \oplus D_4.
$$
This overlattice contains forbidden component (68), i.e. $E_8\oplus D_4\oplus A_2(2)$. 
\item $A_{1,1}^2 A_{3,2}^3 C_{3,2}$: $C=2$, $\bQ=2A_1\oplus 3A_3(2)\oplus 3A_1(2)<L$. Therefore, $L$ is contained in a maximal even overlattice of $2A_1\oplus 3A_3(2)\oplus 3A_1(2)$ in the genus $A_1\oplus D_5\oplus E_8$. This overlattice is forbidden component (69). 
\item $A_{2,1}^3 B_{2,1}^2 F_{4,3}$: $C=3$, $L=D_4(3)\oplus K$ with $3A_2\oplus 4A_1 < K$. The maximal even overlattice of $3A_2\oplus 4A_1$ is $E_6\oplus D_4$. Therefore, 
$$
L<E_6\oplus D_4\oplus D_4(3)<E_6\oplus 2D_4<E_6\oplus E_8
$$
using the rules $D_4(3)<D_4$ and $2D_4<E_8$. The overlattice $E_8\oplus E_6$ of $L$ is forbidden component (70). 
\end{enumerate}

\subsection{Rank 16} There are \textbf{5} extraneous root systems of rank $16$.
\begin{enumerate}[resume]
\item $A_{1,1}^{10} C_{3,2}^2$: $C=2$, $L$ is bounded by
$$
10A_1\oplus 6A_1(2) < L < \ZZ^{10}(1/2)\oplus 2A_3'(4).
$$
Then $L$ is contained in a maximal integral sublattice of $\ZZ^{10}(1/2)\oplus 2A_3'(4)$ which is of genus $\ZZ^{10}\oplus 2A_3'(4)$. Therefore, $L$ is contained in the maximal even sublattice of that, which is in the genus of $2A_3\oplus D_7$ and contains forbidden component (54), i.e $A_3\oplus D_7$. 
\item $A_{1,1}^{10} G_{2,2}^3$: $C=2$, $L=3A_2(2)\oplus K$ with $10A_1<K$. Any maximal even overlattice of $10A_1$ is in the genus of $E_8\oplus 2A_1$. Therefore,
$$
L<E_8\oplus 2A_1 \oplus 3A_2(2). 
$$
This overlattice contains forbidden component (58), i.e. $A_1\oplus A_2(2)\oplus E_8$. 
\item $A_{1,1}^{7} A_{3,2}^2 C_{3,2}$: $C=2$, $\bQ=7A_1\oplus 2A_3(2)\oplus 3A_1(2)$. Then $L$ is contained in a maximal even overlattice of $\bQ$ which lies in the genus of $A_1\oplus D_7\oplus E_8$. This overlattice is forbidden component (74). 
\item $A_{1,1}^{10} D_{4,3} G_{2,2}$: $C=2$, $L=A_2(2)\oplus K$ with $10A_1\oplus D_4(3) < K$.  The maximal even overlattices of $10A_1\oplus D_4(3)$ are all in the genus of $E_8\oplus D_6$. Therefore,
$$
L<E_8\oplus D_6 \oplus A_2(2). 
$$
This is forbidden component (73). 
\item $A_{3,1} A_{7,2} G_{2,1}^3$: $C=4$, $L=3A_2\oplus K$ with $A_3\oplus A_7(2) < K$.  The maximal even overlattices of $A_3\oplus A_7(2)$ are all in the genus of $2A_1\oplus E_8$.  Therefore,
$$
L<2A_1\oplus 3A_2\oplus E_8<2A_1\oplus E_6\oplus E_8. 
$$
This overlattice contains forbidden component (70), i.e. $E_8\oplus E_6$. 
\end{enumerate}

\subsection{Rank 18} There are \textbf{7} extraneous root systems of rank $18$.
\begin{enumerate}[resume]
\item $A_{1,1}^9 A_{3,2}^3$: $C=2$, $\bQ=9A_1\oplus 3A_3(2)$, and $\bQ<L$. The maximal even overlattice of $\bP$ is in the genus $2A_1\oplus 2E_8$. Then $L$ is contained in a certain lattice $K$ in the genus $2A_1\oplus 2E_8$. Since there is a unique (up to powers) reflective Borcherds product on 
$$
2U\oplus K \cong 2U\oplus 2A_1\oplus 2E_8
$$
and this product has weight $42$, we have $L\neq K$. Note that $\mathrm{discr}(\bQ)=2^{24}$, and therefore $\mathrm{discr}(L)=2^a$ with an integer $a$. Therefore, $L$ is contained in an index-two sublattice of $K$, which is either in the genus of $2E_8\oplus 2A_1(2)$ or in the genus of $E_8\oplus D_8\oplus 2A_1$. These two lattices are respectively forbidden components (78) and (76), so neither case can occur.
\item $A_{1,1}^{12} A_{3,2} C_{3,2}$: $C=2$, $\bQ=12A_1\oplus A_3(2)\oplus 3A_1(2)$. Since $\bQ<L$, the lattice $L$ is contained in a maximal even overlattice of $\bQ$ which lies in the genus of $A_1\oplus A_1(2)\oplus 2E_8$. This overlattice is forbidden component (77). 
\item $A_{2,1}^5 D_{4,2}^2$: $C=3$, $\bQ=5A_2\oplus 2D_4(2)$, $\bP=5A_2'\oplus 2D_4$, and $\bQ<L<\bP$. The maximal even overlattices of $\bQ$ lie in the genus of $A_2\oplus 2E_8$. Since $L$ is contained in a certain lattice $K$ of this genus and also in $\bP$, it follows that $L$ is contained in an index $4$ sublattice of $K$ in the genus of $A_2\oplus 2D_4\oplus E_8$. This overlattice is forbidden component (75). 
\item $A_{2,1} B_{2,1}^6 D_{4,2}$: $C=3$, $\bP=A_2'\oplus \ZZ^{12}\oplus D_4$. Since $L$ is an even sublattice of $\bP$, we obtain
$$
L<A_2\oplus D_{12}\oplus D_4.
$$
Note that 
$$
2U\oplus A_2\oplus D_{12}\oplus D_4\cong 2U\oplus A_2\oplus 2D_4\oplus E_8.
$$
This case can then be excluded using forbidden component (75), i.e. $A_2 \oplus 2D_4 \oplus E_8$. 
\item $A_{3,1} C_{3,1}^5$: $C=4$, $\bQ=A_3\oplus 15A_1$, and $\bQ<L$. Then $L$ is contained in a maximal even overlattice of $\bQ$ which lies in the genus of $A_1\oplus A_1(2)\oplus 2E_8$. This overlattice is forbidden component (77). 
\item $A_{3,1} C_{3,1} G_{2,1}^6$: $C=4$, $L=6A_2\oplus K$ with $A_3\oplus 3A_1 < K$. The maximal even overlattice of $A_3\oplus 3A_1$ is $D_5\oplus A_1$. Therefore, 
$$
L<D_5\oplus A_1\oplus 6A_2. 
$$
This overlattice contains forbidden component (57), i.e. $D_5\oplus 3A_2$.
\item $A_{3,1} C_{3,1}^3 G_{2,1}^3$: $C=4$, $L=3A_2\oplus K$ with $A_3\oplus 9A_1 < K$. The maximal even overlattices of $A_3\oplus 9A_1$ are all in the genus of $E_8\oplus A_3\oplus A_1$. Therefore, 
$$
L<E_8\oplus A_3\oplus A_1 \oplus 3A_2 < E_8\oplus E_6\oplus A_3\oplus A_1.
$$
This overlattice contains forbidden component (70), i.e. $E_8\oplus E_6$. 
\end{enumerate}

\subsection{Rank larger than 18} There are \textbf{10} extraneous root systems of rank greater than 18, but they cannot admit reflective Borcherds products of singular weight by \cite[Theorem 1.5]{WW23}.

%% file: chap12.tex
\chapter{The classification of symmetric root systems}\label{sec:exclude-root}
In this chapter, we prove the symmetric case of Theorem \ref{th:non-existence}:

\begin{theorem}\label{th:non-existence-symmetric}
If $2U\oplus L$ has a symmetric reflective Borcherds product of singular weight whose Jacobi form input has non-negative $q^0$-term, then the associated semi-simple Lie algebra $\mathfrak{g}$ defined in Theorem \ref{th:root-system} satisfies $1/C \in \mathbb{Z}$. 
\end{theorem}
The Lie algebras $\mathfrak{g}$ as above are listed in Tables \ref{tab:sym-cycle} and \ref{tab:extra}. 

To prove the theorem we need the following generalization of \cite[Lemma 3.3]{Wan23a}. 

\begin{lemma}\label{lem:sym-overlattice}
Let $U\oplus K\oplus L$ be an even lattice of signature $(l,2)$ with $l\geq 3$ and set $M=U\oplus K$. Suppose that $M\oplus L$ has a reflective Borcherds product which vanishes on some $\lambda^\perp$ with $\lambda \in M'$. Let $L_1$ be an even overlattice of $L$. Then $M\oplus L_1$ also has a reflective Borcherds product vanishing on $\lambda^\perp$.  
\end{lemma}

\begin{proof}
Let $f$ be the input of the Borcherds product on $M\oplus L$ as a vector-valued modular form. Recall that the principal part of $f$ was described in \cite[Lemma 2.1]{Wan23a}. Note that 
$$
M\oplus L < M\oplus L_1 < M'\oplus L_1' < M'\oplus L'.
$$
By applying the arrow operator of \cite[Lemma 5.6]{Bru02} to $f$, we obtain a weakly holomorphic modular form of weight $1-l/2$ for the Weil representation $\rho_{M\oplus L_1}$ that is given by
$$
f|\uparrow_{M\oplus L}^{M\oplus L_1}\; = \sum_{\gamma \in (M'\oplus L_1') / (M\oplus L_1)}\; \sum_{n\in \ZZ-\gamma^2/2} \; \sum_{x \in (\gamma+M\oplus L_1)/(M\oplus L)} c(x,n)q^n e_\gamma,
$$
where $c(x,n)$ are the coefficients of $q^n e_x$ in the Fourier expansion of $f$. We will write 
$$
f|\uparrow_{M\oplus L}^{M\oplus L_1}\; = \sum_{\gamma \in (M' \oplus L_1') / (M \oplus L_1)} c'(\gamma, n) q^n e_\gamma.
$$

It is not hard to see that the nonzero coefficients of $f|\uparrow_{M\oplus L}^{M\oplus L_1}$ are reflective (similarly to the proof of \cite[Lemma 3.3]{Wan23a}); it is less obvious that $f|\uparrow_{M\oplus L}^{M\oplus L_1}$ is not identically zero. Suppose without loss of generality that $\lambda$ is primitive in $M'$, and write $\lambda^2=2/d$ where $d \in \mathbb{N}$. We have two cases to consider. 
\begin{enumerate}
\item[(a)] If the order of $\lambda$ in the discriminant group of $M\oplus L$ is $\ord(\lambda)=d$, then we know by \cite[Section 2.1 and Lemma 3.3]{Wan23a} that $c(\lambda, -1/d) > 0$ and thus
$$
c'(\lambda, -1/d) = \sum_{\substack{x \in (\lambda+M\oplus L_1)/(M\oplus L) \\ \ord(x)=d}} c(x,-1/d) \geq c(\lambda, -1/d) > 0,
$$
because every $c(x,-1/d)$ is non-negative. 
\item[(b)] If $\ord(\lambda)\neq d$, then $\ord(\lambda)=d/2$ and $d/2$ is even. In this case, 
$$
c(2\lambda, -4/d) + c(\lambda, -1/d) > 0
$$
and therefore
\begin{align*}
c'(2\lambda, -4/d)+c'(\lambda, -1/d) =& \sum_{\substack{x \in (\lambda+M\oplus L_1)/(M\oplus L)\\ \ord(x)=d}} c(x,-1/d) \\
&+ \sum_{\substack{y \in (\lambda+M\oplus L_1)/(M\oplus L)\\ \ord(y)=d/2}} \Big( c(2y,-4/d) + c(y,-1/d)\Big) \\
\geq&\, c(2\lambda, -4/d) + c(\lambda, -1/d) > 0,
\end{align*}
because $c(x,-1/d)\geq 0$ and $c(2y,-4/d) + c(y,-1/d)\geq 0$. 
\end{enumerate}
In particular, $\Borch(f|\uparrow_{M\oplus L}^{M\oplus L_1})$ is a  reflective Borcherds product on $M\oplus L_1$ that vanishes on $\lambda^\perp$. Note that the coefficients of $q^{-1}e_0$ in the Fourier expansions of $f$ and $f|\uparrow_{M\oplus L}^{M\oplus L_1}$ are the same, so $\Borch(f|\uparrow_{M\oplus L}^{M\oplus L_1})$ remains symmetric if $\Borch(f)$ was. 
\end{proof}

\begin{proof}[Proof of Theorem \ref{th:non-existence-symmetric}]
To prove the theorem we have to rule out the $5$ extraneous semi-simple $\mathfrak{g}$ of symmetric type with non-integral $1/C$. We do this by cases. 
\begin{enumerate}
\item $A_{2,4}B_{2,4}$: The lattice $L$ is bounded by
$$
A_2(4)\oplus 2A_1(4) < L < A_2'(4)\oplus 2A_1(2). 
$$
Since $L$ is even, we conclude
$$
A_2(4)\oplus 2A_1(4) < L < A_2(4)\oplus 2A_1(2).
$$
It follows that $L=A_2(4)\oplus 2A_1(4)$ or $A_2(4)\oplus 2A_1(2)$. By a direct calculation, we can prove that $2U\oplus A_2(4)\oplus 2A_1(2)$ has no symmetric reflective Borcherds product of any weight. Lemma \ref{lem:sym-overlattice} then shows that $2U\oplus L$ has no symmetric reflective Borcherds products.

\item $A_{2,2}D_{4,4}$: The lattice $L$ is bounded by
$$
A_2(2)\oplus D_4(4) < L < A_2'(2)\oplus D_4(2),
$$
and further by 
$$
A_2(2)\oplus D_4(4) < L < A_2(2)\oplus D_4(2). 
$$
Since both $L$ and $L(3/2)$ are integral, $L(1/2)$ is also integral. Since
$$
A_2\oplus D_4(2) < L(1/2) < A_2\oplus D_4,
$$
$L(1/2)$ can only be $A_2\oplus D_4(2)$, $A_2\oplus 4A_1$ or $A_2\oplus D_4$, and therefore $L=A_2(2)\oplus D_4(4)$, $A_2(2)\oplus 4A_1(2)$ or $A_2(2)\oplus D_4(2)$. We were able to check by a direct calculation that there are no symmetric reflective Borcherds products of singular weight on $2U\oplus L$ in the latter two cases; in the first case, the discriminant was prohibitively large and we needed a more subtle argument. Let $v_2$ be a $2$-root of $D_4$ and write $K$ for the lattice generated by $D_4$ and $v_2/2$, such that $K(4)$ is an even lattice of discriminant $4^4 = 256$. We were able to compute that there are no symmetric reflective Borcherds products of any weights on $2U\oplus A_2(2)\oplus K(4)$. By Lemma \ref{lem:sym-overlattice}, there are also no symmetric reflective Borcherds products of any weight on $2U\oplus A_2(2)\oplus D_4(4)$.

\item $A_{2,2}^2 B_{2,2}^2$: The lattice $L$ is bounded by
$$
2A_2(2)\oplus 4A_1(2) < L < 2A_2'(2)\oplus 4A_1.
$$
Since both $L$ and $L(3/2)$ are integral, $L(1/2)$ is also integral. From
$$
2A_2\oplus 4A_1 < L(1/2) < 2A_2'\oplus \ZZ^4,
$$
and the integrality of $L(1/2)$, it follows that
$$
2A_2\oplus 4A_1 < L(1/2) < 2A_2\oplus \ZZ^4.
$$
This forces $L(1/2)$ to be one of $2A_2\oplus 4A_1$, $A_2\oplus 2A_1\oplus \ZZ^2$, $2A_2\oplus D_4$ or $2A_2\oplus \ZZ^4$. Therefore, $L=2A_2(2)\oplus 4A_1(2)$, $2A_2(2)\oplus 2A_1(2)\oplus 2A_1$, $2A_2(2)\oplus D_4(2)$ or $2A_2(2)\oplus 4A_1$.  Here, $2A_2(2)\oplus D_4$ was the forbidden component (39) - there is no reflective Borcherds product (symmetric or anti-symmetric) of any weight on $2U\oplus 2A_2(2)\oplus D_4$. Since $L$ is of type $2A_2(2)\oplus K$ and $K<D_4$, we can use Lemma \ref{lem:sym-overlattice} to conclude that $2U\oplus L$ has no symmetric reflective Borcherds products of any weight.

\item $A_{4,2} C_{4,2}$: The lattice $L$ is bounded by
$$
A_4(2)\oplus 4A_1(2) < L < A_4'(2)\oplus D_4(2).
$$
Since $L$ is even, we have
$$
A_4(2)\oplus 4A_1(2) < L < A_4(2)\oplus D_4(2).
$$
It follows that $L=A_4(2)\oplus 4A_1(2)$ or $A_4(2)\oplus D_4(2)$. By direct calculation, we were able to verify that there are no reflective Borcherds products of any weight on $2U\oplus A_4(2)\oplus D_4$. Lemma \ref{lem:sym-overlattice} then shows that $2U\oplus L$ also has no symmetric reflective Borcherds products of any weight.

\item $A_{6,2} B_{4,2}$: The lattice $L$ is bounded by
$$
A_6(2)\oplus D_4(2) < L < A_6'(2)\oplus 4A_1.
$$
Since $L$ is even, we have
$$
A_6(2)\oplus D_4(2) < L < A_6(2)\oplus 4A_1.
$$
It follows that $L=A_6(2)\oplus D_4(2)$ or $A_6(2)\oplus 4A_1$. By direct calculation, we were able to verify that there are no symmetric reflective Borcherds products of any weight on $2U\oplus A_6(2)$. Using the pullback from $2U\oplus L$ to $2U\oplus A_6(2)$, this rules out the root system $A_{6,2} B_{4,2}$.
\end{enumerate}
\end{proof}

%% file: chap13.tex
\chapter{Application: anti-symmetric Siegel paramodular forms of weight 3}\label{sec:paramodular}

The restrictions of reflective Borcherds products of singular weight on appropriate lattices yield anti-symmetric Siegel paramodular forms of degree $2$ and weight $3$. This gives an application of our previous results. 

Siegel paramodular forms of degree two and level $t$ are holomorphic functions on the genus-two Siegel upper half space
$$
\HH_2=\left\{Z=\begin{psmallmatrix}
\tau & z \\ 
  z & \omega    
\end{psmallmatrix}
\in M(2,\CC): \im Z>0\right\}
$$
which are modular under the level $t$ paramodular group
$$
\Gamma_t=\left(\begin{array}{cccc}
  * & t* & * & * \\ 
  * & * & * & */t \\ 
  * & t* & * & * \\ 
  t* & t* & t* & *
  \end{array}   \right) \cap \Sp_2(\QQ),\quad \text{all }\  *\in \ZZ.
$$
These can be realized as modular forms on orthogonal groups of the  lattice $2U\oplus A_1(t)$ of signature $(3,2)$. More precisely, modular forms on $\widetilde{\Orth}^+(2U\oplus A_1(t))$ correspond exactly to Siegel paramodular forms that are modular under the normal extension 
$$
\Gamma_t^+=\Gamma_t\cup \Gamma_t V_t,\qquad V_t=\frac{1}{\sqrt{t}}\left(
\begin{matrix}
0 & t & 0 & 0 \\ 
-1 & 0 & 0 & 0 \\ 
0 & 0 & 0 & 1 \\ 
0 & 0 & -t & 0
\end{matrix}\right).
$$
Let $\chi_t: \Gamma_t^+\to \{\pm 1\}$ be the unique nontrivial character with kernel $\Gamma_t$. Then $M_k(\Gamma_t)$ is decomposed into the direct sum of plus and minus $V_t$-eigenspaces, that is,
$$
M_k(\Gamma_t)=M_k(\Gamma_t^+)\oplus M_k(\Gamma_t^+, \chi_t).
$$
Moreover, symmetric modular forms of weight $k$ on $\widetilde{\Orth}^+(2U\oplus A_1(t))$ correspond to $M_k(\Gamma_t^+, \chi_t^{k})$. Therefore, we call Siegel paramodular forms in $M_k(\Gamma_t^+, \chi_t^{k})$ \textit{symmetric} and Siegel paramodular forms in $M_k(\Gamma_t^+, \chi_t^{k+1})$ \textit{anti-symmetric}. For a paramodular eigenform, the distinction between symmetry and anti-symmetry is exactly the sign in the functional equation of the associated $L$-function.

Anti-symmetric Siegel paramodular forms of weight $3$ have applications to birational geometry. Let $t$ be squarefree. By \cite[Proposition 1.5]{GH98}, 
the modular variety  
$${\mathcal A}_t^+=\Gamma_t^+\setminus \HH_2$$
is isomorphic to {\it the moduli space of polarized $K3$ surfaces with a polarization of type} $\latt{2t}\oplus\, 2E_8(-1)$.
\cite[Theorem 1.5]{GH98} further yields that if $t$ is prime then the above modular variety is actually {\it the moduli space of
Kummer surfaces associated to $(1,t)$-polarized abelian surfaces}. 
We know from \cite{Fre83} that the geometric genus of the variety ${\mathcal A}_t^+$ is given by
$$
h^{3,0}\left(\overline{{\mathcal A}_t^+}\right)
=\hbox{dim}_\CC\, S_3(\Gamma_t^+).
$$ 
Recently, Ibukiyama \cite{Ibu22} found a dimension formula for $S_3(\Gamma_t^+)$ for squarefree $t$. In particular, he proved \cite[Proposition 6.1]{Ibu22} that $\hbox{dim}_\CC\, S_3(\Gamma_t^+)>0$ for prime $t$ if and only if $t>163$ and $t\neq 179, 181, 191, 193, 199, 211, 229, 241$. In particular the moduli space ${\mathcal A}_t^+$ always has positive geometric genus if $t>241$ is prime. 

In this chapter, we construct a large infinite family of anti-symmetric Siegel paramodular forms of weight $3$ that seems to be new. 

Recall that there are exactly two semi-simple Lie algebras $\mathfrak{g}$ of rank $6$ with integral $C$ in Schellekens' list. Let $z\in \CC$. For any nonzero $v\in L_\mathfrak{g}$ with $\vartheta_\mathfrak{g}(\tau,vz)\neq 0$, the pullback of the corresponding singular-weight reflective Borcherds product $\Borch(\chi_V)$ along the embedding
$$
2U\oplus \ZZ v \hookrightarrow 2U\oplus L_\mathfrak{g}
$$
defines a nonzero anti-symmetric paramodular form of weight $3$ and level $v^2/2$. The case $\mathfrak{g}=A_{6,7}$ was considered in \cite[Theorem 2.1]{GW19}. Here we use the other Lie algebra, $\mathfrak{g}=A_{1,2}D_{5,8}$. In this case, the orbit lattice is
$$
L_\mathfrak{g}=A_1(2)\oplus D_5'(8).
$$
By restricting to vectors in $L_\mathfrak{g}$ we obtain the following theorem.

\begin{theorem}\label{th:Siegel-3}
For $\mathbf{a}=(a_1,a_2,a_3,a_4,a_5,a_6)\in \ZZ^6$,
the theta block
\begin{equation}
\begin{split}
\Theta_\mathbf{a}=\,&\vartheta_{a_1}\vartheta_{a_2}\vartheta_{a_3}\vartheta_{a_4}
\vartheta_{a_5}\vartheta_{a_1+a_2}\vartheta_{a_1+a_2+a_3}
\vartheta_{a_1+a_2+a_3+a_4}\vartheta_{a_2+a_3}\\
&\vartheta_{a_2+a_3+a_4}
\vartheta_{a_3+a_4}\vartheta_{a_1+a_2+a_3+a_5}\vartheta_{a_2+a_3+a_5}\vartheta_{a_3+a_5}\\
&\vartheta_{a_1+2a_2+2a_3+a_4+a_5}\vartheta_{a_1+a_2+2a_3+a_4+a_5}\vartheta_{a_1+a_2+a_3+a_4+a_5}\\
&\vartheta_{a_2+2a_3+a_4+a_5}\vartheta_{a_2+a_3+a_4+a_5}\vartheta_{a_3+a_4+a_5}\vartheta_{2a_6} /\eta^{15} \\
=\,&q^2(\cdots)  \in  J_{3, N(\mathbf{a})}
\end{split}
\end{equation}
of type $\frac{21-\vartheta}{15-\eta}$ is a holomorphic Jacobi form of weight $3$ and index $N(\mathbf{a})$ for $A_1$, where 
$$
\vartheta_b(\tau,z):=\vartheta(\tau,bz), \quad b\in \ZZ,
$$
and the index $N(\mathbf{a})$ is half the sum of squares of the subscripts in $\vartheta_*$. If this theta block is not identically zero, then there 
exists an anti-symmetric Siegel paramodular form 
$F_{\mathbf{a}}$ of weight $3$ and 
level $N(\mathbf{a})$ with trivial character whose leading Fourier--Jacobi coefficient is exactly $\Theta_\mathbf{a}$. Moreover, if $N(\mathbf{a})$ is squarefree, then $F_{\mathbf{a}}$ is a cusp form.
\end{theorem}

In squarefree levels $t < 300$, Theorem \ref{th:Siegel-3} produces $64$ anti-symmetric Siegel paramodular forms of weight $3$ that are listed in Table \ref{tableanti3prime}. Using \cite[Tables 1-3]{GW19} and Table \ref{tableanti3prime}, we obtain a basis of $S_3(\Gamma_t^+)$ consisting of Borcherds products for the prime levels
\begin{align*}
t=&\,167,\, 173,\, 197,\, 223,\, 227,\, 239,\, 251,\, 257,\, 263,\, 269,\, 271,\, 277,\, 283,\, 293.
\end{align*} 
Note that Ibukiyama \cite{Ibu22} also showed that $\hbox{dim}_\CC\, S_3(\Gamma_t^+)=1$ for $t=233$, $281$; however, the corresponding paramodular forms cannot be obtained from the two infinite series \cite[Theorem 2.2]{GW19} and Theorem \ref{th:Siegel-3}. 

Anti-symmetric Siegel paramodular forms of weight $2$ and trivial character are very interesting due to their role in the paramodular conjecture. Unfortunately, there is no semi-simple $V_1$ structure $\mathfrak{g}$ of rank $4$ with integral $C$, so the above method does not work. However, there does exist a unique semi-simple $V_1$ structure $\mathfrak{g}$ of rank $4$ with non-integral $C$. In particular, we can use a similar argument to construct anti-symmetric Siegel paramodular forms of weight $2$ with a character of order $2$ by considering the pullbacks of $\Borch(\chi_V)$ where $V_1=\mathfrak{g}=C_{4,10}$.

%% file: chap14.tex
\chapter{Remarks, questions and conjectures}\label{sec:problems}

In this chapter, we raise some questions and conjectures that are related to our work in this paper.

\section{Uniqueness of hyperbolizations}
We have proved that there are exactly $81$ affine Kac--Moody algebras $\hat{\mathfrak{g}}$ which have hyperbolizations. For each of these algebras, we have constructed a natural hyperbolization with underlying lattice $L_\mathfrak{g}$.

\begin{question}
Does every affine Lie algebra have a unique hyperbolization? More concretely, let $\mathfrak{g}$ be one of the $81$ affine Lie algebras with a hyperbolization, and suppose that $L$ is an even positive-definite lattice for which there is a singular-weight reflective Borcherds product $F$ on $2U\oplus L$ whose associated semi-simple Lie algebra is $\mathfrak{g}$.
\begin{enumerate}
\item[(a)]  Is the lattice $L$ uniquely determined? 
\item[(b)]  Is the modular form $F$ uniquely determined?
\end{enumerate}
\end{question}
Note in (b) that a single function can have interpretations as a modular form on $\Orth^+(2U\oplus L)$ for different lattices $L$.

Recall that the lattice $L$ must satisfy the bounds $\mathbf{Q}_\mathfrak{g}<L<\mathbf{P}_\mathfrak{g}$. When $\mathfrak{g}$ is of symmetric type, there are examples for which the lattice $L$ is not unique:
\begin{enumerate}
    \item $\mathfrak{g}=A_{1,16}$: In this case, $\bQ=A_1(16)$ and $\bP=A_1(4)$. Then $L$ has only two possibilities: $A_1(16)$ and $A_1(4)$. We verify by direct calculation that both $2U\oplus A_1(16)$ and $2U\oplus A_1(4)$ have unique reflective Borcherds products of singular weight; however, the two products are expansions of the same modular form.
    \item $\mathfrak{g}=A_{1,4}^4$: We have constructed a reflective Borcherds product of singular weight on $2U\oplus L_\mathfrak{g}$ for $L_\mathfrak{g}=4A_1$.  This product can also be viewed as a reflective Borcherds product on $2U\oplus D_4(2)$. 
    \item $\mathfrak{g}=A_{1,8}^2$: the singular Borcherds product we constructed can be defined on $2U\oplus L$ for both $L=L_\mathfrak{g}=2A_1(2)$ and $L=2A_1(4)$. 
    \item $\mathfrak{g}=A_{1,2}^8$: besides $L_\mathfrak{g}=D_8$, we can also take any of $L=2D_4, D_8'(2),$ or $E_8(2)$. 
    \item $\mathfrak{g}=A_{2,3}^3$: besides $L_\mathfrak{g}=3A_2$, we can also take $L=E_6'(3)$.
\end{enumerate}
When $\mathfrak{g}=B_{2,3}G_{2,4}$, we have $\mathbf{Q}_\mathfrak{g}=2A_1(3)\oplus A_2(4)$ and $\mathbf{P}_\mathfrak{g}=\ZZ^2(3)\oplus A_2(4)$. Since $L$ is even, it is contained in the maximal even sublattice of $\bP$, which coincides with $\bQ$. Therefore, $L=2A_1(3)\oplus A_2(4)=L_\mathfrak{g}$. It follows that $L$ is unique.  

Regarding the question above, we propose the following conjecture.
\begin{conjecture}
The singular Borcherds product $F$ is always unique as functions on symmetric domains. For any semi-simple $\mathfrak{g}$ of anti-symmetric type, the lattice $L$ is unique, i.e. $L=L_\mathfrak{g}$. Furthermore, the hyperbolization of $\mathfrak{g}$ is unique. 
\end{conjecture}
Some cases of this conjecture can be proved easily, but in general it appears difficult to check. For example, \cite[Theorem 4.7]{Wan21b} implies the uniqueness of $L$ and $F$ for the $23$ semi-simple $\mathfrak{g}$ of rank $24$. When $\mathfrak{g}=E_{8,2}B_{8,1}$, we have $\mathbf{Q}_\mathfrak{g}=E_8(2)\oplus D_8$ and $\mathbf{P}_\mathfrak{g}=E_8(2)\oplus \ZZ^8$, which uniquely determines $L=E_8(2)\oplus D_8$. The uniqueness of $F$ follows from \cite[Lemma 3.2]{Wan23a}.

We proved in Chapter \ref{sec:construct-anti} that the $69$ singular Borcherds products on $2U\oplus L_\mathfrak{g}$ for anti-symmetric $\mathfrak{g}$ come from only $11$ different modular forms. The above conjecture would therefore imply the following classification.
\begin{conjecture}
There are exactly $23$ distinct modular forms that can be realized as reflective Borcherds products of singular weight on a lattice of type $2U\oplus L$ whose input forms have non-negative principal parts.
\end{conjecture}

\section{Other questions and conjectures}
We have the following conjecture related to Remark \ref{rem:Lam-paper}.
\begin{conjecture}
Let $[g]$ be a $\mathrm{Co}_0$-conjugacy class such that $\Lambda^g$ is nonzero.  Let $M_g$ be an even lattice of signature $(\rank(\Lambda^g)+2,2)$ whose discriminant form is isomorphic to $(R(V_{\Lambda_g}^{\hat{g}}),-\mathfrak{q})$ as introduced by Lam \cite{Lam19}. Then the multiplicative theta lift of the (vector-valued) characters of $V_{\Lambda_g}^{\hat{g}}$ (divided by $\eta^{\rank(\Lambda^g)}$) is a reflective Borcherds product of singular weight on $M_g$. Moreover, the Fourier expansion of this product at a certain $0$-dimensional cusp is the $g$-twisted denominator of the fake monster algebra.
\end{conjecture}
As we mentioned in Chapter \ref{sec:construct-anti}, this conjecture was proved in \cite{WW22} whenever $g$ has level equal to its order. However, the proof is indirect.

Remarks \ref{rem:fE8} and \ref{rem:Conway-SCFT} lead us to ask the following questions:
\begin{question}
Let $[g]$ be a $\mathrm{Co}_0$-conjugacy class whose level $n_g$ is equal to its order. By \cite{WW22}, the $g$-twisted denominator of the fake monster algebra is the Fourier expansion of a singular-weight reflective Borcherds product $\Phi_g$ for $\Orth^+(U_1(n_g)\oplus U\oplus \Lambda^g)$ at the $0$-dimensional cusp determined by $U_1(n_g)$.  The function $\Phi_g$ can also be viewed as a singular-weight reflective Borcherds product on $U_1\oplus U(n_g)\oplus {(\Lambda^g)}'(n_g)$, and we can consider its Fourier expansion there at the $0$-dimensional cusp determined by $U(n_g)$. When ${(\Lambda^g)}'(n_g)\neq \Lambda^g$, this Fourier expansion may define the denominator of a new BKM superalgebra, denoted $\mathbb{G}_g'$, which is different from the $g$-twist of the fake monster algebra. Is $\mathbb{G}_g'$ related to a vertex algebra? If so, can we express its input Jacobi form
$$
\phi_d(\tau,\mathfrak{z}) \in J^!_{0,{(\Lambda^g)}'(n_g)}(\Gamma_0(n_g/d)), \quad d|n_g
$$
in terms of characters of the vertex algebra? Does $\mathbb{G}_g'$ have a natural construction by the BRST cohomology? Note that the corresponding vertex algebra is $V^{fE_8}$ when $g$ has cycle shape $1^{-8}2^{16}$ and it is the Conway SCFT $V^{f\natural}$ when $g$ has cycle shape $1^{-24}2^{24}$. 
\end{question}

We have the following question concerning Chapter \ref{sec:construct-symmetric-C<1}. 
\begin{question}
Let $\mathcal{G}_\mathfrak{g}$ denote the BKM superalgebra corresponding to any of the four affine Lie algebras possessing exceptional modular invariants. Can we realize $\mathcal{G}_\mathfrak{g}$ as the BRST cohomology related to a vertex algebra? Our expression of the Jacobi form input in terms of affine characters may be a hint towards finding such a desirable vertex algebra. 
\end{question}

%% file: tables.tex
\chapter{Long tables}

\begin{table}[ht]
\caption{The 17 solutions of Equation~\eqref{symC} in the order of increasing $C$. Those allowing hyperbolization are colored blue.}
\label{table:sym}
\renewcommand\arraystretch{1.5}
\[
\begin{array}{r|l} 
C & \mathfrak{g} \\\hline\hline
{1}/{8} & \color{blue}A_{1,16}\\ 
1/4 & \color{blue}A_{1,8}^2 \\
1/3 & \color{blue}A_{2,9} \\
1/2 & \color{blue}A_{1,4}^4 \\
3/4 & A_{2,4} B_{2,4} \\
1 &  \color{blue}A_{1,2}^8 \\
1 & \color{blue} A_{2,3}^3\\
1 & \color{blue}A_{4,5} \\
1 & \color{blue}A_{3,4}A_{1,2}^3\\
1 & \color{blue}B_{2,3}G_{2,4}\\
1 & \color{blue}B_{2,3}A_{2,3}A_{1,2}^2\\
1 & \color{blue}B_{3,5}A_{1,2}\\ 
1 &  \color{blue}C_{3,4}A_{1,2} \\ 
3/2  &  A_{2,2} D_{4,4} \\ 
3/2  &  A_{2,2}^2 B_{2,2}^2 \\ 
5/2  &  A_{4,2}C_{4,2} \\
7/2  &  A_{6,2} B_{4,2} \\ \hline
\end{array}
\]
\end{table}

\clearpage

\begin{table}[ht]
\caption{The 221 solutions of Equation~\eqref{antiC} in the order of increasing $C$. Those allowing hyperbolization are colored blue (continued on next page).}
\label{table:anti1}
\begin{minipage}[t]{0.303\textwidth}
\[
\begin{array}{r|l}
C & \mathfrak{g} \\\hline\hline
{1}/{24} & A_{1,48}^3A_{2,72}^2
\\ 
{1}/{24} & A_{1,48}^5 B_{2,72}\\ 
{1}/{24} & A_{1,48}A_{2,72}G_{2,96} \\ 
{1}/{24} & A_{3,96}B_{2,72} \\ 
\hline
1/12 &  A_{1,24}^6A_{2,36}\\ 
{1}/{12} &
A_{1,24}^4G_{2,48}\\ 
{1}/{12} &
A_{1,24}^2B_{2,36}^2\\ 
{1}/{12} &
A_{2,36}^2B_{2,36}\\ 
{1}/{12} &
A_{1,24}A_{2,36}A_{3,48}
 \\
\hline
1/8 & A_{1,16}B_{2,24}G_{2,32}\\ 
{1}/{8} &
A_{1,16}A_{2,24}^3\\ 
{1}/{8} &
A_{1,16}^2C_{3,32}\\ 
{1}/{8} &
A_{1,16}^3A_{2,24}B_{2,24}\\ 
{1}/{8} &
A_{1,16}^9\\ 
{1}/{8} &
A_{1,16}^4A_{3,32}\\ 
{1}/{8} &
A_{1,16}^2B_{3,40}\\ 
{1}/{8} &
A_{1,16}A_{4,40}
 \\\hline
1/6 &   A_{1,12}^2 A_{2,18} G_{2,24}\\ 
{1}/{6} &
A_{1,12}^6 B_{2,18}\\ 
{1}/{6} &
A_{1,12}^4 A_{2,18}^2\\ 
{1}/{6} &
D_{4,36}\\ 
{1}/{6} &
G_{2,24}^2\\ 
{1}/{6} &
A_{1,12} A_{3,24} B_{2,18}\\ 
{1}/{6} &
A_{2,18} B_{2,18}^2
\\\hline
1/4 &  A_{1,8}^3 C_{3,16}\\ 
{1}/{4} &
A_{1,8}^{10}\\ 
{1}/{4} &
A_{1,8}^2 A_{2,12}^3\\ 
{1}/{4} &
A_{1,8}^5 A_{3,16}\\ 
{1}/{4} &
A_{1,8}^2 B_{2,12} G_{2,16}\\ 
{1}/{4} &
A_{1,8}^3 B_{3,20}\\ 
{1}/{4} &
A_{3,16}^2\\ 
{1}/{4} &
A_{1,8}^2 A_{4,20}\\ 
{1}/{4} &
A_{1,8}^4 A_{2,12} B_{2,12}\\ 
{1}/{4} &
A_{2,12}^2 G_{2,16}\\ 
{1}/{4} &
B_{2,12}^3
\\\hline
1/3 &   A_{2,9} B_{2,9} G_{2,12} \\
1/3 &A_{1,6}^8 A_{2,9}\\
1/3 &
A_{1,6} A_{2,9} C_{3,12}\\
1/3 &
A_{2,9} A_{4,15} \\
1/3 &
A_{1,6} A_{2,9} B_{3,15} \\
1/3 &
A_{1,6}^6 G_{2,12}
 \\
\end{array}
\]
\end{minipage}
\begin{minipage}[t]{0.34\textwidth}
\[
\begin{array}{r|l}
C & \mathfrak{g} \\\hline\hline
1/3 & A_{1,6}^2 A_{2,9}^2 B_{2,9}\\
1/3 &
A_{1,6}^4 B_{2,9}^2\\
1/3 &
A_{1,6}^3 A_{2,9} A_{3,12}\\
1/3 &
A_{2,9}^4\\
1/3 &
A_{1,6} A_{3,12} G_{2,12} \\ \hline
1/2 & A_{2,6} G_{2,8}^2\\
1/2 &
A_{1,4}^2 A_{3,8}^2\\
1/2 &
A_{1,4}^4 A_{2,6}^3\\
1/2 &
A_{1,4}^4 A_{4,10}\\
1/2 &
A_{1,4} A_{2,6} A_{3,8} B_{2,6}\\
1/2 &
A_{1,4}^7 A_{3,8}\\
1/2 &
\color{blue}C_{4,10}\\
1/2 &
A_{1,4}^5 B_{3,10}\\
1/2 &
A_{2,6}^2 B_{2,6}^2\\
1/2 &
A_{3,8} B_{3,10}\\
1/2 &
A_{1,4}^6 A_{2,6} B_{2,6}\\
1/2 &
A_{1,4}^2 A_{2,6}^2 G_{2,8}\\
1/2 &
B_{4,14}\\
1/2 &
\color{blue}A_{1,4}^{12}\\
1/2 &
A_{1,4}^4 B_{2,6} G_{2,8}\\
1/2 &
\color{blue}A_{2,6} D_{4,12}\\
1/2 &
A_{1,4}^5 C_{3,8}\\
1/2 &
A_{1,4}^2 B_{2,6}^3\\
1/2 &
A_{3,8} C_{3,8}\\ \hline
2/3 & A_{1,3}^4 G_{2,6}^2
 \\
2/3 & A_{1,3}^4 D_{4,9}\\ \hline
3/4 & A_{2,4}^4 B_{2,4}\\ \hline
1 & A_{1,2} A_{2,3}^3 B_{3,5}\\
1&
A_{1,2}^3 A_{2,3} B_{2,3} C_{3,4}\\
1&
A_{1,2} A_{3,4} B_{2,3}^3\\
1&
A_{1,2} A_{4,5} C_{3,4}\\
1&
\color{blue}A_{6,7}\\
1&
A_{1,2}^4 A_{2,3} G_{2,4}^2\\
1&
A_{1,2}^4 B_{4,7}\\
1&
\color{blue}A_{1,2}^{16}\\
1&
\color{blue}A_{1,2} A_{5,6} B_{2,3}\\
1&
\color{blue}A_{1,2} D_{5,8}\\
1&
A_{1,2}^6 A_{2,3}^2 G_{2,4}\\
1&
A_{1,2}^2 A_{2,3} B_{2,3}^2 G_{2,4}\\
1&
A_{1,2}^9 B_{3,5}\\
1&
B_{2,3}^2 D_{4,6} \\
\end{array}
\]
\end{minipage}
\begin{minipage}[t]{0.32\textwidth}
\[
\begin{array}{r|l}
C & \mathfrak{g} \\\hline\hline
1 &  A_{1,2} B_{2,3} B_{3,5} G_{2,4}\\
1&
A_{2,3}^3 A_{4,5}\\
1&
A_{1,2} A_{4,5} B_{3,5}\\
1&
A_{1,2}^2 A_{2,3} A_{4,5} B_{2,3}\\
1&
A_{1,2}^2 B_{3,5}^2\\
1&
A_{1,2} B_{2,3} C_{3,4} G_{2,4}\\
1&
A_{1,2}^3 A_{2,3}^3 A_{3,4}\\
1&
\color{blue}A_{1,2} A_{3,4}^3\\
1&
A_{1,2}^8 A_{2,3}^3\\
1&
\color{blue}A_{2,3}^6\\
1&
A_{1,2}^{11} A_{3,4}\\
1&
A_{1,2}^5 A_{2,3} A_{3,4} B_{2,3}\\
1&
A_{1,2}^8 B_{2,3} G_{2,4}\\
1&
A_{1,2} A_{2,3}^3 C_{3,4}\\
1&
A_{2,3}^3 B_{2,3} G_{2,4}\\
1&
A_{1,2}^4 A_{2,3} D_{4,6}\\
1&
A_{1,2}^4 A_{3,4} C_{3,4}\\
1&
A_{1,2}^3 A_{3,4} A_{4,5}\\
1&
A_{1,2}^2 G_{2,4}^3\\
1&
A_{1,2}^2 C_{3,4}^2\\
1&
A_{1,2}^6 A_{3,4}^2\\
1&
A_{2,3} A_{3,4}^2 B_{2,3}\\
1&
A_{1,2}^2 A_{2,3}^4 B_{2,3}\\
1&
A_{1,2}^2 B_{3,5} C_{3,4}\\
1&
A_{1,2}^8 A_{4,5}\\
1&
\color{blue}A_{4,5}^2\\
1&
A_{1,2}^4 A_{2,3}^2 B_{2,3}^2\\
1&
A_{2,3} B_{2,3}^4\\
1&
B_{2,3}^2 G_{2,4}^2\\
1&
A_{1,2}^9 C_{3,4}\\
1&
A_{1,2}^4 A_{3,4} B_{3,5}\\
1&
A_{4,5} B_{2,3} G_{2,4}\\
1&
A_{1,2} A_{2,3}^2 A_{3,4} G_{2,4}\\
1&
A_{1,2}^4 C_{4,5}\\
1&
A_{1,2}^3 A_{2,3} B_{2,3} B_{3,5}\\
1&
A_{1,2}^{10} A_{2,3} B_{2,3}\\
1&
A_{1,2}^3 A_{3,4} B_{2,3} G_{2,4}\\
1&
A_{1,2}^2 D_{4,6} G_{2,4}\\
1&
A_{1,2}^6 B_{2,3}^3 \\ \hline
4/3 & C_{3,3}^2 G_{2,3} \\
4/3 & G_{2,3}^4 \\ \hline
\end{array}
\]
\end{minipage}
\end{table}
\addtocounter{table}{-1}
\begin{table}[ht]
\caption{(continued).}
\begin{minipage}[t]{0.303\textwidth}
\[
\begin{array}{r|l}
C & \mathfrak{g} \\\hline\hline
{3}/{2} & A_{2,2}^5 B_{2,2}^2 \\
{3}/{2} &
\color{blue}A_{2,2}^4 D_{4,4}\\
{3}/{2} &
\color{blue}B_{2,2}^6\\
{3}/{2} &
\color{blue}A_{2,2} F_{4,6} \\ \hline
2 & \color{blue} A_{1,1}^2 C_{3,2} D_{5,4} \\
2&
A_{1,1}^{10} C_{3,2}^2\\
2&
A_{1,1}^5 A_{3,2} C_{3,2}^2\\
2&
A_{1,1}^9 D_{5,4}\\
2&
A_{3,2}^2 G_{2,2}^3\\
2&
A_{1,1}^{14} A_{3,2}^2\\
2&
A_{1,1}^{10} G_{2,2}^3\\
2&
 A_{1,1}^{19} A_{3,2}\\
2&
 A_{1,1}^5 A_{3,2} G_{2,2}^3\\
2&
\color{blue}A_{1,1}^3 A_{5,3} D_{4,3}\\
2&
 A_{1,1}^9 A_{3,2}^3\\
2&
\color{blue}A_{1,1}^{24} \\
2&
 A_{1,1}^3 C_{3,2} D_{4,3} G_{2,2}\\
2&
 A_{1,1}^7 A_{3,2}^2 C_{3,2}\\
2&
\color{blue} A_{1,1}^3 A_{7,4}\\
2&
\color{blue} A_{1,1}^2 D_{6,5}\\
2&
 A_{1,1}^5 A_{3,2} D_{4,3} G_{2,2}\\
2&
A_{1,1}^3 C_{3,2} G_{2,2}^3\\
2&
A_{1,1}^3 C_{3,2}^3\\
2&
A_{3,2}^2 D_{4,3} G_{2,2}\\
2&
\color{blue}A_{1,1}^4 A_{3,2}^4\\
2&
 A_{1,1}^{17} C_{3,2}\\
2&
\color{blue} A_{1,1} C_{5,3} G_{2,2}\\
2&
A_{3,2}^2 C_{3,2}^2\\
2&
 A_{1,1}^2 A_{3,2}^3 C_{3,2}\\
2&
A_{1,1}^3 A_{5,3} G_{2,2}^2\\
2&
 A_{1,1}^{12} A_{3,2} C_{3,2}\\
2&
 A_{1,1}^{10} D_{4,3} G_{2,2}\\
2&
 A_{1,1}^4 A_{3,2} D_{5,4} \\ \hline
\end{array}
\]
\end{minipage}
\begin{minipage}[t]{0.34\textwidth}
\[
\begin{array}{r|l}
C & \mathfrak{g} \\\hline\hline
5/2 & \color{blue}A_{4,2}^2 C_{4,2} \\
5/2 & \color{blue}B_{3,2}^4 \\ \hline
3 & \color{blue}A_{2,1}^2 A_{8,3} \\
3&
\color{blue} A_{2,1}  B_{2,1} E_{6,4}\\
3&
 A_{2,1}^5 D_{4,2}^2\\
3&
A_{2,1}^2  B_{2,1}^8 \\
3&
 A_{2,1}^7  B_{2,1}^4 \\
3&
 A_{2,1}^6  B_{2,1}^2 D_{4,2}\\
3&
\color{blue} A_{2,1}^{12} \\
3&
 A_{2,1}^2 D_{4,2} F_{4,3}\\
3&
 A_{2,1}^3  B_{2,1}^2 F_{4,3}\\
3&
\color{blue} B_{2,1}^4 D_{4,2}^2\\
3&
\color{blue}A_{2,1}^2 A_{5,2}^2 B_{2,1}\\
3&
 A_{2,1}  B_{2,1}^6 D_{4,2} \\ \hline
7/2 & \color{blue}B_{4,2}^3 \\ \hline 
4 &\color{blue}  A_{3,1} C_{7,2} \\
4&
A_{3,1}^5 D_{5,2}\\
4&
 A_{3,1}  C_{3,1}^5 \\
4&
\color{blue} A_{3,1}^2 D_{5,2}^2\\
4&
\color{blue} A_{3,1} A_{7,2}  C_{3,1}^2 \\
4&
 C_{3,1}^2 E_{6,3}\\
4&
\color{blue}A_{3,1}^8 \\
4&
 A_{3,1}  C_{3,1}  G_{2,1}^6 \\
4&
 A_{3,1} A_{7,2}  G_{2,1}^3 \\
4&
\color{blue} A_{3,1} D_{7,3}  G_{2,1} \\
4&
 A_{3,1}  C_{3,1}^3  G_{2,1}^3 \\
4&
\color{blue}E_{6,3}  G_{2,1}^3 \\ \hline
9/2 & \color{blue}A_{8,2} F_{4,2}
\\ \hline
5 &  A_{4,1}  B_{3,1}^4  C_{4,1} \\
5&
\color{blue}A_{4,1}^6\\
5&
 A_{4,1}^3  C_{4,1}^2 \\
5&
\color{blue} A_{4,1} A_{9,2} B_{3,1}\\
5&
 \color{blue}C_{4,1}^4 \\
\end{array}
\]
\end{minipage}
\begin{minipage}[t]{0.32\textwidth}
\[
\begin{array}{r|l}
C & \mathfrak{g} \\\hline\hline
5&
\color{blue} B_{3,1}^2  C_{4,1} D_{6,2} \\ \hline
 11/2 & \color{blue}B_{6,2}^2 \\ \hline
 6 & \color{blue} D_{4,1}^6 \\
 6&
\color{blue} A_{5,1}  C_{5,1} E_{6,2}\\
 6&
\color{blue} A_{5,1} E_{7,3}\\
 6&
\color{blue} A_{5,1}^4  D_{4,1} \\ \hline
7 & \color{blue} B_{4,1}^2 D_{8,2} \\
7&
\color{blue}B_{4,1} C_{6,1}^2\\
7&
\color{blue}A_{6,1}^4\\
7&
A_{6,1} B_{4,1}^4  \\ \hline
8 &\color{blue} A_{7,1} D_{9,2} \\
8&\color{blue} A_{7,1}^2 D_{5,1}^2 \\ \hline
9 &\color{blue} C_{8,1} F_{4,1}^2 \\
9& \color{blue}B_{5,1} E_{7,2} F_{4,1}\\
9& \color{blue}A_{8,1}^3 \\ \hline
10 &\color{blue} D_{6,1}^4 \\
10& \color{blue}A_{9,1}^2 D_{6,1} \\ \hline
11 & \color{blue}B_{6,1} C_{10,1} \\ \hline
23/2 &\color{blue} B_{12,2} \\ \hline
12 & \color{blue} E_{6,1}^4 \\
12 & \color{blue} A_{11,1} D_{7,1} E_{6,1}
 \\ \hline
13 & \color{blue} A_{12,1}^2  \\ \hline
14 & \color{blue} D_{8,1}^3
  \\ \hline
15 & \color{blue} B_{8,1} E_{8,2}
  \\ \hline
16 & \color{blue} A_{15,1} D_{9,1}  \\ \hline
18 & \color{blue}D_{10,1} E_{7,1}^2\\
18&\color{blue} A_{17,1} E_{7,1}
\\ \hline
22 & \color{blue} D_{12,1}^2 \\ \hline
25 &\color{blue} A_{24,1}\\ \hline
30& \color{blue} E_{8,1}^3 \\
30&\color{blue} D_{16,1} E_{8,1} \\
\hline
46&\color{blue} D_{24,1}\\ \hline
\end{array}
\]
\end{minipage}
\end{table}

\clearpage

\begin{footnotesize}
\begin{table}[ht]
\caption{H\"{o}hn's construction: the $8$ conjugacy classes with the same order and level.} 
\label{table:8cycleshapes}
\begin{minipage}[t]{0.46\textwidth}
\[
\renewcommand\arraystretch{1.3}
\begin{array}{c|c|l|c}
g & \text{genus}  &  {\mathfrak{g}} &  C  \\ \hline
\multirow{24}{*}{$1^{24}$} & \multirow{24}{*}{$\II_{24,0}(1)$}
     & D_{24,1}        & 46            \\ \cline{3-4}
   & & D_{16,1}E_{8,1}      & 30       \\ \cline{3-4}
      &   & E_{8,1}^3           & 30        \\ \cline{3-4}
   &   & A_{24,1}            &25        \\ \cline{3-4}
   &    & D_{12,1}^2           &22       \\ \cline{3-4}
   &    & A_{17,1}E_{7,1}       &18     \\ \cline{3-4}
   &    & D_{10,1}E_{7,1}^2     & 18      \\ \cline{3-4}
   &   & A_{15,1}D_{9,1}    & 16          \\ \cline{3-4}
 &   & D_{8,1}^3                & 14   \\ \cline{3-4}
   &   & A_{12,1}^2         &13          \\ \cline{3-4}
   &   & A_{11,1}D_{7,1}E_{6,1} &12     \\ \cline{3-4}
  &   & E_{6,1}^4               & 12     \\ \cline{3-4}
   &   & A_{9,1}^2D_{6,1}       &10     \\ \cline{3-4}
   &   & A_{8,1}^3                 &9   \\ \cline{3-4}
   &   & A_{7,1}^2D_{5,1}^2          &8 \\ \cline{3-4}
   &   & A_{6,1}^4         &7           \\ \cline{3-4}
   &   & A_{5,1}^4D_{4,1}    &6         \\ \cline{3-4}
   &   & D_{4,1}^6             &6      \\ \cline{3-4}
   &   & A_{4,1}^6         &5          \\ \cline{3-4}
   &   & A_{3,1}^8        &4            \\ \cline{3-4}
   &    & A_{2,1}^{12}    &3             \\ \cline{3-4}
  &    & A_{1,1}^{24}       &2          \\ \cline{3-4}
   &     & \mathbb{C}^{24}    &1                \\ \hline
\multirow{6}{*}{$1^82^8$} & \multirow{6}{*}{$\II_{16, 0}(2_{\II}^{+10})$}
     & E_{8,2}B_{8,1}   &    15       \\ \cline{3-4}
 &  & C_{10,1}B_{6,1}    &     11    \\ \cline{3-4}
  & & C_{8,1}F_{4,1}^2     &     9  \\ \cline{3-4}
&   & E_{7,2}B_{5,1}F_{4,1}   &   9  \\ \cline{3-4}
 &  & D_{9,2}A_{7,1}        &  8 \\ \cline{3-4}
 &  & D_{8,2}B_{4,1}^2      &   7   
\end{array}
\]
\end{minipage}
\begin{minipage}[t]{0.51\textwidth}
\[
\renewcommand\arraystretch{1.3}
\begin{array}{c|c|l|c}
g & \text{genus}  &  {\mathfrak{g}} &  C  \\ \hline
\multirow{11}{*}{$1^82^8$} & \multirow{11}{*}{$\II_{16, 0}(2_{\II}^{+10})$} 
   & C_{6,1}^2B_{4,1}     &  7   \\ \cline{3-4}
 & & E_{6,2}C_{5,1}A_{5,1}   & 6 \\ \cline{3-4}
 &   & A_{9,2}A_{4,1}B_{3,1}      & 5  \\ \cline{3-4}
&  & D_{6,2}C_{4,1}B_{3,1}^2    &  5 \\ \cline{3-4}
&   & C_{4,1}^4                  &  5 \\ \cline{3-4}
&   & A_{7,2}C_{3,1}^2A_{3,1}    &  4 \\ \cline{3-4}
&   & D_{5,2}^2A_{3,1}^2         &  4 \\ \cline{3-4}
&   & A_{5,2}^2B_{2,1}A_{2,1}^2  &  3\\ \cline{3-4}
&    & D_{4,2}^2B_{2,1}^4         &  3 \\ \cline{3-4}
&    & A_{3,2}^4A_{1,1}^4         &  2 \\ \cline{3-4}
&  & A_{1,2}^{16}               & 1 \\ \hline
\multirow{6}{*}{$1^63^6$} & \multirow{6}{*}{$\II_{12,0}(3^{-8})$}
  &   E_{7,3}A_{5,1}             & 6\\ \cline{3-4}
&  & D_{7,3}A_{3,1}G_{2,1}      &   4\\ \cline{3-4}
&   & E_{6,3}G_{2,1}^3           & 4  \\ \cline{3-4}
&   & A_{8,3}A_{2,1}^2           &  3 \\ \cline{3-4}
&   & A_{5,3}D_{4,3}A_{1,1}^3    &  2 \\ \cline{3-4}
&   & A_{2,3}^{6}                &  1 \\ \hline
\multirow{5}{*}{$1^42^24^4$} & \multirow{5}{*}{$\II_{10,0}(2_{2}^{+2}4_{\II}^{+6})$}
  &   C_{7,2}A_{3,1}             & 4 \\ \cline{3-4}
&   & E_{6,4}B_{2,1}A_{2,1}      &  3 \\ \cline{3-4}
&   & A_{7,4}A_{1,1}^3           &  2 \\ \cline{3-4}
&   & D_{5,4}C_{3,2}A_{1,1}^2    & 2  \\ \cline{3-4}
&   & A_{3,4}^3A_{1,2}           & 1 \\ \hline
\multirow{2}{*}{$1^45^4$} & \multirow{2}{*}{$\II_{8,0}(5^{+6})$}
    & D_{6,5}A_{1,1}^2           &  2 \\ \cline{3-4}
&   & A_{4,5}^2                  &  1 \\ \hline 
\multirow{2}{*}{$1^22^23^26^2$} & \multirow{2}{*}{$\II_{8,0}(2_{\II}^{+6}3^{-6})$}
  &   C_{5,3}G_{2,2}A_{1,1}      & 2 \\ \cline{3-4}
&   & A_{5,6}B_{2,3}A_{1,2}      & 1 \\ \hline 
\multirow{1}{*}{$1^37^3$} & \multirow{1}{*}{$\II_{6,0}(7^{-5})$}
  &     A_{6,7}                   &  1 \\ \hline 
\multirow{1}{*}{$1^22^14^18^2$} & \multirow{1}{*}{$\II_{6,0}(2_{5}^{+1}4_{1}^{+1}8_{\II}^{+4})$}
  &    D_{5,8}A_{1,2}             &  1 \\ \hline
\end{array}
\]
\end{minipage}
\end{table}
\end{footnotesize}

\clearpage

\begin{table}[ht]
\caption{H\"{o}hn's construction: the 3 conjugacy classes with distinct order and level.}
\label{table:3cycleshapes}
\renewcommand\arraystretch{1.6}
\[
\begin{array}{c|c|c|c|l|c}
g & \text{genus} &    L_{\mathfrak{g}} & n(L_{\mathfrak{g}}) & \mathfrak{g}  & C  \\ \hline
\multirow{9}{*}{$2^{12}$} & \multirow{9}{*}{$\II_{12,0}(2_{\II}^{-10}4_{\II}^{-2})$} 
  & \multirow{6}{*}{ $D_{12}(2)$} &  \multirow{6}{*}{6} & B_{12,2}                   &  23/2 \\  \cline{5-6}
& &   &  & B_{6,2}^2                  &  11/2 \\ \cline{5-6}
& &   &  & B_{4,2}^3                  &  7/2 \\ \cline{5-6}
& &    &  & B_{3,2}^4                  &  5/2  \\ \cline{5-6}
& &    & & B_{2,2}^6                  & 3/2   \\ \cline{5-6}
& &    &  & A_{1,4}^{12}               &  1/2  \\ \cline{3-6}
& &   \multirow{3}{*}{ $E_{8}(2)\oplus D_4(2)$} & \multirow{3}{*}{3} & A_{8,2}F_{4,2}         & 9/2  \\ \cline{5-6}
& &  &    & C_{4,2}A_{4,2}^2           &  5/2 \\ \cline{5-6}
& &     &  & D_{4,4}A_{2,2}^4           &  3/2 \\ \hline
\multirow{2}{*}{$2^36^3$} & \multirow{2}{*}{$\II_{6,0}(2_{\II}^{+4}4_{\II}^{-2}3^{+5})$}
   &   \multirow{2}{*}{ $D_4(6)\oplus A_2(2)$} & \multirow{2}{*}{$2$} & F_{4,6}A_{2,2}             &  1/2 \\ 
&    &   &  & D_{4,12}A_{2,6}            &  1/2 \\ \hline 
\multirow{1}{*}{$2^210^2$} & \multirow{1}{*}{$\II_{4,0}(2_{\II}^{-2}4_{\II}^{-2}5^{+4})$}
    & D_4(10)  & 1 & C_{4,10}                   &   1/2 \\ \hline 
\end{array}
\]
\end{table}

\clearpage

\begin{table}[ht]
\caption{The hyperbolizations related to $F_{24}$}
\label{tab:sym-cycle}
\renewcommand\arraystretch{1.6}
\[
\begin{array}{c|c|c|c|c} 
 g & \mathfrak{g} & \text{genus of $L_\mathfrak{g}$}  & L_\mathfrak{g} & \delta_{L_\mathfrak{g}} \\ \hline
1^{-1}5^5 & A_{4,5} & \II_{4,0}(5^{+3}) & A_4'(5) & 2\\ \hline
1^{-2}2^35^210^1 & A_{1,2}B_{3,5} &  \II_{4,0}(2_{\II}^{+2}5^{+3})& L_1 & 4  \\ \hline
1^{-2}2^34^18^2 & A_{1,2}C_{3,4} & \II_{4,0}(2_3^{-1}4_1^{+1}8_{\II}^{-2}) & A_1(2)\oplus A_3'(8) & 7/2 \\ \hline
1^{-2}2^23^24^112^1 & B_{2,3}G_{2,4} & \II_{4,0}(2_6^{+2}4_{\II}^{-2}3^{-3})& 2A_1(3)\oplus A_2(4) & 17/3 \\ \hline
1^{-3}3^9 & A_{2,3}^3 & \II_{6,0}(3^{-3}) & 3A_2 & 2 \\ \hline                      
1^{-4}2^64^4 & A_{1,2}^3A_{3,4} & \II_{6,0}(2_6^{+2}4_{\II}^{-2}) & L_2 & 2 \\ \hline  
1^{-4}2^53^46^1 & A_{1,2}^2 A_{2,3} B_{2,3} & \II_{6,0}(2_{\II}^{+2}3^{-3}) & L_3 & 8/3 \\ \hline  
1^{-8}2^{16} & A_{1,2}^8 & \II_{8,0}(2_{\II}^{+2}) & D_8 & 2\\ \hline
\end{array}
\] 
\smallskip

\begin{align*}
&L_1=\begin{psmallmatrix}
4 & 2 & 2 & 2 \\ 
2 & 6 & 1 & 1 \\ 
2 & 1 & 6 & 1 \\ 
2 & 1 & 1 & 6    
\end{psmallmatrix}&
&L_2=\begin{psmallmatrix}
2 & 0 & 1 & 1 & 1 & 0 \\ 
0 & 2 & 1 & 1 & 1 & 0 \\ 
1 & 1 & 4 & 2 & 2 & 3 \\ 
1 & 1 & 2 & 4 & 0 & 1 \\ 
1 & 1 & 2 & 0 & 4 & 1 \\ 
0 & 0 & 3 & 1 & 1 & 4    
\end{psmallmatrix}&
&L_3=\begin{psmallmatrix}
4 & 2 & 0 & 0 & -2 & 0 \\ 
2 & 4 & 0 & 0 & -1 & 0 \\ 
0 & 0 & 2 & -1 & 0 & 0 \\ 
0 & 0 & -1 & 2 & 0 & 0 \\ 
-2 & -1 & 0 & 0 & 2 & 1 \\ 
0 & 0 & 0 & 0 & 1 & 4    
\end{psmallmatrix}
\end{align*}
\end{table}

\clearpage

\begin{table}[ht]
\caption{The $64$ antisymmetric paramodular cusp forms of weight 3 and squarefree level $<300$ constructed from Theorem \ref{th:Siegel-3} (continued on the next page).} \label{tableanti3prime}
\renewcommand\arraystretch{1.3}
\[
\begin{array}{c|c|c}
N(\mathbf{a}) & \mathbf{a}=(a_1,...,a_6) & \text{Theta block}\cdot\eta^{15} \\ 
\hline 
  122 &(1,1,1,1,1,1) & 1^5 2^5 3^4 4^3 5^2 6 7 \\
 138 &(-7,1,1,1,1,3) & 1^5 2^4 3^4 4^3 5^2 6^2 7 \\
 158 &(-8,1,1,1,1,1) & 1^4 2^5 3^4 4^2 5^3 6 7 8 \\
 167 &(-8,1,1,1,2,1) & 1^4 2^5 3^3 4^3 5^2 6^2 7 8 \\
 170 &(-7,1,1,1,1,5) & 1^5 2^4 3^4 4^3 5^2 6 7 (10) \\
 173 &(-8,1,1,1,2,2) & 1^4 2^4 3^3 4^4 5^2 6^2 7 8 \\
 174 &(-8,1,1,1,1,3) & 1^4 2^4 3^4 4^2 5^3 6^2 7 8 \\
 183 &(-8,1,1,1,2,3) & 1^4 2^4 3^3 4^3 5^2 6^3 7 8 \\
 186 &(-9,2,1,1,1,1) & 1^4 2^4 3^3 4^3 5^3 6^2 7 9 \\
 197 &(-8,1,1,1,2,4) & 1^4 2^4 3^3 4^3 5^2 6^2 7 8^2 \\
 202 &(-9,1,1,1,1,1) & 1^4 2^4 3^4 4^2 5^2 6^2 7 8 9 \\
 202 &(-9,2,1,1,1,3) & 1^4 2^3 3^3 4^3 5^3 6^3 7 9 \\
 206 &(-9,1,2,1,1,1) & 1^4 2^3 3^3 4^4 5^3 6 7 8 9 \\
 206 &(-8,1,1,1,1,5) & 1^4 2^4 3^4 4^2 5^3 6 7 8 (10) \\
 213 &(-9,1,1,1,2,2) & 1^3 2^4 3^3 4^4 5^2 6^2 7 8 9 \\
 215 &(-8,1,1,1,2,5) & 1^4 2^4 3^3 4^3 5^2 6^2 7 8 (10) \\
 218 &(-9,1,1,1,1,3) & 1^4 2^3 3^4 4^2 5^2 6^3 7 8 9 \\
 218 &(-9,1,1,2,2,1) & 1^3 2^5 3^3 4^2 5^3 6 7^2 8 9 \\
 218 &(-7,1,1,1,1,7) & 1^5 2^4 3^4 4^3 5^2 6 7 (14)\\
 222 &(-9,1,1,1,3,1) & 1^4 2^4 3^3 4^2 5^2 6^2 7^2 8 9 \\
 222 &(-9,1,2,1,1,3) & 1^4 2^2 3^3 4^4 5^3 6^2 7 8 9 \\
 223 &(-9,1,1,1,2,3) & 1^3 2^4 3^3 4^3 5^2 6^3 7 8 9 \\
 230 &(-10,2,1,1,1,1) & 1^3 2^5 3^2 4^3 5^2 6^3 7 8 (10) \\
 237 &(-9,1,1,1,2,4) & 1^3 2^4 3^3 4^3 5^2 6^2 7 8^2 9 \\
 237 &(-8,1,1,1,2,6) & 1^4 2^4 3^3 4^3 5^2 6^2 7 8 (12)\\
 238 &(-9,1,1,1,3,3) & 1^4 2^3 3^3 4^2 5^2 6^3 7^2 8 9 \\
 239 &(-10,2,1,1,2,1) & 1^3 2^4 3^3 4^3 5^2 6^2 7^2 8 (10) \\
 246 &(-10,2,1,1,1,3) & 1^3 2^4 3^2 4^3 5^2 6^4 7 8 (10) \\
 254 &(-10,1,1,1,1,1) & 1^4 2^4 3^3 4^2 5^2 6 7^2 8 9 (10) \\
 254 &(-9,1,2,1,1,5) & 1^4 2^2 3^3 4^4 5^3 6 7 8 9 (10) \\
 254 &(-8,1,1,1,1,7) & 1^4 2^4 3^4 4^2 5^3 6 7 8 (14)\\
 255 &(-10,1,1,1,2,1) & 1^3 2^4 3^3 4^3 5^2 6^2 7 8 9 (10) \\
 \hline
 \end{array}
\]
\end{table}

\clearpage

\begin{table}[ht]
\caption{(continued).}
\renewcommand\arraystretch{1.3}
\[
\begin{array}{c|c|c}
N(\mathbf{a}) & \mathbf{a}=(a_1,...,a_6) & \text{Theta block}\cdot\eta^{15} \\ 
\hline 
 255 &(-10,2,1,1,2,3) & 1^3 2^3 3^3 4^3 5^2 6^3 7^2 8 (10) \\
 262 &(-10,1,1,2,2,1) & 1^2 2^5 3^3 4^3 5 6^3 7 8 9 (10) \\
 263 &(-10,1,2,1,2,1) & 1^3 2^4 3^2 4^3 5^3 6^2 7 8 9 (10) \\
 263 &(-8,1,1,1,2,7) & 1^4 2^4 3^3 4^3 5^2 6^2 7 8 (14)\\
 266 &(-10,1,1,1,3,1) & 1^3 2^4 3^3 4^2 5^3 6 7^2 8 9 (10) \\
 266 &(-10,1,2,1,1,3) & 1^3 2^2 3^4 4^3 5^2 6^3 7^2 9 (10) \\
 266 &(-8,-3,4,1,1,1) & 1^4 2^3 3^2 4^2 5^3 6^3 7^2 8 (11) \\
 269 &(-10,1,2,1,2,2) & 1^3 2^3 3^2 4^4 5^3 6^2 7 8 9 (10) \\
 269 &(-10,2,1,1,2,4) & 1^3 2^3 3^3 4^3 5^2 6^2 7^2 8^2 (10) \\
 271 &(-10,1,1,1,2,3) & 1^3 2^3 3^3 4^3 5^2 6^3 7 8 9 (10) \\
 277 &(-9,1,1,1,2,6) & 1^3 2^4 3^3 4^3 5^2 6^2 7 8 9 (12)\\
 278 &(-10,1,1,2,2,3) & 1^2 2^4 3^3 4^3 5 6^4 7 8 9 (10) \\
 278 &(-10,2,1,1,1,5) & 1^3 2^4 3^2 4^3 5^2 6^3 7 8 (10)^2 \\
 282 &(-10,1,1,1,3,3) & 1^3 2^3 3^3 4^2 5^3 6^2 7^2 8 9 (10) \\
 282 &(-9,-2,3,1,1,1) & 1^3 2^4 3^3 4^2 5^2 6^2 7^2 8 9 (11) \\
 282 &(-9,2,1,1,1,7) & 1^4 2^3 3^3 4^3 5^3 6^2 7 9 (14)\\
 282 &(-8,-3,4,1,1,3) & 1^4 2^2 3^2 4^2 5^3 6^4 7^2 8 (11) \\
 282 &(-7,1,1,1,1,9) & 1^5 2^4 3^4 4^3 5^2 6 7 (18)\\
 285 &(-10,1,1,2,3,2) & 1^3 2^3 3^3 4^3 5^2 6^2 7 8^2 9 (10) \\
 286 &(-9,-2,4,1,1,1) & 1^3 2^3 3^3 4^2 5^3 6^3 7 8 9 (11) \\
 287 &(-10,1,1,1,4,1) & 1^4 2^4 3^2 4^2 5 6^2 7^2 8^2 9 (10) \\
 287 &(-10,2,1,1,2,5) & 1^3 2^3 3^3 4^3 5^2 6^2 7^2 8 (10)^2 \\
 287 &(-9,-2,3,1,2,1) & 1^2 2^5 3^2 4^3 5^2 6^2 7^2 8 9 (11) \\
 293 &(-10,1,1,1,4,2) & 1^4 2^3 3^2 4^3 5 6^2 7^2 8^2 9 (10) \\
 293 &(-10,1,2,1,2,4) & 1^3 2^3 3^2 4^3 5^3 6^2 7 8^2 9 (10) \\
 293 &(-9,-2,3,1,2,2) & 1^2 2^4 3^2 4^4 5^2 6^2 7^2 8 9 (11) \\
 293 &(-8,1,1,1,2,8) & 1^4 2^4 3^3 4^3 5^2 6^2 7 8 (16)\\
 295 &(-10,1,1,2,3,3) & 1^3 2^3 3^3 4^2 5^2 6^3 7 8^2 9 (10) \\
 298 &(-10,1,2,1,1,5) & 1^3 2^2 3^4 4^3 5^2 6^2 7^2 9 (10)^2 \\
 298 &(-9,-2,3,1,1,3) & 1^3 2^3 3^3 4^2 5^2 6^3 7^2 8 9 (11) \\
 298 &(-9,-2,3,2,2,1) & 1^2 2^4 3^4 4 5^3 6^2 7 8^2 9 (11) \\
 298 &(-9,1,1,1,1,7) & 1^4 2^3 3^4 4^2 5^2 6^2 7 8 9 (14)\\
\hline
\end{array} 
\]
\end{table}

%% file: main.bbl
\providecommand{\bysame}{\leavevmode\hbox to3em{\hrulefill}\thinspace}
\providecommand{\MR}{\relax\ifhmode\unskip\space\fi MR }
\providecommand{\MRhref}[2]{%
  \href{http://www.ams.org/mathscinet-getitem?mr=#1}{#2}
}
\providecommand{\href}[2]{#2}
\begin{thebibliography}{100}

\bibitem{Bar03}
Alexander~Graham Barnard, \emph{The singular theta correspondence, {L}orentzian lattices and {B}orcherds-{K}ac-{M}oody algebras}, ProQuest LLC, Ann Arbor, MI, 2003, Thesis (Ph.D.)--University of California, Berkeley.

\bibitem{BLS22}
Koichi Betsumiya, Ching~Hung Lam, and Hiroki Shimakura, \emph{Automorphism groups and uniqueness of holomorphic vertex operator algebras of central charge 24}, Comm. Math. Phys. \textbf{399} (2023), no.~3, 1773--1810.

\bibitem{Bor88}
Richard Borcherds, \emph{Generalized {K}ac-{M}oody algebras}, J. Algebra \textbf{115} (1988), no.~2, 501--512.

\bibitem{Bor86}
Richard~E. Borcherds, \emph{Vertex algebras, {K}ac-{M}oody algebras, and the {M}onster}, Proc. Nat. Acad. Sci. U.S.A. \textbf{83} (1986), no.~10, 3068--3071.

\bibitem{Bor90}
\bysame, \emph{The monster {L}ie algebra}, Adv. Math. \textbf{83} (1990), no.~1, 30--47.

\bibitem{Bor92}
\bysame, \emph{Monstrous moonshine and monstrous {L}ie superalgebras}, Invent. Math. \textbf{109} (1992), no.~2, 405--444.

\bibitem{Bor95}
\bysame, \emph{Automorphic forms on {${\rm O}_{s+2,2}({\bf R})$} and infinite products}, Invent. Math. \textbf{120} (1995), no.~1, 161--213.

\bibitem{Bor96}
\bysame, \emph{The moduli space of {E}nriques surfaces and the fake {M}onster {L}ie superalgebra}, Topology \textbf{35} (1996), no.~3, 699--710.

\bibitem{Bor98}
\bysame, \emph{Automorphic forms with singularities on {G}rassmannians}, Invent. Math. \textbf{132} (1998), no.~3, 491--562.

\bibitem{Bor99}
\bysame, \emph{The {G}ross-{K}ohnen-{Z}agier theorem in higher dimensions}, Duke Math. J. \textbf{97} (1999), no.~2, 219--233.

\bibitem{Bor00}
\bysame, \emph{Reflection groups of {L}orentzian lattices}, Duke Math. J. \textbf{104} (2000), no.~2, 319--366.

\bibitem{BKP98}
Richard~E. Borcherds, Ludmil Katzarkov, Tony Pantev, and N.~I. Shepherd-Barron, \emph{Families of {$K3$} surfaces}, J. Algebraic Geom. \textbf{7} (1998), no.~1, 183--193.

\bibitem{Bou82}
Nicolas Bourbaki, \emph{\'{E}l\'{e}ments de math\'{e}matique}, Masson, Paris, 1981, Groupes et alg\`ebres de Lie. Chapitres 4, 5 et 6. [Lie groups and Lie algebras. Chapters 4, 5 and 6].

\bibitem{Bru02}
Jan~H. Bruinier, \emph{Borcherds products on {O}(2, {$l$}) and {C}hern classes of {H}eegner divisors}, Lecture Notes in Mathematics, vol. 1780, Springer-Verlag, Berlin, 2002.

\bibitem{Bru14}
Jan~Hendrik Bruinier, \emph{On the converse theorem for {B}orcherds products}, J. Algebra \textbf{397} (2014), 315--342.

\bibitem{Normaliz}
W.~Bruns, B.~Ichim, C.~S\"oger, and U.~von~der Ohe, \emph{Normaliz. algorithms for rational cones and affine monoids}, Available at \url{https://www.normaliz.uni-osnabrueck.de}.

\bibitem{Cappelli}
A.~Cappelli, C.~Itzykson, and J.-B. Zuber, \emph{The {${\rm A}$}-{${\rm D}$}-{${\rm E}$} classification of minimal and {$A^{(1)}_1$} conformal invariant theories}, Comm. Math. Phys. \textbf{113} (1987), no.~1, 1--26.

\bibitem{CG13}
F.~Cl\'{e}ry and V.~Gritsenko, \emph{Modular forms of orthogonal type and {J}acobi theta-series}, Abh. Math. Semin. Univ. Hambg. \textbf{83} (2013), no.~2, 187--217.

\bibitem{CS99}
J.~H. Conway and N.~J.~A. Sloane, \emph{Sphere packings, lattices and groups}, third ed., Grundlehren der mathematischen Wissenschaften [Fundamental Principles of Mathematical Sciences], vol. 290, Springer-Verlag, New York, 1999, With additional contributions by E. Bannai, R. E. Borcherds, J. Leech, S. P. Norton, A. M. Odlyzko, R. A. Parker, L. Queen and B. B. Venkov.

\bibitem{CDR18}
Thomas Creutzig, John F.~R. Duncan, and Wolfgang Riedler, \emph{Self-dual vertex operator superalgebras and superconformal field theory}, J. Phys. A \textbf{51} (2018), no.~3, 034001, 29.

\bibitem{CKM22}
Thomas Creutzig, Shashank Kanade, and Robert McRae, \emph{Gluing vertex algebras}, Adv. Math. \textbf{396} (2022), Paper No. 108174, 72.

\bibitem{CKS07}
Thomas Creutzig, Alexander Klauer, and Nils~R. Scheithauer, \emph{Natural constructions of some generalized {K}ac-{M}oody algebras as bosonic strings}, Commun. Number Theory Phys. \textbf{1} (2007), no.~3, 453--477.

\bibitem{CFT}
Philippe Di~Francesco, Pierre Mathieu, and David S\'{e}n\'{e}chal, \emph{Conformal field theory}, Graduate Texts in Contemporary Physics, Springer-Verlag, New York, 1997.

\bibitem{Dit19}
Moritz Dittmann, \emph{Reflective automorphic forms on lattices of squarefree level}, Trans. Amer. Math. Soc. \textbf{372} (2019), no.~2, 1333--1362.

\bibitem{DHS15}
Moritz Dittmann, Heike Hagemeier, and Markus Schwagenscheidt, \emph{Automorphic products of singular weight for simple lattices}, Math. Z. \textbf{279} (2015), no.~1-2, 585--603.

\bibitem{DW21}
Moritz Dittmann and Haowu Wang, \emph{Theta blocks related to root systems}, Math. Ann. \textbf{384} (2022), no.~3-4, 1157--1180.

\bibitem{dong2007uniqueness}
Chongying Dong, Robert~L Griess, and Ching~Hung Lam, \emph{Uniqueness results for the moonshine vertex operator algebra}, Amer. J. Math. \textbf{129} (2007), no.~2, 583--609.

\bibitem{DM04}
Chongying Dong and Geoffrey Mason, \emph{Holomorphic vertex operator algebras of small central charge}, Pacific J. Math. \textbf{213} (2004), no.~2, 253--266.

\bibitem{DM04b}
\bysame, \emph{Rational vertex operator algebras and the effective central charge}, Int. Math. Res. Not. (2004), no.~56, 2989--3008.

\bibitem{DM06}
\bysame, \emph{Integrability of {$C_2$}-cofinite vertex operator algebras}, Int. Math. Res. Not. (2006), Art. ID 80468, 15.

\bibitem{DS22}
Thomas~Maximilian Driscoll-Spittler, \emph{Reflective modular forms and vertex operator algebras}, Ph.D. Thesis, Technische Universität Darmstadt, 2022.

\bibitem{Dun07}
John~F. Duncan, \emph{Super-moonshine for {C}onway's largest sporadic group}, Duke Math. J. \textbf{139} (2007), no.~2, 255--315.

\bibitem{EZ85}
Martin Eichler and Don Zagier, \emph{The theory of {J}acobi forms}, Progress in Mathematics, vol.~55, Birkh\"{a}user Boston, Inc., Boston, MA, 1985.

\bibitem{FF83}
Alex~J. Feingold and Igor~B. Frenkel, \emph{A hyperbolic {K}ac-{M}oody algebra and the theory of {S}iegel modular forms of genus {$2$}}, Math. Ann. \textbf{263} (1983), no.~1, 87--144.

\bibitem{Fre83}
E.~Freitag, \emph{Siegelsche {M}odulfunktionen}, Grundlehren der mathematischen Wissenschaften [Fundamental Principles of Mathematical Sciences], vol. 254, Springer-Verlag, Berlin, 1983.

\bibitem{FLM88}
Igor Frenkel, James Lepowsky, and Arne Meurman, \emph{Vertex operator algebras and the {M}onster}, Pure and Applied Mathematics, vol. 134, Academic Press, Inc., Boston, MA, 1988.

\bibitem{Gannon}
Terry Gannon, \emph{The classification of affine {${\rm SU}(3)$} modular invariant partition functions}, Comm. Math. Phys. \textbf{161} (1994), no.~2, 233--263.

\bibitem{Gannon:1994sp}
\bysame, \emph{Towards a classification of {${\rm su}(2)\oplus\cdots\oplus{\rm su}(2)$} modular invariant partition functions}, J. Math. Phys. \textbf{36} (1995), no.~2, 675--706.

\bibitem{gannon2023exotic}
\bysame, \emph{Exotic quantum subgroups and extensions of affine lie algebra {VOA}s--{p}art {I}}, Preprint, 2023.

\bibitem{GH14}
V.~Gritsenko and K.~Hulek, \emph{Uniruledness of orthogonal modular varieties}, J. Algebraic Geom. \textbf{23} (2014), no.~4, 711--725.

\bibitem{GHS09}
V.~Gritsenko, K.~Hulek, and G.~K. Sankaran, \emph{Abelianisation of orthogonal groups and the fundamental group of modular varieties}, J. Algebra \textbf{322} (2009), no.~2, 463--478.

\bibitem{Gri18}
V.~A. Gritsenko, \emph{Reflective modular forms and their applications}, Uspekhi Mat. Nauk \textbf{73} (2018), no.~5(443), 53--122.

\bibitem{GHS07}
V.~A. Gritsenko, K.~Hulek, and G.~K. Sankaran, \emph{The {K}odaira dimension of the moduli of {$K3$} surfaces}, Invent. Math. \textbf{169} (2007), no.~3, 519--567.

\bibitem{GN96a}
V.~A. Gritsenko and V.~V. Nikulin, \emph{Igusa modular forms and ``the simplest'' {L}orentzian {K}ac-{M}oody algebras}, Mat. Sb. \textbf{187} (1996), no.~11, 27--66.

\bibitem{GN02}
\bysame, \emph{On the classification of {L}orentzian {K}ac-{M}oody algebras}, Uspekhi Mat. Nauk \textbf{57} (2002), no.~5(347), 79--138.

\bibitem{GH98}
Valeri Gritsenko and Klaus Hulek, \emph{Minimal {S}iegel modular threefolds}, Math. Proc. Cambridge Philos. Soc. \textbf{123} (1998), no.~3, 461--485.

\bibitem{GN96b}
Valeri~A. Gritsenko and Viacheslav~V. Nikulin, \emph{Siegel automorphic form corrections of some {L}orentzian {K}ac-{M}oody {L}ie algebras}, Amer. J. Math. \textbf{119} (1997), no.~1, 181--224.

\bibitem{GN98a}
\bysame, \emph{Automorphic forms and {L}orentzian {K}ac-{M}oody algebras. {I}}, Internat. J. Math. \textbf{9} (1998), no.~2, 153--199.

\bibitem{GN98}
\bysame, \emph{Automorphic forms and {L}orentzian {K}ac-{M}oody algebras. {II}}, Internat. J. Math. \textbf{9} (1998), no.~2, 201--275.

\bibitem{Gri10}
Valery Gritsenko, \emph{Reflective modular forms in algebraic geometry}, Preprint, 2010.

\bibitem{Gri12}
\bysame, \emph{24 faces of the {B}orcherds modular form ${\Phi_{12}}$}, Preprint, 2012.

\bibitem{GN18}
Valery Gritsenko and Viacheslav~V. Nikulin, \emph{Lorentzian {K}ac-{M}oody algebras with {W}eyl groups of 2-reflections}, Proc. Lond. Math. Soc. (3) \textbf{116} (2018), no.~3, 485--533.

\bibitem{GPY15}
Valery Gritsenko, Cris Poor, and David~S. Yuen, \emph{Borcherds products everywhere}, J. Number Theory \textbf{148} (2015), 164--195.

\bibitem{GSZ19}
Valery Gritsenko, Nils-Peter Skoruppa, and Don Zagier, \emph{Theta blocks.}, J. Eur. Math. Soc. publish online, 2024.

\bibitem{GW19}
Valery Gritsenko and Haowu Wang, \emph{Weight 3 antisymmetric paramodular forms}, Mat. Sb. \textbf{210} (2019), no.~12, 43--66.

\bibitem{GW20}
\bysame, \emph{Theta block conjecture for paramodular forms of weight 2}, Proc. Amer. Math. Soc. \textbf{148} (2020), no.~5, 1863--1878.

\bibitem{F24}
Sarah~M. Harrison, Natalie~M. Paquette, Daniel Persson, and Roberto Volpato, \emph{Fun with {$ F_{24}$}}, J. High Energy Phys. (2021), no.~2, Paper No. 039, 47.

\bibitem{HPV19}
Sarah~M. Harrison, Natalie~M. Paquette, and Roberto Volpato, \emph{A {B}orcherds-{K}ac-{M}oody superalgebra with {C}onway symmetry}, Comm. Math. Phys. \textbf{370} (2019), no.~2, 539--590.

\bibitem{HV96}
Jeffrey~A. Harvey and Gregory~W. Moore, \emph{{Algebras, BPS states, and strings}}, Nucl. Phys. B \textbf{463} (1996), 315--368.

\bibitem{HV98}
\bysame, \emph{{On the algebras of BPS states}}, Commun. Math. Phys. \textbf{197} (1998), 489--519.

\bibitem{HM20}
Gerald H\"{o}hn and Sven M\"{o}ller, \emph{Systematic orbifold constructions of {S}chellekens' vertex operator algebras from {N}iemeier lattices}, J. Lond. Math. Soc. (2) \textbf{106} (2022), no.~4, 3162--3207.

\bibitem{HS03}
Gerald H\"{o}hn and Nils~R. Scheithauer, \emph{A natural construction of {B}orcherds' fake {B}aby {M}onster {L}ie algebra}, Amer. J. Math. \textbf{125} (2003), no.~3, 655--667.

\bibitem{HS14}
\bysame, \emph{A generalized {K}ac-{M}oody algebra of rank 14}, J. Algebra \textbf{404} (2014), 222--239.

\bibitem{Hoh17}
Gerald Höhn, \emph{On the genus of the moonshine module}, Preprint, 2017.

\bibitem{Ibu22}
Tomoyoshi Ibukiyama, \emph{Dimensions of paramodular forms and compact twist modular forms with involutions}, Preprint, 2022.

\bibitem{Igu62}
Jun-ichi Igusa, \emph{On {S}iegel modular forms of genus two}, Amer. J. Math. \textbf{84} (1962), 175--200.

\bibitem{Igu64}
\bysame, \emph{On {S}iegel modular forms genus two. {II}}, Amer. J. Math. \textbf{86} (1964), 392--412.

\bibitem{Kac90}
Victor~G. Kac, \emph{Infinite-dimensional {L}ie algebras}, third ed., Cambridge University Press, Cambridge, 1990.

\bibitem{Kato}
Akishi Kato, \emph{Classification of modular invariant partition functions in two dimensions}, Modern Phys. Lett. A \textbf{2} (1987), no.~8, 585--600.

\bibitem{KLL18}
Kazuya Kawasetsu, Ching~Hung Lam, and Xingjun Lin, \emph{{$\mathbb{Z}_2$}-orbifold construction associated with {$(-1)$}-isometry and uniqueness of holomorphic vertex operator algebras of central charge 24}, Proc. Amer. Math. Soc. \textbf{146} (2018), no.~5, 1937--1950.

\bibitem{KM12}
Matthew Krauel and Geoffrey Mason, \emph{Vertex operator algebras and weak {J}acobi forms}, Internat. J. Math. \textbf{23} (2012), no.~6, 1250024, 10.

\bibitem{Lam11}
Ching~Hung Lam, \emph{On the constructions of holomorphic vertex operator algebras of central charge 24}, Comm. Math. Phys. \textbf{305} (2011), no.~1, 153--198.

\bibitem{Lam19}
\bysame, \emph{Cyclic orbifolds of lattice vertex operator algebras having group-like fusions}, Lett. Math. Phys. \textbf{110} (2019), no.~5, 1081--1112.

\bibitem{LL20}
Ching~Hung Lam and Xingjun Lin, \emph{A holomorphic vertex operator algebra of central charge 24 with the weight one {L}ie algebra {$F_{4,6}A_{2,2}$}}, J. Pure Appl. Algebra \textbf{224} (2020), no.~3, 1241--1279.

\bibitem{LS12}
Ching~Hung Lam and Hiroki Shimakura, \emph{Quadratic spaces and holomorphic framed vertex operator algebras of central charge 24}, Proc. Lond. Math. Soc. (3) \textbf{104} (2012), no.~3, 540--576.

\bibitem{LS15}
\bysame, \emph{Classification of holomorphic framed vertex operator algebras of central charge 24}, Amer. J. Math. \textbf{137} (2015), no.~1, 111--137.

\bibitem{LS16b}
\bysame, \emph{A holomorphic vertex operator algebra of central charge 24 whose weight one {L}ie algebra has type {$A_{6,7}$}}, Lett. Math. Phys. \textbf{106} (2016), no.~11, 1575--1585.

\bibitem{LS16a}
\bysame, \emph{Orbifold construction of holomorphic vertex operator algebras associated to inner automorphisms}, Comm. Math. Phys. \textbf{342} (2016), no.~3, 803--841.

\bibitem{LS19}
\bysame, \emph{Reverse orbifold construction and uniqueness of holomorphic vertex operator algebras}, Trans. Amer. Math. Soc. \textbf{372} (2019), no.~10, 7001--7024.

\bibitem{LS22}
\bysame, \emph{Inertia groups and uniqueness of holomorphic vertex operator algebras}, Transform. Groups \textbf{25} (2020), no.~4, 1223--1268.

\bibitem{LS20}
\bysame, \emph{On orbifold constructions associated with the {L}eech lattice vertex operator algebra}, Math. Proc. Cambridge Philos. Soc. \textbf{168} (2020), no.~2, 261--285.

\bibitem{Lin17}
Xingjun Lin, \emph{Mirror extensions of rational vertex operator algebras}, Trans. Amer. Math. Soc. \textbf{369} (2017), no.~6, 3821--3840.

\bibitem{Ma17}
Shouhei Ma, \emph{Finiteness of 2-reflective lattices of signature {$(2,n)$}}, Amer. J. Math. \textbf{139} (2017), no.~2, 513--524.

\bibitem{Ma18}
\bysame, \emph{On the {K}odaira dimension of orthogonal modular varieties}, Invent. Math. \textbf{212} (2018), no.~3, 859--911.

\bibitem{Miy00}
Masahiko Miyamoto, \emph{A modular invariance on the theta functions defined on vertex operator algebras}, Duke Math. J. \textbf{101} (2000), no.~2, 221--236.

\bibitem{Mol21}
Sven M\"{o}ller, \emph{Natural construction of ten {B}orcherds-{K}ac-{M}oody algebras associated with elements in {$M_{23}$}}, Comm. Math. Phys. \textbf{383} (2021), no.~1, 35--70.

\bibitem{SM19}
Sven M\"{o}ller and Nils~R. Scheithauer, \emph{Dimension formulae and generalised deep holes of the {L}eech lattice vertex operator algebra}, Ann. of Math. (2) \textbf{197} (2023), no.~1, 221--288.

\bibitem{SM21}
Sven M\"oller and Nils~R. Scheithauer, \emph{A geometric classification of the holomorphic vertex operator algebras of central charge 24}, Algebra Number Theory \textbf{18} (2024), no.~10, 1891--1922.

\bibitem{Moore:1988ss}
Gregory Moore and Nathan Seiberg, \emph{Naturality in conformal field theory}, Nuclear Phys. B \textbf{313} (1989), no.~1, 16--40.

\bibitem{Nie02}
Peter Niemann, \emph{Some generalized {K}ac-{M}oody algebras with known root multiplicities}, Mem. Amer. Math. Soc. \textbf{157} (2002), no.~746, x+119.

\bibitem{OS19}
Sebastian Opitz and Markus Schwagenscheidt, \emph{Holomorphic {B}orcherds products of singular weight for simple lattices of arbitrary level}, Proc. Amer. Math. Soc. \textbf{147} (2019), no.~11, 4639--4653.

\bibitem{Ray06}
Urmie Ray, \emph{Automorphic forms and {L}ie superalgebras}, Algebra and Applications, vol.~5, Springer, Dordrecht, 2006.

\bibitem{SS16}
Daisuke Sagaki and Hiroki Shimakura, \emph{Application of a {$\mathbb{Z}_3$}-orbifold construction to the lattice vertex operator algebras associated to {N}iemeier lattices}, Trans. Amer. Math. Soc. \textbf{368} (2016), no.~3, 1621--1646.

\bibitem{Sch00}
Nils~R. Scheithauer, \emph{The fake monster superalgebra}, Adv. Math. \textbf{151} (2000), no.~2, 226--269.

\bibitem{Sch04}
\bysame, \emph{Generalized {K}ac-{M}oody algebras, automorphic forms and {C}onway's group. {I}}, Adv. Math. \textbf{183} (2004), no.~2, 240--270.

\bibitem{Sch06}
\bysame, \emph{On the classification of automorphic products and generalized {K}ac-{M}oody algebras}, Invent. Math. \textbf{164} (2006), no.~3, 641--678.

\bibitem{Sch09}
\bysame, \emph{The {W}eil representation of {${\rm SL}_2(\mathbb Z)$} and some applications}, Int. Math. Res. Not. IMRN (2009), no.~8, 1488--1545.

\bibitem{Sch15}
\bysame, \emph{Some constructions of modular forms for the {W}eil representation of {${\rm SL}_2(\mathbb{Z})$}}, Nagoya Math. J. \textbf{220} (2015), 1--43.

\bibitem{Sch17}
\bysame, \emph{Automorphic products of singular weight}, Compos. Math. \textbf{153} (2017), no.~9, 1855--1892.

\bibitem{Sch93}
A.~N. Schellekens, \emph{Meromorphic {$c=24$} conformal field theories}, Comm. Math. Phys. \textbf{153} (1993), no.~1, 159--185.

\bibitem{Schellekens:1989uf}
A.~N. Schellekens and S.~Yankielowicz, \emph{Field identification fixed points in the coset construction}, Nuclear Phys. B \textbf{334} (1990), no.~1, 67--102.

\bibitem{ELM21}
Jethro van Ekeren, Ching~Hung Lam, Sven M\"{o}ller, and Hiroki Shimakura, \emph{Schellekens' list and the very strange formula}, Adv. Math. \textbf{380} (2021), Paper No. 107567, 33.

\bibitem{EMS20}
Jethro van Ekeren, Sven M\"{o}ller, and Nils~R. Scheithauer, \emph{Construction and classification of holomorphic vertex operator algebras}, J. Reine Angew. Math. \textbf{759} (2020), 61--99.

\bibitem{EMS20b}
\bysame, \emph{Dimension formulae in genus zero and uniqueness of vertex operator algebras}, Int. Math. Res. Not. IMRN (2020), no.~7, 2145--2204.

\bibitem{Verstegen:1990my}
D.~Verstegen, \emph{New exceptional modular invariant partition functions for simple {K}ac-{M}oody algebras}, Nuclear Phys. B \textbf{346} (1990), no.~2-3, 349--386.

\bibitem{Wan21a}
Haowu Wang, \emph{The classification of free algebras of orthogonal modular forms}, Compos. Math. \textbf{157} (2021), no.~9, 2026--2045.

\bibitem{Wan21b}
\bysame, \emph{Reflective modular forms: a {J}acobi forms approach}, Int. Math. Res. Not. IMRN (2021), no.~3, 2081--2107.

\bibitem{Wan21}
\bysame, \emph{Weyl invariant {J}acobi forms: a new approach}, Adv. Math. \textbf{384} (2021), Paper No. 107752, 13.

\bibitem{Wan22}
\bysame, \emph{Reflective modular forms on lattices of prime level}, Trans. Amer. Math. Soc. \textbf{375} (2022), no.~5, 3451--3468.

\bibitem{Wan23b}
\bysame, \emph{2-reflective lattices of signature {$(n, 2)$} with {$n\geq 8$}}, Int. Math. Res. Not. IMRN (2023), no.~20, 17953--17971.

\bibitem{Wan23a}
\bysame, \emph{On the classification of reflective modular forms.}, to appear in Eur. J. Math., 2023.

\bibitem{Wan19}
\bysame, \emph{The classification of 2-reflective modular forms}, J. Eur. Math. Soc. (JEMS) \textbf{26} (2024), no.~1, 111--151.

\bibitem{Wan23-star}
\bysame, \emph{There are no extremal eutactic stars other than root systems}, Bull. Lond. Math. Soc. \textbf{57} (2025), no.~1, 16--22.

\bibitem{WW23}
Haowu Wang and Brandon Williams, \emph{On the non-existence of singular borcherds products.}, to appear in Amer. J. Math., 2023.

\bibitem{WW20c}
\bysame, \emph{Simple lattices and free algebras of modular forms}, Adv. Math. \textbf{413} (2023), 108835, 51.

\bibitem{WW21}
\bysame, \emph{Modular forms with poles on hyperplane arrangements}, Algebr. Geom. \textbf{11} (2024), no.~4, 506--568.

\bibitem{WW22}
\bysame, \emph{The fake monster algebra and singular {B}orcherds products}, Adv. Math. \textbf{461} (2025), 110083, 52.

\bibitem{Wil2018}
Brandon Williams, \emph{Poincar\'{e} square series for the {W}eil representation}, Ramanujan J. \textbf{47} (2018), no.~3, 605--650.

\bibitem{Zhu96}
Yongchang Zhu, \emph{Modular invariance of characters of vertex operator algebras}, J. Amer. Math. Soc. \textbf{9} (1996), no.~1, 237--302.

\end{thebibliography}
